\documentclass{article}
\usepackage{amsmath}
\usepackage{amsthm}
\usepackage{amssymb}
\usepackage{amsfonts}
\usepackage{graphicx}
\usepackage{mathtools}
\usepackage{verbatim}
\usepackage{pgf}
\usepackage{stmaryrd}
\usepackage{amscd}
\usepackage[mathscr]{euscript}
\usepackage{graphicx}
\usepackage{tikz-cd}
\usepackage{tikz}
\usepackage{scalerel}
\usetikzlibrary{decorations.pathmorphing}
\usepackage{xfrac}
\usepackage[utf8]{inputenc}
\usepackage[verbose]{wrapfig}
\usepackage{geometry} 
\usepackage{fancyhdr}
\usepackage[bf,small]{titlesec}
\usepackage{multirow}
\usepackage{enumitem}
\usepackage{cancel} 
\DeclarePairedDelimiter{\ceil}{\lceil}{\rceil}
\DeclarePairedDelimiter{\floor}{\lfloor}{\rfloor}
\usepackage{bbm}
\usepackage{todonotes}
\usepackage{tocloft}
\usepackage{microtype}
\usepackage{mathpazo}
\usepackage[english]{babel}
\usepackage{csquotes}
\usepackage[numbib]{tocbibind}
\usepackage[style=alphabetic,
			backend=biber,
			url=false,
			isbn=false,
			doi=false,
			maxnames=30]{biblatex}
\usepackage[hidelinks]{hyperref}

\setitemize{noitemsep}

\begingroup
		\makeatletter
		\@for\theoremstyle:=definition,remark,plain\do{%
				\expandafter\g@addto@macro\csname th@\theoremstyle\endcsname{%
						\addtolength\thm@preskip\parskip
						}%
				}
\endgroup

\newcommand{\R}{{\mathbb{R}}}
\newcommand{\C}{{\mathbb{C}}}
\newcommand{\Z}{{\mathbb{Z}}}
\newcommand{\Q}{{\mathbb{Q}}}
\newcommand{\T}{{\mathbb{T}}}
\newcommand{\nat}{{\mathbb{N}}}

\renewcommand{\L}{{\mathbb{L}}}
\newcommand{\sphere}{{\mathbb{S}}}
\newcommand{\E}{{\mathbb{E}}}
\newcommand{\F}{{\mathbb{F}}}
\newcommand{\D}{{\mathbb{D}}}

\newcommand{\cat}{{\mathcal{C}}}

\newcommand{\Spc}{\mathrm{Spc}}

\newcommand{\Alg}{\mathrm{Alg}}
\newcommand{\coAlg}{\mathrm{coAlg}}
\newcommand{\biAlg}{\mathrm{biAlg}}
\newcommand{\Op}{\mathrm{Op}}
\newcommand{\Assoc}{\mathrm{Assoc}}
\newcommand{\Com}{\mathrm{Com}}

\newcommand{\Span}{\mathrm{Span}}

\newcommand{\id}{\mathrm{id}}
\newcommand{\ev}{\mathrm{ev}}
\newcommand{\op}{\mathrm{op}}
\newcommand{\vop}{\mathrm{vop}}

\newcommand{\loops}{\Omega}

\newcommand{\cofib}{\mathrm{cofib}}

\newcommand{\surj}{\twoheadrightarrow}

\newcommand{\Spectra}{\mathrm{Sp}}

\newcommand{\DAlg}{\mathrm{DAlg}}

\newcommand{\Fun}{\mathrm{Fun}}
\newcommand{\Fin}{\mathrm{Fin}}

\newcommand{\Cat}{\mathscr{C}\mathrm{at}}
\renewcommand{\Pr}{\mathrm{Pr}}

\newcommand{\End}{\mathrm{End}}

\newcommand{\Proj}{\mathrm{Proj}}

\newcommand{\Hom}{\text{Hom}}
\newcommand{\Map}{\mathrm{Map}}

\DeclareMathOperator*{\fiberproduct}{\times}

\newcommand{\pShv}{{\mathrm{pShv}}}

\newcommand{\Pic}{\mathrm{Pic}}

\newcommand{\Sym}{\mathrm{Sym}}
\newcommand{\LSym}{\mathrm{LSym}}
\newcommand{\CSym}{\mathrm{CSym}}

\newcommand{\Mack}{\mathrm{Mack}}
\newcommand{\Ar}{\mathrm{Ar}}
\newcommand{\TwAr}{\mathrm{TwAr}}

\newcommand{\THH}{\mathrm{THH}}
\newcommand{\THR}{\mathrm{THR}}
\newcommand{\TC}{\mathrm{TC}}
\newcommand{\HH}{\mathrm{HH}}
\newcommand{\HC}{\mathrm{HC}}
\newcommand{\HCR}{\mathrm{HC}\mathbbm{R}}
\newcommand{\HR}{\mathrm{HR}}
\newcommand{\HP}{\mathrm{HP}}
\newcommand{\HPR}{\mathrm{HP}\mathbbm{R}}
\newcommand{\HD}{\mathrm{HD}} 

\newcommand{\Fil}{\mathrm{Fil}}
\newcommand{\fil}{\mathrm{fil}}
\newcommand{\Gr}{\mathrm{Gr}}
\newcommand{\gr}{\mathrm{gr}}
\newcommand{\DG}{\mathrm{DG}}
\newcommand{\fib}{\mathrm{fib}}
\newcommand{\spl}{\mathrm{spl}}
\newcommand{\und}{\mathrm{und}}

\newcommand{\ins}{\mathrm{ins}}
\newcommand{\Nm}{\mathrm{Nm}}
\newcommand{\HKR}{\mathrm{HK}\mathbbm{R}}
\newcommand{\tr}{\mathrm{tr}}
\newcommand{\rslice}{\mathrm{rslice}}
\newcommand{\slice}{\mathrm{slc}}
\newcommand{\fd}{\mathrm{fd}}

\newcommand{\conn}{\mathrm{conn}}
\newcommand{\Tr}{\mathrm{Tr}}
\newcommand{\Res}{\mathrm{Res}}
\newcommand{\SSeq}{\mathrm{SSeq}}

\newcommand{\invdeRham}{{}^\sigma\mathrm{dR}}
\newcommand{\cyc}{\mathrm{cyc}}

\newcommand{\inj}{\hookrightarrow}

\DeclareMathOperator*{\colim}{colim}

\DeclareMathOperator{\Mod}{Mod}
\DeclareMathOperator{\LMod}{LMod}
\DeclareMathOperator{\RMod}{RMod}
\DeclareMathOperator{\coMod}{coMod}
\DeclareMathOperator{\coLMod}{coLMod}

\newcommand{\rlarrows}{\mathrel{\substack{\textstyle\longrightarrow\\[-0.6ex]
											\textstyle\longleftarrow}}}

\makeatletter
\newcommand{\ostar}{\mathbin{\mathpalette\make@circled\star}}
\newcommand{\make@circled}[2]{%
  \ooalign{$\m@th#1\smallbigcirc{#1}$\cr\hidewidth$\m@th#1#2$\hidewidth\cr}%
}
\newcommand{\smallbigcirc}[1]{%
  \vcenter{\hbox{\scalebox{0.77778}{$\m@th#1\bigcirc$}}}%
}
\makeatother

\numberwithin{equation}{subsection}

\theoremstyle{plain} \newtheorem{theorem}[equation]{Theorem}
\theoremstyle{plain} \newtheorem*{theorem*}{Theorem}
\theoremstyle{definition} \newtheorem{defn}[equation]{Definition}
\theoremstyle{plain} \newtheorem{prop}[equation]{Proposition}
\theoremstyle{plain} \newtheorem*{prop*}{Proposition}
\theoremstyle{plain} \newtheorem{lemma}[equation]{Lemma}
\theoremstyle{plain} \newtheorem{cor}[equation]{Corollary}
\theoremstyle{definition} \newtheorem{ex}[equation]{Example}
\theoremstyle{definition} 
\theoremstyle{definition} 
\theoremstyle{definition} \newtheorem{rmk}[equation]{Remark}
\theoremstyle{remark} 
\theoremstyle{definition} 
\theoremstyle{definition} \newtheorem{obs}[equation]{Observation}
\theoremstyle{plain} 
\theoremstyle{remark} 
\theoremstyle{definition} 
\theoremstyle{definition} 
\theoremstyle{definition} \newtheorem{ntn}[equation]{Notation}
\theoremstyle{plain} 
\theoremstyle{remark} 
\theoremstyle{definition} 
\theoremstyle{definition} \newtheorem{cons}[equation]{Construction}
\theoremstyle{definition} \newtheorem{variant}[equation]{Variant}
\theoremstyle{definition} \newtheorem{warning}[equation]{Warning}
\theoremstyle{definition} \newtheorem{recollection}[equation]{Recollection}

\titleformat{\subsubsection}[runin]{\bfseries}{\thesubsubsection}{1em}{}

\usepackage{xcolor}
\definecolor{seagreen}{RGB}{46,139,87}
\definecolor{maroon}{RGB}{128,0,0}
\definecolor{darkviolet}{RGB}{210,150,180}

\title{\texorpdfstring{\vspace{-2em}}{}\large A filtered Hochschild-Kostant-Rosenberg theorem for real Hochschild homology}
\author{\normalsize Lucy Yang}
\date{\normalsize March 4, 2025 \texorpdfstring{\vspace{-0.5em}}{}}

\DeclareUnicodeCharacter{0301}{\Huge [ICI]}

\begin{document}
\maketitle
\vspace{-1em}
\begin{abstract}
	In this paper, we introduce a notion of derived involutive algebras in $ C_2 $-Mackey functors which simultaneously generalize commutative rings with involution and the (non-equivariant) derived algebras of Bhatt--Mathew and Raksit. 
	We show that the $ \infty $-category of derived involutive algebras admits involutive enhancements of the cotangent complexes, de Rham complex, and de Rham cohomology functors; furthermore, their real Hochschild homology is defined. 
	We identify a filtration on the real Hochschild homology of these derived involutive algebras via a universal property and show that its associated graded may be identified with the involutive de Rham complex. 
	Using $ C_2 $-$ \infty $-categories of Barwick--Dotto--Glasman--Nardin--Shah, we show that our filtered real Hochschild homology specializes to the HKR-filtered Hochschild homology considered by Raksit.  
\end{abstract}

\tableofcontents

\section{Introduction} 
\subsection{The Hochschild--Kostant--Rosenberg theorem} 
There is a two-way dialogue between vector bundles and topology: 
\begin{itemize}
	\item the collection of all isomorphism classes of finite rank vector bundles on a compact topological space $ X $ reflects the homotopy type of $ X $: For instance, the collection of complex line bundles on $ X $ up to isomorphism is in bijection with cohomology classes in $ H^2(X; \Z) $; for a line bundle $ \mathscr{L} $ on $ X $, its associated invariant in $ H^2(X;\Z) $ is its first \emph{Chern class} $ c_1(\mathscr{L}) $. 
	\item vector bundles can in turn be used to study topology: If $ \pi \colon E \to X $ is a fiber bundle with compact fibers, then for each $ \ell $, the cohomology of the fibers $ H^\ell(E_x; \C) $ assemble into a complex vector bundle over $ X $ which reflects `how twisted' the fiber bundle $ \pi $ is. 
\end{itemize} 
While two complex vector bundles on a compact space $ X $ which have the same Chern classes are not necessarily stably equivalent, we do know this:  
The homomorphism $ \mathrm{ch} \colon \mathrm{KU}^0(X) \to H^{\mathrm{even}}(X; \Z) $ induced by sending a complex vector bundle $ \mathcal{V} $ on $ X $ to its Chern classes is a rational equivalence, i.e. $ \mathrm{ch} \otimes \Q \colon \mathrm{KU}^0(X) \otimes \Q \to H^{\mathrm{even}}(X; \Q) $ is an isomorphism. 

The collection of algebraic vector bundles on a scheme $ X $ reflects both the arithmetic and algebro-geometric structure of $ X $. 
Because algebraic vector bundles are significantly less well-understood than their topological counterparts, algebraic Chern classes are crucial to understanding algebraic K-theory. 
The \emph{algebraic Chern character} is constructed as follows for affine schemes. 
Let $ A $ be a commutative ring and $ B $ an commutative $ A $-algebra; \emph{Hochschild homology} $ \HH(B/A) $ is a simplicial commutative $ A $-algebra with $ S^1 $-action.  
There is a map $ K(B) \to \HH(B/A) $ called the Dennis trace, which factors through the homotopy fixed points of the $ S^1 $-action: $ K(B) \to \HH(B/A)^{hS^1} $. 
That the Dennis trace and its corresponding lift to $ \HH(B/A)^{hS^1} $ comprise an algebraic analogue of the Chern character in part relies on the following classical result. 
\begin{theorem} [\cites{MR142598}] \label{thm:hkr_most_classical}
		Let $ B $ be a smooth $ A $-algebra. 
		Then there are canonical isomorphisms $ \HH_*(B/A) \simeq \Omega^i_{B/A} $. 
		Moreover, the $ S^1 $-action on $ \HH(B/A) $ induces the de Rham differential on $ \HH_{*}(B/A) \simeq \Omega^*_{B/A} $.  
\end{theorem} 
Theorem \ref{thm:hkr_most_classical} is also a starting point for many recent advances in mixed characteristic cohomological invariants for schemes; let us note a few. 
This characterization of $ \HH(-/\Z) $ implies powerful descent results for topological Hochschild homology, which in turn were used by Bhatt--Morrow--Scholze to relate $ \TC $ and $ \TC^{-} $ to syntomic and $ A_{\mathrm{inf}} $-cohomology, resp. \cite{BMS2}. 
Furthermore, that the aforementioned trace map can be lifted to a filtered map is crucial to Elmanto--Morrow's construction of motivic cohomology theory for not necessarily smooth schemes \cite{ElmantoMorrow}. 

Observe that Theorem \ref{thm:hkr_most_classical} could just as well be rephrased as: there exists a filtration on $ \HH(B/A) $ whose associated graded is identified with the de Rham complex $ \Omega^\bullet_{B/A} $. 
In recent work, Raksit has proven a systematic generalization of Theorem \ref{thm:hkr_most_classical} which is functorial in $ B $ and simultaneously characterizes $ \HH $, the derived de Rham complex, and filtered Hochschild homology by universal property \cite{Raksit20}. 
Raksit's result completely and precisely describes how the filtration interacts with both the $ S^1 $-action and algebra structure on Hochschild homology. 

Vector bundles often are naturally equipped with additional structure.  
For instance, if $ \pi \colon E \to X $ is a bundle of compact orientable $ 4n $-manifolds over $ X $ equipped with a compatible family of orientations, then the vector bundle on $ X $ determined by $ H^{2n}(E_x; \R) $ acquires a non-degenerate symmetric bilinear form via Poincaré duality. 
Such vector bundles with symmetric bilinear forms assemble into a cohomology theory called \emph{real K-theory}; it may be regarded as an algebraic analogue of $ \mathrm{KO} $. 
More broadly, non-degenerate symmetric bilinear forms (or variants such as conjugate linear forms) play an important role in number theory, algebraic geometry, and algebraic surgery theory \cite{MR206940,MR1803358,MR3658988,MR260844,MR296955,Voevodsky:1996,MR145540,MR1581519}.  
Recent progress in enriched enumerative geometry uses real K-theory to explain the failure of certain counting formulas over non-algebraically closed fields \cite{MR4635342,MR4198841,MR4211099,MR4237952}.  

In this work, we will be interested in characteristic classes for real K-theory and the cohomology group(s) that they take values in.  
Hesselholt--Madsen have introduced real Hochschild homology, which receives a natural transformation from real K-theory \cites{HMreal,Harpaz_Nikolaus_Shah}. 
Using said natural transformation, Cortiñas defined Stiefel--Whitney characteristic classes for real algebraic vector bundles \cite{MR1251695}. 
The purpose of this paper is to prove a functorial Hochschild--Kostant--Rosenberg-style theorem for real Hochschild homology. 
In other words, we identify a natural filtration on real Hochschild homology and characterize its associated graded. 

Before we proceed, let us comment on the nature of the desired result. 
Just as any vector bundle with a nondegenerate symmetric bilinear form has an underlying vector bundle, real K-theory has an underlying spectrum given by ordinary algebraic K-theory, real Hochschild homology has an underlying functor given by ordinary Hochschild homology, and so on. 
This paper shows not only that real Hochschild homology admits a functorial filtration, but also that on underlying objects, said filtration \emph{agrees} with that of Theorem \ref{thm:hkr_most_classical}. 
The theory of $ C_2 $-Mackey functors and $ C_2 $-$ \infty $-categories (the latter due to Barwick--Dotto--Glasman--Nardin--Shah) allows us to systematically keep track of the relationship between our definitions and results and their non-involutive counterparts. 
Moreover, working with $ C_2 $-$ \infty $-categories has theoretical significance, in addition to facilitating bookkeeping: The universal property of real Hochschild homology is internal to $ C_2 $-$ \infty $-categories (compare \cite{QSparam_Tate}). 

\subsection{The main result}\label{subsection:intro_mainresult}
Let us elucidate a setting in which a real Hochschild--Kostant--Rosenberg theorem can be expected to hold. 
Real topological Hochschild homology is a functor which takes $ C_2 $-$ \E_\infty $-algebras to $ C_2 $-$ \E_\infty $-algebras with $ S^\sigma $-action; real Hochschild homology is a \emph{relative} variation on the same construction. 
In order for real Hochschild homology to inherit coherently commutative multiplication, it must be taken relative to a base which has sufficient algebraic structure. 
Furthermore, while Hochschild homology is defined for any $ \E_\infty 
$-$ \Z $-algebra, the (non-equivariant) Hochschild--Kostant--Rosenberg theorem applies only to commutative $ \Z $-algebras and their derived counterparts. 
Thus, to formulate and prove a real enhancement of the Hochschild--Kostant--Rosenberg theorem, we require these auxiliary results and constructions:  
\begin{itemize}
	\item Fix a discrete commutative ring $ k $ with an involution. 
	Our earlier work implies that the fixed point $ C_2 $-Mackey functor $ \underline{k} $ admits a canonical enhancement to a $ C_2 $-$ \E_\infty $-algebra--in particular, there exists a notion of $ C_2 $-$ \E_\infty $-$ \underline{k} $-algebras \cite{LYang_normedrings}. 
	It follows that the real topological Hochschild homology relative to $ \underline{k} $ is defined and takes $ C_2 $-$ \E_\infty $-$ \underline{k} $-algebras to $ C_2 $-$ \E_\infty $-$ \underline{k} $-algebras. 
	\item We will extend the theory of nonconnective simplicial commutative $ k $-algebras of Bhatt--Mathew to $ C_2 $-Mackey functors over the fixed point $ C_2 $-Mackey functor $ \underline{k} $ (Definition \ref{defn:involutive_derived_alg}). 
	We will call these objects \emph{derived involutive $\underline{k}$-algebras}\footnote{In keeping with \cite{CDHHLMNNSI}, one might prefer to call these \emph{derived commutative $ \underline{k} $-algebras with genuine involution}.} and denote the category of such objects by $ \DAlg_{\underline{k}}^\sigma $. 
	The objects of $ \DAlg_{\underline{k}}^\sigma $ bear a certain resemblance to Tambara functors, but they are less general; they may be thought of as derived versions of \emph{cohomological} $ C_2 $-Tambara functors (see Variant \ref{variant:otherdalg_inv}\ref{varitem:discrete_dalg_inv}). 
	On the other hand, any derived $ k $-algebra can be regarded canonically as a derived involutive $ \underline{k} $-algebra by endowing it with the trivial $ C_2 $-action (Remark \ref{rmk:const_inv_dalg}).  
	The category $ \DAlg_{\underline{k}}^\sigma $ admits forgetful functors both to $ C_2 $-$ \E_\infty $-$ \underline{k} $-algebras and to the derived $ k $-algebras of \cite{Raksit20}. 
	It also admits all $ C_2 $-colimits and $ C_2 $-limits (Proposition \ref{prop:invdalg_C2_colims}); hence, the real Hochschild homology of (the underlying $ C_2 $-$\E_\infty$-algebra of) a derived involutive $ \underline{k} $-algebra $ A $ is itself a derived involutive $ \underline{k} $-algebra. 
	\item For any derived involutive algebra $ A $, we introduce a notion of $ h_\sigma^+ $-differential graded $ A $-module. 
	This is a twisted analogue of a homotopy coherent cochain complex: An object consists of a graded $ A $-module $ \{X_*\}_{* \in \Z} $ with differentials $ d \colon \Sigma^\sigma X_i \to X_{i+1} $ which square to zero in a homotopy coherent manner. 
	On underlying objects, the differential $ \Sigma^1 X_i^e \to X_{i+1}^e $ is antilinear with respect to the $ C_2 $-action (see Remark \ref{rmk:inv_dg_objects_unravelled}). 
	Furthermore, there is a notion of a derived involutive algebra in the category of $ h_\sigma^+ $-dg $ A $-modules (Definition \ref{defn:dg_involutive_alg}), which we refer to as $ h^\sigma_+ $-dg derived involutive $ A $-algebras and denote by $\DG^{\sigma}_+\DAlg_{A}^\sigma $. 
\end{itemize}
With these notions in place, we prove: 
\begin{theorem} \label{thm:real_hkr} 
	Let $ k $ be a discrete commutative ring with an involution, and let $ \underline{k} $ be the associated fixed point $ C_2 $-Mackey functor (Example \ref{ntn:fixpt_green_functor}). 
	\begin{enumerate}[label=(\arabic*)]
		\item (see \S\ref{subsection:fil_inv_circ}) There exists a functor called real Hochschild homology
		\begin{equation*}
			\HR(-/\underline{k}): \DAlg_{\underline{k}}^\sigma \to \DAlg_{\underline{k}}^\sigma
		\end{equation*} 
		which is a linearization of $ \THR $ (Remark \ref{rmk:real_HH_as_THR_linearization}) and whose underlying object is Hochschild homology. 
		\item (Proposition \ref{prop:sq_zero_has_left_adjoint}) There exists a functor called the involutive cotangent complex
		\begin{equation*}
			\L_{-/\underline{k}}:\DAlg_{\underline{k}}^\sigma \to \Mod_{\underline{k}}\,;
		\end{equation*}
		for a derived involutive $ \underline{k} $-algebra $ A $, the underlying $ k $-module of $ \L_{A/\underline{k}} $ is the ordinary derived cotangent complex $ \L_{A^e/k} $.
		\item (Theorem \ref{thm:invdeRham_equals_LSymcotangent}) There exists a functor called the involutive derived de Rham complex 
		\begin{equation*}
			\L\Omega^{\sigma,\bullet}_{-/\underline{k}}:\DAlg_{\underline{k}}^\sigma \to \DG^{\sigma, \geq 0}_+\DAlg_{\underline{k}}^\sigma 
		\end{equation*}
		whose underlying $ h_+ $-dg $k$-algebra is the derived de Rham complex of $ A^e $ over $ k $ in the sense of \cite[Definition 5.3.3]{Raksit20}. 
		For a derived involutive $ \underline{k} $-algebra $ A $, the underlying graded derived involutive $\underline{k}$-algebra of the involutive derived de Rham complex of $ A $ over $ \underline{k} $ can be computed as 
		\begin{equation*}
			\L\Omega^{\sigma,\bullet}_{A/\underline{k}} \simeq\LSym_{A}^\sigma( \Sigma^\sigma \L_{A/k}(1)),
		\end{equation*} 
		where $ \LSym_{A}^\sigma $ is the \emph{free graded derived involutive $ A $-algebra} functor. 
		\item (Theorem \ref{thm:gr_of_filtered_HR}) \label{mainthmitem:gr_of_filtered_HR} Given $ A  \in \DAlg_{\underline{k}}^\sigma $, there exists a natural decreasing $\Z_{\geq 0}$-indexed filtration $ \fil^{\geq \bullet}\HR(A/\underline{k}) $ on $ \HR(A/\underline{k}) $ with natural isomorphisms 
		\begin{equation*}
			\gr^\bullet \HR(A/\underline{k}) \simeq \L\Omega^{\sigma,\bullet}_{A/\underline{k}} \qquad \qquad  \fil^0 \HR(A/\underline{k}) \simeq \HR(A/\underline{k}) \,.
		\end{equation*}
		Moreover, the underlying filtered object of $ \fil^{\geq \bullet}\HR(A/\underline{k}) $ is the filtered Hochschild homology of $ A^e $ over $ k $ in the sense of \cite{Raksit20}. 
	\end{enumerate}
\end{theorem} 
We refer to the filtration in Theorem \ref{thm:real_hkr}\ref{mainthmitem:gr_of_filtered_HR} as the \emph{real Hochschild--Kostant--Rosenberg theorem}, or the $ \HKR $ filtration. 
We expect a characterization of the individual terms in the involutive de Rham complex in terms of an involutive analogue of derived exterior powers, but we defer this question to future work (see Remark \ref{rmk:involutive_deRham_cplx_and_exterior_powers_conj}). 
\begin{rmk}
		[The filtered involutive circle] 
		As in \cite{Raksit20}, we characterize filtered real Hochschild homology by a universal property. 
		In order to do this, we introduce the notion of a filtered $ S^\sigma $-action (roughly, an action of $ S^\sigma $ which increases the filtration degree) on a $ \underline{\Z} $-module. 
		A filtered $ S^\sigma $-action on a filtered $ \underline{\Z} $-module determines a $ S^\sigma $-action on its underlying $ \underline{\Z} $-module; its associated graded $ \underline{\Z} $ has a ``twisted'' differential graded structure. 
		The universal property of filtered real Hochschild homology is in the $ C_2 $-$ \infty $-category of filtered derived involutive algebras with filtered $ S^\sigma $-action. 
		This definition is made possible by the existence of certain involutive bialgebra structure on $ \underline{\Z}^{S^\sigma} $, which itself hinges on a remarkable coincidence: the regular slice and Postnikov connective covers agree on $ \underline{\Z}^{S^\sigma} $. 
\end{rmk}
\begin{rmk}
		In the non-involutive setting, one may \emph{a posteriori} identify the resulting filtration on $ \HH(A/k) $ (where $ k, A$ are discrete and $ A $ is smooth over $k$) with the Postnikov filtration. 
		Here, we hope to take up the questions of whether our filtration on $ \HR $ arises from a filtration intrinsic to the $ C_2 $-$ \infty $-category of $ \underline{k} $-modules and whether our filtration is complete in future work. 
		However, let us note that existing work suggests that addressing these questions will be less straightforward than in the non-equivariant case (compare \cite{PHrealTHH_perfectoid}, in particular see Lemma 4.26). 
\end{rmk} 
We define filtrations on real negative cyclic homology and real periodic cyclic homology, which are equivariant analogues of negative cyclic homology and periodic cyclic homology. 
In this direction, we show: 
\begin{theorem}\label{thm:real_inv_deRham_and_HC_HP}
		Let $ k $ be a discrete commutative ring with an involution, and let $ \underline{k} $ be the associated fixed point $ C_2 $-Mackey functor (Example \ref{ntn:fixpt_green_functor}). 
		\begin{enumerate}[label=(\arabic*)]
			\item \label{thmitem:inv_deRham_coh} (see \S\ref{subsection:C2deRham}) There exists a functor called the Hodge-filtered Hodge complete involutive de Rham cohomology $ \invdeRham^{\wedge, \geq i}_{-/\underline{k}} \colon \DAlg^\sigma_{\underline{k}} \to \E_\infty\Alg \left(\Fil^\wedge\left(\Mod_{\underline{k}}\right)\right) $ whose underlying object is ordinary Hodge-filtered Hodge-completed derived de Rham cohomology.  
			\item \label{thmitem:fil_gr_on_HCminus_HP} (see \S\ref{subsection:filteredHCHPetc}) 
			Let $ A $ be a derived involutive algebra over $ \underline{k} $. 
			There are decreasing $ \Z $-indexed filtrations on real negative cyclic homology $ \fil^\bullet \HCR^{-} $ and real periodic cyclic homology $ \fil^\bullet \HPR $ with associated graded pieces given by
				\begin{equation*}
						\gr^i\HCR^{-}(A/\underline{k}) \simeq \invdeRham^{\wedge, \geq i}_{A/\underline{k}}[\rho i] \qquad \qquad \gr^i \HPR(A/\underline{k}) \simeq \invdeRham^{\wedge}_{A/\underline{k}}[\rho i] \qquad i \in \Z \,.
				\end{equation*}
				Here $ (-)[\rho i] $ denotes a shift by $ i $ copies of the regular representation sphere. 
				These isomorphisms are compatible with the equivalences of \cite[Proposition 1.2.4]{Raksit20}.
		\end{enumerate}
\end{theorem}
\begin{rmk}
		We expect the $ \E_\infty $-algebra structure on the involutive derived de Rham cohomology of a derived involutive algebra \ref{thmitem:inv_deRham_coh} to promote to a $ C_2 $-$ \E_\infty $ algebra structure (Remark \ref{rmk:involutive_derham_alg_structure_pending_conjecture}). 
		Our results in Theorem \ref{thm:real_inv_deRham_and_HC_HP}\ref{thmitem:fil_gr_on_HCminus_HP} are only partial: We were unable to delineate precise conditions under which the filtrations in \ref{thmitem:fil_gr_on_HCminus_HP} are complete and exhaustive, and hope to return to this in future work. 
\end{rmk} 
At first glance, the main results of this paper are purely formal and take their ideas directly from \cite{Raksit20}. 
However, underpinning the main theorems is a whole host of technical results devoted to showing that certain $ C_2 $-$ \infty $-categories, $ C_2 $-$ \infty $-operads, and objects therein have good structural properties. 
The reader who is more interested in parametrized $ \infty $-category theory and/or genuine equivariant homotopy theory might find these results to be of independent interest; let us survey a selection here. 
\begin{theorem}
		[{Corollary \ref{cor:param_gr_fil_day_convolution}, Proposition \ref{prop:param_assoc_gr_is_C2_monoidal}}] 
		The $ C_2 $-$ \infty $-category $ \Gr\left(\underline{\Spectra}^{C_2}\right) $ (resp. $ \Fil\left(\underline{\Spectra}^{C_2}\right) $) of graded (resp. filtered) $ C_2 $-spectra admits a canonical $ C_2 $-symmetric monoidal structure given by parametrized Day convolution. 
		Moreover, the associated graded functor promotes canonically to a $ C_2 $-symmetric monoidal $ C_2 $-functor. 
\end{theorem}
\begin{prop}[{Proposition \ref{prop:regsliceright_complete_sep}}]
		Let $ A $ denote a connective $ \E_\infty $-algebra in $ \Spectra^{C_2} $. 
	\begin{enumerate}[label=(\arabic*)]
		\item The regular slice filtration on $ A $-modules is \emph{right separated}, i.e. the intersection $ \bigcap_{n \in \Z} \tau^\rslice_{\leq n} \Mod_A\left(\Spectra^{C_2}\right) $ is zero. 
		\item The regular slice filtration on $ \Mod_{A} $ is \emph{right complete}, i.e. the canonical functor $$ \displaystyle \Mod_{A} \to \lim \left(\cdots \to \Mod_{A, \rslice \geq -1} \to \Mod_{A, \rslice \geq 0}\right) $$ is an equivalence. 
	\end{enumerate}
\end{prop}

\subsection{Related work} 
This work was very much inspired by \cite{Raksit20}; the reader who is familiar with that work will immediately notice how many of the ideas in this text (such as the proof of Theorem \ref{thm:gr_of_filtered_HR}) are borrowed from there. 
However, the usage of genuine equivariant homotopy theory and $ C_2 $-$ \infty $-categories inevitably introduces additional complexity (for instance, see Warning \ref{warning:diff_equivariant_linearizations}) and requires us to prove new foundational results outside of the scope of the earlier cited work. 
Furthermore, a number of the results presented here are weaker than might be expected from examining their non-equivariant counterparts; often this is due to aforementioned subtleties, or because an existing result in higher algebra has not (yet) been generalized to parametrized higher algebra. 

A common subtlety which arises when considering equivariant enhancements of a non-equivariant notion is: There may be more than one choice (possibly infinitely many). 
Indeed, another construction which may be regarded as an equivariant generalization of Hochschild homology is to use the definition of Hochschild homology \emph{in the symmetric monoidal category of $\underline{k}$-modules.} 
(Contrast: $	A \otimes_{A \otimes A} A $ with $ A \otimes_{N^{C_2}A} A $.)
Related work of Mehrle--Quigley--Stahlhauer studies the former (which exists for arbitrary finite groups $ G $) \cite{MQS_Koszul_Tambara}, while we study the latter. 
Similar considerations apply when comparing our work with that of Blumberg--Gerhardt--Hill--Lawson \cite{MR3990042}. 
Finally, Michael Hill has introduced a notion of Kähler differentials for Tambara functors for any finite group $ G $ in \cite{Hill_Tambara_Kahlerdiffl}. 
Hill's Kähler differentials is broader in scope than our involutive cotangent complex; for instance, ours is not defined on the category of all $ C_2 $-Tambara functors--only the cohomological ones. 
We hazard a guess that our notions agree where they are both defined, but we do not compare the two theories here as it is out of the scope of the current work. 

The relationship between real Hochschild homology and ordinary Hochschild homology with (naïve) $ C_2 $-action simplifies substantially when 2 acts invertibly on the base ring. 
In this setting, our work is related to and builds off of earlier work of Cortiñas, Loday, Lodder, and Solotar--Vigué-Poirrier \cites{MR1251695,MR1229500,MR917736,MR1384461,SVP96}; we discuss how they compare in \S\ref{subsection:computations_comparisons}. 

Some of the main results in this paper were developed independently by Hornbostel and Park \cite{PHrealTHH_perfectoid}. 
In Theorem 1.1 of their paper, they exhibit a filtration on real Hochschild homology with associated graded expressed in terms of the cotangent complex, just as we do here. 
Let us remark on some key differences between our works here: 
\begin{itemize}
		\item The scopes of our works are different: Hornbostel--Park not only prove a Hochschild--Kostant--Rosenberg theorem for real Hochschild homology, but also use their result on real Hochschild homology to prove results about real topological Hochschild homology, while our work does not address these questions.  
		On the other hand, while \cite{PHrealTHH_perfectoid} consider arbitrary rings with involution, their main results are more thorough for the case of rings with trivial involution (compare Theorem 4.31 with the discussion in and after Remark 4.36 \emph{ibid.}). 
		The perspective in this paper puts all rings with involution on equal footing, regardless of whether or not the involution is trivial. 
		\item The methods by which the respective filtrations are constructed diverge; while \cite{PHrealTHH_perfectoid} uses the slice filtration to construct a filtration on $ \HR $, we endow $ \HR $ with a filtration by first showing $ \tau_{\geq *} \underline{\Z}^{S^\sigma} $ has the structure of a filtered involutive bialgebra, then defining $ \fil_{\HKR} \HR $ to be the universal filtered derived involutive algebra with an action of the filtered involutive circle. 
		In particular, we introduce a notion of derived (not necessarily connective) involutive algebras in order to make sense of this algebraic structure on  $ \tau_{\geq *} \underline{\Z}^{S^\sigma} $, while \cite{PHrealTHH_perfectoid} rely on existing theories of derived equivariant rings. 
\end{itemize}
Angelini-Knoll, Kong, and Quigley have forthcoming work on the real motivic filtration for real THH and real TC; in the future, we hope to investigate if and how these filtrations are related. 

\subsection{Outline} 
What follows is a brief summary of the contents of the paper; for more context and motivation for the contents of each section, we invite the reader to peruse their introductions. 
\S\ref{section:genuineeqvt} is devoted to preliminaries concerning equivariance. 
In particular, \S\ref{subsection:param_infty_cats} recalls the ideas, constructions, and notation regarding parametrized $ \infty $-categories we will use throughout the remainder of the text, while \S\ref{subsection:genuine_homotopy} collects structural properties of the genuine equivariant stable category. 
In \S\ref{subsection:gen_eqvt_alg}, we discuss homotopy-coherent genuine equivariant algebraic structures. 
In \S\ref{subsection:filgr}, we discuss filtered and graded objects in a $ C_2 $-$ \infty $-category and a $ C_2 $-parametrized version of the Day convolution monoidal structure. 
\S\ref{subsection:param_bialg} introduces $ C_2 $-bialgebras and monoidal structures on their module categories. 
In \S\ref{section:scr_with_inv}, we introduce the theory of derived involutive algebras, a simultaneous generalization of derived algebras to the equivariant setting and cohomological $ C_2 $-Tambara functors to the derived setting. 
We will use this theory to endow the filtered involutive circle with additional structure. 
We define the $ C_2 $-analogues of the cotangent complex and the de Rham complex and clarify the [involutive] cochain complex structure inherited by the involutive de Rham complex in \S\ref{section:cochainwinvolution}. 
We introduce real Hochschild homology and the filtered involutive circle in \S\ref{subsection:fil_inv_circ} and arrive at a definition of filtered real Hochschild homology and prove the main theorem in \S\ref{subsection:realHKR}, and it is subsequently extended to the real analogues of $ \HP $ and $ \HC^- $ in \S\ref{subsection:filteredHCHPetc}. 
In \S\ref{subsection:computations_comparisons}, we revisit classical computations of real Hochschild homology and dihedral homology to put our results in context. 

\subsection{Notation \& Conventions}
We use freely the language of $ \infty $-categories as developed in \cite{LurHTT}.  
We review the theory of parametrized $ \infty $-categories and parametrized algebras as developed by Barwick, Dotto, Glasman, Nardin, and Shah \cite{BDGNSintro,BDGNS1,Nardinthesis,NS22,Shah18}, but the reader should consult the former references for more details. 
Remark \ref{rmk:terminology_promoting_functors} gives a precise delineation of what it means for a $ C_2 $-$ \infty $-category or $ C_2 $-functor to promote an ordinary $ \infty $-category or functor of $ \infty $-categories. 
In particular, if a $ C_2 $-$ \infty $-category is a $ C_2 $-equivariant enhancement of an ordinary $ \infty $-category, our convention is to denote it with an underline (see, for instance, Notation \ref{ntn:param_module_cats} regarding parametrized module categories or Definition \ref{defn:involutive_derived_alg}); this is not to be mistaken for \cite[Definition 7.4]{BDGNS1}. 

We co-opt most of the notation of \cite{Raksit20}, adding a superscript $ (-)^\sigma $ to indicate objects which should be thought of as a direct $ C_2 $-equivariant analogue of an existing definition. 

\subsection{Acknowledgements} 
The author would like to thank Araminta Amabel, Elden Elmanto, Peter Haine, Denis Nardin, Arpon Raksit, and Jay Shah for their generosity with their time and ideas. 
The author also benefited from numerous discussions with Emanuele Dotto, Jeremy Hahn, Michael Hopkins, Piotr Pstragowski, and Dylan Wilson. 
The author owes much of her understanding of the slice and regular slice filtrations to Michael Hill. 
The intellectual debt this work owes to \cite{Raksit20} is obvious and profound. 
The author would like to thank David Mehrle, J.D. Quigley, and Doosung Park for discussions pertaining to related work, and Andrew Blumberg, S{\o}ren Galatius, and Elden Elmanto for feedback on an earlier draft. 
The author gratefully acknowledges support from the NSF Graduate Fellowship Research program under Grant No. 1745303. 
Finally, the author thanks Northwestern University for its hospitality during an extended visit during which part of this work was completed.  

\section{Genuine equivariance}\label{section:genuineeqvt}
Let $ k $ be a field and let $ V $ be a $ k $-vector space. 
The space of bilinear maps $ \hom_{k \otimes k }(V \otimes V, k) $ has a canonical $ C_2 $-action induced by the isomorphism $ \tau \colon V \otimes V \simeq V \otimes V$; in coordinates, it is given by sending a matrix to its transpose. 
An element of $ \hom_{k \otimes k }(V \otimes V, k) $ is \emph{symmetric} if it is fixed under this $ C_2 $-action. 
In this way, we see that actions of the cyclic group of order two is indispensible to the study of symmetric bilinear forms. 

In this section, we introduce the background and language we will use to describe higher algebraic structures equipped with an action of the group $ C_2 $.  
We will use the language of \emph{parametrized $ \infty $-categories} of Barwick--Glasman--Nardin--Shah \cites{Nardinthesis,Shah18,Shah_paramII,NS22} because it will allow us to compare our constructions and results to those of \cite{Raksit20} in a systematic way.
In \S\ref{subsection:param_infty_cats}, we recall the basic definitions of $ C_2 $-$ \infty $-categories; in particular, ordinary $ \infty $-categorical notions such as functors, adjunctions between functors, and limits and colimits have versions internal to $ C_2 $-$ \infty $-categories.  
\S\ref{subsection:genuine_homotopy} is devoted to examples in and structural results pertaining to genuine equivariant homotopy theory. 
We survey existing theories of equivariant algebra and homotopy coherent equivariant algebra in \S\ref{subsection:gen_eqvt_alg}; this will set the stage for our notion of strictly commutative equivariant algebras, to be introduced in \S\ref{section:scr_with_inv}. 

\subsection{Parametrized \texorpdfstring{$ \infty $}{∞}-categories}\label{subsection:param_infty_cats}
Let $ G $ be a finite group. 
\begin{recollection}
	The orbit category $ \mathcal{O}_G $ is the category with objects finite transitive $ G $-sets and morphisms $ G $-equivariant maps. 
	We let $ \Fin_G $ denote the finite coproduct completion of $ \mathcal{O}_G $, i.e. the category of finite $ G $-sets and $ G $-equivariant maps. 
	We recall that $ \mathcal{O}^\op_G $ is an \emph{orbital} $ \infty $-category in the sense of Definition 1.2 of \cite{Nardinthesis}. 
\end{recollection}
\begin{defn} [{\cite[Definition 1.3]{BDGNS1}}]
	A \emph{$ G $-$ \infty $-category} is a cocartesian fibration $ p \colon \cat \to \mathcal{O}_G^\op $. 

	A morphism of $ G $-$ \infty $-categories is a map $ F $ of $ \infty $-categories over $ \mathcal{O}_G^\op $:
	\begin{equation*}
	\begin{tikzcd}[column sep=tiny,row sep=small]
		\cat \ar[rd,"p"'] \ar[rr,"F"] & & \mathcal{D} \ar[ld,"q"] \\
		& \mathcal{O}_G^\op &
	\end{tikzcd}
	\end{equation*}
	which takes $ p $-cocartesian arrows in $ \cat $ to $ q $-cocartesian arrows in $ \mathcal{D} $. 
	We will write $ G \Cat_\infty $ for the large $ \infty $-category of small $ G $-$ \infty $-categories. 

	Note that for each $ G/H \in \mathcal{O}_G^\op $, the pullback of the cocartesian fibration $ p $ along the inclusion $ \{G/H\} \hookrightarrow \mathcal{O}^\op_G $ determines an ordinary $ \infty $-category, which we denote interchangeably by $ \cat^H $ or $ \cat_{G/H} $. 
	When $ H = \{e\} $ is the trivial subgroup, we will denote $ \cat^H $ by $ \cat^e $ and refer to it as the \emph{underlying $ \infty $-category of the $ G $-$ \infty $-category $ \cat $.} 
\end{defn}
\begin{rmk}
	[{\cites[\S3.2.2]{LurHTT}[Example 2.5]{Shah18}}] \label{rmk:param_unstraighten}
	Let $ \Cat_\infty $ denote the large $ \infty $-category of small $ \infty $-categories. 
	There is a universal cocartesian fibration $ \mathcal{U} \to \Cat_\infty $ such that pullback induces an equivalence
	\begin{equation*}
		\Fun\left(\mathcal{O}_G^\op, \Cat_\infty\right) \simeq \Cat_{\infty/\mathcal{O}_G^\op}^{\mathrm{cocart}} .
	\end{equation*}
	Informally, a $ C_2 $-$ \infty $-category is the data of 
	\begin{itemize}
		\item an $ \infty $-category $ \cat^{C_2} $,
		\item an $ \infty $-category with $ C_2 $-action $ \cat^{e} $, and 
		\item a functor $ \cat^{C_2} \to \cat^e $ which lifts along the $ C_2 $ homotopy fixed points $ (\cat^e)^{hC_2} \to \cat^e $. 
		In particular, if $ \cat^e $ is endowed with the trivial $ C_2 $-action, then $ (\cat^e)^{hC_2} \simeq (\cat^e)^{BC_2} \simeq \Fun(BC_2 , \cat^e) $ comprises objects in $ \cat^e $ with (naïve) $ C_2 $-action.
	\end{itemize}
	In particular, we see that a cocartesian section $ \sigma \colon \mathcal{O}^\op_{C_2} \to \cat $ is determined by its value on $ \sigma(C_2/C_2) $. 
	Informally, we regard the category of cocartesian sections of $ \cat $ as the category of objects in $ \cat $. 
\end{rmk}
\begin{ntn}
		To lighten notational burden, we write $ G $-$ \infty $-category instead of $ \mathcal{O}^\op_G $-$ \infty $-category, and similarly for other phrases. 
		If we want to say that a $ G $-$ \infty $-category $ \cat $ is such that $ \cat_t $ has a certain property for all $ t \in \mathcal{O}^\op_G $, we say that it has that property \emph{fiberwise} or \emph{pointwise}. 
		The same applies to modifiers for functors of $ G $-$ \infty $-categories. 
\end{ntn} 
\begin{recollection}
		[{\cites[\S10]{BDGNS1}[Recollection 5.18]{Shah18}}]\label{rec:vertical_op}
		Let $ \cat $ be a $ G $-$ \infty $-category. 
		Then the \emph{vertical opposite} $ \cat^{\vop} $ to $ \cat $ is the $ G $-$ \infty $-category characterized by canonical equivalences $ \left(\cat^{\vop}\right)^H \simeq \cat^{H,\op} $ for each $ G/H \in \mathcal{O}_G $ so that under these equivalences, the restriction functors in $ \cat^\vop $ classified by morphisms in $ \mathcal{O}^\op_G $ are identified with (the opposite of) the restriction functors in $ \cat $. 
		In other words, if $ \cat $ is classified by the functor $ F \colon \mathcal{O}^\op_G \to \Cat_\infty $ (see Remark \ref{rmk:param_unstraighten}), then $ \cat^\vop $ is classified by the composite $ \mathcal{O}^\op_G \xrightarrow{F} \Cat_\infty \xrightarrow{(-)^\op} \Cat_\infty $.  
\end{recollection}
\begin{rmk}\label{rmk:terminology_promoting_functors}
		Let $ \cat, \mathcal{D} $ be $ C_2 $-$ \infty $-categories, suppose $ F \colon \cat \to \mathcal{D} $ is a $ C_2 $-functor and $ G \colon \cat^e \to \mathcal{D}^e $. 
		If there is an equivalence $ F^e \simeq G $, we say that \emph{$ F $ recovers $ G $ on underlying $ \infty $-categories} or that \emph{$ G $ can be promoted to a $ C_2 $-functor}. 
		If $ H \colon \cat^{C_2} \to \mathcal{D}^{C_2} $ and there is an equivalence $ F^{C_2} \simeq H $, we say that \emph{$F$ recovers $ H $ on $ C_2 $-fixed points.}
\end{rmk}
\begin{ntn}
		The $ \infty $-category of $ C_2 $-$ \infty $-categories has finite limits \cite[\S9]{BDGNS1}. 
		If $ \cat $, $ \mathcal{D} $ are two $ C_2 $-$ \infty $-categories, we write $ \cat \times \mathcal{D}  $ for the product $ C_2 $-$ \infty $-category, and likewise for fiber products. 
		We do not use the less ambiguous notation $ (-) \underline{\times} (-) $ of \emph{loc. cit.} in order to streamline notation for fiber products.   
\end{ntn}
We will use the notion of \emph{parametrized functor categories} of \cite[\S3]{Shah18} to define real Hochschild homology. 
\begin{prop} [{\cites[Proposition 3.1]{Shah18}[Construction 5.2]{BDGNS1}}]\label{prop:param_functors}
	Let $ \cat \to \mathcal{O}^\op_G $, $ \mathcal{D} \to \mathcal{O}^\op_G $ be cocartesian fibrations. 
	Then there exists a cocartesian fibration $ \underline{\Fun}(\cat, \mathcal{D}) \to \mathcal{O}^\op_G $ such that under the straightening-unstraightening equivalence of Remark \ref{rmk:param_unstraighten}, $ \underline{\Fun}(\cat, \mathcal{D}) $ represents the presheaf
	\begin{equation*}
		\mathcal{E} \mapsto \hom_{\mathcal{O}^\op_G}\left(\mathcal{E} \times_{\mathcal{O}^\op_G} \cat, \mathcal{D} \right).
	\end{equation*}
\end{prop}
\begin{defn}
	[{\cite[Definition 7.3.2.2]{LurHA}}]\label{defn:rel_adjunction}
	Suppose given a commutative diagram of $ \infty $-categories
	\begin{equation}\label{diagram:rel_adjunction}
	\begin{tikzcd}[column sep=small]
		\cat \ar[rd,"q"'] & & \mathcal{D} \ar[ld,"p"] \ar[ll,"G"'] \\
		& \mathcal{E} & 
	\end{tikzcd} .
	\end{equation}
	We will say that $ G $ \emph{admits a left adjoint relative to $ \mathcal{E} $} if it satisfies one of the following equivalent conditions:
	\begin{enumerate}[label=(\arabic*)]
		\item The functor $ G $ admits a left adjoint $ F $. 
		Moreover, for every object $ C \in \cat $, the functor $ q $ carries the unit of the adjunction $ u_C \colon C \to G(F(C)) $ to an equivalence in $ \mathcal{E} $.
		\item There exists a functor $ F \colon \cat \to \mathcal{D} $ and a natural transformation $ u \colon \id_{\cat} \to G \circ F $ which exhibits $ F $ as a left adjoint to $ G $, with the property that $ q(u) $ is the identity transformation from $ q $ to itself. 
	\end{enumerate}
\end{defn}
\begin{defn}[{\cite[Definition 8.3]{Shah18}}]\label{defn:param_adjunction}
	Suppose given a diagram (\ref{diagram:rel_adjunction}) so that $ \mathcal{E} = \mathcal{O}^\op_G $ and suppose that $ \cat, \mathcal{D} $ are $ G $-$ \infty $-categories. 
	The functors $ F, G $ comprise \emph{a $ G $-adjunction} if they satisfy both
	\begin{itemize}
		\item The functors $ F $ and $ G $ are both $ G $-functors
		\item The functors $ F $ and $ G $ form a relative adjunction relative to $ \mathcal{O}^\op_G $. 
	\end{itemize}  
\end{defn} 
While the left adjoint in a relative adjunction over $ \mathcal{O}^\op_G $ is automatically a $ G $-functor, the right adjoint does not necessarily preserve cocartesian arrows (contrast the dual to \cite[Proposition 7.3.2.6]{LurHA} with \cite[Proposition 7.3.2.11]{LurHA}). 
The following results provide conditions under which the existence of a $ G $-adjunction can be checked pointwise. 
\begin{prop}\label{prop:param_left_adjoint_local_crit}
		Let a diagram as in (\ref{diagram:rel_adjunction}) be given and assume that $ q $ and $ p $ are coCartesian categorical fibrations and that $ G $ takes $ p $-cocartesian edges to $ q $-cocartesian edges. 
		Assume further that for every morphism $ \alpha \colon E \to E' \in \mathcal{E} $: 
		\begin{enumerate}[label=(\arabic*)]
			\item The functors $ \alpha^* \colon \cat_E \to \cat_{E'} $ and $ \alpha^* \colon \mathcal{D}_E \to \mathcal{D}_{E'} $ admit right adjoints $ \alpha_* \colon \cat_{E'} \to \cat_{E} $ and $ \alpha_* \colon \mathcal{D}_{E'} \to \mathcal{D}_{E} $. 
			\item \label{prop_assumption:param_left_adjoint_local_crit} The canonical natural transformation $ G_{E'} \circ \alpha_* \to \alpha_* G_{E} $ is an equivalence for each $ \alpha \in \mathcal{E} $. 
		\end{enumerate}
		Then if for each object $ E \in \mathcal{E} $, $ G_E \colon \mathcal{D}_E \to \cat_E $ admits a left adjoint $ F_E $, then $ G $ admits a left adjoint $ F $ relative to $ \mathcal{E} $ which agrees with $ F_E $ fiberwise. 
		Moreover, $ F $ takes $ q $-cocartesian edges to $ p $-cocartesian edges.  
\end{prop}
\begin{cor}\label{cor:C2_left_adjoint_local_crit}
		Suppose $ G \colon \mathcal{D} \to \cat $ is a $ C_2 $-functor so that $ \mathcal{D} $, $ \cat $ admit finite $ C_2 $-products which are preserved by $ G $. 
		If $ G $ admits left adjoints fiberwise, then it admits a left $ C_2 $-adjoint. 
\end{cor}
\begin{proof}
		The assumption that $ \mathcal{D} $ and $ \cat $ admit finite $ C_2 $-products which are preserved by $ G $ is a restatement of the assumptions of Proposition \ref{prop:param_left_adjoint_local_crit}; the result follows immediately.  
\end{proof}
There is a dual statement, proved in a similar way. 
\begin{cor}\label{cor:C2_right_adjoint_local_crit}
		Suppose $ G \colon \mathcal{D} \to \cat $ is a $ C_2 $-functor so that $ \mathcal{D} $, $ \cat $ admit finite $ C_2 $-coproducts which are preserved by $ G $. 
		If $ G $ admits right adjoints fiberwise, then it admits a right $ C_2 $-adjoint. 
\end{cor}
\begin{proof} [Proof of Proposition \ref{prop:param_left_adjoint_local_crit}]
		Let $ C \in \cat_E $, and write $ D = F_E(C) $ and $ u $ for the counit $ u \colon C \to G(F_E(C)) $. 
		By the proof of \cite[Proposition 7.3.2.11]{LurHA}, it suffices to show that for all $ D' \in \mathcal{D} $, $ u $ induces an equivalence 
		\begin{equation*}
				\Hom_{\mathcal{D}}(D, D') \simeq \Hom_{\cat}(C, G(D'))\,. 
		\end{equation*} 
		By assumption, we know that for all $ D'\in \mathcal{D}_E $, we have an equivalence
		\begin{equation*}
				\Hom_{\mathcal{D}_E}(D, D') \simeq \Hom_{\cat_E}(C, G(D'))\,. 
		\end{equation*} 
		Since $ p $ and $ q $ are cocartesian fibrations, for any $ D' \in \mathcal{D} $, we have maps 
		\begin{equation}\label{diagram:rel_left_adjoint_local_crit_maps}
		\begin{tikzcd}[row sep=small]
				\Hom_{\mathcal{D}}(D,D') \ar[r, "{h_p}"] \ar[d] & \Hom_{\mathcal{E}}(p(D), p(D')) \ar[d,equals] \\
				\Hom_{\cat}(C, G(D')) \ar[r,"{h_q}"] & \Hom_{\mathcal{E}}(q(C), q\circ G(D'))		
	 	\end{tikzcd} 	
	  \end{equation}
	  To show that the left vertical arrow in (\ref{diagram:rel_left_adjoint_local_crit_maps}) is an equivalence, it suffices to show that for any $ \alpha \in \Hom_{\mathcal{E}}(p(D), p(D')) = \Hom_{\mathcal{E}}(q(C), q\circ G(D')) $, the fibers of the horizontal maps over $ \alpha $ are equivalent. 
	  The result follows from the identifications (where $ \overset{(1)}{\simeq} $ indicates that an equivalence follows from assumption (1) in the statement of the proposition) 
	  \begin{equation*}
	  \begin{split}
	  	 \fib_{\alpha}(h_q) &\simeq \Hom_{\cat_{p(D')}}(\alpha^*(C), G_{p(D')}(D')) \\
					& \overset{(1)}{\simeq} \Hom_{\cat_{p(D)}}(C, \alpha_* G_{p(D')}(D')) \\
					& \overset{(2)}{\simeq} \Hom_{\cat_{p(D)}}(C, G_{p(D)}\alpha_* (D')) \\
					& \simeq \Hom_{\mathcal{D}_E}(F_E(C), \alpha_*(D')) \\ 
					& \overset{(1)}{\simeq} \Hom_{\mathcal{D}_{p(D')}} (\alpha^*F_E(C), D') \simeq \fib_{\alpha}(h_p) \,. \qedhere
	  \end{split}
	  \end{equation*}
\end{proof}
There are parametrized notions of limits and colimits. 
\begin{ntn}
	[{\cite[Notation 2.29]{Shah18}}]\label{ntn:param_fiber}
	Let $ \mathcal{B} $ be a $ G $-$ \infty $-category and write $ \Ar^\mathrm{cocart}(\mathcal{B}) $ for the full subcategory of $ \Ar(\mathcal{B}) $ on the cocartesian edges of $ \mathcal{B} $. 
	Given an object $ b \in \mathcal{B}_t $, write
	\begin{equation*}
		\underline{b} := \{b\} \times_{\mathcal{B}, \ev_0} \Ar^\mathrm{cocart}(\mathcal{B})
	\end{equation*}
	which we think of as a ``$ G $-point of $ \mathcal{B} $.'' 
	Given a morphism $ \pi \colon \mathcal{E} \to \mathcal{B} $ of $ G $-$ \infty $-categories, we write
	\begin{equation*}
		\mathcal{E}_{\underline{b}} := \underline{b} \times_{\ev_1,\mathcal{B},\pi} \mathcal{E}
	\end{equation*}
	for the \emph{parametrized fiber} of $ \pi $ over $ b $. 
	Note that $ \mathcal{E}_{\underline{b}} $ is a $ (\mathcal{O}_G^\op)^{t/-} $-$ \infty $-category. 
\end{ntn}
\begin{defn}
	[{\cite[Definition 5.1-2]{Shah18}}] \label{defn:Gcolims} 
	Let $ \cat \to \mathcal{O}^\op_G $ be a $ G $-$ \infty $-category and let $ \sigma: \mathcal{O}^\op_G \to \cat $ be a cocartesian section. 
	We say that $ \sigma $ is \emph{$ S $-initial} if $ \sigma(s) $ is an initial object for all $ s \in \mathcal{O}^\op_G $. 
	Dually, we say that $ \sigma $ is \emph{$ S $-final} if $ \sigma(s) $ is a final object for all $ s \in \mathcal{O}^\op_G $. 

	Let $ K , \cat $ be $ G $-$ \infty $-categories and let $ \overline{p} \colon K \star_{\mathcal{O}^\op_G} \mathcal{O}^\op_G  \to \cat $ be an extension of a $ G $-functor $ p: K \to \cat $. 
	There is a commutative diagram
	\begin{equation*}
	\begin{tikzcd}
		S \ar[d,equals] \ar[r,"{\sigma_{\overline{p}}}"] & \underline{\Fun}_G(K \star_{\mathcal{O}^\op_G}\mathcal{O}^\op_G, \cat ) \ar[d] \\
		S \ar[r,"{\sigma_p}"] & \underline{\Fun}_G(K, \cat)
	\end{tikzcd}
	\end{equation*}
	so that $ \sigma_{\overline{p}} $ defines a cocartesian section $ \sigma_{\overline{p}} $ of $ \cat^{(p,\mathcal{O}^\op_G)/-} $. 
	We say that $ \overline{p} $ is a \emph{$ G $-colimit diagram} if $ \sigma_{\overline{p}} $ is a $ G $-final object. 

	If $ \overline{p} $ is a $ G $-colimit diagram, we say that $ \overline{p}|_{\sigma_{\overline{p}}} $ is a $ G $-colimit of $ p $. 

	Dually, substituting $ \mathcal{O}^\op_G \star_{\mathcal{O}^\op_G} K $ for $ K \star_{\mathcal{O}^\op_G} \mathcal{O}^\op_G $ leads to the definition of a \emph{$ G $-limit diagram}. 
\end{defn}
Relative adjunctions and parametrized colimits interact in the expected way:
\begin{prop}
	[{\cite[Corollary 8.9]{Shah18}}]
	Let $ \cat, \mathcal{D} $ be $ \mathcal{T} $-$ \infty $-categories, and let $ F\colon \cat \rlarrows \mathcal{D} \colon G $ be a $ \mathcal{T} $-adjunction $ F \dashv G $. 
	Then $ F $ preserves $ \mathcal{T} $-colimits and $ G $ preserves $ \mathcal{T} $-limits. 
\end{prop}
The following is a consequence of \cite[Proposition 5.11]{Shah18}. 
Note that our notation differs somewhat from that of Shah's--what we think of as pullback/restriction $ \alpha^* \colon \cat^K \to \cat^H $ along a morphism $ \alpha \colon G/H \to G/K $ is denoted by $ \alpha_! $ in their work. 
\begin{prop} [{\cite[Proposition 5.11]{Shah18}}]
	Let $ \cat $ be a $ G $-$ \infty $-category and let $ \alpha \colon G/H \surj G/K $ be a morphism in $ \mathcal{O}_G $. 
	Let $ {\pi \colon M \to \Delta^1} $ be the \textbf{cartesian} morphism classified by $ \alpha^* \colon \cat^K \to \cat^H $. 
	Let 
	\begin{equation*}
		p \colon \left(\mathcal{O}^\op_G\right)^{G/H} \to \cat \times_{\mathcal{O}^\op_G} \left(\mathcal{O}^\op_G\right)^{G/K} 
	\end{equation*}
	be a $ \left(\mathcal{O}^\op_G\right)^{G/K} $-functor and let $ x = p(\id_{G/H} ) \in \cat^H $. 
	Then the data of a $ \left(\mathcal{O}^\op_G\right)^{G/K} $-colimit diagram extending $ p $ yields a $ \pi $-cocartesian edge $ e $ in $ M $ with $ d_0(e) = x $ and lifting $ 0 \to 1 $. 
\end{prop}
\begin{ex} [{\cite[Proposition 5.11]{Shah18}}]
	Let $ G= C_2 $ and $ \cat = \underline{\Spc}^{C_2} $. 
	Let $ X \colon \left(\mathcal{O}_{C_2}^{\op}\right)^{C_2/e} \to \cat $ classify a $ C_2 $-space. 
	Then the $ C_2 $-colimit of the diagram $ X $ exists and is given by the $ C_2 $-space $ \varnothing \to X \sqcup X $ where $ C_2 $ acts by swapping the factors of $ X $ in the disjoint union.  
\end{ex}
The notion of a parametrized colimit diagram not only generalizes Definition \ref{defn:Gcolims} but also allows us to show that $ G $-colimits have suitable functoriality. 
\begin{defn}
	[{\cite[Definition 9.13]{Shah18}}] \label{defn:param_Gcolims}
	Let $ \phi \colon \cat \to \mathcal{D} $ be a $ G $-cocartesian fibration. 
	A $ G $-functor $ \overline{F} \colon \cat \star_{\mathcal{D}} \mathcal{D} \to \mathcal{E} $ is a \emph{$ \mathcal{D} $-parametrized $ G $-colimit diagram} if for every object $ x \in \mathcal{D} $, the $ \underline{x} $-functor $ \overline{F}|_{\cat_{\underline{x}} \star_{\underline{x}} \underline{x}} \colon \cat_{\underline{x}} \star_{\underline{x}} \underline{x} \to \mathcal{E}_{\underline{s}} $ is a $ \underline{s} $-colimit diagram.  
\end{defn}
The existence of parametrized $ G $-colimits can be checked `fiberwise.'
\begin{theorem}
	[{\cite[Theorem 9.15]{Shah18}}]
	Let $ \phi \colon \cat \to \mathcal{D} $ be a $ G $-cocartesian fibration and let $ F \colon \cat \to \mathcal{E} $ be a $ G $-functor. 
	Suppose that for every object $ x \in \mathcal{D} $, $ {F} \colon \cat_{\underline{x}} \to \mathcal{E}_{\underline{s}} $ admits a $ \underline{s} $-colimit. 
	Then there exists a $ \mathcal{D} $-parametrized $ G $-colimit diagram $ \overline{F} \colon \cat \star_{\mathcal{D}} \mathcal{D} \to \mathcal{E} $ extending $ F $. 
\end{theorem}

\subsection{\texorpdfstring{$ G $}{G}-spaces and \texorpdfstring{$ G $}{G}-spectra.}\label{subsection:genuine_homotopy} 
In this section, we consider two examples of $ G $-$ \infty $-categories: $ G $-spaces and $ G $-spectra. 
In particular, we introduce several objects in these categories which we will see later on, and discuss structural properties of these categories which we will need.   
\begin{recollection}\label{rec:G_spc_G_spectra}
		A $ G $-space is a functor $ \mathcal{O}_G^\op \to \Spc $, or equivalently, a functor $ \Fin_G^\op \to \Spc $ which take coproducts in $ \mathcal{O}_G $ to products in $ \Spc $. 
		We will denote the category of $ G $-spaces by $ \Spc^G $. 
		By \cite[Theorem 7.8]{BDGNS1}, there is a $ G $-$ \infty $-category $ \underline{\Spc}^G $ so that the fiber over $ G/H $ is the ordinary $ \infty $-category $ \Spc^H $. 

		Given a $ G $-space $ X: \mathcal{O}_G^\op \to \Spc $ and a $ G $-equivariant map $ S \xrightarrow{f} T $, we will refer to $ {f^* = X(f) \colon X(T) \to X(S)} $ as the \emph{restriction map associated to $ f $.} 

		The $ G $-$ \infty $-category of $ G $-spectra $ \underline{\Spectra}^G $ is \cite[Definition 7.3 \& Corollary 7.4.1]{BDGNS4} applied to $ D = \underline{\Spc}^G $. 
\end{recollection}
\begin{ex}[Sign representation sphere] \label{ex:signrepsphere} 
	Let $ G = C_2 = \{1, \sigma\} $.
	There is a $ C_2 $-space, denoted by $ S^\sigma $, whose underlying space with $C_2$-action is given by $ (S^\sigma)^e = S^1 = \{z \in \C \mid |z| =1\} $, with the $ C_2 $ action via complex conjugation. 
	Its $ C_2 $-fixed points are given by $ (S^\sigma)^{C_2} = \{\pm 1\} = S^0 $, and the restriction map $ (S^\sigma)^{C_2} \to (S^\sigma)^{e} $ is given by the inclusion.  
	This $ C_2 $-space is denoted $ S^\sigma $ because it is the one-point compactification of the sign representation of $ C_2 $ on $ \R $. 

	Because complex conjugation respects multiplication of complex numbers, the group structure on $ S^1 $ lifts to a $ C_2 $-monoid structure on $ S^\sigma $ in the category of $ C_2 $-spaces.
\end{ex}
\begin{ex} [Regular representation sphere] \label{ex:regrep}
	Consider the regular representation $ \R[C_2] $ of $ C_2 $, and write $ S^\rho $ for its one-point compactification.  
	Then $ (S^{\rho})^e \simeq S^2 $ with $ C_2 $-action given by reflection about a plane $ P $, and the $ C_2 $ fixed points $ (S^\rho)^{C_2} $ is given by $ S^1 $, and the restriction map $ (S^\rho)^{C_2} \to (S^{\rho})^e $ is given by the natural identification $ S^2 \cap P \simeq S^1 $. 
\end{ex} 
The category of $ G $-spectra is obtained from the category of $ G $-spaces by asking for the existence of transfers indexed by finite $ G $-sets. 
We begin by recalling that such an indexing category is well-defined as an $ \infty $-category. 
\begin{prop} \label{prop:spancatcoherence}
	Let $ G $ be a finite group. 
	Then there exists an $ \infty $-category $ \Span(\Fin_G) $ having
	\begin{itemize} 
		\item the same objects as $ \Fin_G $
		\item homotopy classes of morphisms from $ V $ to $ U $ in $ \Span(\Fin_G) $ are in bijection with diagrams $ {V \leftarrow T \to U} $ up to isomorphism of diagrams fixing $ V $ and $ U $. 
		\item The composite of $ V \leftarrow T \to U $ and $ U \leftarrow S \to W $ is equivalent to the diagram $ V \leftarrow T \times_U S \to W $.
	\end{itemize}
	Moreover, $ \Span(\Fin_G) $ is semiadditive, i.e. finite coproducts and products are isomorphic, and are given on underlying $ G $-sets by the disjoint union. 
\end{prop}
\begin{proof}
	The construction of $ \Span(\Fin_G) $ is \cite[Proposition 5.6]{BarwickMackey} applied to \cite[Example 5.4]{BarwickMackey}. 
	The (0-)semiadditivity of $ \Span(\Fin_G) $ follows from noticing that $ \Span(\Fin_G) $ is a module over $ \Span(\Fin) $ and \cite[Corollary 3.19]{MR4133704}. 
\end{proof}
\begin{defn} 
	Let $ G $ be a finite group and let $ \Span(\Fin_G) $ be the span category of Proposition \ref{prop:spancatcoherence}. 
	Let $ \cat $ be an additive $ \infty $-category.  
	Then the category of $ \cat $-valued \emph{$ G $-Mackey functors} is given by 
	\begin{equation*} 
		\Mack_G(\cat) := \Fun^\Sigma\left(\Span(\Fin_{G}), \cat\right) 
	\end{equation*}
	where the right-hand side denotes the full subcategory on functors which take direct sums in $ \Span(\Fin_{G}) $ to products in $ \cat $. 
	We will denote the category of \emph{genuine equivariant $ G $-spectra} by $ \Spectra^G = \Mack_G(\Spectra) $. 

	Let $ G/H $ be a transitive $ G $-set. 
	Note there is a map $ \Span(\Fin) \to \Span (\Fin_G) $ given by $ S \mapsto \sqcup_{S} G/H $. 
	Since this map preserves coproducts, restriction along this map induces a functor $ (-)^H \colon \Mack_G(\cat) \xrightarrow{} \Spectra(\cat) $ which lands in spectrum objects in $ \cat $. 
	This is the $ H $-\emph{categorical fixed points.}
	When $ H = \{e\} $ will sometimes call this the \emph{underlying object}.
\end{defn} 
Since our primary group of interest is $ G = C_2 $, hereafter we will understand `Mackey functor' (without a specified $ G $) as referring to a $ C_2 $-Mackey functor. 
Note that given a $ G $-Mackey functor $ X: \Span\left(\Fin_G\right) \to \Spc $, restriction along the inclusion $ \Fin_G^\op = \Span\left(\Fin_G, all, =\right) \to \Span\left(\Fin_G\right) $ gives a $ G $-space. 
However, a $ G $-Mackey functor comprises more data: given any map of finite $ G $-sets $ f:S \to T $, we will refer to its covariant image $ X(f): X(S) \to X(T) $ as the \emph{transfer map associated to $ f $}. 
\begin{ex} [Constant Mackey functors]\label{ex:constMackey}
	Given a discrete abelian group $ k $, we consider it equipped with the trivial involution and associate to it the constant Mackey functor $ \underline{k} $ with $ \underline{k}^{C_2} = \underline{k}^e = k $. 
	Since the inclusion $ \Mod_k^\heartsuit \to \Mod_k $ preserves direct sums, taking Eilenberg-Mac Lane spectra gives us a Mackey functor $ \underline{k} \in \Mack_{C_2}(\Spectra) $. 
\end{ex}
\begin{warning}\label{warning:diff_equivariant_linearizations}
	There are several competing notions of `linearization' in the genuine equivariant setting. 
	The $ C_2 $-Mackey functor $ \underline{\Z} $ should not be confused with $ \Z \otimes \sphere^0 $ (or $ \mathrm{triv}_{C_2}(\Z) $ in the notation of \cite{MR4419053}; also see \cite{MR2918295}) or the $ C_2 $-Burnside Mackey functor $ \mathbb{A}_{C_2} $ with $ \Z $-coefficients \cite{MR3212584}; the former is \emph{cohomological} in the sense of \cite[\S7]{MR1759612}. 

	On module categories, we see the forgetful functor $ \Mod_{\underline{k}} \to \Mack_{C_2}(\Mod_k) \simeq \Mod_{k \otimes \sphere^0} $ is far from being an equivalence. 
	In particular, the ($ k $-linearized) Burnside functor is not in the image of the forgetful functor. 
	See \cite[Proposition 16.3]{MR1261590} for a 1-categorical version of this statement. 
	Explicitly, the $ C_2 $-Mackey functors $ \underline{\Z} $ and $ \mathbb{A}_{C_2} $ satisfy
	\begin{equation*}
			\underline{\Z}^{C_2} = \Z \qquad  \mathbb{A}_{C_2}^{C_2} = \Z\{C_2, C_2/C_2\}
	\end{equation*}
	while the $ C_2 $-fixed points of $ \mathrm{triv}_{C_2}(\Z) $ lives in a pullback diagram (\ref{eq:recollementsq})
	\begin{equation*}
	\begin{tikzcd}[cramped]
			\mathrm{triv}_{C_2}(\Z)^{C_2} \ar[d] \ar[r] & \Z \ar[d] \\
			\Z^{hC_2} \ar[r] & \Z^{tC_2}
	\end{tikzcd} \,;
	\end{equation*}
	in particular, it is not discrete. 
\end{warning}
\begin{ntn} [{\cite[p.57]{MR979507}}] \label{ntn:lewis_diagram}
	 The data of a $ C_2 $-Mackey functor $ M $ can be summarized in a diagram of the form
	\begin{equation*}
	\begin{tikzcd}
		M^{C_2} \ar[d, bend right=20,"\mathrm{Res}"'] \\
		M^e \ar[u, bend right=20,"\mathrm{Tr}"'] \ar[loop below,"{m \mapsto \sigma(m)}",out=-60, in=240,distance=15]
	\end{tikzcd}  
	\end{equation*}
	where we write $ \sigma $ for the nontrivial element of $ C_2 $ and the lower loop expresses the $ C_2 $-action on the underlying object of $ M $. 
	Such a diagram is referred to as a \emph{Lewis diagram}. 
\end{ntn}
\begin{rmk}  
		Though our description of $ G $-spectra is quite algebraic, the Guillou--May theorem \cite[Theorem 0.1]{GuillouMay} implies that the category of $ G $-spectra may alternatively be regarded as a stabilization of the category of $ G $-spaces with respect to smashing by representation spheres. 
		In particular, the classical adjunction between $ G $-suspension spectra lifts to a $ G $-adjunction $ \Sigma^\infty_{G} \colon \underline{\Spc}^G_* \rlarrows \underline{\Spectra}^G \colon \loops^\infty_G $ whose value on each orbit $ G/H $ is the $ (\Sigma^\infty_{H},\loops^\infty_H) $ adjunction between $ H $-spaces and $ H $-spectra. 
\end{rmk}
\begin{recollection}
		If $ \cat $ is a presentable symmetric monoidal $ \infty $-category, then $ \Mack_{G}(\cat) $ inherits a symmetric monoidal structure via Day convolution so that \cite{BarwickGlasmanShah}. 
		We defer more detailed discussion of algebraic structures (and their parametrized enhancements) to \S\ref{subsection:gen_eqvt_alg}. 
\end{recollection}
\begin{ntn}
	If $ V $ is a $ G $-representation on a finite-dimensional real vector space, we write $ \Sigma^V(-) = \Sigma_{G}^\infty S^V \otimes - \colon \Spectra^G \to \Spectra^G $ and refer to it as suspension by the representation $ V $. 
	In particular, we have $ \Sigma^\rho(-) $ and $ \Sigma^\sigma(-) $ of Examples \ref{ex:regrep} and \ref{ex:signrepsphere}, resp. 
	We will use the notations $ \Sigma^\rho(-) $ and $ (-)[\rho] $ interchangeably, and likewise for arbitrary representation spheres. 
\end{ntn}
There is an alternative way of understanding $ \Mack_{G}(\cat) $ as a recollement \emph{when $ \cat $ is stable} and cocomplete. 
Since the complexity of this formula grows with the complexity of the lattice of subgroups of $ G $, we recall the result for $ G = C_p $ cyclic of prime order. 
\begin{prop}\label{prop:eqvtspectrarecollement} 
	There is an equivalence of stable $ \infty $-categories
	\begin{equation}\label{eq:eqvtspectrarecollement}
	\begin{split}
		&\Spectra^{C_p} = \Mack_{C_p}(\Spectra) \to \Spectra^{BC_p} \times_{\Spectra} Ar(\Spectra) \\
		&X \mapsto \left(X^e, \Phi^{C_p}X := \cofib\left((X^e)_{hC_p} \xrightarrow{\tr} X^{C_p}\right) \to (X^e)^{tC_p} \right)
	\end{split}
	\end{equation}  
	where the map $ Ar(\Spectra) \to \Spectra $ is evaluation at the target. 
	We call $ \Phi^{C_p} X $ the \emph{[$C_p$]-geometric fixed points} of $ X $. 
\end{prop}
Given a triple $ \left(X^e, \Phi^{C_p}X \to (X^e)^{tC_p} \right) $ on the right hand side of (\ref{eq:eqvtspectrarecollement}), we can reconstruct the $ C_p $-categorical fixed points of the image of the triple under the inverse equivalence of Proposition \ref{prop:eqvtspectrarecollement} from the following pullback square
\begin{equation}\label{eq:recollementsq}
\begin{tikzcd}
	X^{C_2} \ar[r] \ar[d] & \Phi^{C_2}X \ar[d] \\
	(X^e)^{hC_2} \ar[r] & X^{tC_2}
\end{tikzcd}.
\end{equation}
We will refer to the pullback square (\ref{eq:recollementsq}) as the \emph{recollement square for $X$}. 
We will refer to the exact sequence of spectra
\begin{equation}\label{diagram:isotropysep}
	X_{hC_2} \longrightarrow X^{C_2} \longrightarrow \Phi^{C_2}X 
\end{equation}
as the \emph{isotropy separation sequence for $ X $.}
There is also a relative version of Proposition \ref{prop:eqvtspectrarecollement}:
\begin{prop}\label{prop:rel_eqvt_recollement}
	Let $ A $ be an $ \E_\infty $-ring in $ \Spectra^{C_p} $. 
	Then there is a stable recollement
	\begin{equation*}
	\begin{tikzcd}
		\Mod_{A^e}^{BC_2} \ar[r, shift right=.2em] & \Mod_A(\Spectra^{C_2}) \ar[l,shift right=.2em,"{(-)^e}"'] \ar[r,"{\Phi^{C_2}}"', shift right=.2em] & \Mod_{\Phi^{C_2}A}(\Spectra) \ar[l,shift right=.2em]
	\end{tikzcd} .
	\end{equation*}  
\end{prop}
\begin{proof}
		The result follows from \cite[Theorem 1.2]{ShahRS}, which implies that the recollement of Proposition \ref{prop:eqvtspectrarecollement} is symmetric monoidal in the sense of \cite[Definition 2.20]{ShahRS}.  
\end{proof}
\begin{recollection}\label{rec:genuinepostnikov}
	(Postnikov t-structure on genuine $ C_2 $-spectra) \cites[Corollary 1.3.16]{DHLPS}[Example 6.3]{BarwickGlasmanShah} 
	There is a t-structure on $ \Spectra^{C_2} $ where 
	\begin{itemize}
		\item An object $ X $ is connective if $ X^e $ and $ X^{C_2} $ are both connective. 
		\item An object $ X $ is coconnective iff $ X^{e} $ and $ X^{C_2} $ are both coconnective. 
	\end{itemize}
	The heart of this t-structure is equivalent to $ C_2 $-Mackey functors valued in abelian groups. 

	Suppose $ X \in \Spectra^{C_2} $ such that $ X^e $ is $ n $-connective. 
	By the isotropy separation sequence (\ref{diagram:isotropysep}), $ X^{C_2} $ is $ n $-connective if and only if $ \Phi^{C_2} X $ is $ n $-connective. 
	It follows that the Postnikov filtration is symmetric monoidal, i.e. if $ X, Y \in \Spectra^{C_2} $ are $ n, m $-connective respectively, then $ X \otimes Y $ is $ (n+m) $-connective. 
\end{recollection} 
\begin{lemma}\label{lemma:postnikov_rightcomplete}
	Let $ A $ be a connective $ \E_\infty $-ring in $ \Spectra^{C_2} $. 
	Then the t-structure on $ \Mod_{A} = \Mod_{A}\left( \Spectra^{C_2}\right) $ is right-complete. 
\end{lemma}
\begin{proof}
	Because $ (-)^e $ and $ (-)^{C_2} $ are jointly conservative, the Postnikov t-structure on $ \Mod_{A} $ is right-separated. 
	Now notice that $ \Mod_{A} $ admits countable coproducts, and moreover countable coproducts of coconnective objects remain coconnective. 
	Then by \cite[Lemma A.6]{antieau-nikolaus}, the t-structure on $ \Mod_{A} $ is right-complete. 
\end{proof}
We recall the regular slice filtration on $ \Spectra^{C_2} $ of \cite{UllmanThesis}. 
\begin{defn}
	[{\cite[\S3]{UllmanThesis}}] \label{defn:reg_slice_cells}
	Let $ H \subseteq C_2 $ be a subgroup. 
	We will refer to a $ C_2 $-genuine equivariant spectrum of the form $ C_2 \otimes_H S^{n\rho_H} $ as a \emph{regular slice cell}, and say that the regular slice cell has \emph{dimension} $ n \cdot | H | $. 

	We write $ \tau^\rslice_{\geq n}\Spectra^{C_2} $ for the localizing subcategory of $ \Spectra^{C_2} $ generated by regular slice cells of dimension $ \geq n $, and describe the objects therein as \emph{regular slice $n$-connective}. 

	We say that a $ C_2 $-spectrum $ X $ is \emph{regular slice $ n $-coconnective} (cf. \cites[Definition 11.1.1]{MR4273305}[Definition 4.8]{MR4273305}) if, for all regular slice cells $ S $ of dimension $ > n $, the mapping space $ \hom_{\Spectra^{C_2}}\left(S, X\right) $ is contractible. 
\end{defn}
\begin{ex}\label{ex:spheres_as_slices}
	Let $ n \leq 0 $. 
	Then $ \sphere^{n \sigma} $ is a regular $ n $-slice, i.e. it is regular slice $ n $-connective but not regular slice $ (n+1) $-connective. 

	To show that $ \sphere^{n \sigma} $ is not regular slice $ (n+1) $-connective, it suffices to observe that if $ M \in \Spectra^{C_2}_{\geq n+1} $, then $ M^e $ must be (non-equivariantly) $ (n+1) $-connective.  
	To show that $ \sphere^{n \sigma} $ is regular slice $ n $-connective, we induct on $ |n| $ and use the exact sequence
	\begin{equation*}
	\begin{tikzcd}
		\sphere^{-n\sigma} \to \sphere^{-(n-1)\sigma} \to \Sigma^{-(n-1)\sigma} C_2 \simeq \Sigma^{-(n-1)} C_2
	\end{tikzcd}. 
	\end{equation*}
	For future reference, we note that for any $ m $, $ \sphere^{m\rho} $ is a regular $ (2m) $-slice. 
\end{ex}
\begin{obs}\label{obs:rslice_shift_by_rep_sphere}
		The equivalence $ -\otimes \sphere^{\rho} \colon \Spectra^{C_2} \xrightarrow{\sim} \Spectra^{C_2} $ restricts to an equivalence of full subcategories $ \tau^\rslice_{\geq n} \Spectra^{C_2} \xrightarrow{\sim} \tau^\rslice_{\geq n+2} \Spectra^{C_2} $. 
\end{obs}
\begin{rmk}
		The inclusion of regular slice $ n $-connective $ C_2 $-spectra into all $ C_2$-spectra admits a right adjoint $ \tau^\rslice_{\geq n } $. 
		It follows immediately from the definitions that a $ C_2 $-spectrum $ X $ is regular slice $ n $-coconnected if and only if $ \tau^\rslice_{\geq n}X \simeq 0 $. 
\end{rmk}
\begin{obs}\label{obs:rslice_more_connective}
	There are inclusions of regular slice-connective localizing subcategories as follows: $ \tau^\rslice_{\geq k}\Spectra^{G} \subset \tau^\rslice_{\geq n}\Spectra^G $ ($ \tau^\rslice_{\geq k}\Mod_{A} \subset \tau^\rslice_{\geq n }\Mod_{A}$) for all $ k \geq n $. 
\end{obs}
The following result allows us to check regular slice connectivity of a $ C_2 $-spectrum one orbit at a time; it should be compared to \cites[\S4.4.1]{MR3505179}[Theorem 6.14 \& Remark 6.15]{MR3007090}. 
\begin{lemma}\label{lemma:reg_slice_conn_conditions}  
	Let $ X \in \Spectra^{C_2} $ (resp. $ \Mod_A(\Spectra^{C_2}) $ where $ A $ is an $ \E_\infty $-algebra in $ \Spectra^{C_2} $ so that $ A^e $ and $ \Phi^{C_2}A $ are both connective). 
	Then the following are equivalent: 
	\begin{enumerate}[label=(\arabic*)]
		\item \label{lemitem:reg_slice_conditions_defn} The object $ X $ is regular slice $ n $-connective. 
		\item \label{lemitem:reg_slice_conditions_connectivity} $ X^e $ is $ n $-connective and $ \Phi^{C_2} X $ is $ \ceil{\frac{n}{2}} $-connective.
	\end{enumerate}
\end{lemma}
\begin{proof} [Proof of Lemma \ref{lemma:reg_slice_conn_conditions}] 
	We discuss the proof in the case $ A = \sphere^0 $; the general case follows from the same argument in view of \cite[Proposition 7.1.1.13]{LurHA}. 
	That \ref{lemitem:reg_slice_conditions_defn} implies \ref{lemitem:reg_slice_conditions_connectivity} follows readily from Definition \ref{defn:reg_slice_cells} and the fact that $ \Phi^{C_2} \left(\Sigma^\infty S^{m \rho}\right) = \Sigma^\infty \left( (S^{m\rho})^{C_2}\right) = \sphere^m $. 

	To show the converse statement, we introduce some notation which will be used in the proof. 
	Write $ \widetilde{E}\mathcal{P}_{C_2} $, $  E\mathcal{P}_{C_2} $ for the $ C_2 $-spaces 
	\begin{align*}
		\left(\widetilde{E}\mathcal{P}_{C_2}\right)^{C_2} &= * \qquad &\left(\widetilde{E}\mathcal{P}_{C_2}\right)^{e} &= S^0 \\
		\left(E\mathcal{P}_{C_2}\right)^{C_2} &= * \qquad& \left(E\mathcal{P}_{C_2}\right)^e &= \varnothing 
	\end{align*}
	of \cite[{\S}V.4]{MR1413302} (compare \cite[\S6]{MNN}). 
	We first show that the geometric spectrum $ \widetilde{E}\mathcal{P}_{C_2} \otimes \sphere^m $ is regular slice $ \geq n $ if $ m \geq \ceil{\frac{n}{2}} $. 
	The inequality implies that $ 2m \geq n $, hence in the exact sequence
	\begin{equation*}
		{E}\mathcal{P}_{C_2} \otimes \sphere^{m\rho} \to \sphere^{m \rho} \to \widetilde{E}\mathcal{P}_{C_2} \otimes \sphere^{m\rho} \simeq \widetilde{E}\mathcal{P}_{C_2} \otimes \sphere^{m}
	\end{equation*}
	the left and middle terms are both regular slice $ \geq n $. 
	By definition of a localizing subcategory, it follows that the right-hand term is also regular slice $ \geq n $. 
	Now for an arbitrary $ C_2 $-spectrum $ X $, consider the exact sequence
	\begin{equation*}
		{E}\mathcal{P}_{C_2} \otimes X \to X \to \widetilde{E}\mathcal{P}_{C_2} \otimes X \, .
	\end{equation*} 
	By our assumption that $ X^e $ is $ n $-connective, the left-hand term is regular slice $ \geq n $. 
	Now by the previous argument and our assumption that $ \Phi^{C_2} X $ is $ \ceil{\frac{n}{2}} $-connective, the right hand term is regular slice $ \geq n $. 
	Again by definition of a localizing subcategory, it follows that $ X $ is also regular slice $ \geq n $.
\end{proof}
\begin{lemma}\label{lemma:reg_slice_coconn_conditions}  
	Let $ X \in \Spectra^{C_2} $ (resp. $ \Mod_A(\Spectra^{C_2}) $ where $ A $ is an $ \E_\infty $-algebra in $ \Spectra^{C_2} $ so that $ A^e $ and $ \Phi^{C_2}A $ are both connective). 
	Suppose that $ X $ is regular slice $ n $-coconnective for $ n \leq 0 $. 
	Then $ X^e $ is $ n $-coconnective and $ X^{C_2} $ is $ \floor{\frac{n}{2}} $-coconnective. 
\end{lemma}
\begin{rmk}
		In view of Lemma \ref{lemma:reg_slice_conn_conditions}, the usage of categorical fixed points in Lemma \ref{lemma:reg_slice_coconn_conditions} (instead of geometric fixed points) may be surprising. 
		An illustrative example is $ \underline{\Z} $: Despite $ \Phi^{C_2}\underline{\Z} $ not being $ n $-coconnective for any $ n \geq 0 $, $ \underline{\Z} $ is regular slice $ 0 $-coconnective. 
		In particular, even though there exists a nontrivial map $ \Phi^{C_2}\left(\Sigma^2 \underline{\Z}\right) \to \Phi^{C_2}\left(\underline{\Z}\right) $, it does not arise as $ \Phi^{C_2} $ of any map of $ \underline{\Z} $-modules. 
		To see this, note that any map $ f \colon \Sigma^2 \underline{\Z} \to \underline{\Z} $ determines a diagram
		\begin{equation*}
		\begin{tikzcd}
		 		\tau_{\geq 2} \Z^{tC_2}\simeq \Phi^{C_2}\left(\Sigma^2 \underline{\Z}\right) \ar[r,"{\Phi^{C_2}f}"] \ar[d] & 	\Phi^{C_2}\left(\underline{\Z}\right) \simeq \tau_{\geq 0} \Z^{tC_2} \ar[d] \\ 
		 		(\Sigma^2 \Z)^{tC_2} \ar[r,"{(f^e)^{tC_2}}"] & \Z^{tC_2} \,.
	 	\end{tikzcd} 	
		\end{equation*} 
		Since $ \Z $ is coconnective, $ f^e $ is trivial, hence so is $ (f^e)^{tC_2} $. 
		Since the vertical maps are injective on $ \pi_* $, $ \Phi^{C_2} f $ must be the zero map. 
\end{rmk}
\begin{proof} [Proof of Lemma \ref{lemma:reg_slice_coconn_conditions}]
		We discuss the proof in the case $ A = \sphere^0 $; the general case follows from the same argument in view of \cite[Proposition 7.1.1.13]{LurHA}. 
		Recall that $ C_2 \otimes \sphere^m $ is a regular slice $ m $-cell. 
		If $ X $ is regular slice $ n $-coconnective, then the mapping space $ \hom_{\Spectra^{C_2}}\left(C_2 \otimes \sphere^m, X \right) \simeq \Omega^\infty \Sigma^{-m} X^e $ is contractible for all $ m > n $. 
		In particular, it follows immediately that $ X^e $ is $ n $-coconnective. 

		To prove the statement about $ X^{C_2} $, let us replace $ n $ by $ 2n+1 $ and induct on $ |n| $. 
		Suppose $ X $ is regular slice $ (-1) $-coconnective. 
		Then $ \sphere^{0\rho} = \sphere^0 $ is a regular slice $ 0 $-cell and the mapping space $ \hom_{\Spectra^{C_2}}\left(\sphere^0, X \right) \simeq \Omega^\infty X^{C_2} $ is contractible, so $ X^{C_2} $ is $ (-1) $-connective. 

		Now suppose $ X $ is regular slice $ (2n-1) $-coconnective. 
		By the inductive hypothesis, $ \Sigma^\rho X $ is regular slice $ (2n+1) $-coconnective, so $ (\Sigma^\rho X)^{C_2} $ is $ \floor{\frac{2n+1}{2}} = n $-coconnective. 
		Writing $ \sphere^\rho \simeq \Sigma \sphere^\sigma \simeq \cofib(\Sigma C_2 \to \Sigma C_2/C_2) $, we have an exact sequence of spectra $ \Sigma X^e \to \Sigma (X^{C_2}) \to (\Sigma^\rho X)^{C_2}  $.
		In particular, in the long exact sequence $ \cdots \to \pi_{\ell} X^e \to \pi_{\ell} X^{C_2} \to \pi_{\ell + 1} \left(\Sigma^\rho X\right)^{C_2} \to \cdots $, the left and right terms are both zero for all $ \ell + 1 \geq n+1 $, hence $ X^{C_2} $ is $ n-1 = \floor{\frac{2n-1}{2}} $-coconnective as desired. 		
\end{proof}

\subsection{Genuine equivariant algebra}\label{subsection:gen_eqvt_alg} 
In this section, we introduce $ G $-$ \E_\infty $-algebras and show that the Eilenberg--Mac Lane spectra associated to certain discrete Mackey functors (Notation \ref{ntn:fixpt_green_functor}) inherit a $ G $-$ \E_\infty $-algebra structure. 
Just as $ G $-Mackey functors are (roughly) obtained from abelian groups by allowing addition indexed by finite sets with $ G $-action (in addition to the existing addition indexed by finite sets), $ G $-$ \E_\infty $-rings are (roughly) obtained from $ \E_\infty $-rings by allowing multiplication indexed by finite $ G $-sets. 
In particular, the Hill--Hopkins--Ravenel norms $ N_{H}^G $ play the role of a smash product indexed by the $ G $-set $ G/H $.  
We will use the notion of $ C_2 $-$ \E_\infty $ rings introduced in \cite{NS22}, which are expected to agree with $ N_\infty $-algebras of \cite{MR3406512}. 

Since we do not need the full strength of parametrized $ \infty $-operads in this work, we sketch the definitions and constructions we will need for this paper; the interested reader is invited to peruse \cite{NS22} for more detail. 
\begin{recollection} [{\cite[\S2]{NS22}}]
		Write $ \Fin_{C_2} $ for the category of finite sets with $ C_2 $-action, i.e. the finite coproduct completion of $ \mathcal{O}_{C_2} $. 
		There is a parametrized $ \infty $-category $ \underline{\Fin}_{C_2,*} $ whose fiber over $ C_2/e $ is $ \left(\Fin_{C_2}\right)_{C_2/-/C_2} $ and whose fiber over $ C_2/C_2 $ is $ \left(\Fin_{C_2}\right)_{*/-/*} $. 
		The restriction map is given by pullback. 
		Now, given an object $ U \in \left(\Fin_{C_2}\right)_{T/-/T} $ and an orbit $ W \subseteq U $, write $ i_W \colon U \to W $ for the map which collapses $ U \setminus W $ to the basepoint $ T $. 
		A \emph{$ C_2 $-symmetric monoidal $ C_2 $-$ \infty $-category} is a cocartesian fibration $ p \colon \cat^\otimes \to \underline{\Fin}_{C_2,*} $ so that for all $ T \in \mathcal{O}_{C_2} $ and for all $ U \in \left(\Fin_{C_2}\right)_{T/-/T} $, the $ p $-cocartesian maps over $ i_W $ induce equivalences
		\begin{equation*}
				\cat^{\otimes}_{U} \xrightarrow{\sim} \prod_{W \in \mathrm{Orbit}(U)} \cat^\otimes_W \,. 
		\end{equation*}
		Given a $ C_2 $-symmetric monoidal $ C_2 $-$ \infty $-category $ p \colon \cat^\otimes \to \underline{\Fin}_{C_2,*} $, the $ C_2 $-$ \infty $-category $ C_2\E_\infty \underline{\Alg}(\cat) $ of $ C_2 $-$ \E_\infty $-algebras in $ \cat $ is the $ C_2 $-$ \infty $-category of sections of $ p $ which carry inert morphisms in $ \underline{\Fin}_{C_2,*} $ to $ p $-cocartesian morphisms. 
		If $ p \colon \cat^\otimes \to \underline{\Fin}_{C_2,*} $, $ q \colon \mathcal{D}^\otimes \to \underline{\Fin}_{C_2,*} $ are $ C_2 $-symmetric monoidal $ C_2 $-$ \infty $-categories, a $ C_2 $-symmetric monoidal functor from $ \cat $ to $ \mathcal{D} $ is a morphism of cocartesian fibrations from $ p $ to $ q $. 
\end{recollection}
\begin{ex}
		[{\cite[Example 2.4.2]{NS22}}] There is a $ C_2 $-symmetric monoidal $ C_2 $-$ \infty $-category $ \left(\underline{\Spectra}^{C_2}\right)^\otimes \to \underline{\Fin}_{C_2,*} $ whose underlying $ C_2 $-$ \infty $-category is $ \underline{\Spectra}^{C_2} $. 
		In particular, the $ p $-cocartesian morphism associated to the map $ C_2/e \to C_2/C_2 $ classifies the Hill--Hopkins--Ravenel norm $ N_e^{C_2} \colon \Spectra \to \Spectra^{C_2} $ \cite[Definition A.52]{MR3505179}. 
\end{ex} 
The (large) $ \infty $-category of presentable $ \infty $-categories with left adjoint functors, equipped with the symmetric monoidal Lurie tensor product, is an indispensable tool to higher category theory. 
Next, we recall the $ C_2 $-parametrized analogue of these ideas. 
We begin with the parametrized analogue of the condition that a functor $ F \colon \cat \times \mathcal{D} \to \mathcal{E} $ of ordinary $ \infty $-categories preserves colimits separately in each variable. 
\begin{recollection}[{\cites[Definition 3.2.4]{NS22}[\S3.3]{Nardinthesis}[Definition 5.16]{QSparam_Tate}}] \label{rec:distributivity_param_sym_mon_cat} 
		Let $ \cat^\otimes $ be a $ C_2 $-symmetric monoidal $C_2$-$ \infty $-category. 
		The $ C_2 $-symmetric monoidal structure on $ \cat $ is said to be $ C_2 $-\emph{distributive} if, roughly, it preserves $ C_2 $-colimits separately in each variable. 
		Let us note that $ \underline{\Spc}^{C_2}_* $ with the smash product $ C_2 $-symmetric monoidal structure and $ \underline{\Spectra}^{C_2} $ with the smash product and Hill--Hopkins--Ravenel norm $ C_2 $-symmetric monoidal structure are distributive (\cite[p.34]{NS22} and \cite[Corollary 3.28]{Nardinthesis}, resp.). 
\end{recollection}
\begin{recollection}
	There is a $ C_2 $-symmetric monoidal $ C_2 $-$ \infty $-category $ C_2\Pr^L $ of presentable $ C_2 $-$ \infty $-categories with morphisms given by distributive functors \cite[Definition 3.24]{Nardinthesis}. 
	By \cite[Theorem 5.1.4(3)]{NS22}, $ C_2\E_\infty\underline{\Alg}\left(C_2\Pr^L \right) $ has coproducts. 
	Given presentable $ C_2 $-$ \infty $-categories$ \cat $ and $ \mathcal{D} $, regard them as objects of $ C_2\E_\infty\underline{\Alg}\left(C_2\Pr^L \right) $ via $ C_2 $-cocartesian $ C_2 $-symmetric monoidal structure \cite[Example 2.4.1]{NS22}.
	Denote their coproduct in $ C_2\E_\infty\underline{\Alg}\left(C_2\Pr^L \right) $ by $ \cat \otimes_{C_2} \mathcal{D} $. 
\end{recollection}
Next, we show that $ C_2 $-Green functors which arise as the fixed point functors associated to discrete rings with involution may be regarded as $ C_2 $-$ \E_\infty $-ring spectra in a canonical way. 
\begin{ntn}\label{ntn:fixpt_green_functor}
		Let $ k $ be a commutative ring with involution endowed with an involution. 
		Write $ \underline{k} $ for the $ C_2 $-Green functor with $ \underline{k}^{C_2} = k^{C_2} $, where $ k^{C_2} $ denotes the strict fixed points of the $ C_2 $-action on $ k $, and $ \underline{k}^e = k $. 	
		When $ k $ is given the trivial involution, this agrees with Example \ref{ex:constMackey}. 
\end{ntn}
\begin{prop}\label{prop:constMackeyisnormed}
	Let $ k $ be a discrete commutative ring with a given involution. 
	Then the fixed point $ C_2 $-Green functor $ \underline{k} $ (Notation \ref{ntn:fixpt_green_functor}) canonically acquires the structure of a $ C_2 $-$ \E_\infty $-algebra. 
	Moreover, suppose that $ k' $ is another discrete commutative ring with an involution and $ f \colon k \to k' $ is a map of commutative rings respecting the involution. 
	Then $ f $ canonically induces a map $ \underline{k} \to \underline{k}' $ of $ C_2 $-$ \E_\infty $-rings. 
\end{prop}
\begin{proof}
		The result follows from a nearly identical argument to that of \cite[Theorem 5.1]{LYang_normedrings}; one only needs to observe that the strict fixed points $ k^{C_2} $ of the involution on $ k $ satisfies $ k^{C_2} \simeq \pi_0\left(k^{hC_2}\right) \simeq \tau_{\geq 0} k^{hC_2} $. 
\end{proof}
The $ C_2 $-$\E_\infty $-structure on $ \underline{k} $ allows us to define a relative norm. 
\begin{defn}\label{defn:relativenorm}
	Let $ A $ be a $ C_2 $-$ \E_\infty $-ring. 
	We define the \emph{relative norm} to be
	\begin{align*}
		\underline{N}^{C_2} \colon \Mod_{A^e}\left(\Spectra\right) & \to \Mod_{A}\left(\Spectra^{C_2}\right) \\
		M & \mapsto A \otimes_{N^{C_2}A^e} N^{C_2} M .
	\end{align*}
\end{defn}
\begin{prop}\label{prop:param_module_cat_are_distributive_sym_mon}
		Let $ \cat $ be a distributive $ C_2 $-symmetric monoidal $ C_2 $-$ \infty $-category and let $ A $ be a $ C_2 $-$ \E_\infty $-algebra in $ \cat $. 
		Then the $ C_2 $-$ \infty $-category $ \underline{\Mod}_A\left(\cat\right) $ admits a $ C_2 $-symmetric monoidal refinement. 
		Moreover, the $ C_2 $-symmetric monoidal structure on $ \underline{\Mod}_{A} $ is distributive in the sense of Recollection \ref{rec:distributivity_param_sym_mon_cat}. 
\end{prop}
\begin{ntn}\label{ntn:param_module_cats}
		In the setting of Proposition \ref{prop:param_module_cat_are_distributive_sym_mon}, we will write $ \underline{\Mod}_A $ for the $ C_2 $-$ \infty $-category $ \underline{\Mod}_A\left(\cat\right) $, we will write $ \Mod_A $ for the $ \infty $-category $ \underline{\Mod}_A^{C_2} $, and we will write $ \Mod_A^e $ for the underlying $ \infty $-category $ \Mod_{A^e} $. 
\end{ntn}
\begin{proof} [Proof of Proposition \ref{prop:param_module_cat_are_distributive_sym_mon}]
		The first statement can be proved using the same strategy as \cite[Proposition A.8]{LYang_normedrings}. 
		The latter statement follows from the fact that $ \cat $ was assumed to be distributive and $ C_2 $-colimits in $ \underline{\Mod}_A $ are computed in $ \cat $ Proposition \ref{prop:param_modules_colimits}.  
\end{proof}
\begin{rmk}\label{rmk:Tambara_model}
	For the purposes of this paper, a `$ C_p $-Tambara functor' is a $ C_p $-$ \E_\infty $-algebra object in $ \Mod_{\underline{k}}^\heartsuit $ (see Variant \ref{var:eqvtmodulespostnikov}) with respect to the box product and the norm on discrete Mackey functors \cite{MThesis}. 
	This agrees with the definition in terms of polynomial functors; see \cite{CHLL_norm_and_Tambara}.  
\end{rmk}
\begin{defn}[{\cites[Appendix A]{MR4411877}[Definition 3.13]{sulyma_prisms_tambara}}]\label{defn:cohomological_Tambara} 
		A $ G $-Tambara functor $ B $ is said to be \emph{cohomological} if for all subgroups $ K \leq H \leq G $ and $ x \in B^{H} $, $ N_{K}^H \mathrm{Res}^H_K(x) = x^{[H:K]} $.   
\end{defn}
The condition appearing in Definition \ref{defn:cohomological_Tambara} can be regarded as the multiplicative analogue of asking for Mackey functors to be cohomological. 
If $ M $ is a cohomological $ G $-Mackey functor, the free $ G $-Tambara functor on $ M $ is not necessarily cohomological \emph{as a Tambara functor}. 

\begin{variant}\label{var:eqvtmodulespostnikov}
	By \cites[Proposition 1.4.4.11]{LurHA}[Proposition A.15]{antieau-nikolaus}, there is a t-structure on $ \Mod_{\underline{k}}(\Spectra^{C_2}) $ where an object $ X $ is connective if $ X^e $ and $ X^{C_2} $ are both connective in $ \Spectra^{G} $, that is $ \Mod_{\underline{k}}(\Spectra^G)_{\geq 0} = \Mod_{\underline{k}}(\Spectra^G_{\geq 0}) $. 
	The heart of this t-structure is equivalent to modules over $ \underline{k} $ in $ C_2 $-Mackey functors. 
\end{variant}
The following is a special case of \cite[Theorem 1.3]{MR2205726}. 
\begin{prop}
	The $ \Mod_{\underline{k}}^\heartsuit $-tensor product given on Mackey functors $ M, N $ by $ \pi_0(M \otimes N) $ can be identified with the box product on $ C_2 $-Mackey functors of Lewis \cites[5]{lewis-green}[61]{MR979507}[9]{MR2941379}.
\end{prop}
\begin{rmk}
	Commutative algebras with respect to the symmetric monoidal structure on $ \Mod_{\underline{k}}^\heartsuit $ are often referred to as \emph{Green functors} \cite[19]{MR0384917}.
\end{rmk}
\begin{ntn}
	Let $ \underline{k} $ be the constant $ C_2 $-Mackey functor associated to a commutative ring $ k $. 
	Suppose $ S $ is a finite $ C_2 $-set. 
	We will write $ \underline{k}[S] := \underline{k} \otimes_{\sphere^0} \Sigma^\infty_{C_2,+} S \in \Mod_{\underline{k}}(\Spectra^{C_2}) $ and abbreviate $ \underline{k} = \underline{k}[C_2/C_2] $. 
\end{ntn}
\begin{warning}
	Let $ M \in \Mod_{\underline{k}} $. 
	We write $ M^\vee = \hom_{\underline{k}}(M, \underline{k}) $ for the dual in Mackey functors, which is not to be confused with a different type of duality which utilizes precomposing with the anti-autoequivalence $ \Span(\Fin_G) \simeq \Span(\Fin_G)^\op $ \cite[\S4]{MR1261590}.
\end{warning}

\section{Filtered and graded objects}\label{subsection:filgr}
Just as filtered and graded objects are indispensable to the ordinary Hochschild--Kostant--Rosenberg theorem, filtered and graded objects in parametrized $ \infty $-categories will play a key role in our main theorem. 
In \S\ref{subsection:fil_gr_definitions}, we introduce filtered and graded objects in a $ C_2 $-$ \infty $-category and show that they inherit parametrized enhancements of structural properties possessed by ordinary $ \infty $-categories of filtered and graded objects. 
In \S\ref{subsection:slice_trunc_on_fil_gr}, we introduce a parametrized version of the filtrations considered in \cite{Wil17} and discuss various filtrations on $ C_2 $-$ \infty $-categories of filtered and graded objects which will be useful to us later. 

While writing this section, this author wanted to prove a $ C_2 $-symmetric monoidal enhancement of the equivalence of \cite[Theorem 3.2.14]{Raksit20}, but was stymied by the absence of a parametrized Tannakian reconstruction result. 
We content ourselves with Proposition \ref{prop:complete_fil_as_cochain_cplx_monoidal} for now. 
While we expect similar statements to hold for filtered and graded objects in suitable $ G $-$ \infty $-categories for other finite groups $ G $, we do not pursue this matter here. 

\subsection{Definitions}\label{subsection:fil_gr_definitions}
In this section, we introduce filtered and graded objects in $ C_2 $-$ \infty $-categories and prove $ C_2 $-parametrized enhancements of many of the results contained in \cites[\S3.1-2]{Lurie-Rot}. 
In particular, we show that if $ \cat $ is a distributive $ C_2 $-symmetric monoidal $ C_2 $-$ \infty $-category, then the $ C_2 $-$ \infty $-categories of graded and filtered objects in $ \cat $ themselves admit $ C_2 $-symmetric monoidal structures given by parametrized Day convolution (Corollary \ref{cor:param_gr_fil_day_convolution}). 
Moreover, the associated graded $C_2$-functor is $ C_2 $-symmetric monoidal (Proposition \ref{prop:param_assoc_gr_is_C2_monoidal}). 
\begin{recollection}
	Consider $ \Z $ as a (1-)category with objects the integers and a unique morphism $ n \to m $ if $ n \geq m $. 
	We will abuse notation and denote the $ \infty $-categorical nerve of $ \Z $ similarly. 
	Let $ \Z^\delta $ be the category with the same objects but no nontrivial morphisms, i.e. a discrete category. 
	There is evidently an inclusion $ \iota: \Z^\delta \to \Z $. 
	The operation of addition on the integers endows both $ \Z^\delta $ and $ \Z $ with symmetric monoidal structures. 
\end{recollection}
\begin{defn}\label{defn:param_fil_gr}
	Let $ \cat $ be a $ C_2 $-$\infty$-category, and write $ \Z_{C_2} $, $ \Z^\delta_{C_2} $ for the constant $ C_2 $-$ \infty $-categories at $ \Z $ and $ \Z^\delta $, respectively \cite[Example 2.2]{BDGNS1}. 
	The $ C_2 $-$\infty$-category of \emph{filtered objects in $ \cat $} is given by $ {\Fil(\cat) := \underline{\Fun}(\Z_{C_2}, \cat)} $ and the $C_2 $-$\infty$-category of \emph{graded objects in $ \cat $} is $ \Gr(\cat) = \underline{\Fun}(\Z^\delta_{C_2}, \cat) $, where $ \underline{\Fun}(-,-) $ denotes the parametrized functor categories of Proposition \ref{prop:param_functors}.  
\end{defn}
\begin{rmk}\label{rmk:param_fil_gr_is_pointwise}
		Let $ \cat $ be a $ C_2 $-$\infty$-category. 
		A $ C_2 $-object in $ \Fil(\cat) $ is a morphism $ X \colon \Z \times \mathcal{O}^\op_{C_2} \to \cat $ of cocartesian fibrations over $ \mathcal{O}^\op_{C_2} $, or equivalently, an ordinary functor $ \Z \to \cat^{C_2} $. 
		Thus we see that under the equivalence of Remark \ref{rmk:param_unstraighten}, the $ C_2 $-$ \infty $-category $ \Fil(\cat) $ is given by the diagram $ \Fil\left(\cat^{C_2}\right) \to \Fil\left(\cat^e\right) \curvearrowleft C_2 $, where the $ C_2 $-action on $ \Fil\left(\cat^e\right) $ is inherited from $ \cat^e $. 
		We may abuse notation and denote a filtered or graded object by $ X_* $ where $ * \in \Z $. 
\end{rmk}
\begin{recollection}\label{rec:splitfiltobj} Suppose that $ \cat $ is a $ C_2 $-$\infty $-category which admits sequential colimits pointwise, which are preserved by the restriction functor $ \cat^{C_2} \to \cat^e $. 
\begin{enumerate}[label=(\alph*)]
	\item Given a filtered object $ X \colon \mathcal{O}^\op_{C_2} \to \Fil(\cat) $, we can associate to it its parametrized colimit $ |X|= \colim_{n \in \Z} X_{n} $ (compare \cite[Corollary 5.9]{Shah18}). 
	Since the underlying $ C_2 $-$ \infty $-category of $ \Z^{+}_{C_2} $ is constant, by \cite[Proposition 5.8]{Shah18} the parametrized colimit is computed pointwise by the usual colimit of filtered objects. 

	\item \label{recitem:exhaustive_filt} Restriction along the constant map $ c \colon \Z_{C_2} \to \underline{\{*\}} $ defines a $ C_2 $-functor $ c \colon \cat \to \Fil(\cat) $ sending every object $ Y \colon \mathcal{O}^\op_{C_2} \to \cat $ to the \emph{constant} filtered object $ c(Y) $. 
	A \emph{filtration on $ Y $} is a morphism of filtered objects $ \alpha \colon X_* \to c(Y) $. 
	A filtration on $ Y $ is \emph{exhaustive} if $ \alpha $ induces an equivalence on colimits $ |X| \simeq |c(Y)| \simeq Y $. 
\end{enumerate}
By \cite[Theorem 10.5]{Shah18}, there is a $ C_2 $-adjunction $ \colim \dashv c $. 
\end{recollection} 
\begin{ntn}\label{ntn:ins_ev}
	Given $ n \in \Z $, we can restrict along the inclusion $ \ev_n: \Gr(\cat) \to \cat $, $ \ev_n: \Fil(\cat) \to \cat $. 
	These functors admit fully faithful left $C_2$-adjoints $ \ins_n: \cat \to \Fil(\cat), \Gr(\cat) $. 
\end{ntn}
\begin{rmk}
Let $ \cat $ be a $ C_2 $-presentable $ C_2 $-$ \infty $-category. 
We can form the \emph{associated graded} of a filtered object, which participates in a $ C_2 $-adjunction:
\begin{equation}\label{eq:param_assoc_graded}
\begin{split} 	
	X_* &\mapsto \gr(X)_n := \cofib(X_{n+1} \to X_n) \\
	\gr \colon \Fil(\cat) & \rlarrows  \Gr(\cat) \colon \zeta \\
	\{\cdots \xrightarrow{0} X_n \xrightarrow{0} X_{n-1} \xrightarrow{0} \cdots \} & \mapsfrom  X_* \,.
\end{split}
\end{equation}	
\end{rmk}
\begin{defn}\label{defn:complete_fil}
		Let $ \cat $ be a $ C_2 $-stable $ C_2 $-$ \infty $-category which admits sequential limits, which are preserved by the restriction functor $ \cat^{C_2} \to \cat^e $.  
		A filtered object $ X_{*} $ is said to be \emph{complete} if $ \lim_{n \to \infty} X_n = 0 $. 
		We denote the full $C_2$-subcategory on complete filtered objects by $ {\Fil}^\wedge(\cat) $. 
\end{defn}
The next observation is an immediate consequence of its non-parametrized counterpart \cite[Lemma 2.15]{MR3806745} and Corollary \ref{cor:C2_left_adjoint_local_crit}. 
\begin{obs}
		Let $ \cat $ be a $ C_2 $-presentable $ C_2 $-stable $ C_2 $-$ \infty $-category which admits sequential limits, which are preserved by the restriction functor $ \cat^{C_2} \to \cat^e $.  
		The inclusion $ {\Fil}^\wedge(\cat) \inj {\Fil}(\cat) $ admits a left $C_2$-adjoint, which we call completion and denote by $ {(-)^\wedge} $.
\end{obs}
\begin{rmk}\label{rmk:und_spl}
		Restriction along the canonical fiberwise inclusion $ \Z^{\delta}_{C_2} \to \Z_{C_2} $ induces a $C_2$-functor $ \mathrm{und} \colon \Fil(\cat) \to \Gr(\cat) $ which associates to a filtered object its underlying graded object. 
		The functor $ \und $ admits a left $C_2$-adjoint $ \spl: \Gr(\cat) \to \Fil(\cat) $ given on objects by $ \spl(X_*)_j \simeq \bigsqcup_{i \geq j} X_i $.  
		We will say that a filtered object $ X $ is \emph{split} if there is an equivalence $ \spl(\gr(X)) \simeq X $ in $ \Fil(\cat) $. 
\end{rmk}
We need a bit more preparation in order to be able to make sense of filtered and graded equivariant algebras. 
\begin{cons}\label{cons:param_gr_fil_indexing_cats}
Let $ \Z^\delta $, $ \Z $ denote the monoidal categories used to define filtered and graded objects in Definition \ref{defn:param_fil_gr}.  
Consider the functors
\begin{align*}
	g \colon \Span\left(\Fin_{C_2} \right) &\to \Cat_\infty &\qquad f \colon \Span\left(\Fin_{C_2} \right) &\to \Cat_\infty \\
	U \simeq \bigsqcup_{W \in \mathrm{Orbit}(U)} W &\mapsto \prod_{W \in \mathrm{Orbit}(U)} \left(\Z^{\delta}\right) &\qquad U \simeq \bigsqcup_{W \in \mathrm{Orbit}(U)} W &\mapsto \prod_{W \in \mathrm{Orbit}(U)} \Z \\
	\left(W^{\sqcup n} \xrightarrow{\nabla} W \right) &\mapsto \left( \left(\Z^{\delta}\right)^n \xrightarrow{+} \Z^\delta \right) &\qquad \left(W^{\sqcup n} \xrightarrow{\nabla} W \right) &\mapsto \left( \Z^n \xrightarrow{+} \Z\right)  \\
	(W^{\sqcup n} \hookleftarrow W \colon \iota_i) & \mapsto \left( \left(\Z^{\delta}\right)^n \xrightarrow{p_i} \Z^{\delta}\right) &\qquad (W^{\sqcup n} \hookleftarrow W \colon \iota_i ) & \mapsto \left( \Z^n \xrightarrow{p_i} \Z\right) \\ 
	\left(C_2 \to C_2/C_2\right) & \mapsto \left( \Z^\delta \xrightarrow{\cdot 2} \Z^\delta\right) & \qquad \left(C_2 \to C_2/C_2\right) & \mapsto \left( \Z \xrightarrow{\cdot 2} \Z\right) \\
	\left(C_2/C_2 \leftarrow C_2\right) & \mapsto \left( \Z^\delta = \Z^\delta\right) & \qquad \left(C_2/C_2 \leftarrow C_2\right) & \mapsto \left( \Z = \Z\right)
\end{align*}
where $ \iota_i $ denotes inclusion of the $ i $th component, $ p_i $ is projection onto the $ i $th component, and $ W $ denotes some object in the orbit category $ \mathcal{O}_{C_2}^\op $. 
Notice that $ g $ and $ f $ are determined by their values on the aforementioned objects and morphisms. 

Consider the composites $ \underline{\Fin}_{C_2,*} \xrightarrow{s}  \Span\left(\Fin_{C_2} \right) \xrightarrow{g} \Cat_\infty  $ and $ \underline{\Fin}_{C_2,*} \xrightarrow{s}  \Span\left(\Fin_{C_2} \right) \xrightarrow{f} \Cat_\infty  $. 
Denote the corresponding Grothendieck constructions by $ \Z^{\delta,+}_{C_2} \to \underline{\Fin}_{C_2,*} $ and $ \Z^+_{C_2} \to \underline{\Fin}_{C_2,*} $, respectively. 
Notice that the inclusion $ \Z^\delta \to \Z $ induces a natural transformation $ g \implies f $, which in turn induces a morphism $ \Z^{\delta,+}_{C_2} \to \Z^+_{C_2} $ of cocartesian fibrations. 
\end{cons}
\begin{lemma}\label{lemma:gr_fil_promonoidal}
	The $ C_2 $-$ \infty $-categories $ \Z^{\delta,+}_{C_2} \to \underline{\Fin}_{C_2,*} $ and $ \Z^+_{C_2} \to \underline{\Fin}_{C_2,*} $ of Construction \ref{cons:param_gr_fil_indexing_cats} are $ \underline{\Fin}_{C_2,*} $-promonoidal in the sense of Definition 3.1.1 of \cite{NS22}. 
\end{lemma}
\begin{proof}
	Since $ \Z^{\delta,+}_{C_2} \to \underline{\Fin}_{C_2,*} $ and $ \Z^+_{C_2} \to \underline{\Fin}_{C_2,*} $ were defined as cocartesian fibrations over $ \underline{\Fin}_{C_2,*} $, it suffices to show that the structure morphisms exhibit $ \Z^{\delta, +}_{C_2} $ and $ \Z^+_{C_2} $ as $ C_2 $-$ \infty $-operads. 
	This follows from unravelling the definitions of $ f $ and $ g $, as the morphism spaces in $ \Z^\delta $ and $ \Z $ are either empty or contractible. 
\end{proof}
\begin{cor}\label{cor:param_gr_fil_day_convolution}
	For any distributive $ C_2 $-symmetric monoidal $ \infty $-category $ \cat^\otimes $ (Recollection \ref{rec:distributivity_param_sym_mon_cat}), there are $ C_2 $-symmetric monoidal $ \infty $-categories $ \Gr(\cat)^{\ostar} := \widetilde{\Fun}\left(\Z^\delta_{C_2}, \cat \right) $ and $ \Fil(\cat)^{\ostar} := \widetilde{\Fun}\left(\Z_{C_2}, \cat \right) $. 
	Their underlying $ C_2 $-$ \infty $-categories are given by $ \Gr(\cat) = \underline{\Fun}(\Z^\delta_{C_2}, \cat) $ and $ \Fil(\cat) = \underline{\Fun}(\Z_{C_2}, \cat) $, respectively. 
	Moreover, the symmetric monoidal structures on $ \Gr(\cat)_t $ and $ \Fil(\cat)_t $ for any $ t \in \mathcal{O}^\op_{C_2} $ agree with the usual Day convolution symmetric monoidal structure. 
\end{cor}
\begin{proof} 
	The first statement follows from Theorem 3.2.6 of \cite{NS22} and Lemma \ref{lemma:gr_fil_promonoidal}. 
	The descriptions of the underlying $ C_2 $-$ \infty $-categories follows from Proposition 3.1.9 of \emph{loc. cit.} 
	The last statement follows by definition of a $ C_2 $-colimit diagram (compare the proof of \cite[Corollary 5.9]{Shah18}). 
\end{proof} 
\begin{cor}
		Let $ A $ be a $ C_2 $-$ \E_\infty $-algebra in $ \Spectra^{C_2} $. 
		Then parametrized Day convolution induces a $ C_2 $-symmetric monoidal structure on $ \Gr\left(\underline{\Mod}_A\right) $ and $ \Fil\left(\underline{\Mod}_A\right) $. 
		On underlying $ \infty $-categories, this recovers the Day convolution symmetric monoidal structure on $ \Gr\left(\Mod_{A^e}(\Spectra)\right) $ and $ \Fil\left(\Mod_{A^e}(\Spectra)\right) $, resp.
\end{cor}
\begin{proof}
		Follows from Proposition \ref{prop:param_module_cat_are_distributive_sym_mon} and Corollary \ref{cor:param_gr_fil_day_convolution}. 
\end{proof}
\begin{rmk}
		In the situation of Notation \ref{ntn:ins_ev}, if $ \cat $ is in addition endowed with a $ C_2 $-symmetric monoidal structure, then $ \ins_0 $ is a $C_2$-symmetric monoidal functor, so $ \ev_0 $ is lax $C_2$-symmetric monoidal. 
\end{rmk}
\begin{variant}
		Suppose $ \cat $ is pointed. 
		We denote $ \Fil^{\geq 0}(\cat) = \Fun\left(\Z_{\geq 0,C_2}, \cat\right) $, equivalently given by the full $C_2$-subcategory of $ \Fil(\cat) $ on filtered objects $ X_* $ such that the map $ X_n \simeq X_{n-1} $ for $ n\leq 0$. 
		Similarly, we have a $ C_2 $-$ \infty $-category $ \Gr^{\geq 0}(\cat) = \Fun\left(\Z_{\geq 0,C_2}^\delta, \cat\right) $, equivalently given by the full $C_2$-subcategory of $ \Gr(\cat) $ on filtered objects $ X_* $ such that $ X_n \simeq * $ for $ n< 0$.  
		Restriction along the inclusions $ \Z_{\geq 0,C_2} \hookrightarrow \Z_{C_2} $ and $ \Z_{\geq 0, C_2}^\delta \hookrightarrow {\Z}^\delta_{C_2} $ induce $C_2$-adjunctions $ \ins^{\geq 0} \colon \Fil^{\geq 0}(\cat) \rlarrows \Fil(\cat) \colon \ev^{\geq 0} $ and $ \ins^{\geq 0} \colon \Gr^{\geq 0}(\cat) \rlarrows \Gr(\cat) \colon \ev^{\geq 0} $. 

		When $ \cat $ is $ C_2 $-symmetric monoidal, the functors $ \ins^{\geq 0} $ are $ C_2 $-symmetric monoidal, so $ \ev^{\geq 0} $ are lax $ C_2 $-symmetric monoidal.  
\end{variant} 
As with ordinary Day convolution, the norm on graded and filtered objects (i.e., Day convolution indexed by the $ C_2 $-set $ C_2 $) admits an explicit description. 
\begin{lemma}\label{lemma:norm_on_fil_gr_formula}
Let $ \left(\Spectra^{C_2}\right)^\otimes $ be the $ C_2 $-$ \infty $-category of $ C_2 $-spectra, endowed with the distributive $ C_2 $-symmetric monoidal structure of \cite[Example 2.4.2]{NS22}.
\begin{itemize}
	\item Let $ A \in \Spectra $ and write $ A(\ell) \in \Gr(\Spectra) $ for the graded object which is concentrated in degree $ \ell $. 
	Under the $ C_2 $-symmetric monoidal structure of Corollary \ref{cor:param_gr_fil_day_convolution} and (\ref{eq:recollementsq}), we have 
	\begin{equation*}
			\left(N^{C_2}\left(A(\ell)\right)\right)^e \simeq \left(A(\ell)\right)^{\ostar 2} \simeq \left(A \otimes A\right)(2 \ell) \qquad \Phi^{C_2}\left(N^{C_2}\left(A(\ell)\right)\right) \simeq A(2\ell) \,.
	\end{equation*}
	In particular, if $ B $ is any graded object, then $ \left(N^{C_2}(B)\right)^e_{m} \simeq \bigoplus_{i+j = m} B_i^e \otimes B_j^e $ and $ \Phi^{C_2}\left(N^{C_2}B\right)_{2m} \simeq B_m^e $. 
	\item Let $ A $ be a filtered object in $ \Spectra $. 
	Then the norm of $ A $ as a filtered object in $ \Spectra^{C_2} $ satisfies 
	\begin{align*}
			\left(N^{C_2}(A)\right)^e_m &\simeq \left(A^e \ostar^{\tau} A^e\right)_{m} &&\qquad \text{ in } \Spectra^{BC_2} \\
			\Phi^{C_2}\left(N^{C_2}(A)\right)_{2m} \simeq \Phi^{C_2}\left(N^{C_2}(A)\right)_{2m + 1} &\simeq A^e_m &&\qquad \text{ in } \Spectra \,,
	\end{align*}
	where $ C_2 $ acts on $ A^e \ostar^{\tau} A^e $ by \emph{both} permuting the factors of $ A^e $ and acting on both components with the given action on $ A^e $. 
\end{itemize}
\end{lemma}
\begin{proof}
		We prove the case of filtered objects; the case for graded objects is similar. 

		Taking $ \underline{\Fin}_{C_2,*} = \mathcal{O}^\otimes $, $ \cat^\otimes = \Z_{C_2}^+ $, and $ x = C_2 $, $ y = C_2/C_2 $ in Proposition 3.2.2 of \cite{NS22} gives a formula for the norm of a filtered object in terms of a parametrized left Kan extension. 
		Unraveling definitions, we see that $ \cat^\otimes_{\underline{x}} $ has \emph{non parametrized} fibers $ \left(\cat^\otimes_{\underline{x}}\right)_{C_2/C_2} \simeq \Z $, $ \left(\cat^\otimes_{\underline{x}}\right)_{C_2} \simeq \Z \times \Z $ while on the other hand, $ \left(\cat^\otimes_{\underline{y}} \right)_{C_2/C_2} \simeq \left(\cat^\otimes_{\underline{y}} \right)_{C_2} \simeq \Z $. 
		By a similar argument to that of the proof of \cite[Theorem 4.15]{LYang_normedrings}, a $ C_2 $-colimit of an $ \underline{\Spectra}^{C_2} $-valued diagram may be computed, under the recollement of Proposition \ref{prop:eqvtspectrarecollement}, as two ordinary colimit diagrams satisfying compatiilities. 
		The result follows immediately from these considerations. 
\end{proof}
\begin{obs}\label{obs:spl_und_param_monoidal}
		Consider the adjunction $ (\spl,\und) $ of Remark \ref{rmk:und_spl} and assume that $ \cat $ is a distributive $ C_2 $-symmetric monoidal $ C_2 $-$ \infty $-category \cite[Definition 3.2.4]{NS22}. 
		By \cite[Proposition 1.1.2.2 \& Corollary 2.4.6.5]{LurHTT}, $ \Z^{\delta,+}_{C_2} \to \Z^+_{C_2} $ is a fibration of $ C_2 $-$ \infty $-operads (\cite[Definition 2.2.1]{NS22}). 
		Taking $ \cat^\otimes = \Z^{\delta,+}_{C_2} $, $ \mathcal{O}^\otimes = \Z^+_{C_2} $, and $ \mathcal{P}^\otimes = \underline{\Fin}_{C_2,*} $ in \cite[Remark 4.3.6]{NS22}, it follows that $ \spl $ is $ C_2 $-symmetric monoidal and $ \und $ is lax $ C_2 $-symmetric monoidal. 
\end{obs}
\begin{lemma}\label{lemma:Rees_alg_C2_sym_mon}
	Write $ \sphere[t] = \mathrm{und}(\sphere^\fil) $ where $ \sphere^\fil $ is the unit object in filtered $ C_2 $-spectra. 
	Since $ \mathrm{und} $ is lax symmetric monoidal (Observation \ref{obs:spl_und_param_monoidal}), $ \sphere[t] $ has a $ C_2 $-$ \E_\infty $-algebra structure. 
	The forgetful $ C_2 $-functor 
	\begin{equation*}
		\mathrm{und} \colon \Fil\left(\underline{\Spectra}^{C_2}\right) \to \Gr\left(\underline{\Spectra}^{C_2}\right)
	\end{equation*}
	promotes to a $ C_2 $-symmetric monoidal equivalence
	\begin{equation*}
		\theta \colon \Fil\left(\underline{\Spectra}^{C_2}\right) \xrightarrow{\sim} \underline{\Mod}_{\sphere[t]}\left(\Gr\left(\underline{\Spectra}^{C_2}\right)\right)\,.
	\end{equation*} 
	On underlying $ \infty $-categories, $ \theta $ recovers the ordinary symmetric monoidal equivalence of \cite[Proposition 3.1.6]{Lurie-Rot} (also denoted $ \theta $). 
\end{lemma}
\begin{proof}
	That $ \theta $ is lax $ C_2 $-symmetric monoidal follows from Corollary \ref{cor:param_gr_fil_day_convolution}. 
	Building on that, it suffices to show that if $ X $ is a filtered $ C_2 $-spectrum, the canonical map $ N^{C_2} (\theta (X)) \ostar_{N^{C_2}(\sphere[t])} \sphere[t] \simeq \theta(N^{C_2}(X)) $ is an equivalence. 
	The result follows from Lemma \ref{lemma:norm_on_fil_gr_formula} and the observation that, as a module over $ N^{C_2} \left(\sphere[t] \right) $, there is an equivalence $ \sphere[t] \simeq  N^{C_2} \left(\sphere[t] \right) \oplus  N^{C_2} \left(\sphere[t] \right)(-1) $. 
\end{proof}
\begin{prop}\label{prop:param_assoc_gr_is_C2_monoidal}
	Let $ \cat $ be a distributive $ C_2 $-stable $ C_2 $-symmetric monoidal $ C_2 $-$ \infty $-category. 
	Then the associated graded functor (\ref{eq:param_assoc_graded}) 
	promotes to a $ C_2 $-symmetric monoidal functor $ \gr \colon \Fil(\cat)^{\ostar} \to \Gr(\cat)^{\ostar} $, where filtered and graded objects in $ \cat $ are given the $ C_2 $-symmetric monoidal structures of Corollary \ref{cor:param_gr_fil_day_convolution}. 
\end{prop}
\begin{proof}
	Our proof will be quite similar to \cite[\S3.2]{Lurie-Rot}. 
	We will consider the universal case $ \cat = \underline{\Spectra}^{C_2} $; the result for general $ \cat $ follows from the equivalences $ \Gr(\cat) \simeq \Gr\left(\underline{\Spectra}^{C_2}\right) \otimes \cat $ and $ \Fil(\cat) \simeq \Fil\left(\underline{\Spectra}^{C_2}\right) \otimes \cat $ in the (large) $ \infty $-category of $ C_2 $-presentable $ C_2 $-$ \infty $-categories. 

	Now we may regard $ \underline{\Spectra}^{C_2} \simeq \underline{\Fun}\left( \{0\}, \underline{\Spectra}^{C_2} \right) $ as a $ C_2 $-localization of $ \Fil\left(\underline{\Spectra}^{C_2}\right) $ via the inclusion $ \{0\} \subseteq \Z $. 
	This localization is compatible with the $ C_2 $-symmetric monoidal structures on $ \underline{\Spectra}^{C_2} $ and $ \Fil\left(\underline{\Spectra}^{C_2}\right) $ in the sense of \cite[\S2.9, in particular see Remark 2.9.3]{NS22}; this follows from the description of the norm in Lemma \ref{lemma:norm_on_fil_gr_formula}. 
	As a consequence, there is a unique $ C_2 $-$ \E_\infty $-algebra structure on the filtered $ C_2 $-spectrum $ \mathbb{A} $ characterized by 
	\begin{equation*}
		\mathbb{A}^n = \begin{cases} \sphere & \text{ if }n = 0 \\ 0 & \text{ else }\end{cases} 
	\end{equation*}
	so that the unit map $ \sphere^\fil \to \mathbb{A} $ restricts to an equivalence $ \sphere \xrightarrow{\sim} \mathbb{A}^0 $. 
	On underlying spectra, the map of algebras $ \sphere^\fil \to \mathbb{A} $ agrees with that of \cite[Proposition 3.2.5]{Lurie-Rot}. 

	Now the previous argument together with Lemma \ref{lemma:Rees_alg_C2_sym_mon} imply that there is an equivalence
	\begin{equation}\label{eq:modA_in_filtered_is_graded}
		\underline{\Mod}_{\mathbb{A}}\left(\Fil(\underline{\Spectra}^{C_2})\right) \xrightarrow{\sim} \underline{\Mod}_{\mathbb{A}}\underline{\Mod}_{\sphere[t]}\left(\Gr\left(\underline{\Spectra}^{C_2}\right)\right) \simeq \Gr\left(\underline{\Spectra}^{C_2}\right)\,.
	\end{equation}

	The result follows from noting that the associated graded $ C_2 $-functor is equivalent to the composite
	\begin{equation*}
		\Fil\left(\underline{\Spectra}^{C_2}\right)\xrightarrow{-\ostar \mathbb{A}} \underline{\Mod}_{\mathbb{A}}\left(\Fil\left(\underline{\Spectra}^{C_2}\right)\right) \xrightarrow{\simeq} \Gr\left(\underline{\Spectra}^{C_2}\right)\,,
	\end{equation*}
	where the latter equivalence is (\ref{eq:modA_in_filtered_is_graded}). 
\end{proof}
\begin{rmk}\label{rmk:bar_cons_as_bialgebra}
		Let $ \cat $ be a $ C_2 $-symmetric monoidal $ C_2 $-$ \infty $-category admitting geometric realizations. 
		Let $ A $ be an augmented $ C_2 $-$ \E_\infty $ algebra object in $ \cat $. 
		Since the forgetful $ C_2 $-functor $ \underline{\E}_1 \underline{\Alg}\left(C_2\E_\infty\underline{\Alg}(\cat)\right) \to C_2\E_\infty\underline{\Alg}(\cat) $ is an equivalence by \cite{Stewart_Ninfty}, using Proposition \ref{prop:param_cats_monoidal_const_operads}, we may take the bar construction in $ C_2\E_\infty\underline{\Alg}(\cat) $ to obtain $ \mathrm{Bar}(A) $ as a $ C_2 $-object in $ \underline{\coAlg} \left(C_2\E_\infty\underline{\Alg}(\cat)\right)$, i.e. a $ C_2 $-$ \E_\infty $-co-$ \underline{\E}_1 $-bialgebra object in $ \cat $.  
		Since $ C_2\E_\infty\underline{\Alg}(\cat) \to \cat $ preserves geometric realizations, the underlying $ \underline{\E}_1 $-coalgebra object of $ \mathrm{Bar}(A) $ agrees with the ordinary bar construction for $ \mathcal{O}^\op_{C_2} $-cocartesian families of augmented algebra objects.  
\end{rmk}
\begin{ntn}\label{ntn:homotopy_dual_numbers}
		Let $ \cat $ be a distributive $C_2$-presentable $C_2$-symmetric monoidal $C_2$-$ \infty $-category. 
		Write $ \D^\vee_{-} = \mathrm{Bar}(\mathbbm{1}^{\gr}[t]) $; by Remark \ref{rmk:bar_cons_as_bialgebra}, it is a $ C_2 $-$ \E_\infty $-co-$ \underline{\E}_1 $-bialgebra object in $ \Gr(\cat) $. 
		We write $ \D_{-} $ for the dual of $ \D^\vee_{-} $, which we regard as an $ \underline{\E}_1 $-co-$C_2$-$ \E_\infty $-bialgebra in $ \Gr(\cat) $ by Proposition \ref{prop:param_dual_takes_coalg_to_alg}. 
		It also follows from 
		Remark \ref{rmk:bar_cons_as_bialgebra} that the underlying object of $ \D_{-} $ admits a formula as in \cite[Notation 3.2.11]{Raksit20}:
		\begin{equation*}
				\D_{-} \simeq \begin{cases} \mathbbm{1} & \text{ if }n = 0 \\
													\mathbbm{1}[-1] & \text{ if }n = -1 \\
													0 & \text{ else. }  \end{cases}
		\end{equation*}
\end{ntn}
\begin{prop}\label{prop:complete_fil_as_cochain_cplx_monoidal}
	Let $ \cat $ be a distributive $C_2$-stable $C_2$-presentable $C_2$-symmetric monoidal $C_2$-$ \infty $-category, and let $ \mathbbm{1} $ denote the unit. 
	Suppose further that $ \cat $ admits $ C_2 $-limits indexed by $ \Z_{C_2} $. 
	There is a canonical $ \mathrm{Com}_{\mathcal{O}_{C_2}^\simeq} $-monoidal equivalence $ \overline{\gr} \colon \Fil^\wedge(\cat) \to \underline{\Mod}_{\D_-}(\cat) $ (see Notation \ref{ntn:homotopy_dual_numbers}) making the following diagram commute
	\begin{equation*}
	\begin{tikzcd}
		\Fil(\cat) \ar[d,"(-)^\wedge"'] \ar[r,"\gr"] & \Gr(\cat) \\ 
		\Fil^\wedge(\cat) \ar[r,"\overline{\gr}"] & \underline{\Mod}_{\D_-}(\Gr (\cat))\,. \ar[u,"U"']
	\end{tikzcd}
	\end{equation*}
\end{prop} 
\begin{rmk}\label{rmk:complete_fil_as_cochain_cplx_C2monoidal_conj}
		As noted at the beginning of \S\ref{subsection:filgr}, we expect the equivalence of Proposition \ref{prop:complete_fil_as_cochain_cplx_monoidal} to refine to one of $ C_2 $-symmetric monoidal $ C_2 $-$ \infty $-categories. 
\end{rmk}
\begin{proof} [Proof of Proposition \ref{prop:complete_fil_as_cochain_cplx_monoidal}]
		By Proposition \ref{prop:param_cats_monoidal_const_operads}, the result is equivalent to exhibiting an equivalence of $ \mathcal{O}_{C_2}^\op $-cocartesian families of symmetric monoidal $ \infty $-categories. 
		The result follows from observing that \cite[Theorem 3.2.14]{Raksit20} is suitably natural in colimit-preserving symmetric monoidal functors. 
\end{proof}

\subsection{Slices and truncations}\label{subsection:slice_trunc_on_fil_gr}
In this section, we introduce various filtrations on categories of graded and filtered objects which will be useful later. 
Unlike in \cite{Raksit20}, where the filtrations under consideration arise from t-structures, we will consider filtrations which do not arise from a t-structure. 
This can be thought of as reflecting the peculiarities of the $ C_2 $-equivariant category (which admits infinitely many possible filtrations which measure connectivity of underlying and geometric fixed point spectra at different speeds, cf. \cite{Wil17}). 
Alternatively, this might be regarded as reflecting the `motivic' nature of working $ C_2 $-equivariantly \cite{MR4045363}.  
\begin{defn}
		[{\cite[Definition 1.41]{Wil17}}] \label{defn:filtration_on_C2_cat}
		Let $ \cat $ be a $ C_2 $-stable $ C_2 $-$ \infty $-category. 
		A \emph{filtration} $ \mathscr{F} $ on $ \cat $ is a sequence of full $ C_2 $-subcategories 
		\begin{equation*}
		 	\cdots \subseteq \cat_{\geq n} \subseteq \cat_{\geq n-1} \subseteq \cdots \cat
		 \end{equation*} 
		 so that for each $ n $, the inclusion $ \cat_{\geq n} \subseteq \cat $ admits a right $ C_2 $-adjoint $ P^n \colon \cat \to \cat_{\geq n} $. 
		 Write $ \cat_{\leq n-1} $ for the full $ C_2 $-subcategory of $ \cat $ on those objects $ X \in \cat_t $ which $ P^n(X) = 0 \in (\cat_{\geq n})_t $. 
		 Write $ \mathscr{F}_n = \cat_{\geq n} \cap \cat_{\leq n} $ and refer to this as the subcategory of $ n $-slices. 
		 We will refer to $ \mathscr{F}_0 $ as the \emph{heart} of the filtration. 
\end{defn} 
This generalizes the filtrations we are familiar with. 
\begin{prop}\label{prop:filtrations_are_parametrized}
		Consider the sequences of full $ C_2 $-subcategories
		\begin{enumerate}[label=(\alph*)]
			\item $ \underline{\Spectra}^{C_2}_{\rslice \geq n} $ generated by those $ C_2 $-spectra which are regular slice $ n $-connective (Definition \ref{defn:reg_slicefiltration}).  
			\item $ \underline{\Spectra}^{C_2}_{\geq n} $ generated by those $ C_2 $-spectra which are $ n $-connective with respect to the Postnikov t-structure (Definition \ref{rec:genuinepostnikov}).  
		\end{enumerate} 
		Both sequences of $ C_2 $-subcategories define a filtration on $ \underline{\Spectra}^{C_2} $ in the sense of Definition \ref{defn:filtration_on_C2_cat}. 
		A similar result holds for $ \underline{\Mod}_A $, where $ A $ is a connective $ \E_\infty $-algebra in $ \Spectra^{C_2} $. 
\end{prop}
\begin{proof}
		We sketch the proof for $ A = \sphere^0 $; the general case is identical. 
		Note that the left adjoint $ C_2 \otimes - \colon \Spectra \to \Spectra^{C_2} $ to the restriction functor evaluated on an $ n $-connective ordinary spectrum is both regular slice $ n $-connective (see Lemma \ref{lemma:reg_slice_conn_conditions}) and Postnikov $ n $-connective. 
		The result now follows from Corollary \ref{cor:C2_right_adjoint_local_crit}. 
\end{proof}
\begin{defn}\label{defn:compatible_filtration}
	Let $ \cat $ be a $ C_2 $-symmetric monoidal $ C_2 $-$ \infty $-category. 
	A filtration $ \mathscr{F} $ on the underlying $ C_2 $-$ \infty $-category $ \cat $ is \emph{compatible} with the $C_2$-symmetric monoidal structure if the following conditions are satisfied:
	\begin{enumerate}[label=(\roman*)]
		\item The subcategory of coconnective objects $ \cat_{\leq 0} $ is closed under filtered colimits; 
		\item The unit object $ \mathbbm{1}_\cat $ lies in $ \cat_{\geq 0} $; 
		\item If $ X, Y \in \cat_{\geq 0} $, then $ X \otimes Y \in \cat_{\geq 0} $. 
		\item If $ X \in \cat_{\geq 0}^e $, then $ N^{C_2}_e X \in \cat_{\geq 0} $. 
	\end{enumerate} 
\end{defn} 
\begin{rmk}
		If $ \cat $ is a $ C_2 $-symmetric monoidal $ C_2 $-$ \infty $-category with a filtration which is compatible with the $ C_2 $-symmetric monoidal structure, then $ \cat^e $ is an ordinary symmetric monoidal $ \infty $-category equipped with a compatible filtration.
\end{rmk}
The following is an immediate consequence of \cite[Theorem 2.9.2]{NS22}. 
\begin{prop}\label{prop:compatible_filt_implies_C2_monoidal_heart}
		Let $ \cat^\otimes $ be a $ C_2 $-symmetric monoidal $ C_2 $-$ \infty $-category. 
		Suppose $ \cat $ is equipped with a filtration $ (\cat, \cat_{\leq 0}, \cat_{\geq 0}) $ which is compatible with the $C_2$-symmetric monoidal structure in the sense of Definition \ref{defn:compatible_filtration}. 
		Then 
		\begin{enumerate}[label=(\alph*)]
				\item $ \cat_{\geq 0} $ inherits a canonical $ C_2 $-symmetric monoidal structure so that the inclusion $ \cat_{\geq 0} \to \cat $ admits a $ C_2 $-symmetric monoidal structure. 
				\item The heart $ (\cat_{\geq 0})_{\leq 0} $ inherits the structure of a $C_2 $-symmetric monoidal $ \infty $-category so that $ \tau_{\leq 0} \colon \cat_{\geq 0} \to (\cat_{\geq 0})_{\leq 0} $ admits a $ C_2 $-symmetric monoidal structure. 
		\end{enumerate}
\end{prop}
\begin{defn}\label{defn:param_filt_on_gr}
Given a filtration $ \mathscr{F} $ on a $ C_2 $-$ \infty $-category $ \cat $, $ \Gr(\cat) $ inherits several filtrations. 
\begin{enumerate}[label=(\alph*)]
	\item A graded object $ X_* \in \Gr(\cat) $ is $ n $-connective in the \emph{neutral $ \mathscr{F} $-filtration} if $ X_* \in \cat_{\geq n} $ for all $ n $.
	\item An object $ X_* \in \Gr(\cat) $ is $n$-connective in the \emph{positive $ \mathscr{F}$-filtration} if $ X_* \in \cat_{\geq n+*} $ for all $ n $. 
	Denote the $ C_2 $-$ \infty $-category $ \Gr(\cat) $ equipped with this filtration by $ \Gr(\cat)_{\mathscr{F}^+} $.
	\item An object $ X_* \in \Gr(\cat) $ is $n$-connective in the \emph{negative $ \mathscr{F}$-filtration} if $ X_* \in \cat_{\geq n-*} $ for all $ n $. 
	Denote the $ C_2 $-$ \infty $-category $ \Gr(\cat) $ equipped with this filtration by $ \Gr(\cat)_{\mathscr{F}^-} $.
\end{enumerate} 
When $ \cat = \underline{\Spectra}^{C_2} $ or $ \underline{\Mod}_A $ for $ A $ a connective $ \E_\infty $-ring spectrum in $ \Spectra^{C_2} $, we will denote the latter two by $ \Gr \left(\underline{\Spectra}^{C_2}\right)_{\rslice^\pm} $ or $ \Gr \left(\underline{\Mod}_A \right)_{\rslice^\pm} $, respectively. 
\end{defn} 
\begin{rmk}\label{rmk:heart_of_regslice_filt_on_gr}
		Let $ A $ be a connective $ \E_\infty $-ring spectrum in $ \Spectra^{C_2} $. 
		Consider the negative regular slice filtration of Definition \ref{defn:param_filt_on_gr} for $ \cat = \underline{\Mod}_{A} $. 
		The heart of this filtration can be identified with the full $ C_2 $-subcategory of $ \underline{\Fun}\left(\Z^\delta_{C_2}, \Mod_{A}\right) $ on those functors $ F \colon \Z^\delta \times \left(\mathcal{O}^\op_{C_2}\right)_{t/} \to \left(\Mod_{A}\right)_{\underline{t}} $ so that for each $ n $, $ F(n,-) $ takes values in regular $ (-n) $-slices. 
		Similarly, the heart of the positive regular slice filtration of Definition \ref{defn:param_filt_on_gr} can be identified with the full $ C_2 $-subcategory of $ \underline{\Fun}\left(\Z^\delta_{C_2}, \underline{\Mod}_{A}\right) $ on those functors $ F \colon \Z^\delta \times \left(\mathcal{O}^\op_{C_2}\right)_{t/} \to \left(\underline{\Mod}_{A}\right)_{\underline{t}} $ so that for each $ n $, $ F(n,-) $ takes values in regular $ n $-slices. 
		We will denote these $ C_2 $-$ \infty $-categories by $ \Gr \left(\underline{\Mod}_A \right)_{\rslice^\pm}^\heartsuit $, or $ \Gr \left(\rslice_{\mp}^\heartsuit\right) $ if $ A $ is understood. 
\end{rmk}
\begin{obs}\label{obs:graded_rslice_heart_as_graded_Mackey_functors} 
		Let $ \underline{k} $ denote the fixed point $ C_2 $-Green functor associated to a discrete ring with involution. 
		Recall \cite[Corollary 8.9]{UllmanThesis} which characterizes regular $ 1 $-slices as those discrete cohomological $ C_2 $-Mackey functors whose restriction maps are injective and the fact that $ \otimes S^\rho $ defines an equivalence from the category of regular $ n $-slices to the category of $ (n+2) $-slices. 
		It follows that the functor $ - \left[-\ceil{\frac{n}{2}}\rho +\sigma \right] \simeq - \left[-\floor{\frac{n}{2}}\rho -1 \right] $ defines an equivalence\footnote{In writing this equivalence, we made a choice; we could have used the functor $ - \left[-\ceil{\frac{n}{2}}\rho +1 \right] \simeq - \left[-\floor{\frac{n}{2}}\rho -\sigma \right] $ which identifies regular $ n$-slices with those cohomological $ C_2 $-Mackey functors for which the transfer map $ M^e \to M^{C_2} $ is surjective.} from the category of regular $ n $-slices to the category of $ 1 $-slices when $ n $ is odd. 
		Thus, we have $ C_2 $-functors 
		\begin{equation*}
		\begin{split}	
			(- )\left[-\ceil{\frac{*}{2}}\rho +\varepsilon(*)\sigma \right] \colon \Gr \left(\underline{\Mod}_{\underline{k}} \right)_{\rslice^+}^\heartsuit &\to \Gr\left(\underline{\Mod}_{\underline{k}}^\heartsuit \right) \\
			X_* &\mapsto F(X)_n := X_n \left[-\ceil{\frac{n}{2}}\rho +\varepsilon(n)\sigma \right] \\
			(- )\left[\ceil{\frac{*}{2}}\rho -\varepsilon(*)1 \right] \colon \Gr \left(\underline{\Mod}_{\underline{k}} \right)_{\rslice^-}^\heartsuit &\to \Gr\left(\underline{\Mod}_{\underline{k}}^\heartsuit \right) 
		\end{split}
		\end{equation*}		
		(where $ \varepsilon(n) = 1 $ if $ n $ is odd and zero otherwise) which are fully faithful and have the same essential image: those graded $ \underline{k} $-modules $ M_n $ in $ C_2 $-Mackey functors so that the restriction map on $ M_n $ is injective for all $ n $ odd. 
		Write $ \Gr \left(\underline{\Mod}_{\underline{k}}\right)^{\heartsuit}_{r^\pm} $ for their essential image.  

		Write $ \Gr \left(\underline{\Mod}_{\underline{k}}\right)^{\heartsuit}_{r} $ for the heart of the neutral regular slice filtration on $ \Gr \left(\underline{\Mod}_{\underline{k}}\right) $; this category is canonically identified with $ \Gr \left(\underline{\Mod}_{\underline{k}}^{\heartsuit}\right) $.  
\end{obs}
\begin{prop}\label{prop:filtrations_are_compatible} 
		Let $ A $ a $ C_2 $-$ \E_\infty $-algebra in $ \Spectra^{C_2} $ which is connective with respect to the Postnikov t-structure. 
		Then the neutral, positive, and negative regular slice filtrations of Definition \ref{defn:param_filt_on_gr} (also see Proposition \ref{prop:filtrations_are_parametrized}) are all compatible with the $C_2$-symmetric monoidal structure on $ \Gr \left(\underline{\Mod}_{A}\right) $ from Corollary \ref{cor:param_gr_fil_day_convolution}. 	
\end{prop} 
\begin{rmk}
		The analogue of Proposition \ref{prop:filtrations_are_compatible} does not hold for the Postnikov filtration, because $ N^{C_2} $ does not take $ n $-connective spectra to $ 2n $-connective $ C_2 $-spectra unless $ n \leq 0 $. 
\end{rmk}
\begin{proof} [Proof of Proposition \ref{prop:filtrations_are_compatible}]
		The result follows from the definition of compatibility in view of Lemma \ref{lemma:reg_slice_conn_conditions} and Lemma \ref{lemma:norm_on_fil_gr_formula}. 
\end{proof}
\begin{rmk}\label{rmk:equivariant_koszul_rule}
		Let $ \underline{k} $ be the fixed point $ C_2 $-Mackey functor associated to a discrete ring with an involution. 
		By Proposition \ref{prop:compatible_filt_implies_C2_monoidal_heart}, there are induced $ C_2 $-symmetric monoidal structures on the hearts of the neutral, positive, and negative regular slice filtrations on $ \Gr\left(\underline{\Mod}_{\underline{k}}\right) $. 
		In particular, the positive and negative regular slice filtrations induce a \emph{Koszul} $ C_2 $-symmetric monoidal structure $ \ostar_K $ on $ \Gr \left(\underline{\Mod}_{\underline{k}}\right)^{\heartsuit}_{r^\pm} $, while the neutral slice filtration induces the Day convolution $ C_2 $-symmetric monoidal structure $ \ostar_D $ on $ \Gr \left(\underline{\Mod}_{\underline{k}}^{\heartsuit}\right) $. 
		Identifying $ \Gr \left(\underline{\Mod}_{\underline{k}}\right)^{\heartsuit}_{r^\pm} $ with full subcategories of $ \Gr\left(\underline{\Mod}_{\underline{k}}^\heartsuit\right) $ using Observation \ref{obs:graded_rslice_heart_as_graded_Mackey_functors}, this agrees with the usual $ C_2 $-symmetric monoidal structure on graded Mackey functors, with tensor product and norm given by parametrized Day convolution, but with two changes. 
		In addition to the symmetry isomorphism incorporating the Koszul sign convention, the $ C_2 $-action on the norm is twisted by $ (-1) $.   
		In particular, if $ A $ is a $ C_2 $-commutative algebra in $ \Gr\left(\underline{\Mod}_{\underline{k}}^\heartsuit\right) $ with respect to the Koszul sign rule and $ a \in A^e_{2\ell+1} $ is in odd degree, then $ n(a) = - n(\sigma(a)) \in A^{C_2}_{4\ell+2} $. 
		We will write $ \Gr \left(\underline{\Mod}_{\underline{k}}\right)^{\heartsuit, \ostar_K}_{r^\pm} $ for the $ C_2 $-category $ \Gr \left(\underline{\Mod}_{\underline{k}}\right)^{\heartsuit}_{r^\pm} $ with the Koszul $ C_2 $-symmetric monoidal structure.\footnote{We have not compared our notion of $ C_2 $-graded commutative algebra with that of \cite{MR3818296}, but let us note that Angeltveit--Bohmann consider a more general notion: Their indexing category is the representation ring $ RO(G) $.} 
\end{rmk}
An important aspect of \cite{Raksit20} is relating positive and negative t-structures on graded objects; in particular, this allows Raksit to define cohomology for a $ h_+ $-cochain complex. 
We prove an involutive enhancement of this correspondence in Proposition \ref{prop:shear_gr_by_reg_rep_sphere_is_mult}. 
In order to do so, we first introduce a parametrized enhancement of the Picard space for a $ C_2 $-symmetric monoidal $ C_2 $-$ \infty $-category. 
\begin{ntn}\label{ntn:picard_C2_spectrum}
		Let $ \cat $ be a $ C_2 $-symmetric monoidal $C_2$-$ \infty $-category. 
		Write $ \underline{\Pic}(\cat) $ for its corresponding \emph{Picard space}; at each orbit, $ \underline{\Pic}(\cat) $ is the ordinary Picard space of $ \cat^{C_2/H} $. 
		Since invertible objects are closed under symmetric monoidal functors, $ \underline{\Pic}(\cat) $ is a grouplike $ C_2 $-$ \E_\infty $-monoid in $ C_2 $-spaces. 
		In particular, we may regard it as the infinite loop space associated to a connective $ C_2 $-spectrum. 
\end{ntn}
\begin{obs}\label{obs:C2_picard_of_module_cat}
		Let $ A $ be a connective $ C_2 $-$ \E_\infty $-algebra in $ \Spectra^{C_2} $. 
		As in \cite[Example 2.2.2]{MR3590352}, there is an isomorphism $ \underline{\pi}_1 \Pic\left(\Mod_A\right) \simeq \underline{\pi}_0\left(\loops^\infty A\right)^\times $ of $ C_2 $-Mackey functors, where $ \underline{\pi}_0\left(\loops^\infty A\right)^\times $ is the $ C_2 $-Mackey functor of units in the $C_2$-Tambara functor $ \underline{\pi}_0 A $. 
		Furthermore, there is an equivalence of $ C_2 $-spaces $ \tau_{\geq 1}\loops \underline{\Pic}\left(\underline{\Mod}_A\right) \simeq \tau_{\geq 1} \left(\loops^\infty A\right)^\times $. 
\end{obs}
Let $ \cat $ be a $ \underline{\Z} $-linear $ C_2 $-stable $ C_2 $-presentable distributive $ C_2 $-symmetric monoidal $ C_2 $-$ \infty $-category and consider the $ C_2 $-functor which shears  
\begin{equation}\label{eq:shear_gr_by_reg_rep_sphere}
\begin{split}	
	\Gr\left(\cat\right) &\xrightarrow{[\pm \rho *]} \Gr\left(\cat\right) \\
	X_* &\mapsto F(X)_n := \Sigma^{\pm \rho n} X_n .	
\end{split}
\end{equation}
\begin{prop}\label{prop:shear_gr_by_reg_rep_sphere_is_mult} 
Let $ \cat $ be a $ \underline{\Z} $-linear $ C_2 $-stable $ C_2 $-presentable distributive $ C_2 $-symmetric monoidal $ C_2 $-$ \infty $-category. 
\begin{enumerate}[label=(\arabic*)]
	\item The shear functor $ [\pm \rho *] $ of (\ref{eq:shear_gr_by_reg_rep_sphere}) is a $ C_2 $-symmetric monoidal equivalence of $ C_2 $-$\infty$-categories with inverse $ [\mp \rho *]: X_* \mapsto \Sigma^{\mp \rho *} X_n $. 
	Suppose $ \cat = \underline{\Mod}_A $ for $ A $ a connective $ C_2 $-$ \E_\infty $-$ \underline{\Z} $-algebra. 
	Then the functors $ [\pm \rho *] $ promote to filtered equivalences 
	\begin{equation*}
	\begin{split}
			[\rho *] \colon \Gr \left(\underline{\Mod}_{A}\right)_{\rslice^{-}}\rlarrows \Gr \left(\underline{\Mod}_{A}\right)_{\rslice^{+}} \colon [-\rho *]
	\end{split}	
	\end{equation*}
	between the negative and positive regular slice filtrations of Definition \ref{defn:param_filt_on_gr} on $ \Gr \left(\underline{\Mod}_{A}\right) $. 
	In particular, they induce inverse equivalences on hearts
	\begin{equation*}
		[\rho *] \colon \Gr \left(\underline{\Mod}_{A}\right)_{\rslice^{-}}^\heartsuit \rlarrows \Gr \left(\underline{\Mod}_{A}\right)_{\rslice^+}^\heartsuit \colon [-\rho *] \,.
	\end{equation*}
	\item On underlying spectra, $ [\pm \rho *] $ recovers the equivalences $ [\pm 2 *] $ of \cite[Proposition 3.3.4]{Raksit20}. 
\end{enumerate}
\end{prop} 
\begin{proof} [Proof of Proposition \ref{prop:shear_gr_by_reg_rep_sphere_is_mult}]
\begin{enumerate}[label=(\arabic*)]
		\item Similarly to \cite[Proof of Proposition 3.3.4]{Raksit20}, a $ C_2 $-functor 
		\begin{equation}\label{eq:shear_coefficients}
		\begin{split}
			\psi_{\pm} \colon &\Z^\delta_{C_2} \to \underline{\Pic}(\underline{\Mod}_{\underline{\Z}}) \\
			& n \mapsto \Sigma^{\pm n\rho_{C_2} }\underline{\Z}
		\end{split}
		\end{equation}
		defines an autoequivalence of $ \Gr(\cat) $ via the adjoint of 
		\begin{equation*}
		\begin{split}
			\Z^\delta_{C_2} \times \Gr(\cat) \xrightarrow{(\psi_{\pm}\circ \pi_1, \ev)} \underline{\Pic}(\underline{\Mod}_{\underline{\Z}}) \times \cat \xrightarrow{\otimes} \cat 
		\end{split}. 
		\end{equation*}
		By the universal property of the parametrized Day convolution \cite[Proposition 3.1.7]{NS22}, it suffices to upgrade the map of spaces (\ref{eq:shear_coefficients}) to a morphism of grouplike $C_2$-$ \E_\infty $-$ C_2 $-spaces. 
		Using Observation \ref{obs:C2_picard_of_module_cat}, consider the Whitehead tower of the Picard $ C_2 $-space
		\begin{equation}\label{eq:picard_whitehead_tower}
		\begin{tikzcd}[row sep=small]
			{\underline{\pi}_1\left(\Pic(\Mod_{\underline{\Z}}) \right)[1]} \ar[r] & {\tau_{\geq 1} \Pic(\Mod_{\underline{\Z}})} \ar[r] & \underline{\pi}_0 \Pic(\Mod_{\underline{\Z}})  \ar[d,"\simeq"] \\
			{\underline{\Z/2\Z}[1]} \ar[u,"{\simeq}"] \ar[r] & {\Pic(\Mod_{\underline{\Z}})}\ar[u,"\simeq"] \ar[r] & A
		\end{tikzcd}
		\end{equation}
		where $ A $ is the Burnside $ C_2 $-Mackey functor. 
		To obtain our desired map, it suffices to show that $ \underline{\Z} \to {\pi_0 \Pic(\Mod_{\underline{\Z}})} \simeq A $ classified by $ [C_2] \in A^{C_2} $ admits lifts (as maps of connective $ C_2 $-spectra) along each stage of this tower. 

		Since the map $ \underline{\Z} \xrightarrow{[C_2]} A $ is homotopic to the composite $ \underline{\Z} \xrightarrow{[C_2/C_2]} A \xrightarrow{tr} A $, the composite $ \underline{\Z} \to A \xrightarrow{\kappa} \underline{\Z/2\Z}[2] $ is homotopic to the composite $ \underline{\Z} \xrightarrow{[C_2/C_2]} A \xrightarrow{\kappa} \underline{\Z/2\Z}[2]  \xrightarrow{tr} \underline{\Z/2\Z}[2] $, which is canonically nullhomotopic because the transfer map is zero on the constant Mackey functor with value $ \Z/2\Z $. 
	\item Let $ \psi_{\pm} $ be as in the proof of (1); taking underlying $ \infty $-groupoids, $ \psi_{\pm}^e $ are maps of grouplike $ \E_\infty $-spaces. 
	The result follows from the observation that $ \psi_{\pm}^e $ are canonically identified with the maps $ \pm \phi $ which are used to define the functors $ [\pm 2*] $ in \cite[Proposition 3.3.4]{Raksit20}. \qedhere
\end{enumerate}
\end{proof}
Now, let us consider filtrations on $ \Fil\left(\cat\right) $. 
Doing so will allow us to show that given a $ C_2 $-$ \E_\infty $ algebra $ A $ in $ \Spectra^{C_2} $, the regular slice-connective covers $ \tau_{\geq *}^\rslice A $ interact nicely with the normed multiplication structure on $ A $ (Remark \ref{rmk:regslicelaxsymmmon}). 
\begin{ntn}\label{ntn:reg_slice_connective}
		Let $ A $ be a $ C_2 $-$ \E_\infty $-algebra in $ \Spectra^{C_2} $ which is connective. 
		Let $ \Fil\left(\underline{\Mod}_A\right)_{\geq *}^\rslice $ denote the full $ C_2 $-subcategory of those filtered objects $ X_{\geq *} $ such that $ X_n $ is regular slice $ n $-connective for all $ n $. 
		By Lemma \ref{lemma:reg_slice_conn_conditions}, its fiber over $ C_2 $ is $ \Fil\left(\Mod_{A^e}(\Spectra)\right)^{\mathrm{P}}_{\geq *} $, where the latter denotes the category of \cite[Construction 3.3.6]{Raksit20}, i.e. consisting of those filtered objects $ X^\star $ so that $ X^n $ is $ n $-connective for all $ n $. 
		Write $ \iota $ for the inclusion $ C_2 $-functor $ \iota: \Fil\left(\underline{\Mod}_A\right)_{\geq *}^\rslice \subseteq \Fil\left(\underline{\Mod}_A\right) $, which exists by Observation \ref{obs:rslice_more_connective}.
\end{ntn}
\begin{prop}\label{prop:reg_slice_conn_cover}
		Let $ A $ be a $ C_2 $-$ \E_\infty $-algebra in $ \Spectra^{C_2} $ which is connective. 
		\begin{enumerate}[label=(\alph*)]
			\item \label{propitem:reg_slice_conn_cover_as_param_functor} The inclusion $ C_2 $-functor $ \iota: \Fil\left(\underline{\Mod}_A\right)_{\geq *}^\rslice \subseteq \Fil\left(\underline{\Mod}_A\right) $ admits a right $ C_2 $-adjoint $ \tau_{\geq * }^R $. 
			\item \label{propitem:reg_slice_conn_cover_is_param_monoidal} The $ C_2 $-functor $ \iota $ is $C_2$-symmetric monoidal, and $ \tau_{\geq *}^R $ is lax $ C_2 $-symmetric monoidal. 
		\end{enumerate}
\end{prop}
\begin{proof}
		Note that $ \Fil\left(\underline{\Mod}_A\right) $ and $ \Fil\left(\underline{\Mod}_A\right)_{\geq *}^\rslice $ have $ C_2 $-coproducts, which are preserved by the inclusion $ \iota $. 
		Now \ref{propitem:reg_slice_conn_cover_as_param_functor} follows from Corollary \ref{cor:C2_right_adjoint_local_crit}. 
		To prove \ref{propitem:reg_slice_conn_cover_is_param_monoidal}, note that the subcategory is compatible with the $ C_2 $-symmetric monoidal structures on $ \Fil\left(\underline{\Mod}_A\right) $ (Corollary \ref{cor:param_gr_fil_day_convolution}) in the sense of \cite[Proposition 2.9.1]{NS22} by Lemmas \ref{lemma:reg_slice_conn_conditions} and \ref{lemma:norm_on_fil_gr_formula}. 
		The result follows from Theorem 2.9.2 of \emph{loc. cit.}
\end{proof} 
\begin{defn} [cf. \cite{UllmanThesis,Wil17}] \label{defn:reg_slicefiltration}
	Let $ A $ be a $ C_2 $-$ \E_\infty $-algebra in $ \Spectra^{C_2} $ which is connective. 
	Denote the diagonal $ C_2 $-functor by $ \delta: \underline{\Mod}_A \to \Fil\left(\underline{\Mod}_A\right) $ and let $ \tau_{\geq *}^R $ denote the connective cover with respect to the regular slice filtration (Proposition \ref{prop:reg_slice_conn_cover}).
	use $ \tau_{\geq *}^\rslice $ to denote the composite
	\begin{align*}
		\underline{\Mod}_A \xrightarrow{\delta} \Fil\left(\underline{\Mod}_A\right) \xrightarrow{\tau_{\geq *}^{R}} \Fil\left(\underline{\Mod}_A\right)_{\geq *}^\rslice 
	\end{align*}
	which we call the \emph{regular slice filtration}.
\end{defn}
\begin{rmk}\label{rmk:regslicelaxsymmmon}
	The regular slice filtration functor $ \tau_{\geq *}^\rslice $ is lax $ C_2 $-symmetric monoidal since $ \delta $, $ \tau_{\geq *}^R $ both are. 
	In particular, it takes $ C_2 $-$ \E_\infty $-algebra objects to $ C_2 $-$ \E_\infty $-algebra objects. 
\end{rmk}
\begin{prop}\label{prop:regsliceright_complete_sep}
	Let $ A $ denote a connective $ \E_\infty $-algebra in $ \Spectra^{C_2} $. 
\begin{enumerate}[label=(\arabic*)]
	\item \label{propitem:regslicerightsep} The regular slice filtration is \emph{right separated}, i.e. the intersection $ \bigcap_{n \in \Z} \tau^\rslice_{\leq n} \Mod_A\left(\Spectra^{C_2}\right) $ is zero. 
	\item \label{propitem:regslicerightcomplete} The regular slice filtration on $ \Mod_{A} $ is \emph{right complete}, i.e. the canonical $ C_2 $-functor $$ \displaystyle \underline{\Mod}_{A} \to \lim \left(\cdots \to \underline{\Mod}_{A, \rslice \geq -1} \to \underline{\Mod}_{A, \rslice \geq 0}\right) $$ is an equivalence. 
\end{enumerate}
\end{prop}
\begin{cor}\label{cor:colim_reg_slicefil}
	Let $ A $ denote a connective $ \E_\infty $-algebra in $ \Spectra^{C_2} $. 
	The composite
	\begin{equation*}
		\underline{\Mod}_{A} \xrightarrow{\tau_{\geq *}^\rslice} \Fil\left(\underline{\Mod}_{A}\right) \xrightarrow{\colim} \underline{\Mod}_{A} 
	\end{equation*}
	is canonically equivalent to the identity functor. 
\end{cor}
\begin{proof}
		[Proof of Proposition \ref{prop:regsliceright_complete_sep}] 
		Part \ref{propitem:regslicerightsep} follows from Lemma \ref{lemma:reg_slice_coconn_conditions} and right-separatedness of the Postnikov t-structure on $ \Spectra^{C_2} $. 
		In view of the fact that limits in diagram categories are computed pointwise, Part \ref{propitem:regslicerightcomplete} follows from \ref{propitem:regslicerightsep}. 
\end{proof}

\section{Bialgebras and their modules}\label{subsection:param_bialg} 
Let $ S^\sigma $ denote the circle-with-flip-action, regarded as a $ C_2 $-monoid object in $ \Spc^{C_2} $. 
Then there is an equivalence of $ C_2 $-$ \infty $-categories 
\begin{equation}\label{eq:functor_cat_as_module_cat}
 	\underline{\Mod}_{\underline{\Z}[S^\sigma]} \simeq \underline{\Fun}\left(BS^\sigma, \underline{\Mod}_{\underline{\Z}}\right)
\end{equation} 
which recovers the equivalence $ \Mod_{\Z[S^1]} \simeq \Fun\left(BS^1, \Mod_{\Z} \right) $ on underlying $ \infty $-categories. 
Moreover, we may regard (\ref{eq:functor_cat_as_module_cat}) as an equivalence of $ \mathcal{O}^\op_{C_2} $-families of symmetric monoidal $ \infty $-categories, which recovers the pointwise symmetric monoidal structure on $ \Fun\left(BS^1, \Mod_{\Z} \right) $. 
However, the $ C_2 $-monoid structure endows $ \underline{\Mod}_{\underline{\Z}[S^\sigma]} $ with additional structure: the \emph{twisted diagonal} map $ S^\sigma \to \prod_{C_2} S^1 $, $ x \mapsto (x, \sigma x) $ induces a map of $ C_2 $-$ \E_\infty $-algebras $ \Delta^\tau \colon \underline{\Z}[S^\sigma] \to \underline{N}^{C_2} \Z[S^1] $ where $ \underline{N}^{C_2} $ is the relative norm of Definition \ref{defn:relativenorm}. 
Given a $ \Z[S^1] $-module $ M $, we may canonically regard $ \underline{N}^{C_2} M $ as a $ \underline{\Z}[S^\sigma] $-module via the map $ \Delta^\tau $. 

In \S\ref{subsection:param_cocart_tensor}, we show that the aforementioned map is part of a \emph{$ C_2 $-symmetric monoidal structure} on $ \underline{\Mod}_{\underline{\Z}[S^\sigma]} $. 
Much of our material on $ C_2 $-symmetric monoidal structures on modules over a $ C_2 $-bialgebra directly generalizes the main results of \cite[\S2]{Raksit20} to the $C_2$-parametrized context, and also serves a similar purpose: We will use these notions to define filtered involutive circle actions and involutive chain complexes, which do not admit a convenient description in terms of functor categories. 
In \S\ref{subsection:parametrized_Tate_cons}, we discuss generalizations of orbits, fixed points, and the Tate construction, closely following \cite{QSparam_Tate}. 
Certain constructions and proofs in this section follow from a fiberwise/pointwise application of the corresponding result for ordinary $ \infty $-categories; when this occurs, we only give an indication of the necessary modifications (if any).  

\subsection{Parametrized bialgebras and tensor products}\label{subsection:param_cocart_tensor} 
In this section, we introduce $ C_2 $-bialgebras and show that modules over a $ C_2 $-bialgebra inherit a $ C_2 $-symmetric monoidal structure. 
Our proofs will take us on a lengthy digression on parametrized (co)cartesian symmetric monoidal structures. 
The reader who is interested only in real trace theories may wish to skip the material between Proposition \ref{prop:models_of_bialg} and Definition \ref{defn:bi_C2_commutative_bialg}, though the reader who is more interested in parametrized $ \infty $-categories and parametrized $ \infty $-operads may find that material of independent interest. 
Part of this section is written in slightly greater generality than needed in the rest of the paper; if assuming our finite group $ G $ is $ C_2 $ did not simplify proofs substantially, then we did not make such an assumption. 
However, we did not concern ourselves with proving the most general statement possible for parametrized $ \infty $-categories. 
\begin{defn}\label{defn:eqvt_coalg}
	Let $ \cat $ be a $ G $-symmetric monoidal $ G $-$ \infty $-category. 
	We say that a $ G $-$ \E_\infty $-coalgebra object of $ \cat $ is a $ G $-$ \E_\infty $-algebra object of $ \cat^\vop $. 

	We write $ G\E_\infty \underline{\coAlg}(\cat) $ for the $ G $-$ \infty $-category $ G \E_\infty \underline{\Alg} (\cat^\vop)^\vop $ of $ G $-$ \E_\infty $-coalgebra objects in $ \cat $.  
\end{defn}
\begin{ntn}
		In the following, $ \underline{\coAlg}(\cat) $ and $ \underline{\Alg}(\cat) $ will denote (co)algebras with respect to the operad $ \underline{\E}_1 $ of Definition \ref{defn:operad_induces_G_operad}.
\end{ntn}
The $ G $-symmetric monoidal structure on $ \cat $ induces canonical $ G $-symmetric monoidal structures on the $ G $-$ \infty $-categories $ G\E_\infty \underline{\coAlg}(\cat) $ and $ G\E_\infty \underline{\Alg}(\cat) $ \cite[\S5.3]{NS22}. 
Thus, we may consider parametrized algebras endowed with the structure of both an algebra and a coalgebra. 
\begin{prop}\label{prop:models_of_bialg}
	Let $ \cat $ be a $ G $-symmetric monoidal $ G $-$ \infty $-category. 
	Then there are canonical equivalences of $ G $-symmetric monoidal $ G $-$ \infty $-categories
	\begin{equation*}
	\begin{split}
		 	G\E_\infty \coAlg(G\E_\infty \Alg(\cat)) \simeq  G\E_\infty \Alg\left(G\E_\infty \coAlg(\cat)\right)  \qquad \qquad \qquad \\ 
		 	G\E_\infty \coAlg(\Alg(\cat)) \simeq  \Alg\left(G\E_\infty \coAlg(\cat)\right) \qquad \qquad \coAlg(G\E_\infty \Alg(\cat)) \simeq  G\E_\infty \Alg\left(\coAlg(\cat)\right)
  \end{split}		 
	\end{equation*}
	commuting with the ($G$-symmetric monoidal) forgetful functors to $ \cat $. 
	On underlying symmetric monoidal $ \infty $-categories, this recovers \cite[Proposition 2.1.2]{Raksit20} (see also \cite[Corollary 3.3.4]{EllipticI}). 
\end{prop}
\begin{proof}
		The first equivalence follows from Lemma \ref{lemma:C2_bialg_universality}. 
		The latter two equivalences follow from a parametrized adaptation of the argument in \cite[Proposition 2.1.2]{Raksit20}; let us verify that the necessary ingredients for the argument carry over to the parametrized setting. 
		Firstly, note that Nardin--Shah have shown that for any $ G $-distributive $ G $-symmetric monoidal $ G $-$ \infty $-category $ \mathcal{D} $, there is a free-forgetful $ G $-adjunction $ G \E_\infty\Alg(\mathcal{D}) \to \mathcal{D} $ which is fiberwise monadic \cite[Theorem 4.3.4 \& Corollary 5.1.5]{NS22}. 
		Furthermore, any $ G $-symmetric monoidal $ G $-$ \infty $-category $ \mathcal{D} $ can be embedded into a $ G $-distributive $ G $-symmetric monoidal $ G $-$ \infty $-category of presheaves, endowed with the parametrized Day convolution $ G $-symmetric monoidal structure \cite[Corollary 6.0.12]{NS22}. 
		Additionally, the $ G $-$ \infty $-category $ G \E_\infty\Alg(\mathcal{D}) $ has finite $ G $-coproducts which are computed by the parametrized tensor product \cite[\S5.3]{NS22}. 
		The result follows from universal properties of $ G $-cocartesian $ G $-$ \infty $-operads of Definition \ref{defn:param_cartesian_operad_conditions}, which are proved in Lemma \ref{lemma:alg_over_cocart_operad} and Proposition \ref{prop:alg_in_cocart_operad}. 		 
\end{proof} 
\begin{recollection}
		[{\cite[Definition 5.2.1]{NS22}}]  
		Let $ \mathcal{T} $ be an orbital $ \infty $-category. 
		A $ \mathcal{T} $-$ \infty $-operad $ \mathcal{O}^\otimes $ is \emph{unital} if for all orbits $ V \in \mathcal{T} $ and objects $ x \in \mathcal{O}_V^\otimes $, the space of multimorphisms $ \mathrm{Mul}_{\mathcal{O}}(\varnothing, x) $ is contractible. 
\end{recollection} 
\begin{defn}\label{defn:param_cartesian_operad_conditions}
		Let $ \cat $ be a $ \mathcal{T} $-$ \infty $-category. 
		We will say that a $ \mathcal{T} $-symmetric monoidal structure $ \cat^\otimes $ on $ \cat $ is \emph{$ \mathcal{T} $-cartesian} if it satisfies
		\begin{enumerate}[label=(\arabic*)]
			\item The unit $\mathcal{T}$-object $ 1_\cat $ is $ \mathcal{T} $-final 
			\item For each pair of objects $ X, Y \in \cat_t $, the canonical maps $ X \otimes 1_t \leftarrow X \otimes Y \to 1_t \otimes Y $ exhibit $ X \otimes Y $ as a product of $ X $ and $ Y $ in the $ \infty $-category $ \cat_t $
			\item For each morphism $ \alpha \colon s \to t $ in $ \mathcal{T} $, the functor $ \cat_s \to \cat_t $ classified by $ s = s \xrightarrow{\alpha} t $ in $ \Span(\Fin_{\mathcal{T}}) $ is right adjoint to the restriction functor $ \cat_t \to \cat_s $ classified by $ \alpha $ regarded as a morphism in $ \mathcal{T}^\op $. 
		\end{enumerate}
		There is an analogous dual characterization of those $ \mathcal{T} $-$\infty $-operads which are $ \mathcal{T} $-cocartesian. 
		We will denote $ \mathcal{T} $-cartesian $ \mathcal{T} $-symmetric monoidal structures by $ \cat^\Pi $ and $ \mathcal{T} $-cocartesian $ \mathcal{T} $-symmetric monoidal structures by $ \cat^\sqcup $. 

		A $ \mathcal{T} $-$ \infty $-operad $ \mathcal{O}^\otimes $ is \emph{cartesian} (resp. \emph{cocartesian}) if it is equivalent to $ \cat^\Pi $ (resp. $ \cat^\sqcup $) for some $ \mathcal{T} $-$ \infty $-category $ \cat $.  
\end{defn}
\begin{rmk}
		Unravelling definitions, it is not too difficult to see that Definition \ref{defn:param_cartesian_operad_conditions} agrees with the parametrized (co)cartesian operad of \cite[Example 2.4.1]{NS22}. 
\end{rmk}
\begin{lemma}\label{lemma:C2_bialg_universality}
		Let $ \cat $ be a $ G $-symmetric monoidal $ G $-$ \infty $-category. 
		\begin{enumerate}[label=(\alph*)]
				\item \label{lemma_item:param_bialg_is_param_semiadd} The $ \infty $-category $ G\E_\infty \underline{\coAlg}\left(G\E_\infty \underline{\Alg}(\cat)\right) $ of $ G $-bialgebra objects is $ G $-semiadditive in the sense of \cite[Definition 2.19]{Nardinthesis}. 
				\item \label{lemma_item:tensor_on_bialg_is_bicartesian} The $ G $-symmetric monoidal structure on $ \cat $ induces a $G $-symmetric monoidal structure on $ G\E_\infty \underline{\coAlg}\left(G\E_\infty \underline{\Alg}(\cat)\right) $ which is both $ G $-cartesian and $ G $-cocartesian. 
				\item Let $ \mathcal{E} $ any $ G $-semiadditive $ G $-$ \infty $-category, and regard $ \mathcal{E} $ as a $ G $-symmetric monoidal $ G $-$ \infty $-category with its $ G $-cartesian (equivalently, $ G $-cocartesian) $G$-symmetric monoidal structure. 
				Then the forgetful functor $ G\E_\infty \underline{\coAlg}\left(G\E_\infty \underline{\Alg}(\cat)\right) \to \cat $ induces an equivalence of $G$-$ \infty $-categories $ \underline{\Fun}^{G\otimes}\left(\mathcal{E},G\E_\infty \underline{\coAlg}\left(G\E_\infty \underline{\Alg}(\cat)\right) \right) \xrightarrow{\sim} \underline{\Fun}^{G\otimes}(\mathcal{E},\cat ) $. 
		\end{enumerate}	
\end{lemma}
\begin{proof}
		We prove \ref{lemma_item:tensor_on_bialg_is_bicartesian} first; \ref{lemma_item:param_bialg_is_param_semiadd} follows from \ref{lemma_item:tensor_on_bialg_is_bicartesian} by unravelling definitions. 
		Recall that the $ G $-symmetric monoidal structure on $ G\E_\infty\underline{\Alg}(\cat) $ is $ G $-cocartesian \cite[Corollary 5.3.8]{NS22}. 
		Applying Theorem 5.1.3 \emph{ibid.} to $ G\E_\infty\underline{\Alg}(\cat)^\vop $, we deduce that $ G\E_\infty \underline{\coAlg}\left(G\E_\infty \underline{\Alg}(\cat)\right) $ admits finite $ G $-coproducts, and that the forgetful functor $ G\E_\infty \underline{\coAlg}\left(G\E_\infty \underline{\Alg}(\cat)\right)  \to G \E_\infty \underline{\Alg}(\cat) $ preserves finite $ G $-coproducts. 
		Moreover, \cite[Corollary 5.3.8]{NS22} also implies that the symmetric monoidal structure on $ G\E_\infty \underline{\coAlg}\left(G\E_\infty \underline{\Alg}(\cat)\right) $ is cartesian.  
		To show that the symmetric monoidal structure on $ G\E_\infty \underline{\coAlg}\left(G\E_\infty \underline{\Alg}(\cat)\right) $ is also $ G $-cocartesian, we must check the three conditions of Definition \ref{defn:param_cartesian_operad_conditions}; the verification of the first two conditions are straightforward parametrized generalizations of those in the proof of \cite[Proposition 3.3.3]{EllipticI}. 
		It remains to verify
		\begin{itemize}
			\item [(3)] Let $ G/K \to G/H $ be any map of orbits, and write $ R $ for the restriction functor $ G\E_\infty \underline{\coAlg}\left(G\E_\infty \underline{\Alg}(\cat)\right)_{G/H} \to G\E_\infty \underline{\coAlg}\left(G\E_\infty \underline{\Alg}(\cat)\right)_{G/K} $. 
			The norm functor $ N_K^{H} \colon \cat^K \to \cat^{H} $ lifts to a right adjoint of $ R $ (which we also denote by $ N_K^{H}$) by $ G $-cartesianness of the $ G $-symmetric monoidal structure on $ G \E_\infty \underline{\coAlg}\left(G \E_\infty \underline{\Alg}(\cat)\right) $. 
			We must check that $ N_K^{H} $ is also the left adjoint to $ R $. 
			Since the forgetful functor $ \theta \colon G\E_\infty \underline{\coAlg}\left(G\E_\infty \underline{\Alg}(\cat)\right) \to G\E_\infty\underline{\Alg}(\cat) $ creates $ G $-colimits, it suffices to show that $ N_K^{H} $ computes the left adjoint to $ G\E_\infty\underline{\Alg}(\cat)_{G/H} \to G\E_\infty\underline{\Alg}(\cat)_{G/K} $. 
			This follows immediately from the fact that the $ G $-symmetric monoidal structure on $ G \E_\infty \underline{\Alg}(\cat) $ is cocartesian \cite[Corollary 5.3.8]{NS22}. 
		\end{itemize}
		We now prove part (c). 
		Write the forgetful functor as a composite 
		\begin{equation*}
			\underline{\Fun}^{G\otimes}\left(\mathcal{E},G\E_\infty \underline{\coAlg}(G\E_\infty \underline{\Alg}(\cat)) \right) \xrightarrow{\varphi_1} \underline{\Fun}^{G\otimes}\left(\mathcal{E},G\E_\infty \underline{\Alg}(\cat)\right) \xrightarrow{\varphi_2} \underline{\Fun}^{G\otimes}(\mathcal{E},\cat )  \,.
		\end{equation*} 
		Since the $ G $-symmetric monoidal structure on $ \mathcal{E} $ is $ G $-cocartesian, by Lemma \ref{lemma:alg_over_cocart_operad}, we can identify $ \varphi_2 $ with the forgetful functor $ \underline{\Fun}\left(\mathcal{E}, G\E_\infty\underline{\Alg}\left(G\E_\infty\underline{\Alg}(\cat)\right)\right) \to \underline{\Fun}\left(\mathcal{E}, G\E_\infty\underline{\Alg}(\cat)\right) $. 
		This functor is an equivalence by Proposition \ref{prop:alg_in_cocart_operad}. 
		A similar line of reasoning shows that $ \varphi_1 $ is an equivalence. 
\end{proof}
\begin{ntn}
		Let $ \cat $ be a $ G $-$ \infty $-category and assume that $ \alpha\colon U \to V $ is a map in $ \mathcal{O}_G $. 
		In the following lemma, we will write $ \mathrm{Res}^V_U $ for the restriction functor $ \alpha^* \colon \cat_V \to \cat_U $ associated to $ \alpha $. 
		If $ \alpha^* $ has a right adjoint $ \alpha_* $, we may denote it by $ \mathrm{coInd}_U^V \colon \cat_U \to \cat^V $. 
\end{ntn}
\begin{cons}\label{cons:approx_to_cocart_operad}
		Let $ G $ be a finite group and let $ \cat $ be a $ G $-$ \infty $-category which admits all finite $ G $-coproducts. 
		Let $ \cat^\sqcup \to \underline{\Fin}_{G,*} $ be the associated $ G$-cocartesian operad of \cite[Example 2.4.1]{NS22}. 
		There is a canonical $ G $-functor
		\begin{align*}
				\Theta \colon \cat \times \underline{\Fin}_{G,*} &\to \cat^\sqcup \\
				\left(c \in \cat_T, W \to T\right) &\mapsto\left(\left(\mathrm{Res}_U^T(c)\right)_{U}\in \prod_{U \in \mathrm{Orbit}(W)}\cat_U \right) \,.
		\end{align*}
\end{cons} 
\begin{lemma}\label{lemma:alg_over_cocart_operad} 
		Let $ \mathcal{D} $ be a $ G $-symmetric monoidal $ G $-$ \infty $-category, and let $ \cat $ a $ G $-$ \infty $-category which admits finite $ G $-coproducts. 
		Then restriction along the functor $ \Theta $ of Construction \ref{cons:approx_to_cocart_operad} induces an equivalence $ \underline{\Alg}_{\cat^\sqcup}(\mathcal{D}) \to \underline{\Fun}\left(\cat, G\E_\infty\underline{\Alg}(\mathcal{D})\right) $ of $ G $-$ \infty $-categories. 
\end{lemma}
\begin{rmk}
		While this work was being written, we learned that Natalie Stewart had already proved similar results in \cite[Appendix A]{Stewart_Ninfty}. 
		We leave our results on the universal properties of $ G $-cocartesian operads here, but make no claim to originality. 
\end{rmk}
\begin{proof} [Proof of Lemma \ref{lemma:alg_over_cocart_operad}]
		Throughout this proof, let us write $ \mathcal{T} = \mathcal{O}_{G} $. 
		By a similar maneuver to that in \cite[Corollary C.2]{norms}, we may reduce to the case where the $ G $-symmetric monoidal structure on $ \mathcal{D} $ is $ G $-cartesian. 
		Then there is an equivalence of $ C_2 $-$ \infty $-categories $ G\E_\infty\Alg(\mathcal{D})\simeq \underline{\mathrm{CMon}}_{\mathcal{T}}(\mathcal{D}) $ by \cite[Theorem 2.32]{Nardinthesis}. 
		Emulating \cite[Lemma C.4]{norms}, it suffices to show that restriction along $ \Theta $
		\begin{equation*}
				\underline{\Fun}\left(\cat^\sqcup, \mathcal{D}\right) \to \underline{\Fun}\left(\cat \times \underline{\Fin}_{G,*}, \mathcal{D}\right)
		\end{equation*}
		induces an equivalence between the subcategories
		\begin{itemize}
			\item functors $ f \colon \cat^\sqcup_t \to \mathcal{D}_t $ that preserve $ \mathcal{T}^{/t} $-products 
			\item functors $ \cat_t \to \underline{\mathrm{CMon}}_{\mathcal{T}}(\mathcal{D})_t $ 
		\end{itemize}
		for all $ t \in \mathcal{T} $. 
		Without loss of generality, replace $ \mathcal{T}^{/t} $ by $ \mathcal{T} $.  
		Let $ M \colon \cat \times \underline{\Fin}_{G,*} \to \mathcal{D} $ so that the adjoint $ M^\dag \colon \cat \to \underline{\Fun}(\underline{\Fin}_{G,*}, \mathcal{D}) $ has image in $ \underline{\Fun}^{\times}(\underline{\Fin}_{G,*}, \mathcal{D}) \simeq \underline{\mathrm{CMon}}_{\mathcal{T}}(\mathcal{D}) $. 
		Note that a general object of $ \cat^\sqcup_T $ is a tuple $$ x = \left((x_U)_{U \in \mathrm{Orbit}(W)} \in \prod_{U \in \mathrm{Orbit}(W)}\cat_U, W \to T\right) \in \cat^\sqcup_T\,; $$ we will write $ x_U = \left(x_U \in \cat_U, U \subseteq W \to T\right) $. 
		To prove the claim, it suffices to show that the $ \mathcal{T} $-limit of $$ M^x\colon \left(\cat \times \underline{\Fin}_{G,*} \right) \fiberproduct_{\cat^\sqcup}\left(\cat^\sqcup\right)^{\underline{x}/} \to \mathcal{D} $$ is $ \displaystyle\prod_{U \in \mathrm{Orbit}(W)} \mathrm{coInd}_U^T M(c_U) \in \mathcal{D}_T $.  
		Recall that we have a commutative diagram
		\begin{equation*}
		\begin{tikzcd}
				\cat^{\times,\op} \ar[rr] \ar[d] & & \cat^{\sqcup} \ar[d] \\
				\Fin_{G}^\op \ar[r,hookrightarrow] & \underline{\mathrm{Triv}}_{G}^\otimes \ar[r] & \underline{\Fin}_{G,*}
		\end{tikzcd}
		\end{equation*}
		where $ \cat^\times $ is defined between Propositions 5.11 and 5.12 in \cite{Shah18} (in particular, both vertical arrows are cocartesian fibrations) and $ \underline{\mathrm{Triv}}_{G}^\otimes $ is the trivial operad of \cite[Example 2.4.3]{NS22}. 
		The inclusion
		\begin{equation*}
			\mathcal{Q} := \left(\cat \times \underline{\mathrm{Triv}}_{G}^\otimes \right) \fiberproduct_{\cat^\times}\left(\cat^\times\right)^{\underline{x}/} \to \left(\cat \times \underline{\Fin}_{G,*} \right) \fiberproduct_{\cat^\sqcup}\left(\cat^\sqcup\right)^{\underline{x}/}	
		\end{equation*}
		admits a fiberwise right adjoint by the discussion in the introduction to appendix C of \cite{norms}, hence it is $ G $-coinitial by \cite[Theorem 6.7]{Shah18}. 
		Thus it suffices to compute the $ G $-limit of the restriction of $ M^x $ to $ \mathcal{Q} $. 
		Now an object of $ \cat \times \underline{\mathrm{Triv}}_{G}^\otimes $ is a pair $ \left(c \in \cat_S, W \to S\right) $. 
		Its image in $ \cat^{\times,\op}_W $ is $ \left(\Res^S_Z (c)\right)_{Z \in \mathrm{Orbit}(W)} $. 
		The map $ S \to T $ in $ \mathcal{T} $ classifies the functor 
		\begin{align*}
			\cat^\times_W \simeq \prod_{U \in \mathrm{Orbit}(W)} \cat_U &\to \cat^\times_{S\times_T W} \simeq \prod_{Y \in \mathrm{Orbit}(S\times_T W)} \cat_Y \\
		\mathrm{sending} \qquad \qquad \qquad \qquad	(x_U) &\mapsto \left(\Res^{U_Y}_Y(x_{U_Y})\right) \,.
		\end{align*}
		Therefore, we may regard a $ \left(\mathcal{T}^\op \right)^{S/} $-object of $ \mathcal{Q} $ as a tuple $$ q= \left(c \in \cat_S, V \to S, \sigma, \left\{\Res^{U_Y}_Y(x_{U_Y}) \to \Res^S_Y (c)\right\}_{Y \in \mathrm{Orbit}(S \times_T W)}\right) $$ where $ \sigma $ is a morphism in $ \Fin_{\mathcal{T}} $ making the following diagram
		\begin{equation*}
		\begin{tikzcd}
				V \ar[d] & S \times_T W \ar[d] \ar[l,"\sigma"'] \\
				S \ar[r,equals] & S
		\end{tikzcd} 
		\end{equation*}
		commute. 
		Note that the data of a map $ \Res^{U_Y}_Y(x_{U_Y}) \to \Res^S_Y (c) $ is equivalent to the data of a map $ x_{U_Y} \to \mathrm{coInd}^{U_Y}_{Y}\Res^{S}_Y(c) $. 
		Since $ M $ preserves $ G $-products in its second variable, $ M(q) \simeq \prod_{Z \in \mathrm{Orbit}(V)} \Res^S_Z M(c) $. 
		We claim that the restriction of $ M^x $ to $ \mathcal{Q} $ is $ C_2 $-right Kan extended from the subcategory 
		\begin{equation*}
				\mathcal{R}:= \left(\cat \times \underline{\{*\}} \right) \fiberproduct_{\cat^\times}\left(\cat^\times\right)^{\underline{x}/} \simeq \bigsqcup_{U \in \mathrm{Orbit}(W)} \cat^{\underline{x_U}/}\,. 
		\end{equation*}
		Indeed for any $ q \in \mathcal{Q} $, the parametrized comma category is $ \mathcal{R} \times_{\mathcal{Q}} \mathcal{Q}^{\underline{q}/} \simeq \bigsqcup_{Z \in \mathrm{Orbit}(V)} \cat^{\underline{\Res^S_Z(c)}/} $, and the claim follows from the pointwise formula for parametrized right Kan extensions \cite[\S10]{Shah18}.  
		Now 
		\begin{equation*}
				\underline{\lim}_{\mathcal{R}} M \simeq \prod_{U \in \mathrm{Orbit}(W)} \underline{\lim}_{\cat^{\underline{x_U}/}} M \simeq \prod_{U \in \mathrm{Orbit}(W)} \mathrm{coInd}_U^T M(x_U)
		\end{equation*}
		where the latter equivalence follows from the canonical identification $ \cat^{\underline{x_U}/} \simeq \mathcal{T}^{/U} $ and \cite[Proposition 5.11]{Shah18}. 
\end{proof}
\begin{prop}\label{prop:alg_in_cocart_operad}
		Let $ \mathcal{T} $ be an orbital $ \infty $-category. 
		Let $ \mathcal{O}^\otimes $ be a unital $ \mathcal{T} $-$ \infty $-operad, and let $ \cat^\otimes $ be a cocartesian $ \mathcal{T} $-$ \infty $-operad (Definition \ref{defn:param_cartesian_operad_conditions}). 
		Then the restriction
		\begin{equation*}
				\underline{\Alg}_{\mathcal{O}}(\cat) \to \underline{\Fun}\left(\mathcal{O}, \cat \right) 	
	  \end{equation*} 
	  is an equivalence of $ \mathcal{T} $-$ \infty $-categories. 
\end{prop}
\begin{proof}
		Let $ \cat $ be a $ \mathcal{T} $-$ \infty $-category and recall the definition of $ \cat^{\sqcup} $ from \cite[Example 2.4.1]{NS22}. 
		In particular, we have $ \cat^{\sqcup} = \Span(\cat^\times, all, \cat^\times_{p\text{-cocart}}) \to \Span(\Fin_{\mathcal{T}}) $ where $ p^\times \colon \cat^\times \to \Fin_{\mathcal{T}} $ is a cartesian fibration defined in \cite[Proposition 5.12]{Shah18} and the triple structure on $ \cat^\times $ is given the triple structure where the backward morphisms are `everything' and forward morphisms consist of those morphisms which are \emph{co}cartesian over $ \Fin_{\mathcal{T}} $ \cite[Definition 11.3]{BarwickMackey}. 

		Now if $ \mathcal{O}^\otimes $ is a $ \mathcal{T} $-$ \infty $-operad, then an $ \mathcal{O} $-algebra in $ \cat^{\sqcup} $ is a functor $ B $ making the diagram
		\begin{equation*}
		\begin{tikzcd}[column sep=tiny]
				\mathcal{O}^\otimes \ar[rd, "{q^\dag}"'] \ar[rr,"B"] & & \cat^\sqcup = \Span(\cat^\times,all,\cat^\times_{p\text{-cocart}}) \ar[ld] \\
				& \Span(\Fin_{\mathcal{T}}) & 
		\end{tikzcd}
		\end{equation*}
		 commute. 
		 By \cite[Theorem 2.18]{MR4676218}, this is equivalent to the data of a commutative diagram of adequate triples
		\begin{equation}\label{diagram:alg_in_cocart_opd_mor_adtrip}
		\begin{tikzcd}[column sep=tiny]
				\TwAr^r \left(\mathcal{O}^\otimes\right) \ar[rd,"q"'] \ar[rr,"{B^\dag}"] & & (\cat^\times, all, \cat^\times_{p\text{-cocart}}) \ar[ld] \\
				& \left(\Fin_{\mathcal{T}}, all, all\right) & 
		\end{tikzcd}
		\end{equation}
		where $ \TwAr^r $ is given the triple structure where a morphism $ [f\colon c \to c'] $ to $ [g\colon d \to d'] $ is ingressive (resp egressive) if $ \ev_1 $ (resp. $ \ev_2 $) takes it to an equivalence. 
		Let $ \varphi \colon x \to y $ be an object of $ \TwAr^r\left(\mathcal{O}^\otimes\right) $. 
		Then the image of $ \varphi $ under $ q^\dag $ is a span $ q^\dag(x) \leftarrow q^\dag(\varphi) \to q^\dag(y) $. 
		The functor $ q $ takes $ \varphi $ to the object $ q^\dag(\varphi) $. 
		Now write $ \TwAr^r \left(\mathcal{O}^\otimes\right)_{t = 0} $ for the full subcategory on those arrows whose target is a zero object. 
		Since $ \mathcal{O}^\otimes $ was assumed to be unital, the inclusion $ \iota \colon \TwAr^r \left(\mathcal{O}^\otimes\right)_{t=0} \to \TwAr^r \left(\mathcal{O}^\otimes\right) $ has a left and right adjoint and they agree; call it $ \pi $. 
		Notice that the components of the unit and counit $ \eta \colon \mathrm{id} \to \iota \circ \pi $ and $ \varepsilon \colon \iota \circ \pi \to \mathrm{id} $ are egressive. 
		Therefore, the morphism of adequate triples (\ref{diagram:alg_in_cocart_opd_mor_adtrip}) is equivalent to a diagram of functors 
		\begin{equation}\label{diagram:alg_in_cocart_opd_mor_adtrip_simplified}
		\begin{tikzcd}[column sep=tiny]
				\TwAr^r \left(\mathcal{O}^\otimes\right)_{t = 0} \ar[rd] \ar[rr] & & \cat^\times \simeq (\ev_1)_*(\ev_0)^*\left(\cat^{\vee}\right) \ar[ld,"{p^\times}"] \\
				& \Fin_{\mathcal{T}} & 
		\end{tikzcd}
		\end{equation}
		For the next step, we introduce some notation: write $ X \subset \Ar(\Fin_{\mathcal{T}}) $ for the full subcategory on arrows with source in $ \mathcal{T} $. 
		Then we have maps
		\begin{equation*}
				\Fin_{\mathcal{T}} \xleftarrow{\ev_1} X \xrightarrow{\ev_0} \mathcal{T} \,.
		\end{equation*}
		By \cite[Theorem 2.24]{Shah18}, (\ref{diagram:alg_in_cocart_opd_mor_adtrip_simplified}) is equivalent to the data of a diagram of cartesian\footnote{Note that $ (\ev_0)_! (\ev_1)^* $ is only left Quillen by \emph{op. cit.}, so a priori we must perform fibrant replacement on $ (\ev_0)_! (\ev_1)^*\TwAr^r \left(\mathcal{O}^\otimes\right)_{t=0} $. However, we may use the naïve description of $ (\ev_0)_{!} $ as postcomposition with $ \ev_0 $ because $ \ev_0 $ is a cartesian fibration.} fibrations 
		\begin{equation*}
		\begin{tikzcd}[column sep=tiny]
				\mathcal{Q} := (\ev_0)_! (\ev_1)^*\TwAr^r \left(\mathcal{O}^\otimes\right)_{t=0} \ar[rd] \ar[rr,"\overline{A}"] & & \cat^{\vee} \ar[ld,"{p^\vee}"] \\
				& {\mathcal{T}} & 
		\end{tikzcd}
		\end{equation*}
		 over $ \mathcal{T} $, or equivalently \cite{MR3746613}, a map $ A := \overline{A}^\vee \colon \mathcal{Q}^\vee \to \cat $ of cocartesian fibrations over $ \mathcal{T} $. 

		Now we proceed as in \cite[Proposition 2.4.3.16]{LurHA}. 
		Such a $ \mathcal{T} $-functor $ A $ determines a map of $ \mathcal{T} $-$ \infty $-operads if and only if
		\begin{itemize}
			\item [(*)] Let $ \alpha $ be a morphism in $ \mathcal{Q}^\vee $ whose image under the `source' functor in $ \mathcal{O}^\otimes $ is inert. 
			Then $ A(\alpha) $ is $ p $-cocartesian.\footnote{Note that if $ \alpha $ is \emph{fiberwise} inert, then $ A(\alpha) $ is an equivalence. Hence this recovers the condition (*) in \emph{op. cit.} when $ \mathcal{T} = \{*\} $.}
		\end{itemize}
		Write $ j_0 \colon \mathcal{O}^\vee \to X $ for the composite of the structure map $ \mathcal{O}^\vee \to \mathcal{T} $ with the identity section $ \mathcal{T} \to X $. 
		Moreover, choosing a map $ j_1 \colon \left(\mathcal{O}^\otimes\right)^\vee \to \underline{\TwAr}^\ell(\mathcal{O}^\otimes)^\op \simeq \underline{\TwAr}^r(\mathcal{O}^\otimes) \to \TwAr^r(\mathcal{O}^\otimes)  $ corresponding to a coherent choice of morphism with target zero object (see \cite[Warning 2.2.5]{MR4405646}) canonically factors through $ \TwAr^r \left(\mathcal{O}^\otimes\right)_{t=0} $. 
		Since $ j_0 $ and $ j_1 $ agree after composing with the projection to $ \Fin_{\mathcal{T}} $, they assemble to a $ \mathcal{T} $-functor $ \mathcal{O} \to \mathcal{Q} $. 

		Thus it suffices to show 
		\begin{enumerate}[label=(\alph*)]
				\item A functor $ A $ is a left $ \mathcal{T} $-Kan extension of $ A|_{\mathcal{O}} $ if and only if it satisfies (*) 
				\item Every functor $ A_0 \colon \mathcal{O} \to \cat $ admits a left $ \mathcal{T} $-Kan extension satisfying the conditions of (a)
		\end{enumerate}
		Now, for $ W \in \mathcal{O}^\op_G $, an object of $ \mathcal{Q}_W $ is a pair $ (X, W) $ where $ X \in \mathcal{O}^\otimes_{[S\to V]} $ and $ W \to S $ is a map in $ \Fin_{\mathcal{T}} $ and $ W $ is an orbit. 
		Write $ \rho_W $ for the inert map $ S \leftarrow W  = W $ in $ \Span\left(\Fin_\mathcal{T}\right) $, and choose an inert morphism $ X \to X_W $ lying over $ \rho_W $.  
		This gives rise to a map $ f_W \colon (X, W \to S) \to (X_W, W) $ in $ \mathcal{Q} $. 
		Now for every $ Y \in \mathcal{O}_{[W=W]} $, the composite
		\begin{equation}\label{eq:alg_in_cocart_opd_adjunction}
				\hom_{\mathcal{Q}}\left(Y, (X,W) \right) \xrightarrow{f_W \circ -} \hom_{\mathcal{Q}}\left(Y, (X_W, W)\right) \simeq \hom_{\mathcal{O}_W}(Y, X_W)
		\end{equation}
		is an equivalence. 
		Now taking $ Y = X_W $ in (\ref{eq:alg_in_cocart_opd_adjunction}), we see that $ f_W $ admits a right inverse $ g_W $. 
		Furthermore, the preceding discussion implies that the inclusion $ \mathcal{O}_W \subseteq \mathcal{Q}_W $ admits a right adjoint. 
		Since the inert map $ X \to X_W $ is cocartesian over $ \mathcal{T}^\op $, the aforementioned right adjoints promote to a right $ \mathcal{T} $-adjoint. 
		The remainder of the proof proceeds as in \cite[Proposition 2.4.3.16]{LurHA}; we omit the details here. 
\end{proof}
Hereafter, we only concern ourselves with the case $ G = C_2 $. 
\begin{defn}\label{defn:bi_C2_commutative_bialg}
		Let $ \cat $ be a $ C_2 $-symmetric monoidal $ \infty $-category, and recall the equivalences of . 
		Set $ C_2 \E_\infty \underline{\biAlg}(\cat):= C_2\E_\infty \underline{\coAlg}(C_2\E_\infty \underline{\Alg}(\cat)) $; we will identify this $ C_2 $-$ \infty $-category with $  C_2\E_\infty \underline{\Alg}\left(C_2\E_\infty \underline{\coAlg}(\cat)\right)  $ using Proposition \ref{prop:models_of_bialg}, and refer to objects therein as \emph{$ C_2 $-$ \E_\infty $-bialgebra objects}. 
		We will refer to objects of $	C_2\E_\infty \underline{\coAlg}\left(\underline{\Alg}(\cat)\right) $ (resp. $ \underline{\coAlg}(C_2\E_\infty \underline{\Alg}(\cat)) $) as \emph{$ \underline{\E}_1 $-co-$ C_2 $-$ \E_\infty $-bialgebra objects} (resp. \emph{$ C_2 $-$ \E_\infty $-co-$ \underline{\E}_1 $-bialgebra objects}) of $ \cat $. 
\end{defn}
\begin{ntn}
	The construction $ \cat \mapsto \underline{\coAlg}(\cat) $ determines a $ C_2 $-functor $ \coAlg \colon \Alg\left(\Fun\left(\mathcal{O}^\op_{C_2}, \Cat_\infty\right)\right)\to \Fun\left(\mathcal{O}^\op_{C_2}, \Cat_\infty\right) $, classified by a $ C_2 $-cocartesian fibration that we denote $ C_2\Cat_\infty^{\coAlg} \to \Alg(C_2\Cat_\infty) $. 
	We identify $ \mathcal{O}^\op_{C_2} $-cartesian sections of $ C_2\Cat_\infty^{\coAlg} $ as a $ \mathcal{O}^\op_{C_2} $-coCartesian family $ \cat $ of monoidal $ \infty $-categories and $ A $ is a $ \mathcal{O}^\op_{C_2} $-coCartesian family of coalgebra objects in $ \cat $. 

	Similarly, there is a coCartesian fibration $ C_2\Cat_\infty^{\RMod} \to \Alg(C_2\Cat_\infty) $. 
	We identify $ \mathcal{O}^\op_{C_2} $-cartesian sections of $ C_2\Cat_\infty^{\coAlg} $ as a $ \mathcal{O}^\op_{C_2} $-coCartesian family $ \cat $ of monoidal $ \infty $-categories and a $ C_2 $-$ \infty $-category $ \mathcal{M} $ which is right-tensored over $ \cat $ in the sense of Definition \ref{defn:param_cat_tensored}. 
\end{ntn}
\begin{rmk}
	The $ C_2 $-$\infty$-categories $ C_2\Cat_\infty^{\RMod} $, $ C_2\Cat_\infty^{\coAlg} $ admit finite products as in \cite[Remark 2.2.5]{Raksit20}. 
	Moreover, they also admit finite products indexed by $ C_2 $-sets. 
	For instance, 
	\begin{align*}
		\prod_{C_2}\left(\cat, A\right) &\simeq \left(\cat \xrightarrow{\Delta} \cat \times \cat, A \mapsto (A,A) \right) \\
		\prod_{C_2}\left(\cat, \mathcal{M}\right) &\simeq \left(\cat \xrightarrow{\Delta} \cat \times \cat, \mathcal{M} \xrightarrow{\Delta} \mathcal{M}\times \mathcal{M} \right)\,.
	\end{align*}
	where $ C_2 $ acts on $ \prod_{C_2} \cat $ and $ \prod_{C_2}\mathcal{M} $ by permuting the factors. 	
\end{rmk}
\begin{cons}\label{cons:coalg_to_comod_is_G_sym_mon}
	The assignment $ A \mapsto \underline{\coMod}_A(\cat) $ extends to a $ C_2 $-functor $ C_2\Cat_\infty^{\coAlg} \to C_2\Cat_\infty^{\RMod} $ over $ \Alg(C_2\Cat_\infty) $. 
	The functor preserves finite products indexed by $ C_2 $-sets, thus we can regard it as a $ C_2 $-symmetric monoidal functor $ \mu $. 
\end{cons}
\begin{prop}\label{prop:mod_is_param_sym_monoidal}
	Let $ \cat $ be a small $ C_2 $-symmetric monoidal $ C_2 $-$ \infty $-category. 
	Then the assignment $ A \mapsto \underline{\coMod}_A $ promotes to a $ C_2 $-symmetric monoidal functor $ \mu \colon \underline{\coAlg}(\cat) \to \RMod_{\cat}\left(C_2\Cat_\infty\right) $. 
\end{prop}
\begin{proof}
		By Corollary \ref{cor:param_comodules_tensored}, if $ A $ is an $ \underline{\E}_1 $-coalgebra in $ \cat $, then the $ C_2 $-$ \infty $-category $ \coMod_A(\cat) $ is right-tensored over $ \cat $.  
		By \cite{Stewart_Ninfty}, we can regard $ \cat $ as a $ C_2 $-$ \E_\infty $-algebra in $ \E_1 \Alg(C_2 \Cat_\infty ) $. 
		This induces $ C_2 $-symmetric monoidal structures on the parametrized fibers
		\begin{equation*}
				C_2\Cat_\infty^{\coAlg} \fiberproduct_{\Alg(C_2\Cat_\infty)} \underline{\{\cat\}} \simeq \underline{\coAlg}(\cat) \qquad \qquad  C_2\Cat_\infty^{\RMod} \fiberproduct_{\Alg(C_2\Cat_\infty)} \underline{\{\cat\}} \simeq \underline{\RMod}_{\cat}(C_2 \Cat_\infty)\,.
		\end{equation*}
		Since the functor $ \mu $ of Construction \ref{cons:coalg_to_comod_is_G_sym_mon} is $ C_2 $-symmetric monoidal, it restricts to the desired $ C_2 $-symmetric monoidal functor $\underline{\coAlg}(\cat) \to  \underline{\RMod}_{\cat}(C_2 \Cat_\infty) $.
\end{proof}
\begin{cons}\label{cons:comodule_pointwise_param_tensor}
	Let $ \cat $ be a small $ C_2 $-symmetric monoidal $C_2$-$ \infty $-category and let $ A $ be a $ C_2 $-$ \E_\infty $-bialgebra object of $ \cat $. 
	By Proposition \ref{prop:mod_is_param_sym_monoidal} and by the argument of \cite[Construction 2.2.2]{Raksit20}, we obtain a $ C_2 $-symmetric monoidal structure on $ \underline{\coMod}_A(\cat) $ and a $ C_2 $-symmetric monoidal structure on the forgetful functor $ \underline{\coMod}_A(\cat) \to \cat $.  
\end{cons}
\begin{variant}\label{variant:module_pointwise_param_tensor}
		Let $ \cat $ be a small $ C_2 $-symmetric monoidal $ C_2 $-$ \infty $-category and let $ A $ be a $ C_2 $-$ \E_\infty $-bialgebra object of $ \cat $. 
		Applying Construction \ref{cons:comodule_pointwise_param_tensor} with $ \cat^\vop $ induces a $C_2$-symmetric monoidal structure on $ \underline{\Mod}_A(\cat) $ and a $ C_2 $-symmetric monoidal structure on the forgetful functor $ \underline{\Mod}_A(\cat) \to \cat $. 
\end{variant}
\begin{ex}\label{ex:functor_cat_as_module_cat}
		Let $ G $ be a $ C_2 $-commutative monoid in $ \underline{\Spc}^{C_2} $. 
		Since $ \underline{\Spc}^{C_2} $ has all $ C_2 $-limits (i.e., it has a $ C_2 $-cartesian $ C_2 $-symmetric monoidal structure), we may regard $ G $ as a $C_2 $-$\E_\infty $-co-$C_2$-$\E_\infty $-bialgebra object in $ \underline{\Spc}^{C_2} $. 
		Let $ \cat $ be a $ C_2 $-presentable $ C_2 $-symmetric monoidal $C_2$-$\infty $-category \cite[\S3]{Nardinthesis}. 
		Then there is a unique $ C_2 $-symmetric monoidal $ C_2 $-functor $ F \colon \underline{\Spc}^{C_2}\to \cat $ which preserves $ C_2 $-colimits. 

		We claim that there is a canonical $ C_2 $-symmetric monoidal equivalence $ \underline{\LMod}_{F(G)}(\cat) \simeq \underline{\Fun}(BG,\cat) $, where the latter is equipped with the pointwise $ C_2 $-symmetric monoidal structure of \cite[\S3.3]{NS22}. 
		By Proposition \ref{prop:tensor_of_param_module_cats}, there is an equivalence $\underline{\LMod}_{F(G)}(\cat) \simeq \cat \otimes \underline{\LMod}_G\left(\underline{\Spc}^{C_2}\right) $ in $ C_2\E_\infty\Alg\left(C_2\mathrm{Pr}^{L}\right) $. 
		By \cite[Example 3.26]{Nardinthesis}, there is an equivalence $ \underline{\Fun}(BG,\cat) \simeq \cat \otimes \underline{\Fun}\left(BG,\underline{\Spc}^{C_2}\right) $. 
		It suffices to prove the result for $ \cat = \underline{\Spc}^{C_2} $, wherein the $C_2$-symmetric monoidal structures on $ \underline{\Fun}\left(BG,\underline{\Spc}^{C_2}\right) $ and $ \underline{\LMod}_G\left(\underline{\Spc}^{C_2}\right) $ are $ C_2 $-cartesian, hence we are done. 
\end{ex}
\begin{recollection}
		[{\cite[\S3.4, in particular Proposition 3.25]{Nardinthesis}}] The $ C_2 $-$\infty $-category of presentable $ C_2 $-$\infty $-categories and distributive $ C_2 $-functors admits a $ C_2 $-symmetric monoidal structure. 
		We will denote the $ C_2 $-$ \infty $-category of presentable $ C_2 $-$ \infty $-categories and distributive $ C_2 $-functors by $ C_2 \mathrm{Pr}^L $. 
\end{recollection}
\begin{prop}\label{prop:tensor_of_param_module_cats}
		Let $ F\colon \underline{\Spc}^{C_2} \to \mathcal{D} $ be a morphism in $ C_2\E_\infty\underline{\Alg}(C_2\mathrm{Pr}^L)^{C_2} $, and let $ A $ be an algebra object of $ \underline{\Spc}^{C_2} $.  
		Then the $ C_2 $-symmetric monoidal functor $ \underline{\Mod}_A\left(\underline{\Spc}^{C_2}\right) \to \underline{\Mod}_{F(A)}(\mathcal{D}) $ induces a map $ \mathcal{D} \otimes_{\cat} \underline{\Mod}_A\left(\underline{\Spc}^{C_2}\right) \to \underline{\Mod}_{F(A)}(\mathcal{D}) $ in $ C_2\E_\infty\underline{\Alg} \left(C_2 \mathrm{Pr}^L\right)^{C_2} $ which is an equivalence. 
\end{prop}
\begin{proof}
		It suffices to check that the induced functor is an equivalence at the level of underlying $ C_2 $-presentable $ C_2 $-$ \infty $-categories. 
		In particular, it suffices to show the induced functor is an equivalence over each orbit. 
		The result follows from Lemma \ref{lemma:tensor_C2_pres_cats_formula} and applying \cite[Theorem 4.8.4.6]{LurHA} to each orbit. 
\end{proof}
\begin{lemma}\label{lemma:tensor_C2_pres_cats_formula}
		Let $ \cat,\mathcal{D} $ be two $ C_2 $-presentable $ C_2 $-$ \infty $-categories. 
		Notice that $ \cat^e $ and $ \cat^{C_2} $ are presentable ordinary $ \infty $-categories (and likewise for $ \mathcal{D} $). 
		Then there is an equivalence 
		\begin{equation}\label{eq:tensor_C2_pres_cats_formula}
			\cat \otimes_{C_2} \mathcal{D} \simeq \left(\cat^{C_2} \otimes_{\Spc^{C_2}} \mathcal{D}^{C_2} \xrightarrow{\mathrm{Res} \otimes \mathrm{Res}} \cat^e \otimes_{\Spc} \mathcal{D}^e \right) 
		\end{equation}
		 in $ C_2 \mathrm{Pr}^L $, where on the right-hand side of (\ref{eq:tensor_C2_pres_cats_formula}), $ \otimes_{(-)} $ denotes the tensor product in $ \mathrm{Pr}^L $ and the $ C_2 $-action on $ \cat^e \otimes_{\Spc} \mathcal{D}^e $ is induced by the componentwise $ C_2 $-action on $ \cat^e \times \mathcal{D}^e $. 
\end{lemma}
\begin{rmk}
		Taking $ \mathcal{D} = \underline{\Spc}^{C_2} $ in (\ref{eq:tensor_C2_pres_cats_formula}), we see that indeed $ \underline{\Spc}^{C_2} $ is the unit in $ C_2 \mathrm{Pr}^L $. 
\end{rmk}
\begin{proof} [Proof of Lemma \ref{lemma:tensor_C2_pres_cats_formula}]
		Write $ \cat \otimes' \mathcal{D} $ for the $ C_2 $-$ \infty $-category on the right-hand side of (\ref{eq:tensor_C2_pres_cats_formula}). 
		We will show that $ \cat \otimes' \mathcal{D} $ satisfies the universal property defining $ \cat \otimes_{C_2} \mathcal{D} $, hence they are canonically equivalent. 
		By definition of the tensor product in $ \mathrm{Pr}^L $, we have a $C_2 $-functor $ G \colon \cat \times_{C_2} \mathcal{D} \to \cat \otimes' \mathcal{D} $ which preserves fiberwise colimits separately in each variable. 
		Now if $ L_\cat $ is the left adjoint to the restriction functor $ R_\cat \colon \cat^{C_2} \to \cat^e $, notice that the composite $ L_\cat \circ R_\cat $ agrees with $ \cat^{C_2} \times \{*\} \hookrightarrow \cat^{C_2} \times \Spc \xrightarrow{\id_{\cat^{C_2}} \times \left(\sqcup_{C_2}\right)} \cat^{C_2} \times \Spc^{C_2} \xrightarrow{\alpha} \cat^{C_2} $, and likewise for $ \mathcal{D} $. 
		Since $ \Spc^{C_2} $ is generated under ordinary colimits by $ * $ and $ C_2 $, we see that $ G $ preserves $ C_2 $-coproducts separately in each variable. 
		Moreover, for any $ C_2 $-presentable $ C_2 $-$ \infty $-category $ \mathcal{E} $, the data of a $ C_2 $-functor $ \cat \times \mathcal{D} \to \mathcal{E} $ preserving $ C_2 $-colimits separately in each variable is equivalent to the data of a functor $ \cat \otimes' \mathcal{D}  \to \mathcal{E} $ which preserves $ C_2 $-colimits. 
\end{proof}
\begin{recollection}
Let $ \cat $ be a $ C_2 $-symmetric monoidal $C_2$-$ \infty $-category. 
An object $ X \in \cat $ is \emph{dualizable} if there exists an object $ Y \in \cat $ with maps $ \eta: \mathbbm{1} \to X \otimes Y $ and $ \varepsilon: X \otimes Y \to \mathbbm{1} $ such that the composites
\begin{align*}
	X \xrightarrow{\eta \otimes \id_X} X \otimes Y \otimes X \xrightarrow{\id_X \otimes \varepsilon } X \\
	Y \xrightarrow{\eta \otimes \id_{Y}} Y \otimes X \otimes Y \xrightarrow{\id_Y \otimes \varepsilon} Y
\end{align*}
are both equivalences. 
One can check that if such a $ Y $ exists, it must be unique (up to contractible choice), so we write $ Y \simeq X^\vee $ and call it the dual of $ X $ (in $ \cat $). 
Furthermore, $ X^\vee $ is dualizable with dual $ X $. 
We denote the full $ C_2 $-subcategory on dualizable objects by $ \cat^{\fd} \subseteq \cat $. 
Since fully dualizable objects are preserved under symmetric monoidal functors, $ \cat^{\fd} $ is closed under the norm map and thus inherits a $ C_2 $-symmetric monoidal structure from $ \cat $. 
Furthermore, dualization refines to a $ C_2 $-functor $ \cat^\fd \to (\cat^\fd)^\vop $. 
\end{recollection}
\begin{prop}\label{prop:param_dual_takes_coalg_to_alg}
	Let $ \cat $ be a $ C_2 $-symmetric monoidal $C_2$-$ \infty $-category. 
	There exists an essentially unique $ C_2 $-symmetric monoidal functor $ \beta \colon \underline{\coAlg} \left(\cat_{\fd}\right) \to \underline{\Alg}(\cat_\fd)^\vop \simeq \underline{\coAlg}(\cat^\vop_\fd) $ making the following diagram 
	\begin{equation*}
	\begin{tikzcd}
				\underline{\coAlg} \left(\cat_{\fd}\right)\ar[d, "\beta"]  \ar[r] & \underline{\coAlg} \left(\cat\right) \ar[r,"{\mu}"] & \underline{\RMod}_{\cat}(C_2 \Cat_\infty)_{/\cat} \ar[d,"{(-)^\vop}"] \\
				 \underline{\Alg}(\cat_\fd^\vop)  \ar[r] &  \underline{\Alg}(\cat^\vop ) \ar[r, "{\mu}"] & \underline{\RMod}_{\cat^\vop}(C_2 \Cat_\infty)_{/\cat^\vop}
	\end{tikzcd}
	\end{equation*} 
	commute. 
	Moreover, the diagram of $ C_2 $-$ \infty $-categories 
	\begin{equation*}
	\begin{tikzcd}
			\underline{\coAlg} \left(\cat_{\fd}\right)\ar[r, "\beta"] \ar[d] & \underline{\Alg}(\cat_\fd^\vop) \ar[d] \\
			\cat_\fd \ar[r, "{(-)^\vee}"] & \cat_\fd^\vop
	\end{tikzcd}
	\end{equation*}
	commutes canonically. 
\end{prop}
\begin{cor}\label{cor:mod_comod_equivalence}
	Let $ \cat $ be a $ C_2 $-symmetric monoidal $ C_2 $-category, and $ A $ a $ C_2 $-$ \E_\infty $-$ \underline{\E}_1 $-bialgebra object of $ \cat $ whose underlying object (i.e. forgetting the bialgebra structure) is dualizable. 
	Then there exists an $ \underline{\E}_1 $-co-$ C_2\E_\infty $-bialgebra structure on $ A^\vee $ and a $ C_2$-symmetric monoidal equivalence of $ C_2 $-$ \infty $-categories
	\begin{align*}
		\underline{\Mod}_A(\cat) \simeq \underline{\coMod}_{A^\vee}(\cat). 
	\end{align*} 
	On underlying $ \infty $-categories, this recovers the equivalence of \cite[Corollary 2.3.3]{Raksit20}. 
\end{cor}
\begin{proof}
		[Proof of Proposition \ref{prop:param_dual_takes_coalg_to_alg}] 
		The existence of the functor $ \beta $ follows from an argument similar to the proof of \cite[Proposition 2.3.2]{Raksit20}, using Proposition \ref{prop:comod_functor_ess_image} and the pointwise description of module categories over $ \underline{\E}_1 $-algebras discussed in \S\ref{appendix:param_module_cats}. 
		That $ \beta $ is $ C_2 $-symmetric monoidal follows from the observation that $ \beta \circ \beta $ is canonically equivalent to the identity. 
\end{proof}
\begin{rmk}
		Let $ \cat $ be a $ C_2 $-symmetric monoidal $ C_2 $-$ \infty $-category and suppose $ \mathcal{M},\mathcal{N} $ are $ C_2 $-$ \infty $-categories right tensored over $ \cat $. 
		If $ F \colon \mathcal{M} \to \mathcal{N} $ is a $C_2$-left adjoint functor of $ \cat $-linear $ C_2 $-$ \infty $-categories, then its right adjoint $ G $ is canonically lax $ \cat $-linear (compare \cite[Remark 7.3.2.9]{LurHA}). 
\end{rmk}
\begin{lemma}\label{lemma:end_alg_as_adjoint}
		Let $ (\mathcal{M}, U\colon \mathcal{M} \to \cat) $ be a $ C_2 $-object of $ \underline{\RMod}_{\cat}(C_2 \Cat_\infty)_{/\cat} $ so that $ U $ admits a right $ C_2 $-adjoint $ G \colon \cat \to \mathcal{M} $ which is $ \cat $-linear. 
		Then there is a $ \underline{\E}_1 $-coalgebra $ A $ in $ \cat $ equipped with a map $ \alpha \colon \mathcal{M} \to \underline{\coMod}_A \left(\cat\right) $ in $ \underline{\RMod}_{\cat}(C_2 \Cat_\infty)_{/\cat} $ so that, for any $ B \colon \mathcal{O}^{\mathrm{op}}_{C_2} \to \underline{\E}_1\underline{\coAlg}(\cat) $, the map 
		\begin{equation*}
				\underline{\Map}_{\underline{\E}_1\coAlg}(A,B)\to \underline{\Map}_{\underline{\RMod}_{\cat}(C_2 \Cat_\infty)_{/\cat}}\left(\mathcal{M}, \underline{\coMod}_{B}(\cat)\right)	 	
		\end{equation*} 
		is an equivalence of $ C_2 $-$ \infty $-groupoids. 
\end{lemma}
\begin{proof}
		The result follows from the same argument as \cite[Lemma 2.3.5]{Raksit20}, but using ($ \mathcal{O}^\op_{C_2} $-points of) parametrized functor categories in place of ordinary functor categories. 
\end{proof}
\begin{prop}\label{prop:comod_functor_ess_image}
		Let $ \cat $ be a $ C_2 $-symmetric monoidal $ C_2 $-$ \infty $-category. 
		The $C_2$-functor $ \mu \colon \underline{\coAlg} (\cat) \to \underline{\RMod}_{\cat}(C_2 \Cat_\infty)_{/\cat} $ of Proposition \ref{prop:mod_is_param_sym_monoidal} is fully faithful. 
		Over each fiber $ t \in \mathcal{T} $, its essential image consists of those pairs $ (\mathcal{M}, U )$ where $ \mathcal{M} $ is a $ \mathcal{T}^{/t} $-$\infty $-category right-tensored over $ \cat_{\underline{t}} $ and $ U $ is a $ \cat_{\underline{t}} $-linear $ \mathcal{T}^{/t} $-functor (in the sense of Definition \ref{defn:param_linear_functors}) so that
		\begin{enumerate}[label=(\alph*)]
			\item $ U $ is fiberwise comonadic, in particular it admits a right $ \mathcal{T}^{/t} $-adjoint $ G \colon \cat_{\underline{t}} \to \mathcal{M} $ 
			\item $ G $ is $ \cat_{\underline{t}} $-linear.  
		\end{enumerate}
\end{prop}
\begin{proof}
		We may define a $ C_2 $-subcategory $ \underline{\RMod}_{\cat}(C_2 \Cat_\infty)_{/\cat}^0 $ of $ \underline{\RMod}_{\cat}(C_2 \Cat_\infty)_{/\cat}$ consisting of those pairs $ (\mathcal{M}, U) $ which are in the essential image of $ \mu $. 
		By Lemma 2.3.5 \emph{ibid.}, $ \mu $ admits a left adjoint $ \nu $. 
		Write $ q $ for the structure map $ \underline{\RMod}_{\cat}(C_2 \Cat_\infty)_{/\cat} \to \mathcal{O}^\op_{C_2} $. 
		By Lemma \ref{lemma:end_alg_as_adjoint}, the unit $ u \colon \mathrm{id} \to \mu \circ \nu $ satisfies $ q(u) $ is the identity on $ \mathcal{O}_{C_2}^\op $, hence $ (\mu, \nu) $ is a $ C_2 $-adjunction. 
		The rest of the proof is similar to that of \cite[Proposition 2.3.6]{Raksit20}.  
\end{proof}
\subsection{The Tate construction}\label{subsection:parametrized_Tate_cons}
Hochschild homology $ \HH(A/k) $ is a $ k $-module with $ S^1 $-action; taking the homotopy orbits, fixed points, and Tate construction with respect to the $ S^1 $-action gives rise to cyclic homology, negative cyclic homology, and periodic cyclic homology, respectively. 
One defines real versions of these trace theories using a parametrized version of the Tate construction \cite{QSparam_Tate}. 
In this section, we consider a variant on Quigley--Shah's parametrized Tate construction which will allow us to apply fixed points, orbits, and the Tate construction to filtered objects with filtered $ S^\sigma $-action.  
\begin{recollection}
		[Tate construction] 
		\cites[Construction 2.4.4]{Raksit20}{MNN}[Definition 5.48]{QSparam_Tate}\label{rec:Tateconstruction}
		Let $ A $ be a dualizable bialgebra object of a stable presentable symmetric monoidal $ \infty $-category $ \cat $. 
		Restriction along the counit $ A \to \mathbbm{1} $ defines a forgetful functor $ \cat \simeq \Mod_{\mathbbm{1}}(\cat) \to \Mod_A(\cat) $. 
		Denote the left and right adjoints to the restriction by $ (-)_A, (-)^A $ respectively. 
		Assume further that we are given an equivalence of $ A $-modules $ A \simeq A^\vee \otimes \omega_A $ for some invertible $ \omega_A \in \Pic(\cat) $. 
		Then given any $ A $-module $ M $, we have a morphism
		\begin{align*}
			M_A \otimes \omega_A \simeq (M \otimes_A \mathbbm{1}) \otimes \omega_A \xrightarrow{} M \otimes_A (\mathbbm{1} \otimes \omega_A \otimes A^\vee) \simeq M \otimes_A A \simeq M
		\end{align*}
		Since the $ A $-module structure on the left hand side factors canonically through the counit $ A \to \mathbbm{1} $, the map above is adjoint to a map $ \Nm_M: M \otimes \omega_A \to \hom_A(k, M) $. 
		The \emph{Tate construction of $ M $} is the cofiber of the norm map 
		\begin{equation*}
			M^{tA} := \cofib(\Nm_M) .
		\end{equation*} 
		Since the previous discussion was functorial in $ M $, we have a functor $ {\Nm_{(-)}: \Mod_A(\cat) \to \Fun(\Delta^1, \cat)} $ which induces $ (-)^{tA} \simeq \cofib \circ \Nm_{(-)} : \Mod_A(\cat ) \to \cat $.
\end{recollection}
\begin{prop}\label{prop:tate_cons_is_param_compatibility}
		Let $ \cat $ be a $ C_2 $-symmetric monoidal $ \infty $-category which is $ C_2 $-stable, fiberwise presentable, and $ C_2 $-complete and $C_2$-cocomplete\footnote{See \cite[Remark 6.11]{Shah18} on why assuming presentability is not sufficient.}, and suppose $ A $ is a dualizable cocommutative bialgebra object of $ \cat $ which satisfies the assumptions of Recollection \ref{rec:Tateconstruction}. 
		Then the Tate construction of \emph{loc. cit.} promotes to a $ C_2 $-functor
		\begin{equation*}
				(-)^{tA} \colon \underline{\LMod}_A(\cat) \to \cat \,,
		\end{equation*}
		and the norm map $ \Nm_{-} $ promotes to a natural transformation $ (-)^A \to (-)^{tA} $ of $ C_2 $-functors $ \underline{\LMod}_A(\cat) \to \cat $, where $ (-)^A $ is the right $C_2$-adjoint to base change along the unit map $ \mathbb{1} \to A $ of Proposition \ref{prop:param_modules_colimits}\ref{prop_item:param_mod_colim_change_of_alg_tensor_is_left_adjoint}. 
\end{prop}
\begin{rmk}
		In the situation of Proposition \ref{prop:tate_cons_is_param_compatibility}, unraveling definitions we see that there is a commutative diagram
		\begin{equation*}
		\begin{tikzcd}
				{\LMod}_A\left(\cat^{C_2}\right) \ar[d,"{(-)^e}"] \ar[r,"{(-)^{tA}}"] & {\LMod}_A\left(\cat^{C_2}\right)\ar[d,"{(-)^e}"] \\
				{\LMod}_{A^e}(\cat^{e})  \ar[r,"{(-)^{tA^e}}"] & {\LMod}_{A^e}(\cat^e)	\,.			
		\end{tikzcd}
		\end{equation*}
\end{rmk}
\begin{proof}
		[Proof of Proposition \ref{prop:tate_cons_is_param_compatibility}]
		Since $ \cat $ is assumed to be $ C_2 $-complete and $ C_2 $-cocomplete, in particular it admits both $C_2$-products and $ C_2$-coproducts. 
		By \cite[Proposition 5.11]{Shah18} and its dual, it follows that the restriction map $ \cat^{C_2} \to \cat^e $ has both a left and right adjoint, and therefore it preserves limits and colimits. 
		The result follows from \cite[Remark 2.4.12]{Raksit20}. 
\end{proof}
Next, we show that the parametrized Tate construction of Proposition \ref{prop:tate_cons_is_param_compatibility} is lax monoidal (compare \cite{QSparam_Tate}). 
\begin{defn}
		Let $ \cat $ be a $ C_2 $-stable $ C_2 $-symmetric monoidal $ \infty $-category, and let $ A $ be a $ C_2 $-$ \E_\infty $-algebra object of $ \cat $. 
		Write $ \underline{\LMod}^{\mathrm{ind}}_A(\cat) $ denote the smallest stable full $ C_2 $-subcategory of $ \underline{\LMod}_A(\cat) $ containing the objects $ A \otimes X $ for $ X \in \cat^{C_2} $ and $ A^e \otimes Y $ for $ Y \in \cat^e $. 
		We refer to this as the subcategory of \emph{induced} $ A $-modules. 
		By construction, it is a full $ C_2 $-stable $ C_2 $-subcategory of $ \underline{\LMod}_A(\cat) $. 
\end{defn}
\begin{rmk}
		The category $ \underline{\LMod}^{\mathrm{ind}}_A(\cat) $ is fiberwise a $ \otimes $-ideal; by construction it is closed under the restriction functor $ \underline{\LMod}_A(\cat)^{C_2} \to \underline{\LMod}_A(\cat)^e $. 
		Now write $ \overline{N}^{C_2} $ for the norm map on $ \cat $ and $ N^{C_2} $ for the relative norm map on $ A $-modules (see \cite[Appendix A]{LYang_normedrings}).
		Since $ N^{C_2}(A^e \otimes Y) \simeq \left(\overline{N}^{C_2}(A^e \otimes Y)\right) \otimes_{\overline{N}^{C_2}A^e} A \simeq A \otimes \overline{N}^{C_2}Y $, the subcategory $ \underline{\LMod}^{\mathrm{ind}}_A(\cat) $ is also closed under the relative norm map $ \underline{\LMod}_A(\cat)^e \to \underline{\LMod}_A(\cat)^{C_2} $. 
		Therefore, it is a $ C_2 $-$ \otimes $-ideal in the sense of \cite[Definition 5.24]{QSparam_Tate}. 
\end{rmk}
\begin{obs}\label{obs:tate_vanishes_on_induced}
		Let $ \cat $ be a $ C_2 $-stable $ C_2 $-symmetric monoidal $ \infty $-category, and let $ A $ be a dualizable $ C_2 $-$ \E_\infty $-bialgebra object of $ \cat $ satisfying the assumptions in Recollection \ref{rec:Tateconstruction}.
		Let $ M \in \underline{\LMod}^{\mathrm{ind}}_A(\cat) $. 
		Then $ M^{tA} \simeq 0 $.
\end{obs}
\begin{prop}\label{prop:param_tate_lax_monoidal}
		Let $ \cat $ be a $C_2$-symmetric monoidal $C_2$-stable $C_2$-$ \infty $-category and let $ A $ be a dualizable bialgebra object of $ \cat $. 
		Assume that $ \cat $ is $C_2$-complete and $C_2$-cocomplete, and that $ \cat $ is fiberwise compactly generated. 
		Then there is a unique pair of data: 
		\begin{itemize}
			\item A lax $ C_2 $-symmetric monoidal structure on the functor $ (-)^{tA} \colon \underline{\LMod}_A(\cat) \to \cat $
			\item A lax $ C_2 $-symmetric monoidal structure on the natural transformation $ (-)^A \to (-)^{tA} $. 
		\end{itemize}
\end{prop}
\begin{proof}
		Take $ C =\underline{\LMod}_A(\cat) $, $ D = \underline{\LMod}^{\mathrm{ind}}_A(\cat) $ in \cite[Theorem 5.28]{QSparam_Tate}, $ E = \cat $, and keep the notation $ L $ for the functor $ \Fun_{C_2}^{\mathrm{ex},\mathrm{lax}}(C, E) \to \Fun_{C_2}^{\mathrm{ex},\mathrm{lax}}(C/D, E) $ which is left adjoint to restriction along the quotient $ C \to C/D $. 
		By the theorem cited above, there exists an essentially unique lax $ C_2 $-symmetric monoidal functor $ L \left((-)^A\right) \colon C/D \to E $ and (regarding $ L \left((-)^A\right) $ as a functor $ C \to E $ which vanishes on $ D $) and the unit of the adjunction $ (-)^A \to L \left((-)^A\right) $ acquires a canonical lax $ C_2 $-symmetric monoidal structure. 
		It remains to check that $  L \left((-)^A\right) \simeq (-)^{tA} $. 
		As in the proof of \cite[Theorem I.3.1]{nikolaus-scholze}, it suffices to check the equivalence at the level of $ C_2 $-exact $ C_2 $-functors. 
		This follows by definition of the Tate construction, Observation \ref{obs:tate_vanishes_on_induced}, and Lemma \ref{lemma:param_mod_cat_gen_by_induced}. 
\end{proof}
\begin{rmk}
		It follows from Proposition \ref{prop:param_tate_lax_monoidal} that when $ \cat = \underline{\Spectra}^{C_2} $ and $ A = \sphere[K] $ for $ K $ any finite group or compact Lie group sitting in an extension $ 1 \to K \to \widehat{K} \to C_2 \to 1 $ (compare \cite{QSparam_Tate}), the norm map and Tate construction agree with the norm map and \emph{parametrized Tate construction} of \cite{QSparam_Tate} under the equivalence of Example \ref{ex:functor_cat_as_module_cat}.   
\end{rmk}

\begin{lemma}\label{lemma:param_mod_cat_gen_by_induced}
		Let $ \cat $ be a $C_2$-symmetric monoidal $C_2$-stable $C_2$-$ \infty $-category and let $ A $ be an algebra object of $ \cat $. 
		Assume that $ \cat $ is $C_2$-complete and $C_2$-cocomplete, and that $ \cat $ is fiberwise compactly generated.
		Let $ X \in \underline{\LMod}_A\left(\cat\right)_t $. 
		Then the following maps
		\begin{align*}
				\colim_{Y \in \left(\underline{\LMod}^{\mathrm{ind}}_A(\cat)_t\right)_{/X}} Y &\to X \\
				\colim_{Y \in \left(\underline{\LMod}^{\mathrm{ind}}_A(\cat)_t\right)_{/X}} \cofib(Y \to X)^A &\to \colim_{Y \in \left(\underline{\LMod}^{\mathrm{ind}}_A(\cat)_t\right)_{/X}} \cofib(Y \to X)^{tA} 
		\end{align*}
		are equivalences in $ \cat_t $. 
\end{lemma}
\begin{proof}
		The result follows immediately from \cite[Lemma 2.4.9]{Raksit20}. 
\end{proof}

\begin{rmk}\label{rmk:param_tate_naturality}
		Let $ \mathcal{D} $ be a $C_2$-symmetric monoidal $C_2$-stable $C_2$-$ \infty $-category and let $ F $ be a $ C_2 $-symmetric monoidal functor $ \cat \to \mathcal{D} $. 
		Let $ A $ be a dualizable algebra object of $ \cat $. 

		For any $ X \in \cat$, there is a canonical sequence of transformations $ F(X)_{F(A)} \to F(X_A) \xrightarrow{F(\mathrm{Nm})} F(X^A) \to F(X)^{F(A)} $ whose composite may be identified with the norm map with respect to $ F(A) $. 
		If $ F $ preserves all $ C_2 $-colimits, then the canonical map $ F(X)_{F(A)} \to F(X_A) $ is an equivalence. 
		In particular, there is a natural transformation $ F( (-)^{tA}) \to F(-)^{tF(A)} $. 
		If $ F $ additionally preserves all $ C_2 $-limits, then the aforementioned natural transformation is an equivalence. 
\end{rmk}

\section{Derived involutive algebra}\label{section:scr_with_inv}
In this section, we define derived algebras in $ C_2 $-Mackey functors over $ \underline{k} $ (Notation \ref{ntn:fixpt_green_functor}) which admit norm maps; these are our derived rings with involution. 
Each derived ring with involution over $ \underline{k} $ has an underlying $ C_2 $-$\E_\infty$-algebra over $ \underline{k} $. 

The purpose for defining derived involutive algebras is twofold: Let us reiterate the reasons discussed in \S\ref{subsection:intro_mainresult}. 
First, while real Hochschild homology is defined for all $ C_2 $-$ \E_\infty $-$ \underline{k} $-algebras, we do not expect a real Hochschild--Kostant--Rosenberg theorem to hold at this level of generality: The ordinary Hochschild--Kostant--Rosenberg theorem applies only to derived rings. 
Furthermore, we will need to equip cochains on the involutive circle $ \underline{\Z}^{S^\sigma} $ and its filtered cousin $ \tau_{\geq *} \underline{\Z}^{S^\sigma} $ and $ \D^{\sigma,\vee}_+ $ with additional structure in order to formulate the universal property of filtered real Hochschild homology and the involutive de Rham complex. 
In particular, the latter are not connective; thus it is not sufficient to take our derived involutive algebras to be simplicial objects in cohomological $ C_2 $-Tambara functors (up to some notion of weak equivalence). 
We expect our theory of derived involutive algebra to be relevant to the relationship between involutions in algebraic geometry and genuine involutions in the sense of \cite{CHN25}. 

To construct derived involutive algebras, we modify the formalism of Bhatt--Mathew and Mathew (recorded in \cite[\S4]{Raksit20}). 
Our definition uses the language of filtered monads from \emph{loc. cit.}, which we recall in \S\ref{subsection:filtered_monads} and extend to $ C_2 $-$ \infty $-categories. 
The general construction is contained in \S\ref{subsection:involutive_dalg}. 
In \S\ref{subsection:involutive_dalg_connectivity}, we discuss how derived involutive algebra structures behave under various (co)connectivity assumptions; the results contained therein will be used to endow the filtered involutive circle $ \underline{\Z}^{S^\sigma} $ and its associated graded with the additional structure needed to state and prove the main theorem(s) of this paper.  

\subsection{Filtered monads}\label{subsection:filtered_monads}
Let $ k $ be a discrete commutative ring with an involution. 
We will define (not necessarily connective) derived $ k $-algebras with involution using the formalism of monads, similar to the approach taken in \cite[\S4.1-2]{Raksit20}. 
In particular, we will utilize the notion of \emph{filtered monad} (Definition 4.1.2 of \emph{loc. cit.}). 
The reader unfamiliar with the formalism of monads may find it useful to compare the Barr--Beck--Lurie theorem \cite[Theorem 4.7.3.5]{LurHA} and the characterization of $ \E_\infty $-algebras in a presentable symmetric monoidal $ \infty $-category $ \cat $ as modules over the symmetric algebra monad \cite[Construction 4.1.1]{Raksit20}.  

\begin{defn}\label{defn:param_monad}
	Let $ \cat $ be a $ C_2 $-$ \infty $-category. 
	A \emph{monad} on $ \cat $ is an $ \E_1 $-algebra object in the monoidal $ \infty $-category of $ C_2 $-endofunctors $ \End_{C_2}(\cat) = \Fun_{C_2}\left(\mathcal{O}^\op_{C_2}, \underline{\End}(\cat)\right) $. 

	A \emph{filtered monad} on $ \cat $ is a lax monoidal functor $ \Z^\times_{\geq 0 } \to \End_{C_2}(\cat) $, where $ \Z^\times_{\geq 0 }  $ denotes the partially ordered set of nonnegative integers regarded as a monoidal category via multiplication and $ \Fun_{C_2}\left(\mathcal{O}^\op_{C_2},\underline{\End}(\cat)\right) $ is equipped with the composition monoidal structure. 

	More generally, if $ \mathcal{E} $ is a monoidal full subcategory of $ \End_{C_2}(\cat) $, we will refer to algebra objects of $ \mathcal{E} $ as $ \mathcal{E} $\emph{-monads} and lax monoidal functors $ \Z_{\geq 0} \to \mathcal{E} $ as \emph{filtered $ \mathcal{E} $-monads}.
\end{defn} 
\begin{rmk}\label{rmk:fiberwise_sifted_vs_param_sifted} 
		Let $ \cat $ be a $ \mathcal{T} $-$\infty $-category and suppose that for all $ \alpha \colon s \to t $, the restriction functor $ \alpha^* \colon \cat_{t} \to \cat_s $ preserves all sifted colimits. 
		Let $ F \colon \cat \to \cat $ be a $ \mathcal{T} $-functor. 
		It follows from \cite[Theorem D(2)]{Shah_paramII} that $ F $ preserves all $ \mathcal{T} $-sifted $ \mathcal{T} $-colimits \cite[Definition 1.12]{Shah_paramII} if and only if $ F $ preserves all sifted colimits fiberwise, i.e. if $ F_t \colon \cat_t \to \cat_t $ preserves sifted colimits for every $ t \in \mathcal{T} $. 
		Moreover, a $ \mathcal{T} $-subcategory $ \mathcal{D} \to \cat $ generates $ \cat $ under $ \mathcal{T} $-sifted colimits if and only if, for each $ t \in \mathcal{T} $, $ \mathcal{D}_t $ generates $ \cat_t $ under sifted colimits.  
\end{rmk}
\begin{cons}\label{cons:param_left_modules_over_monad}
		Let $ \cat $ be a $ C_2 $-$ \infty $-category. 
		Under the identification of Proposition \ref{prop:param_cats_monoidal_const_operads}, a monad on $ \cat $ is equivalently a $ \mathcal{O}^\op_{C_2} $-cocartesian family of $ \E_1 $-algebra objects in the $ \mathcal{O}^\op_{C_2} $-cocartesian family of $ \E_1 $-monoidal $ \infty $-categories $ \underline{\End}(\cat)^\circ $. 
		Any monad $ T $ on $ \cat $ admits a $ C_2 $-$ \infty $-category of left modules (for instance, by applying the constructions of \cite[\S4.7.1]{LurHA} for each object $ t \in \mathcal{O}^\op_{C_2} $). 
\end{cons}
\begin{rmk}\label{rmk:restrict_param_endomorphisms}
		Notice that for any $ C_2 $-$ \infty $-category $ \cat $, we have monoidal restriction functors $ \End_{C_2}(\cat) \to \End(\cat^{C_2}) $ and $ \End_{C_2}(\cat) \to \End(\cat^{e}) $. 
\end{rmk}
\begin{lemma}\label{lemma:param_monad_induces_param_adjunction}
		Let $ \cat $ be a $ C_2 $-$ \infty $-category which admits finite $ C_2 $-products and let $ T $ be a monad on $ \cat $ in the sense of Definition \ref{defn:param_monad}. 
		Then there is a $ C_2 $-adjunction $ U \colon \underline{\Mod}_T(\cat) \rlarrows \cat \colon T $ which recovers the adjunction $ \Mod_{T^e}(\cat^e) \rlarrows  \cat^e $ on underlying $ \infty $-categories, where $ T^e $ is the image of $ T $ under the restriction functor of Remark \ref{rmk:restrict_param_endomorphisms}. 
\end{lemma}
\begin{proof}
		Follows from Example \ref{ex:param_endomorphisms_cats}, Proposition \ref{prop:param_modules_limits}, and Corollary \ref{cor:C2_left_adjoint_local_crit}. 
\end{proof}
The following proposition is a straightforward generalization of \cite[Proposition 4.1.4]{Raksit20}, hence we omit its proof. 
\begin{prop}\label{prop:colim_of_filt_param_monad}
		Let $ \cat $, $ \mathcal{E} $ be as in Definition \ref{defn:param_monad}. Assume that
		\begin{enumerate}[label=(\alph*)]
			\item $ \cat $ admits all small $ C_2 $-colimits,
			\item $ \mathcal{E} $ is closed under pointwise sequential colimits, so we have an adjunction $ \colim \colon \Fun(\Z_{\geq 0}, \mathcal{E}) \rlarrows \mathcal{E} \colon \delta $, where $ \delta $ denotes the diagonal functor, 
			\item Each $ F \in \mathcal{E} $ commutes with sequential colimits. 
		\end{enumerate}
		Then the adjunction $ (\colim, \delta ) $ canonically lifts to a relative adjunction $ \Fun(\Z_{\geq 0}, \mathcal{E}^\circ) \rlarrows \mathcal{E}^\circ \colon \delta $ over $ \Assoc^\otimes $. 
		In particular, the colimit of a filtered $ \mathcal{E} $-monad is canonically an $ \mathcal{E} $-monad. 
\end{prop}
\begin{recollection}
	Let $ \underline{\F}_{C_2,*} $ denote the $ C_2 $-$ \infty $-groupoid whose objects consist of finite pointed sets with $ C_2 $-action; it is a $ C_2 $-symmetric monoidal $ \infty $-category via (parametrized) disjoint union.   
	For any presentable $ C_2 $-symmetric monoidal $ \infty $-category $ \cat $, let $ \underline{\SSeq}(\cat) = \underline{\Fun}(\underline{\F}_{C_2,*} , \cat) $ denote the $ C_2 $-$ \infty $-category of $ C_2 $-symmetric sequences in $ \cat $ \cite[\S2.6]{Stewart_Ninfty}, regarded as a $ C_2 $-symmetric monoidal $ \infty $-category via parametrized Day convolution \cite[Definition 3.1.6]{NS22}. 

	$ C_2 $-left Kan extension along the inclusion of the basepoint $ \mathcal{O}_{C_2}^\op \inj \underline{\F}_{C_2,*} $ induces a fully faithful, $ C_2 $-colimit-preserving, $ C_2 $-symmetric monoidal embedding $ \cat \to \underline{\SSeq} (\cat) $. 
	Composing the parametrized Yoneda embedding of \cite[Definition 10.2]{BDGNS1} $  \underline{\F}_{C_2,*} \simeq  \underline{\F}_{C_2,*}^\mathrm{vop} \to \underline{\Fun}\left(\underline{\F}_{C_2,*}, \underline{\Spc}^{C_2}\right) $ with the unique map $ \underline{\Spc}^{C_2} \to \cat $ in $ C_2\E_\infty\Alg\left(\Pr^L \right) $, we obtain a $ C_2 $-functor $ y \colon \underline{\F}_{C_2,*} \to \underline{\SSeq}(\cat) $. 

	There is an additional (non)-symmetric monoidal structure on $ \SSeq(\cat) $ called the \emph{composition monoidal structure}. 
	For any $ \mathcal{D} \in C_2\E_\infty\Alg\left(\Pr^L \right)_{\cat/-} $, evaluation at $ \underline{y(\{[C_2/C_2 = C_2/C_2 ]\})} $ determines an equivalence of $ C_2 $-$ \infty $-categories
	\begin{equation*}
		\Fun_{C_2\E_\infty\Alg\left(\Pr^L \right)_{\cat/-}}\left(\underline{\SSeq}(\cat), \mathcal{D} \right) \simeq \mathcal{D}
	\end{equation*}
	by \cite[Corollary 2.4.5]{NS22}. 
	Taking $ \mathcal{D} = \underline{\SSeq}(\cat) $, the reverse of the composition monoidal structure on the left hand side transports to a monoidal structure on $ \underline{\SSeq}(\cat) $ which we will call the \emph{composition monoidal structure}. 
	Explicitly, if $ A $ and $ B $ are $C_2$-symmetric sequences, the value of their composition product on the $ C_2 $-set $\{*\}$ with trivial action is computed as a parametrized colimit
	\begin{equation*}
		 \left(A \circ B\right)(C_2/C_2) \simeq \colim_{S \in \underline{\Fin}_{C_2}^\simeq} \left( A (S) \otimes B^{\ostar S}\right)_{h\mathrm{Aut}_{C_2}(S)} \,, 
	\end{equation*}
	where $ \otimes $ denotes the tensoring of $ \SSeq (\cat) $ over $ \cat $, $ \ostar $ denotes the Day convolution product, and $ \mathrm{Aut}_{C_2}(S) $ denotes $C_2$-equivariant automorphisms of $ S $. 

	An $ \E_1 $-algebra object with respect to the composition monoidal structure on $ \underline{\SSeq} (\cat) $ is a $ C_2 $-$ \infty $-operad.\footnote{We assume this equivalence for the moment and defer its proof to future work. 
	Such an equivalence is generally expected to hold, (cf. \cite{MR4405646} for the non-parametrized version of the statement). 
	We thank Rune Haugseng for sharing his ideas in this direction.} 
	Finally, there is a fiberwise monoidal functor 
	\begin{equation*}
	\begin{split}
	 	 \theta : \underline{\SSeq}(\cat) &\to \underline{\End}(\underline{\SSeq} (\cat)) \\ 
					A &\mapsto A \circ (-)
	\end{split}
	\end{equation*}
	We also use $ \theta $ to denote the induced functor $ C_2\mathrm{Op}(\cat) \simeq \E_1\Alg\left(\underline{\SSeq}(\cat) \right)  \to \E_1\Alg\underline{\End}(\underline{\SSeq} (\cat)) $. 
\end{recollection}

\begin{cons}\label{cons:freeC2alg_as_C2sseq}
	Let $ \cat $ be a presentable $ C_2 $-symmetric monoidal $ \infty $-category. 
	Let $ A $ be the constant $ C_2 $-symmetric sequence with value the unit object $ \mathbbm{1} $. 

	Write $ G \colon C_2\E_\infty\underline{\Alg}\left(\underline{\SSeq}(\cat) \right) \rlarrows \underline{\SSeq}(\cat) : F $ denote the free-forgetful $ C_2 $-adjunction of \cite[Theorem 4.3.4]{NS22}. 
	Then by Corollary 5.1.5 of \emph{loc. cit.}, $ G \circ F \simeq C_2\Sym_{\SSeq(\cat)} $ carries a canonical algebra (monad) structure. 
	By a similar argument to that of \cite[Construction 4.1.6]{Raksit20}, the monad $ G \circ F $ is of the form $ \theta(A) $ for $ A $ the constant $ C_2 $-symmetric sequence at $ \mathbbm{1} $. 
	Hence $ A $ inherits an operad structure. 
\end{cons}
\begin{cons}\label{cons:filtered_C2_sseq}
	For $ i \in \Z_{\geq 0} $, let $ \underline{\SSeq}^{\leq i}(\cat) \subseteq \underline{\SSeq}(\cat) $ denote the full $ C_2 $-subcategory spanned by those $ C_2 $-symmetric sequences such that $ A(S\to C_2/H) $ is an initial object of $ \cat^H $ whenever $ |S \times_{C_2/H}C_2 |/2 > i $. 

	Now consider the $ \infty $-category $ \Fun_{C_2}\left(\Z_{\geq 0} \times \mathcal{O}^\op_{C_2}, \underline{\SSeq}(\cat)\right) $, and let $ \SSeq^{\leq \ast}(\cat) $ denote the full subcategory spanned by those functors $ F $ such that $ F(i,-) \in \SSeq^{\leq i}(\cat) $ for all $ i \geq 0 $. 
	The inclusion $ \varphi: \SSeq^{\leq \ast}(\cat) \to \Fun_{C_2}\left(\Z_{\geq 0}\times \mathcal{O}^\op, \underline{\SSeq}(\cat)\right) $ admits a right adjoint $ \psi $. 

	Now $ \Fun_{C_2}\left(\Z_{\geq 0}^\times\times \mathcal{O}^\op, \underline{\SSeq}(\cat)^{C_2,\circ}\right) $ acquires a monoidal structure via parametrized Day convolution; write $ p \colon \Fun_{C_2}\left(\Z_{\geq 0}^\times\times \mathcal{O}^\op_{C_2}, \underline{\SSeq}(\cat)^{\circ}\right)^\otimes \to \Assoc^\otimes $ for the corresponding fibration of $ \infty $-operads. 
	As in \cite[Construction 4.1.7]{Raksit20}, the full subcategory $ \SSeq^{\leq \ast}(\cat)^\otimes \subseteq \Fun_{C_2}\left(\Z_{\geq 0}^\times\times \mathcal{O}^\op_{C_2}, \underline{\SSeq}(\cat)^{\circ}\right)^\otimes $ determined by the inclusion $ \varphi $ is closed under the Day convolution product and $ p $ is a locally cocartesian fibration, hence $ \psi $ promotes to a map of $ \infty $-operads $ \Fun_{C_2}\left(\Z_{\geq 0}^\times\times \mathcal{O}^\op_{C_2}, \underline{\SSeq}(\cat)^{\circ}\right)^\otimes \to \SSeq^{\leq \ast}(\cat)^\otimes $.  

	Finally, we will write $ (-)^{\leq \ast} $ for the composite
	\begin{equation*}
		\underline{\SSeq}(\cat)^{C_2,\circ} \xrightarrow{\delta} \Fun_{C_2}\left(\Z_{\geq 0}^\times \times \mathcal{O}^\op_{C_2}, \underline{\SSeq}(\cat)^{\circ}\right)^\otimes \xrightarrow{\psi} \SSeq^{\leq \ast}(\cat)^\otimes
	\end{equation*}
	and for the induced composite on algebra objects
	\begin{equation*}
		C_2\mathrm{Op}(\cat) \simeq \E_1\Alg \left(\underline{\SSeq}(\cat)^{C_2}\right) \to \E_1\Alg\Fun_{C_2}\left(\Z_{\geq 0}^\times \times \mathcal{O}^\op_{C_2}, \underline{\SSeq}(\cat)^{\circ}\right) \simeq \Fun^{\mathrm{lax}}_{C_2}\left(\Z_{\geq 0}^\times\times \mathcal{O}^\op_{C_2}, \underline{\SSeq}(\cat)^{\circ}\right) . 
	\end{equation*}
\end{cons}
\begin{cons}
	[$C_2 $-symmetric powers filtered monad] \label{cons:C2_sym_filtered_monad}
	Let $ \cat $ be a presentable $ C_2 $-symmetric monoidal $ \infty $-category. 
	Let $ \mathcal{E} $ denote the full $ C_2 $-subcategory of $ \underline{\End}(\cat) $ on those endofunctors which preserve sifted colimits pointwise. 

	Let $ \mathcal{E}' $ denote the full $ C_2 $-subcategory of $ \underline{\End}(\underline{\SSeq}(\cat)) $ on those endofunctors which preserve both sifted colimits and the essential image of the embedding $ \cat \to \underline{\SSeq}(\cat) $. 
	Write $ \rho \colon \mathcal{E}' \to \mathcal{E} $ for the canonical inclusion functor. 

	From Construction \ref{cons:freeC2alg_as_C2sseq}, there is a $ C_2 $-$ \infty $-operad $ A $ with an equivalence of monads $ \theta(A) \simeq C_2\Sym_{\underline{\SSeq}(\cat)} $. 
	Applying Construction \ref{cons:filtered_C2_sseq}, we obtain a lax monoidal $C_2$-functor $ A^{\leq \ast} \colon \Z_{\geq 0}^\times \times \mathcal{O}^\op_{C_2} \to \underline{\SSeq}(\cat)^{\circ}  $. 
	Write $ C_2 \Sym^{\leq \ast} $ for the composite
	\begin{equation*}
		\Z_{\geq 0}^\times \xrightarrow{A^{\leq \ast} } \underline{\SSeq}(\cat)^{C_2,\circ} \xrightarrow{\theta} \mathcal{E}' \xrightarrow{\rho} \mathcal{E}. 
	\end{equation*}
	Then a modification of the argument of \cite[Proposition 4.1.4]{Raksit20} implies that there is an equivalence of $ \mathcal{E} $-monads $ \colim C_2 \Sym^{\leq \ast} \simeq C_2 \Sym $. 
\end{cons}

\subsection{Derived rings with involution}\label{subsection:involutive_dalg}
The goal of this subsection is to define \emph{derived rings with (genuine) involution}. 
Just as ordinary derived algebras over $ \Z $ can be regarded as generalizations of discrete commutative rings, our derived rings with involution can be regarded as generalizations of cohomological $ C_2 $-Tambara functors (see Variant \ref{variant:otherdalg_inv}\ref{varitem:discrete_dalg_inv}). 

Let us recall the definition of ordinary derived algebras. 
To define derived rings with involution, Raksit identifies a filtered monad $ \CSym_{\Z}^{\leq *} $ on $ \Proj_{\Z} $, extends $ \CSym_{\Z}^{\leq *} $ to a filtered monad $ \LSym_{\Z}^{\leq *} $ on connective $ \Z $-modules $ \Mod_{\Z}^{\mathrm{cn}} $ via left Kan extension, and extends $ \LSym_{\Z}^{\leq *} $ to a filtered monad on all $ \Z $-modules using Goodwillie calculus. 
Derived algebras are defined to be modules over $ \LSym_{\Z}:= \colim_{*} \LSym_{\Z}^{\leq *} $ the extended monad. 
To define $ \CSym_{\Z}^{\leq *} $, Raksit takes $ \CSym_{\Z}^{\leq *} :=  \pi_0 \Sym^{\leq * }_\Z $ where $ \Sym^{\leq *}_{\Z} $ is the filtered free $ \E_\infty $-algebra monad. 
In particular, the property that $ \CSym_{\Z}^{\leq *} $ preserves the subcategory $ \Proj_{\Z} $ is crucial, for means that $ \CSym_{\Z}^{\leq *} $ inherits a filtered monad structure from $ \Sym^{\leq *}_{\Z} $. 
Furthermore, the resulting map $ \Sym^{\leq * }_\Z \to \LSym^{\leq *}_{\Z} $ implies that any derived algebra $ A $ has an underlying $ \E_\infty $-$ \Z $-algebra. 

By virtue of the domain of definition of real trace theories, a derived involutive algebra should have an underlying $ C_2 $-$ \E_\infty $-algebra. 
By Proposition \ref{prop:constMackeyisnormed}, the constant $ C_2 $-Mackey functor $ \underline{\Z} $ is a suitable base. 
Many of the ideas and constructions in \cite[\S4.2]{Raksit20} generalize directly to $ \Mod_{\underline{\Z}} $; for instance, using the parametrized monads of Definition \ref{defn:param_monad} in place of ordinary monads, and the free $ C_2 $-$ \E_\infty $-$ \underline{\Z} $-algebra in place of the free $ \E_\infty $-$ \Z $-algebra. 
However, despite the fact that $ \Mod_{\underline{\Z}} $ has a t-structure which is compatible with its $ C_2 $-symmetric monoidal structure, it is no longer true that $ \underline{\pi}_0 $ of the free $ C_2 $-$ \E_\infty $-algebra on a free $ \underline{\Z} $-module on some finite $ C_2 $-set is a free $ \underline{\Z} $-module on some finite $ C_2 $-set. 
We use the zeroth slice functor in place of $ \underline{\pi}_0 $ to define $ C_2\CSym $ from the free $ C_2 $-$ \E_\infty $-$ \Z $-algebra monad. 

We will define derived involutive algebras for any $ C_2 $-$ \infty $-category satisfying a list of axioms, though the relevant examples for us consist primarily of $ \underline{\Mod}_{\underline{\Z}} $ and its filtered and graded variants, and are discussed in greater detail in \S\ref{subsection:inv_dalg_context_examples}. 
\begin{defn}
		A \emph{derived involutive algebraic context} consists of a $ C_2 $-stable $ C_2 $-presentable $ C_2 $-symmetric monoidal $ C_2 $-$ \infty $-category, a compatible t-structure (Definition \ref{defn:compatible_filtration}), a localization $ C_2 $-functor $ P \colon \cat_{\geq 0} \to \cat_{\geq 0} $ with essential image contained in $ \cat^\heartsuit $, and a small full subcategory $ \cat^0 \subseteq P(\cat_{\geq 0})=: \cat^s $ satisfying: 
		\begin{enumerate}[label=(\alph*)]
			\item the t-structure is right-complete 
			\item $ P $ is compatible with the $ C_2 $-symmetric monoidal structure on $ \cat_{\geq 0} $ in the sense of \cite[\S2.9]{NS22}
			\item $ \cat^0 \subseteq \cat $ is a $ C_2 $-symmetric monoidal subcategory which is closed under $ \cat^s $ $ C_2 $-symmetric powers (see \cite[Example 4.3.7]{NS22}) 
			\item $ \cat^0 $ is closed under finite coproducts in $ \cat $ and its objects form a set of compact generators for $ \cat_{\geq 0} $. 
		\end{enumerate}
		We will often denote a derived involutive algebraic context by $ \cat $ and regard the additional data as implicit. 
		A \emph{morphism of derived involutive algebraic contexts} is a $ C_2 $-functor $ F \colon \cat \to \mathcal{D} $ which is right t-exact, $ C_2 $-symmetric monoidal, compatible with the localizations $ P_\cat $ and $ P_{\mathcal{D}} $ and satisfies $ F(\cat^0) \subseteq \mathcal{D}^0 $. 
\end{defn} 
\begin{rmk}\label{rmk:inv_dalg_context_norm_properties}
		In this situation, the norm functor $ N^{C_2} \colon \cat^e \to \cat^{C_2} $ is $ 2 $-excisive. 
		First, observe that $ N^{C_2} $ preserves sifted colimits by the distributivity assumption and \cite[Proposition 8.19]{Shah_paramII}. 
		By the proofs of Propositions 4.2.14 and 4.2.15 of \cite{Raksit20}, it suffices to show that $ (\cat^0)^e \to \cat^{C_2} $ is of degree 2. 
		The observation follows from \cite[Example 3.17]{Nardinthesis}.  	
\end{rmk} 
\begin{rmk}
		Let $ \cat $ be a derived involutive algebraic context. 
		If we furthermore assume that $ (P^0)^e \simeq \pi_0^e $ as functors $ \cat^e_{\geq 0} \to \cat^e_{\geq 0} $, then $ \cat^e $ is a derived algebraic context in the sense of \cite[Defintiion 4.2.1]{Raksit20}. 
\end{rmk}
We record an auxiliary lemma which will be used to construct the derived involutive symmetric powers as a monad.
\begin{lemma}\label{lemma:derived_slices} 
	Let $ F \colon \cat_{\geq 0} \to  \cat_{\geq 0} $ be a $C_2$-functor which preserves sifted $ C_2 $-colimits (see Remark \ref{rmk:fiberwise_sifted_vs_param_sifted}). 
	Then for any $ X \in \cat_{\geq 0} $, the canonical map $ X \to P^0 X $ induces an equivalence
	\begin{equation*}
		P^0(F(X)) \to P^0\left(F\left(P^0X\right) \right) \,. 
	\end{equation*}
	On underlying $ \infty $-categories, this recovers the equivalence of \cite[Lemma 4.2.17]{Raksit20}. 
\end{lemma}
\begin{proof}
	Observe that the $ C_2 $-$ \infty $-category $ \cat_{\geq 0} $ is generated pointwise under sifted colimits by $ (\cat^{C_2})^0 $ and $ (\cat^e)^0 $, and $ P^0 \circ F $ and $ P^0 \circ F \circ P^0 $ preserve sifted colimits as functors $ \cat_{\geq 0} \to \cat^s $. 
	Therefore, it suffices to show that they agree on the subcategory $ \cat^0 $. 
	However, $ P^0 $ acts by the identity on $ \cat^0 $ by assumption, hence the result follows. 
\end{proof}
\begin{rmk}\label{rmk:adjunctions_of_endomorphism_cats} 
	Let us regard the (ordinary $\infty$-)category of $C_2$-endofunctors $ \End_{C_2}(\cat) $ as a monoidal $ \infty $-category under composition. 
	Write $ \End^\Sigma_{C_2}\left(\cat_{\geq 0}\right) $ for the full subcategory of $ \End_{C_2}\left(\cat\right) $ on those $C_2$-functors which preserve sifted colimits pointwise. 
	Consider the following full subcategories of $ \End^\Sigma_{C_2}\left(\cat_{\geq 0}\right) $:
	\begin{itemize}
		\item $ \End^\Sigma_{C_2,0}\left(\cat_{\geq 0}\right) $ spanned by those $ C_2 $-functors $ F $ which both preserve sifted colimits pointwise and satisfy $ F\left(\cat^0\right) \subseteq \cat^0 $. 
		\item $ \End^\Sigma_{C_2,2}\left(\cat_{\geq 0}\right) $ spanned by those $ C_2 $-functors $ F $ which both preserve sifted colimits pointwise and satisfy $ P F\left(\cat^0\right) \subseteq \cat^0 $. 
	\end{itemize}
	The monoidal stucture on $ \End_{C_2}\left(\cat_{\geq 0}\right) $ descends to one on $ \End^\Sigma_{C_2,0}\left(\cat_{\geq 0}\right) $, and by Lemma \ref{lemma:derived_slices}, $ \End^\Sigma_{C_2,2}\left(\cat_{\geq 0}\right) $ is a monoidal subcategory of $ \End^\Sigma_{C_2}\left(\cat_{\geq 0}\right) $. 
	We have an inclusion $ \End^\Sigma_{C_2,0}\left(\cat_{\geq 0}\right) \subseteq \End^\Sigma_{C_2,2}\left(\cat_{\geq 0}\right) $ which is moreover monoidal. 
	The inclusion admits a left adjoint\footnote{We break with our notational conventions and write $ \underline{\tau} $ for a non-parametrized functor; the motivation for doing so is reflected in the last sentence of this remark.} $ \underline{\tau} \colon \End^\Sigma_{C_2,2}\left(\cat_{\geq 0}\right) \to \End^\Sigma_{C_2,0}\left(\cat_{\geq 0}\right) $. 
	Since $ \cat^0 $ generates $ \cat_{\geq 0} $ under sifted colimits, we identify sifted colimit-preserving $C_2$-functors $ \cat_{\geq 0} \to \cat_{\geq 0} $ with their restrictions to $ \cat^0 $. 
	From this description, it follows that the left adjoint $ \underline{\tau} $ is given by composition with $ P $. 
	Since the inclusion is monoidal, $ \underline{\tau} $ is oplax monoidal, and Lemma \ref{lemma:derived_slices} furthermore implies that $ \underline{\tau} $ is strictly monoidal. 

	Furthermore, under the restriction functor of Remark \ref{rmk:restrict_param_endomorphisms}, $ \underline{\tau} $ is sent to $ \tau $ of \cite[Remark 4.2.18]{Raksit20} by Corollary \ref{cor:zeroth_slice_param_enhancement}. 
\end{rmk}
\begin{defn}
	Let $ \cat $ be a $ C_2 $-$ \infty $-category and let $ \mathcal{B} $ be a $ C_2 $-stable $ C_2 $-$ \infty $-category. 
	We say that a $ C_2 $-functor $ F \colon \cat \to \mathcal{B} $ is \emph{$ n $-excisive} if both $ F^e $ and $ F^{C_2} $ (see Remark \ref{rmk:restrict_param_endomorphisms}) are $ n $-excisive. 
	We say that a $ C_2 $-functor is \emph{excisively polynomial} if it is $ n $-excisive for some $ n $. 
\end{defn}
\begin{obs}\label{obs:C2Sym_involutive_excisivity} 
	Let $ \cat $ be a derived involutive algebraic context. 
	Then it follows from Remark \ref{rmk:inv_dalg_context_norm_properties} that $ C_2\Sym^{\leq i}_{\cat}(-) $ is $ i $-excisive.  
\end{obs}
\begin{defn}
		Let $ \mathcal{A},\mathcal{B} $ be additive $ C_2 $-$ \infty $-categories and suppose $ \mathcal{B} $ is idempotent-complete. 
		Say that a $ C_2 $-functor $ \mathcal{A} \to \mathcal{B} $ is of \emph{degree $ n $} if $ F^e $ and $ F^{C_2} $ are both of degree $ n $ in the sense of \cite[Definition 4.2.10]{Raksit20}.  
\end{defn}
\begin{ex}\label{ex:param_symmetric_powers_degree}
		Let $ \mathcal{A} $ be a distributive $ C_2 $-symmetric monoidal $ C_2 $-$ \infty $-category. 
		Then the $ C_2 $-functor $ C_2\Sym_{\mathcal{A}}^{\leq i} $ is of degree $ i $. 
\end{ex}
\begin{recollection}\label{rec:extending_epoly_functors}
	Let $ \mathcal{E} \subseteq \End_{C_2}\left(\cat\right) $ denote the full subcategory spanned by those $C_2$-endofunctors which are excisively polynomial, preserve sifted colimits pointwise, and preserve the full $C_2$-subcategory $ \cat_{\geq 0} $. 
	Let $ \mathcal{E}' \subseteq \End_{C_2}\left(\cat_{\geq 0}\right) $ denote the full subcategory spanned by those $C_2$-endofunctors which are excisively polynomial and preserve sifted colimits pointwise. 
	Because the Postnikov t-structure on $ \cat $ is right-complete, we conclude by Proposition 4.2.15 of \cite{Raksit20} that the monoidal restriction functor $ \mathcal{E} \to \mathcal{E}' $ is an equivalence. 
\end{recollection}
\begin{cons}
	[Derived involutive symmetric powers monad] \label{cons:equivariant_dalg} 
	We will construct a filtered $ \mathcal{E} $-monad $ \LSym^{\sigma, \leq *}_{\cat} $ on the $ C_2 $-$ \infty $-category $ \cat $, with a map of filtered $ \mathcal{E} $-monads $ \theta^{\leq *} \colon C_2\Sym^{\leq *}_{\cat}(-) \to \LSym^{\sigma,\leq *}_{\cat}(-) $, where $ C_2\Sym^{\leq *}_{\cat}(-) $ is the $ C_2 $-symmetric powers filtered monad of Construction \ref{cons:C2_sym_filtered_monad}. 

	By Recollection \ref{rec:extending_epoly_functors}, it will suffice to construct filtered $ \mathcal{E}' $-monads $ \LSym^{\sigma, \leq *}_{\cat_{\geq 0}} $ and $ \theta^{\leq i} $ on $ \cat_{\geq 0} $ instead. 
	By Observation \ref{obs:C2Sym_involutive_excisivity}, $ C_2\Sym_{\cat_{\geq 0}}^{\leq *}(-) $ is a filtered $ \mathcal{E}' $-monad. 
	By definition of an derived involutive algebraic context, $ C_2\Sym^{\leq i}_{\cat_{\geq 0}}(-) \in \End^\Sigma_{C_2,2}\left(\cat_{\geq 0}\right) $ for all $ i \geq 0 $. 
	Using Remark \ref{rmk:adjunctions_of_endomorphism_cats}, define $ \LSym^{\sigma, \leq *}_{\cat_{\geq 0}}(-) = \underline{\tau} C_2\Sym^{\leq *}_{\cat_{\geq 0}}(-) $. 
	Because $ \underline{\tau} $ is monoidal, $ \LSym^{\sigma,\leq *}_{\cat}(-) $ inherits a filtered monad structure from $ C_2\Sym^{\leq *}_{\cat}(-)$. 
	Finally, the unit map of the adjunction $ \underline{\tau} $ induces a map of filtered monads $ \theta^{\leq *} \colon C_2\Sym^{\leq *}_{\cat_{\geq 0}}(-) \to \LSym^{\sigma,\leq *}_{\cat_{\geq 0}}(-) $ in $ \End^\Sigma_{C_2,2}\left(\cat_{\geq 0}\right) $. 

	By Example \ref{ex:param_symmetric_powers_degree} and \cite[Proposition 4.2.14]{Raksit20}, $ \LSym_{\cat_{\geq 0}}^{\sigma,\leq i} $ is a filtered $ \mathcal{E}' $-monad, and $ \theta^{\leq *} $ is a map of filtered $ \mathcal{E}' $-monads.  
\end{cons}
\begin{defn}\label{defn:involutive_derived_alg}
	Let $ \cat $ be an derived involutive algebraic context. 
	Let $ \LSym_{\cat}^\sigma $ denote the colimit of the derived involutive symmetric powers filtered monad $ \LSym^{\sigma, \leq *}_{\cat} $ of Construction \ref{cons:equivariant_dalg}. 
	By Proposition \ref{prop:colim_of_filt_param_monad}, $ \LSym_{\cat}^\sigma $ is a monad on $ \cat $. 
	We will refer to left modules over the monad $ \LSym^{\sigma}_{\cat} $ on $ \cat $ as \emph{derived involutive algebras in $ \cat $}, and denote the $C_2$-$\infty$-category of such objects (Construction \ref{cons:param_left_modules_over_monad}) by $ \underline{\DAlg}^\sigma(\cat) $.  
	
	We will denote the fiber of $ \underline{\DAlg}^\sigma(\cat) $ over the $ C_2 $-set $ C_2/C_2 $ by $ \DAlg^\sigma(\cat) $, and we will denote the fiber of $ \underline{\DAlg}^\sigma(\cat) $ over the $ C_2 $-set $ C_2/e $ by $ \DAlg(\cat^e) $ or $ \DAlg(\cat)^e $. 
\end{defn}
Every derived involutive algebra has an underlying $C_2 $-$ \E_\infty $-algebra.
\begin{ntn}\label{ntn:dalg_calg_forget}
	Recall the map of filtered monads $ \theta^{\leq *} \colon C_2\Sym^{\leq *}_{\cat}(-) \to \LSym^{\sigma,\leq *}_{\cat}(-) $ of Construction \ref{cons:equivariant_dalg}. 
	Taking colimits, we obtain a map $ \theta \colon C_2\Sym_{\cat}(-) \to \LSym^\sigma_{\cat}(-) $ of monads on $ \cat $. 
	This induces a forgetful $C_2$-functor $ \Theta \colon \underline{\DAlg}^\sigma (\cat) \to C_2\E_\infty\underline{\Alg}(\cat) $. 
\end{ntn} 
\begin{prop}\label{prop:invdalg_C2_colims}
\begin{enumerate}[label=(\arabic*)]
	\item \label{propitem:invdalg_admits_C2_colims} The $ C_2 $-$ \infty $-category of derived involutive algebras in $ \cat $ strongly admits all $ C_2 $-small $ C_2 $-colimits. 
	\item \label{propitem:invdalg_admits_C2_lims} The $ C_2 $-$ \infty $-category of derived involutive algebras in $ \cat $ strongly admits all $ C_2 $-small $ C_2 $-limits. 
	\item \label{propitem:invdalg_C2_colims_computed_in_calg} The forgetful $ C_2 $-functor $ \Theta \colon \underline{\DAlg}^\sigma (\cat) \to C_2\E_\infty\underline{\Alg}(\cat) $ strongly preserves all small $C_2$-colimits and $ C_2 $-limits. 
\end{enumerate}
\end{prop}
\begin{proof}
	Since we work with a fixed derived involutive algebraic context $ \cat $, we will mostly suppress $ \cat $ from notation and write $ \underline{\DAlg}^\sigma $ instead of $ \underline{\DAlg}^\sigma(\cat) $ throughout the proof. 
	To show \ref{propitem:invdalg_admits_C2_colims}, by Theorem B of \cite{Shah_paramII}, it suffices to check the conditions
	\begin{enumerate}[label=(\alph*)]
		\item \label{item:param_res_colimits} for every $ C_2 $-set $ T = C_2/H $, the fiber $ \underline{\DAlg}^\sigma_T $ admits all small $ \kappa $-colimits, and for every morphism $ \alpha : T \to S $ in $ \mathcal{O}_{C_2} $, the restriction map $ \alpha^* : \underline{\DAlg}^\sigma_S \to \underline{\DAlg}^\sigma_T $ preserves $ \kappa $-small colimits. 
		\item \label{item:param_res_adjoint} for every map of finite $ C_2 $-sets $ \alpha: U \to V $, the restriction functor $ \alpha^* : \underline{\DAlg}^\sigma_V \to \underline{\DAlg}^\sigma_U $ admits a left adjoint $ \alpha_! $
		\item \label{item:param_bc_condition} $ \underline{\DAlg}^\sigma $ satisfies the Beck--Chevalley condition: For every pullback square 
		\begin{equation*}
		\begin{tikzcd}
			U' \ar[d,"{\alpha'}"'] \ar[r,"{\beta'}"] & U \ar[d,"\alpha"] \\
			V' \ar[r,"\beta"] & V
		\end{tikzcd}
		\end{equation*}
		in $ \Fin_{C_2} $, the exchange map \cite[immediately before Definition 7.3.1.1]{LurHTT}
		\begin{equation*}
			\alpha'_! \beta'^* \implies \beta^*\alpha_!
		\end{equation*}
		is an equivalence. 
	\end{enumerate}
	For criterion \ref{item:param_res_colimits}, each fiber $ \underline{\DAlg}^\sigma_{T} $ is presentable by \cite[Proposition 4.1.10]{Raksit20}. 
	Now the only non-isomorphism is $ \alpha: C_2 \to C_2/C_2 $. 
	The functor $ \alpha^* $ preserves sifted colimits because the forgetful functors $ \underline{\DAlg}^\sigma (\cat)^{C_2} \to \cat^{C_2} $ and $ \underline{\DAlg}^\sigma(\cat)^e \to \cat^e $ detect sifted colimits, and the functor $ \alpha^* \colon \cat^{C_2} \to \cat^e $ preserves sifted colimits. 
	To show that the restriction functor $ \alpha^* $ preserves coproducts, we can argue as in \cite[Proposition 4.2.27]{Raksit20} to reduce to showing that the canonical map $ \mu \colon \LSym_{\cat^e}(X^e \oplus Y^e) \to \alpha^*\LSym^\sigma_{\cat}(X \oplus Y) $ is an equivalence for all $ X, Y \in \cat^{C_2} $. 
	Because the map $ \mu $ arises canonically as the colimit of a map $ \LSym_{\cat^e}^{\leq *}(X^e \oplus Y^e) \to \alpha^*\LSym^{\sigma, \leq *}_{\cat} (X \oplus Y) $ in $ \cat^{C_2} \times \cat^{C_2} \to \Fun\left(\Z_{\geq 0}, \cat^e\right) $ which is polynomially excisive in each variable, it suffices to show that $ \mu $ is an equivalence for all $ X, Y \in (\cat^0)^{C_2} $. 
	This follows from the corresponding map where $ \LSym^\sigma_{\cat} $ is replaced by $ C_2\Sym_{\cat} $ is an equivalence. 

	To verify criterion \ref{item:param_res_adjoint}, we show that $ \alpha^* $ satisfies the conditions of the adjoint functor theorem \cite[Theorem 5.5.2.9(2)]{LurHTT}. 
	That $ \underline{\DAlg}_U^\sigma ,\underline{\DAlg}_V^\sigma $ are presentable and $ \alpha^* $ is accessible follow from the verification of criterion \ref{item:param_res_colimits}. 
	It remains to show that $ \alpha^* $ preserves all small limits. 
	Since limits in $ \underline{\DAlg}_U^\sigma \simeq \prod_{W \in \mathrm{Orbit}(U)} \underline{\DAlg}_W^\sigma $ are created by the forgetful functor to $ \prod_{W \in \mathrm{Orbit}(U)} \cat_W $ (and likewise for $ V $), \ref{item:param_res_adjoint} follows from the fact that the restriction functor $ \alpha^*\colon \cat^{C_2} \to \cat^e $ preserves all small limits. 

	We show how to check criterion \ref{item:param_bc_condition} in the special case where the pullback square is given by 
	\begin{equation*}
	\begin{tikzcd}
		C_2 \times C_2 \ar[r,"\pi_2"] \ar[d,"\pi_1"'] & C_2 \ar[d,"\alpha"] \\
		C_2 \ar[r,"\alpha"] & C_2/C_2
	\end{tikzcd}
	\end{equation*}
	in $ \Fin_{C_2} $; the general case is similar. 
	If $ \alpha^* = (-)^e \colon \underline{\DAlg}^\sigma_{C_2/C_2} \to \underline{\DAlg}^\sigma_{C_2/e} $ denotes restriction along $ C_2 \to C_2/C_2 $, let us write $ (-)^{\boxtimes C_2} $ for its left adjoint. 
	Fixing a trivialization $ C_2 \times C_2 \simeq C_2 \sqcup C_2 $, we may identify the left adjoint to restriction along $ \pi_1 $ as a twisted form of the coproduct $ -\boxtimes - \colon \underline{\DAlg}^\sigma_{C_2 \times C_2} \to \underline{\DAlg}^\sigma_{C_2} $. 
	Now unraveling definitions, the exchange map $ \pi_{1!} \pi_2^* \implies \alpha^* \alpha_! $ is given on $ A \in \DAlg^\sigma(\cat)^e $ by a composite 
	\begin{align*}
		A \boxtimes A \xrightarrow{\varepsilon} (A^{\boxtimes C_2})^e \boxtimes (A^{\boxtimes C_2})^e \xrightarrow{} ((A^{\boxtimes C_2})^e)^{\otimes 2} \xrightarrow{\mu} (A^{\boxtimes C_2})^e \,.
	\end{align*}
	Because the functors $ \pi_{1!} $, $ \pi_2^* $, $ \alpha_! $, $ \alpha^* $ all preserve sifted colimits, it suffices to show that the exchange map is an equivalence when $ A = \LSym_{\cat^e}^\sigma(N) $ for some $ N \in \cat^e $. 
	Because $ \LSym_{\cat^e} $ preserves finite coproducts, we have $ \LSym_{\cat^e}^{\sigma}(N) \boxtimes \LSym_{\cat^e}^{\sigma}(N) \simeq \LSym_{\cat^e}(N \oplus N) $. 
	Because $ \left(\LSym_{\cat}^\sigma\right)^e \simeq \LSym_{\cat^e}(-^e) $, the corresponding diagram of left adjoints commutes and we have an equivalence $ \LSym_{\cat^e}^{\sigma}(N)^{\boxtimes C_2} \simeq \LSym_{\cat^e}(C_2 \otimes N) $. 
	Thus we may regard the exchange map as a morphism $ \LSym_{\cat^e}(N \oplus N) \to \LSym_{\cat}^{\sigma }(C_2 \otimes N) $. 
	Furthermore, let us observe that the exchange map arises as the colimit of a filtered map $ \LSym_{\cat^e}^{\leq *}(N \oplus N) \to \LSym_{\cat}^{\sigma,\leq *}(C_2 \otimes N) $. 
	Because each filtered piece preserves filtered colimits in $ N $ and is excisively polynomial, it suffices to show the exchange map is an equivalence for $ N \in (\cat^0)^e $, and this is true by construction. 

	To show \ref{propitem:invdalg_admits_C2_lims}, in view of \cite[Corollary 5.25]{Shah18}, by Theorem B of \cite{Shah_paramII}, it suffices to check the duals to the conditions outlined in the proof of part \ref{propitem:invdalg_admits_C2_colims}. 
	For each $ C_2 $-set $ T \in \mathcal{O}_{C_2} $, the fiber $ \underline{\DAlg}^\sigma_T $ admits all small limits because they are computed in $ \cat_T $. 
	For each $ \alpha \colon T \to S $, the restriction functor $ \alpha^* \colon \underline{\DAlg}^\sigma_S \to \underline{\DAlg}^\sigma_T $ preserves all small limits because they are computed in $ \cat_S $ and $ \cat_T $, respectively and $ \alpha^* \colon \cat_S \to \cat_T $ preserves all small limits. 
	For each map of finite $ C_2 $-sets $ \alpha: U \to V $, the restriction functor $ \alpha^* : \underline{\DAlg}^\sigma_V \to \underline{\DAlg}^\sigma_U $ admits a right adjoint $ \alpha_* $ by the adjoint functor theorem \cite[Theorem 5.5.2.9(1)]{LurHTT} and the proof of part \ref{propitem:invdalg_admits_C2_colims}. 
	Finally (the dual to) condition (c) follows from the fact that limits in $ \underline{\DAlg}^\sigma_T $ are computed in $ \cat_T $. 

	The forgetful $ C_2 $-functor $ \Theta $ preserves small sifted colimits and all limits indexed by constant diagrams because they are reflected by the forgetful $ C_2 $-functors to $ \cat $.  
	By \cites[Lemma 2.8]{norms}[Theorem 8.6(a)]{Shah_paramII}, to prove \ref{propitem:invdalg_C2_colims_computed_in_calg}, it suffices to show that $ \Theta $ preserves finite $C_2$-coproducts. 
	In other words, we need to show that for $ A \in \DAlg^\sigma(\cat)^e $, the canonical map $ \underline{N}^{C_2}\left(\Theta(A)\right) \to \Theta(A^{\boxtimes C_2}) $ is an equivalence, where $ \boxtimes C_2 $ denotes the left adjoint to the forgetful functor $ \DAlg^\sigma(\cat)^{C_2} \to \DAlg^\sigma(\cat)^e $. 
	Since each $ A $ is a geometric realization of free derived algebras and $ \Theta $ and $ (-)^{\boxtimes C_2} $ and $ \underline{N}^{C_2} $ commute with sifted colimits, we may without loss of generality take $ A = \LSym_{\cat^e}(M) $ for $ M \in \cat^e $. 
	That is, we need to show that the canonical map $ {\varepsilon_M \colon \underline{N}^{C_2}\left(\Theta \LSym_{\cat^e}(M)\right) \to \Theta \left(\LSym_{\cat^e}(M)^{\boxtimes C_2}\right) } $ is an equivalence for $ M \in \cat^e $. 
	Since $ \LSym^\sigma_{\cat} $ is a left $ C_2 $-adjoint by Lemma \ref{lemma:param_monad_induces_param_adjunction},	we have a canonical equivalence $ \LSym_{\cat^e}(M)^{\boxtimes C_2} \simeq \LSym^\sigma_{\cat}(C_2 \otimes M) $.
	Thus, arguing as in \cite[Proposition 4.2.27]{Raksit20}, it suffices to show that the map $ \underline{N}^{C_2}\left( \LSym_{\cat^e}(M)\right) \to \left(\LSym^\sigma_{\cat}(C_2 \otimes M)\right) $ is an equivalence. 
	The map $ \varepsilon $ arises as the colimit of a filtered map $ {\varepsilon^{\leq *}_M \colon \left(\Theta \LSym^{\sigma,\leq *}_{\cat^e}(M)\right)^{\ostar C_2} \to \Theta \left(\LSym^{\sigma,\leq *}_{\cat}(C_2 \otimes M)\right) } $, where $ (-)^{\ostar C_2} $ denotes the Day convolution $ C_2 $-symmetric monoidal structure on $ \underline{\Fil}(\cat) $ of Corollary \ref{cor:param_gr_fil_day_convolution}. 
	Since $ \varepsilon^{\leq n}_M $ is excisively polynomial for each $ n $ and preserves sifted colimits, by \cite[Proposition 4.2.15]{Raksit20} it suffices to show that it induces an equivalence on colimits for $ M \in \cat^0 $. 
	This follows from the fact that $ C_2\Sym_{\cat_{\geq 0}} $ takes $ C_2 $-coproducts to the norm as a functor $ \cat^0 \to \cat_{\geq 0} $ and $ P $ is $ C_2 $-symmetric monoidal. 
\end{proof}
Let $ \cat $ be a derived algebraic context, and write $ \Pi_{C_2} $ for the right adjoint to the restriction functor $ \cat^{C_2} \to \cat^e $. 
By Proposition \ref{prop:invdalg_C2_colims}\ref{propitem:invdalg_admits_C2_colims}, the restriction functor $ \alpha^* \colon \DAlg^\sigma(\cat)^{C_2} \to \DAlg^\sigma(\cat)^e $ admits a right adjoint $ \alpha_* $, where $ \alpha \colon C_2/e \to C_2/C_2 $. 
\begin{prop}\label{prop:dalg_coinduction}
	Let $ \cat $ be a derived involutive algebraic context, and write $ \Pi_{C_2} $ for the right adjoint to the restriction functor $ \cat^{C_2} \to \cat^e $. 
	Write $ U \colon \underline{\DAlg}^\sigma(\cat) \to \cat $ for the forgetful $ C_2 $-functor. 		
	Then in the adjunction $ (\alpha^*, \alpha_*) $ between $ \DAlg^\sigma(\cat)^{C_2} $ and $ \DAlg^\sigma(\cat)^e $, the right adjoint $ \alpha_* $ agrees with $ U $ on underlying objects, i.e. there is a commutative diagram 
	\begin{equation*}
	\begin{tikzcd}
			\DAlg^\sigma(\cat)^e \ar[r,"{\alpha_*}"] \ar[d,"{U^e}"] & \DAlg^\sigma(\cat) \ar[d,"{U^{C_2}}"] \\
			\cat^e \ar[r,"{\Pi_{C_2}}"] & \cat^{C_2} \,.
	\end{tikzcd}
	\end{equation*}
\end{prop}
\begin{proof}
	By definition, the functor $ \alpha_* \colon \DAlg^\sigma(\cat)^e \to \DAlg^\sigma(\cat)^{C_2} $ satisfies that for all $ B \in \DAlg^\sigma(\cat)^{C_2} $ and $ A \in \DAlg^\sigma(\cat)^e $, there is an equivalence 
	\begin{equation*}
		\hom_{\DAlg^\sigma(\cat)}(B, \alpha_* (A)) \simeq \hom_{\DAlg^\sigma(\cat)^e}(B^e, A) .
	\end{equation*} 
	Suppose $ B = \LSym^\sigma_{\cat}(M) $ for some $ M \in \cat^{C_2} $. 
	The aforementioned equivalence becomes 
	\begin{equation*}
		\hom_{\cat^{C_2}}(M, U^{C_2} \alpha_* (A)) \simeq \hom_{\cat^e}(M^e, U^{e} A) .
	\end{equation*} 
	Since $ \Pi_{C_2} $ is the right adjoint to $ (-)^e: \cat^{C_2} \to \cat^e $, we have shown that the underlying $ \underline{k} $-module of $ \alpha_* (A) $ is equivalent to $ \Pi_{C_2} U^e A $. 
\end{proof} 
\begin{rmk}\label{rmk:dalg_construction_naturality}
	Given a morphism $ F \colon \cat \to \mathcal{D} $ of derived involutive algebra contexts, there is an induced $ C_2 $-functor $ F' \colon \underline{\DAlg}^\sigma(\cat) \to \underline{\DAlg}^\sigma(\mathcal{D}) $ which strongly preserves $ C_2 $-colimits so that the diagrams
	\begin{equation*}
	\begin{tikzcd}
		\underline{\DAlg}^\sigma(\cat) \ar[r,"{F'}"] \ar[d] & \underline{\DAlg}^\sigma(\mathcal{D}) \ar[d] & & \underline{\DAlg}^\sigma(\cat) \ar[r,"{F'}"]  & \underline{\DAlg}^\sigma(\mathcal{D})\\
		\cat \ar[r,"{F}"] & \mathcal{D} & & \cat \ar[u,"{\LSym^\sigma_\cat}"] \ar[r,"{F}"] & \mathcal{D}\ar[u,"{\LSym^\sigma_{\mathcal{D}}}"']
	\end{tikzcd}
	\end{equation*}
	commute canonically. 
	By Corollary \ref{cor:C2_right_adjoint_local_crit}, $ F' $ and $ F $ admit right $ C_2 $-adjoints $ G' $ and $ G $, and commutativity of the right-hand diagram above implies that the diagram
	\begin{equation*}
	\begin{tikzcd}
			\underline{\DAlg}^\sigma(\cat) \ar[d,"{U_{\cat}}"] & \underline{\DAlg}^\sigma(\mathcal{D}) \ar[d,"{U_{\mathcal{D}}}"] \ar[l,"{G'}"]  \\ 
			\cat  & \mathcal{D} \ar[l,"G"] 
	\end{tikzcd}
	\end{equation*}
	also commutes. 
	On underlying $ \infty $-categories, these diagrams recover those of \cite[Remark 4.2.25]{Raksit20}. 
\end{rmk}
\begin{defn}
	We say that $ A $ is a \emph{derived involutive bialgebra} if $ A $ is a dualizable $ C_2 $-$ \underline{\E}_\infty $-coalgebra object of $ \underline{\DAlg}^\sigma(\cat) $. 
	There is a $ C_2 $-$ \infty $-category of derived involutive bialgebras $ \underline{\coAlg}\underline{\DAlg}^\sigma (\cat) $ whose underlying $ \infty $-category is the $ \infty $-category of derived bicommutative bialgebra objects of \cite[Definition 4.2.30]{Raksit20}. 
\end{defn} 
If $ A $ is a derived involutive bialgebra object, then comodules over $ A $ admits a notion of derived involutive algebra objects. 
\begin{cons}\label{cons:dalg_in_comodules}
	Let $ A $ be a derived involutive bialgebra object over $ k $. 
	Then because the forgetful $ C_2 $-functor $ \underline{\DAlg}^\sigma (\cat) \to \cat $ is canonically $C_2$-symmetric monoidal, the argument of \cite[Construction 4.2.32]{Raksit20} applies to show the existence of a commutative diagram
	\begin{equation*}
	\begin{tikzcd}
		{\underline{\coMod}_{A^\vee}} \ar[r,"{\LSym^\sigma}"] \ar[d] & {\underline{\coMod}_{A^\vee}} \ar[d] \\
		\cat \ar[r,"{\LSym^\sigma_{\cat}}"] & \cat
	\end{tikzcd}
	\end{equation*}
	of $ C_2 $-functors of $ C_2 $-$\infty $-categories, where we have abbreviated $ \underline{\coMod}_A = \underline{\coMod}_A(\cat) $. 
	Note that taking `underlying' recovers the diagram of \cite[Construction 4.2.32]{Raksit20}. 

	Suppose that $ A $ is moreover dualizable. 
	We write $ \underline{\DAlg}^\sigma\left(\underline{\Mod}_A\right) := \underline{\Mod}_{\LSym^\sigma}\left(\underline{\coMod}_{A^\vee}\right) $ and write $ \DAlg^\sigma(\Mod_A) $ for its $ C_2 $-fixed points. 
	We will refer to these as the ($C_2$-)$ \infty $-category of \emph{derived involutive algebra objects of } $ \Mod_A $. 
\end{cons}
\begin{rmk}\label{rmk:dalg_functoriality_in_comodules}
	Let $ F \colon \cat \to \mathcal{D} $ be a morphism of derived involutive algebra contexts. 
	Let $ B $ be a dualizable $ C_2 $-$ \E_\infty $ bialgebra object of $ \cat $ equipped with a lift of $ B^\vee $ to a derived involutive bialgebra object in $ \cat $. 
	Then $ F(B) $ and $ F(B^\vee ) \simeq F(B)^\vee $ inherit the same structure in $ \mathcal{D} $, and there is an induced functor $ F_B \colon \underline{\LMod}_{B}(\cat)\to \underline{\LMod}_{F(B)}(\mathcal{D}) $. 
	By Remark \ref{rmk:dalg_construction_naturality}, there is an induced $ C_2 $-functor
	\begin{equation*}
		F_B' \colon \underline{\DAlg}^\sigma\left(\underline{\LMod}_{B}(\cat)\right) \simeq \underline{\DAlg}^\sigma\left(\underline{\coLMod}_{B^\vee}(\cat)\right) \to \underline{\DAlg}^\sigma\left(\underline{\LMod}_{F(B)}(\mathcal{D})\right) \simeq \underline{\DAlg}^\sigma\left(\underline{\coLMod}_{F(B)^\vee}(\mathcal{D})\right) 
	\end{equation*}
	making the diagrams
	\begin{equation*}
	\begin{tikzcd}[cramped]
		\underline{\DAlg}^\sigma\left(\underline{\LMod}_{B}(\cat)\right)\ar[d] \ar[r,"{F_B'}"] & \underline{\DAlg}^\sigma\left(\underline{\LMod}_{B}(\cat)\right) \ar[d] & & \underline{\DAlg}^\sigma\left(\underline{\LMod}_{B}(\cat)\right) \ar[r,"{F_B'}"] & \underline{\DAlg}^\sigma\left(\underline{\LMod}_{B}(\cat)\right) \\
		\underline{\LMod}_B(\cat) \ar[r,"{F}"] & \underline{\LMod}_{F(B)}(\mathcal{D}) & & \underline{\LMod}_B(\cat) \ar[u,"{\LSym^\sigma}"] \ar[r,"{F}"] & \underline{\LMod}_{F(B)}(\mathcal{D}) \ar[u,"{\LSym^\sigma}"]
	\end{tikzcd}
	\end{equation*}
	commute, where the vertical arrows on the left are forgetful $ C_2 $-functors. 
\end{rmk}

\subsection{Examples}\label{subsection:inv_dalg_context_examples}
In this section, we discuss the derived involutive algebraic contexts which we will use in the rest of the paper.  
Before we discuss the prototypical example (Example \ref{ex:inv_dalg_over_const_Mackey}), we fix some notation and include some auxiliary results and definitions needed to show that modules over a fixed point $ C_2 $-Mackey functor may be regarded as an derived involutive algebraic context. 
We discuss filtered and graded variants on the prototypical example and relate the examples here to their non-equivariant counterparts. 
Finally, we include a few examples of objects in these $ C_2 $-$ \infty $-categories. 
\begin{ntn}
		Throughout this section, $ k $ will denote some discrete ring with involution and $ \underline{k} $ will denote the associated fixed point Green functor of Notation \ref{ntn:fixpt_green_functor}. 
		We will write $ \underline{\Mod}_{\underline{k}}^0 $ for the full $ C_2 $-subcategory of $ \underline{\Mod}_{\underline{k}} $ spanned by the free $ \underline{k} $-modules on finite $ C_2 $-sets. 
		Note that there is an inclusion $ \underline{\Mod}_{\underline{k}}^0 \subseteq \underline{\Mod}_{\underline{k}}^\heartsuit $. 
\end{ntn}
\begin{defn}
	[{\cites{MR3505179}[Definition 2.1]{MR3007090}}] \label{defn:slice_filtration} 
	For each $ n \in \Z $, let $ \slice_{\geq n} $ denote the localizing subcategory of $ \Mod_{\underline{k}}\left( \Spectra^{C_2}\right) $ generated by $ C_2 \otimes_H S^{\ell\rho_H -\varepsilon} \otimes_{\sphere} \underline{k} $, where $ H $ is a subgroup of $ C_2 $, $ \ell \cdot |H| - \varepsilon \geq n $, and $ \varepsilon = 0, 1 $. 
	We say that an object $ X \in \Mod_{\underline{k}} $ is \emph{slice $ n $-connective} if it is in $ \slice_{\geq n} $. 

	We will write $ \underline{\slice}_{\geq n} $ for the full $ C_2 $-subcategory of $ \underline{\Mod}_{\underline{k}} $ on $ \slice_{\geq n} $. 
\end{defn}
A word of warning: the slice filtration does \emph{not} agree with the \emph{regular} slice filtration of Definition \ref{defn:reg_slicefiltration}.
\begin{rmk}\label{rmk:slice_param_enhancement}
		The inclusion $ \underline{\slice}_{\geq n} \subseteq \underline{\Mod}_{\underline{k}}\left( \Spectra^{C_2}\right) $ admits a right $ C_2 $-adjoint $ \tau_{\geq n}^{\slice} $ by a similar argument to Proposition \ref{prop:filtrations_are_parametrized}; the key is that $ \underline{\slice}_{\geq n} $ is closed under $ C_2 $-coproducts \cite[Proposition 2.6]{MR3007090}. 
		We will also write $ \tau_{\geq n}^{\slice} $ for the composite $ \underline{\Mod}_{\underline{k}} \xrightarrow{\tau^\slice_{\geq n}} \underline{\slice}_{\geq n} \subseteq \underline{\Mod}_{\underline{k}} $. 
		Let $ \tau^\slice_{\leq n} $ denote the $ C_2 $-endofunctor $ \cofib\left(\tau^{\slice}_{\geq n+1} \to \mathrm{id}\right) $ of $ \underline{\Mod}_{\underline{k}} $. 
		Since there is an inclusion $ \slice_{\geq n} \subseteq \slice_{\geq n-1} $, there are canonical natural transformations $ \tau^\slice_{\leq n} \to \tau^\slice_{\leq n-1} $. 
		The fiber of the natural transformation $ P^n = \fib(\tau^\slice_{\leq n} \to \tau^\slice_{\leq n-1}) $ will be called the \emph{$ n $th slice}. 
		We will abuse notation and also write $ \tau^\slice_{\leq n} $ for the composite of $ \tau^\slice_{\leq n} $ with the inclusion $ \underline{\slice}_{\geq n} \subseteq \underline{\Mod}_{\underline{k}}\left( \Spectra^{C_2}\right) $, and likewise for $ P^n $. 
\end{rmk}
We will mainly be concerned with the case $ n = 0 $. 
\begin{prop}
	[{\cite[Corollary 2.16]{MR3007090}}] \label{prop:zeroth_slice_description}
	Let $ \underline{k} $ be the fixed point $ C_2 $-Mackey functor associated to a commutative ring $ k $ with an involution (Notation \ref{ntn:fixpt_green_functor}). 
	For any $ M \in \Mod_{\underline{k}} $, the zeroth slice $ P^0 M $ is the largest quotient of $ \underline{\pi}_0 M $ on which the restriction map $ \pi_0 M^{C_2} \to \pi_0 M^e $ is injective. 	
\end{prop} 
The next corollary follows immediately from Remark \ref{rmk:slice_param_enhancement} and the preceding proposition. 
\begin{cor}\label{cor:zeroth_slice_param_enhancement}
		Same assumptions as in Proposition \ref{prop:zeroth_slice_description}. 
		The functor $ P^0 $ admits a $ C_2 $-parametrized enhancement $ \underline{\Mod}_{\underline{k}} \to \underline{\Mod}_{\underline{k}} $ which we also denote by $ P^0 $. 
		The underlying component of the parametrized functor $ P^0 $ is given by $ \pi_0 $. 
\end{cor}
\begin{lemma}\label{lemma:zeroth_slice_compatible_C2_monoidal_structure}
		Let $ \underline{k} $ be the fixed point Green functor associated to a commutative ring with an involution. 
		Then the zeroth slice functor $ P^0 \colon \underline{\Mod}_{\underline{k},\geq 0} \to \underline{\Mod}_{\underline{k}}^{\slice=0} $ is compatible with the $ C_2 $-symmetric monoidal structure on $ \underline{\Mod}_{\underline{k},\geq 0} $. 
\end{lemma}
\begin{proof}
		It suffices to check the conditions in \cite[Remark 2.9.3]{NS22}. 
		Condition (1) follows from the same argument as \cite[Proposition 2.2.1.8]{LurHA}. 
		To verify condition (2), we must show that for any $ X \in \underline{\Mod}_{\underline{k},\geq 0}^e \simeq \Mod_{k^e, \geq 0} $, the canonical map $ X \to \pi_0 X $ induces an equivalence $ P^0 \underline{N}^{C_2} X \to P^0 \underline{N}^{C_2} \pi_0 X $. 
		Because the functors $ P^0 \circ \underline{N}^{C_2} $ and $ P^0 \underline{N}^{C_2}\pi_0 $ preserve sifted colimits as functors $ \Mod_{k^e, \geq 0} \to \Mod_{\underline{k}}^{\slice = 0} $, it suffices to show that they agree on $ \Mod_{k^e}^0 $. 
		Since $ \pi_0 $ acts by the identity on $ \Mod_{k^e}^0 $, the result follows. 
\end{proof}
\begin{ntn}\label{ntn:zero_slices}
	The subcategory of $ \Mod_{\underline{k}}^\heartsuit $ on the zero slices will be denoted $ \Mod_{\underline{k}}^{\slice=0} $. 
	We will similarly write $ \underline{\Mod}_{\underline{k}}^{\slice=0} $ for the full $ C_2 $-subcategory of $ \underline{\Mod}_{\underline{k}} $ on $ \Mod_{\underline{k}}^{\slice=0} $. 
	By Corollary \ref{cor:zeroth_slice_param_enhancement}, we have $ \underline{\Mod}_{\underline{k}}^{\slice=0}(C_2/e) \simeq \Mod_{k^e}^{\heartsuit} $. 
\end{ntn}
\begin{ex}\label{ex:inv_dalg_over_const_Mackey}
		Let $ \cat = \underline{\Mod}_{\underline{k}} $ be equipped with the Postnikov t-structure (Variant \ref{var:eqvtmodulespostnikov}), take $ \cat^{0} $ to be the full $ C_2 $-subcategory spanned by the free $ \underline{k} $-modules on finite $ C_2 $-sets, and let $ \cat^s $ to be the zero slices of Notation \ref{ntn:zero_slices} (hence $ P $ is the zero slice functor of Corollary \ref{cor:zeroth_slice_param_enhancement}). 
		That $ P $ is compatible with the $ C_2 $-symmetric monoidal structure on $ \cat $ follows from Lemma \ref{lemma:zeroth_slice_compatible_C2_monoidal_structure}. 
		We will denote $ \underline{\DAlg}^\sigma\left(\underline{\Mod}_{\underline{k}}\right) $ by $ \underline{\DAlg}_{\underline{k}}^\sigma $ and refer to objects therein as \emph{derived involutive $ \underline{k} $-algebras} or \emph{derived involutive algebras/rings} when $ k = \Z $ with the trivial action. 
		Similarly, we will denote the fiber of $ \underline{\DAlg}^\sigma_{\underline{k}} $ over $ C_2/C_2 $ by $ \DAlg^\sigma_{\underline{k}} $. 
\end{ex}
Unraveling Proposition \ref{prop:invdalg_C2_colims} and Remark \ref{rmk:param_unstraighten}, we see how the formalism of parametrized $ \infty $-categories allows us to compare our derived involutive algebras with the (non-equivariant) derived algebras of \cite[Definition 4.2.22]{Raksit20}. 
\begin{rmk}\label{rmk:inv_dalg_underlying}
	The fiber of $ \underline{\DAlg}_{\underline{\Z}}^\sigma $ over the $ C_2 $-set $ C_2/e $ may be identified with the ordinary $ \infty $-category $ \Mod_{(\LSym^{\sigma}_{\underline{\Z}})^e}\left(\Mod_{\Z}\right) \simeq \Mod_{\LSym_{\Z}}\left(\Mod_{\Z}\right) $, where $ \LSym_{\Z} $ is the monad of \cite[Example 4.3.1]{Raksit20}, and the map $ C_2/e \to C_2/C_2 $ of $ C_2 $-sets classifies a functor $ \DAlg_{\underline{\Z}}^\sigma \to \DAlg_\Z $. 

	Moreover, there is a colimit-preserving functor $ {(-)^e \colon \DAlg_{\underline{k}}^\sigma \to \DAlg_{k^e}} $ such that the diagrams
	\begin{equation*}
	\begin{tikzcd}
		\DAlg_{\underline{k}}^\sigma \ar[d, "U"'] \ar[r,"(-)^e"] & \DAlg_{k^e}^{BC_2} \ar[d,"U"] \\
		\Mod_{\underline{k}} \ar[r,"(-)^e"] & \Mod_{k^e}  
	\end{tikzcd}
	\qquad \text{and} \qquad
	\begin{tikzcd}
		\DAlg_{\underline{k}} \ar[r,"(-)^e"] & \DAlg_{k^e} \\
		\Mod_{\underline{k}} \ar[r,"(-)^e"] \ar[u,"\LSym^\sigma"] & \Mod_{k^e}  \ar[u,"\LSym"']
	\end{tikzcd}
	\end{equation*}
	canonically commute. 
\end{rmk}
\begin{rmk}\label{rmk:const_inv_dalg}
		Let $ \ell $ be a discrete ring with an involution, and let $ \underline{\ell} $ denote the associated $ C_2 $-Mackey functor. 
		Write $ k:= \ell^{C_2} $ for the subring of elements which are fixed by the involution. 
		The assignment $ M \mapsto \left(M \to M \otimes_{k} \ell \right) $ where $ M \otimes_k \ell $ is given the induced $ C_2 $-action promotes to a functor $ \otimes_k \underline{\ell} \colon \Mod_k \to \Mod_{\underline{\ell}} $. 
		Observe that $ \otimes_k \underline{\ell} $ makes the diagram
		\begin{equation*}
		\begin{tikzcd}
		 		\Mod_{k}  \ar[d,"{\otimes_k \underline{\ell}}"'] \ar[r,"{\pi_0}"] & \Mod_{k}  \ar[d,"{\otimes_k \underline{\ell}}"] \\
		 		\Mod_{\underline{\ell}} \ar[r,"{P^0}"] & \Mod_{\underline{\ell}}
		 	\end{tikzcd} 	
		\end{equation*} 
		commute and sends discrete finite free $ k $-modules to free $ \underline{\ell} $-modules on $ C_2 $-sets with trivial action. 
		Moreover, $ \otimes_k \underline{\ell} $ is symmetric monoidal and for any $ \underline{\ell}[S] $ where $ C_2 $ acts trivially on $ S $, the canonical map $ P^0 \Sym_{\underline{\ell}}\left(\underline{\ell}[S]\right) \to P^0 C_2\Sym_{\underline{\ell}}\left(\underline{\ell}[S]\right) $ is an equivalence. 
		It follows that $ \otimes_k \underline{\ell} $ induces a functor of $ \infty $-categories $ \DAlg_k \to \DAlg^\sigma_{\underline{\ell}} $. 
		The precise argument is similar to Remark \ref{rmk:dalg_construction_naturality}, and we leave the details to the reader.  
		When $ \ell $ has the trivial $ C_2 $-action (so $ k^{C_2} = \ell $), this functor sends a derived algebra over $ \ell $ to the constant derived involutive algebra over $ \underline{\ell} $. 
\end{rmk}
\begin{variant}\label{variant:otherdalg_inv} 
	Let $ \DAlg_{\underline{k}}^{\sigma,\conn} = \DAlg_{\underline{k}}^\sigma \times_{\Mod_{\underline{k}}} \Mod_{\underline{k},\geq 0}  $ and $ \DAlg_{\underline{k}}^{\sigma,\slice = 0} = \DAlg_{\underline{k}}^\sigma \times_{\Mod_{\underline{k}}} \Mod_{\underline{k}}^{\slice = 0} $. 
	\begin{enumerate}[label=(\arabic*)]
		\item Let $ \mathcal{D}^0 $ denote the full subcategory of $ \DAlg_{\underline{k}}^{\sigma,\conn} $ spanned by the objects $ \LSym_{\underline{k}}^\sigma(X) $ for $ {X \in \Mod_{\underline{k}}^0} $. 
		Then $ \DAlg_{\underline{k}}^{\sigma,\conn} $ is projectively generated by $ \mathcal{D}^0 $, so there is a canonical equivalence $ \DAlg_{\underline{k}}^{\sigma,\conn} \simeq \pShv^\Sigma(\mathcal{D}^0) $. 

		\item \label{varitem:discrete_dalg_inv} Consider the localization $ P^0 \colon \Mod_{\underline{k},\geq 0} \rlarrows \Mod_{\underline{k}}^{\slice = 0} \colon \iota $. 
		Then by \cite[Proposition 4.1.9]{Raksit20}, there is an equivalence $ \DAlg_{\underline{k}}^{\sigma,\slice = 0} \simeq \Mod_{T_0}\left(\Mod_{\underline{k}}^{\slice = 0}\right) $ where $ T_0 = P^0 \LSym^\sigma \iota $ is the induced monad on $ \Mod_{\underline{k}}^{\slice = 0} $. 
		Unravelling definitions, we see that an object $ A $ of $ \DAlg_{\underline{k}}^{\sigma, \slice = 0} $ consists of a Tambara functor $ A $ so that the restriction map $ R \colon A^{C_2} \to A^e $ is injective and the composite of the restriction with the internal norm is given by squaring, i.e. if $ a \in A^{C_2} $, then $ n(a^e) = a^2 $. 
		In other words, $ A $ is cohomological as a $ C_2 $-Tambara functor (Definition \ref{defn:cohomological_Tambara}). 
	\end{enumerate}
\end{variant}
\begin{cons}\label{cons:gr_fil_dalg_inv} 
		Let $ \cat $ be a derived involutive algebraic context and consider $ \Gr\left( \cat \right) $ and $ \Fil \left( \cat \right) $. 
		These categories inherit t-structures from $ \cat $ where a graded (resp. filtered) object $ X^* $ is connective if and only if each $ X^n $ is connective. 
		Moreover, by Corollary \ref{cor:param_gr_fil_day_convolution}, they inherit $ C_2 $-symmetric monoidal structures from $ \cat $. 
		Take $ \Gr\left( \cat \right)^0 $ (resp. $ \Fil \left( \cat \right)^0 $) to be the full $ C_2 $-subcategory on finite coproducts of $ \ins^n(X) $ for $ X \in \cat^0 $ and let $ \Gr \left(\cat\right)^s $ be the full subcategory on those objects $ A_* $ so that $ \ev^n(A_*) = A_n \in \cat^s $, and similarly for $ \Fil\left(\cat\right) $. 
		These choices allow us to regard $ \Gr\left( \cat \right)^0 $ and $ \Fil \left( \cat \right)^0 $ as derived involutive algebraic contexts. 
		We define graded (resp. filtered) derived involutive algebras of $ \cat $ to be derived involutive algebras in $ \Gr\left( \cat \right)^0 $ (resp. $ \Fil \left( \cat \right)^0 $), which we denote by $ \Gr\underline{\DAlg}^\sigma $ and $ \Fil\underline{\DAlg}^\sigma (\cat) $, respectively. 
		Similarly, we have nonnegatively graded and filtered variants, which we denote by $ \Gr^{\geq 0}\underline{\DAlg}^\sigma $ and $ \Fil^{\geq 0}\underline{\DAlg}^\sigma (\cat) $, respectively. 
\end{cons}
\begin{rmk}\label{rmk:zeroth_graded_dalg}
	The $ C_2 $-functor $ \ins_0 : \cat \to \Gr\left(\cat\right) $ is a morphism of derived involutive algebraic contexts with right $ C_2 $-adjoint $ \ev_0 $, 
	whence we have an induced $ C_2 $-adjunction $ \ins_0 \colon \underline{\DAlg}^\sigma(\cat) \rlarrows \Gr\underline{\DAlg}^\sigma(\cat) \colon \ev_0 $ by Remark \ref{rmk:dalg_construction_naturality}. 
	Similar statements apply to the functors $ \ins^{\geq 0} \colon \Gr^{\geq 0}\left(\cat\right) \to \Gr\left(\cat\right) $, $ \gr \colon \Fil\left(\cat \right) \to \Gr\left(\cat\right) $, and $ \colim \colon \Fil\left( \cat \right) \to \cat $ in place of $ \ins_0 $ (compare \cite[Remark 4.3.6]{Raksit20}). 
\end{rmk}
\begin{ex}\label{ex:dalg_in_gr_01}
		Let $ \cat $ be an derived involutive algebraic context and let $ \Gr^{\{0,1\}}(\cat) $ denote the full $ C_2 $-subcategory of $ \Gr^{\geq 0}(\cat) $ on those nonnegatively graded objects $ X_* $ so that $ X_* = 0 $ if $ * \neq 0, 1 $. 
		Write $ \iota $ for the inclusion functor. 
		Then $ \Gr^{ \{0,1\}}(\cat) $ is a $C_2$-stable subcategory of $ \Gr^{\geq 0}(\cat) $, the slice filtration on $ \Gr^{\geq 0}(\cat) $ restricts to a slice filtration on $ \Gr^{\{0,1\}}(\cat) $ so that the inclusion $ \iota $ preserves connective objects, and $ \iota $ admits a $ C_2 $-left adjoint $ \lambda $ which is compatible with the $ C_2 $-symmetric monoidal structure on $ \Gr^{\geq 0}(\cat) $ in the sense of \cite[Theorem 2.9.2]{NS22}, and $ \lambda $ sends the canonical compact projective generators of $ \Gr^{\geq 0}(\cat) $ to the compact projective generators of $ \Gr^{\{0,1\}}(\cat) $. 
		Then by a similar argument to that of \cite[Example 4.3.8]{Raksit20}, $ \Gr^{\{0,1\}}(\cat) $ is a derived involutive algebra context and $ \lambda, \iota $ are morphisms of derived involutive algebraic contexts which fit into a localizing adjunction $ \Gr^{\{0,1\}}(\cat)\rlarrows\Gr^{\geq 0}(\cat) $. 
		Note in particular that the norm on $ \Gr^{\{0,1\}}(\cat) $ is given by
		\begin{equation*}
			N^{C_2}\left(X^0, X^1\right) \simeq \left(N^{C_2}(X^0), C_2 \otimes (X^0 \otimes X^1)\right) \,,
		\end{equation*}
		where $ C_2 \otimes - $ denotes the left adjoint to the restriction functor $ \cat^{C_2} \to \cat^e $. 
		We will write $ \Gr^{\{0,1\}}\underline{\DAlg}^{\sigma}(\cat) $ for the $ C_2 $-$ \infty $-category of derived involutive algebras in $ \Gr^{\{0,1\}}\left(\cat\right) $. 
\end{ex} 
We close this section with a few examples of ($ C_2 $-)objects in the categories introduced earlier in this section (for instance, in $ \underline{\DAlg}_{\underline{k}}^\sigma $ of Example \ref{ex:inv_dalg_over_const_Mackey}). 
\begin{ex}
	The free derived involutive $ \underline{k} $-algebra $ \underline{k}[x^{triv}] $ on the $ \underline{k} $-module $ \underline{k} $ is given by the Lewis diagram 
	\begin{equation*}
	\begin{tikzcd}
		k[x] \ar[d, bend right=20,"\mathrm{Res}"'] \\
		k[x] \ar[u, bend right=20,"\mathrm{Tr}"'] \ar[loop below,"{x \mapsto x}",out=-60, in=240,distance=15]
	\end{tikzcd}
	\end{equation*}
	The restriction map is the identity and the transfer map is multiplication by $ 2 $. 
	Note that the norm map satisfies $ N(\mathrm{Res}(x)) = x^2 $.  
\end{ex} 
\begin{ex}
	The free derived involutive algebra $ \underline{k} [x, x_\sigma]$ on the $ \underline{k} $-module $ \underline{k}[C_2] $ is given by the Lewis diagram 
	\begin{equation*}
	\begin{tikzcd}
		k[t_i, x \cdot x_\sigma] / 
		\begin{pmatrix} t_i\cdot t_j = t_{i+j} + (xx_\sigma)^j \cdot t_{i-j} \text{ when } i>j \\
						t_i \cdot t_i = t_{2i} + 2(xx_\sigma)^i \end{pmatrix} \ar[d, bend right=20,"\mathrm{Res}"'] \\
		k[x, x_\sigma] \ar[u, bend right=20,"\mathrm{Tr}"'] \ar[loop below,"{x \mapsto x_\sigma}",out=-60, in=240,distance=15]
	\end{tikzcd}
	\end{equation*}
	The restriction map takes $ t_i \mapsto x^i + x_\sigma^i $ and the transfer map takes $ f \mapsto f + \sigma(f) $. 
	Note that $ \underline{k} [x, x_\sigma]^{tC_2} = \bigoplus_{n \geq 0} k^{tC_2}\{x_N^n\} $ and $ \Phi^{C_2}\underline{k} [x, x_\sigma] = \tau_{\geq 0} \underline{k} [x, x_\sigma]^{tC_2} = \tau_{\geq 0} k^{tC_2}[x_N]$. 
\end{ex}
\begin{ex}\label{ex:dalg_inv_cplx_conjugation}
		Let $ A $ be a finite type $ \R $-algebra. 
		Then by Variant \ref{varitem:discrete_dalg_inv}, there is an involutive algebra whose underlying Lewis diagram is 
	\begin{equation*}
	\begin{tikzcd}
		A = \R [x_1, \ldots, x_n]/(f_1,\ldots, f_m) \ar[d, bend right=20,"\mathrm{Res}"'] \\
		\C [x_1, \ldots, x_n]/(f_1,\ldots, f_m) \ar[u, bend right=20,"\mathrm{Tr}"'] \ar[loop below,"{a x_i^n \mapsto \overline{a} x_i^n }",out=-60, in=240,distance=15]
	\end{tikzcd}  \,.
	\end{equation*}
	The restriction is induced by the canonical inclusion $ \R \to \C $ and the transfer is induced by the transfer map $ \C \to \R $ which sends $ a \mapsto a + \overline{a} $. 
\end{ex}
\begin{ex}
	Let $ G $ be an abelian group. 
	Endow $ G $ with the involution given by $ g \mapsto g^{-1} $ and regard $ \Z $ as having the trivial involution. 
	Then $ \underline{\Z}[G] $ is an involutive algebra, where $ {\underline{\Z}}[G]^{C_2} $ is given by the group ring on the 2-torsion subgroup of $ G $. 
\end{ex}

\subsection{Filtered and graded derived involutive rings}\label{subsection:involutive_dalg_connectivity}
Recall that, given an $ \E_\infty $-algebra $ A $ in a category in $ \cat $ and a lax symmetric monoidal functor $ F \colon \cat \to \mathcal{D} $, $ F(A) $ naturally acquires the structure of a $ \E_\infty $-algebra in $ \mathcal{D} $. 
In particular, given a t-structure on $ \cat $ which is compatible with the symmetric monoidal structure on $ \cat $, the filtered object $ \tau_{\geq * } A $ naturally acquires the structure of an $ \E_\infty $-algebra in $ \Fil(\cat) $. 
Similarly, given a filtration (which may or may not come from a t-structure) on a $ C_2 $-symmetric monoidal $C_2$-$\infty $-category $ \cat $ which is compatible with the $C_2$-symmetric monoidal structure on $ \cat $ and a $ C_2 $-$ \E_\infty $-algebra $ A $ in $ \cat $, the filtered object $ \tau_{\geq * } A $ naturally acquires the structure of an $C_2$-$ \E_\infty $-algebra in $ \Fil(\cat) $. 
Unfortunately for us, a lax $C_2$-symmetric monoidal functor between derived involutive algebraic contexts need not induce a functor on categories of derived involutive algebras. 
In particular, $ \tau_{\geq * }A $ is not necessarily a filtered derived algebra when $ A $ is connective but not truncated.\footnote{We thank Arpon Raksit for explaining this point to us.} 

We record some results in this section which allow us to endow graded and filtered $ \underline{\Z} $-modules with \emph{derived} involutive algebra structures from connectivity considerations alone. 
They will be used to define involutive cochain complexes and the (dual) filtered involutive circle later on. 
\begin{lemma}\label{lemma:C2sym_LSym_connectivity_comparison}
		Let $ k \leq 0 $ and $ n \geq 0 $, and let $ \underline{\Z} $ denote the constant $ C_2 $-Mackey functor at $ \Z $. Then
		\begin{itemize}
			\item The fibers of the maps 
			\begin{equation*}
			\begin{split}
				c_{\Z[-k]}\colon& C_2\Sym_{\underline{\Z}}^m\left(\underline{\Z}[-k]\right) \to \LSym^{\sigma,m}_{\underline{\Z}}\left(\underline{\Z}[-k]\right) \\ 
				c_{\Z[C_2][-k]} \colon& C_2\Sym_{\underline{\Z}}^m\left(\underline{\Z}[C_2][-k]\right) \to \LSym^{\sigma,m}_{\underline{\Z}}\left(\underline{\Z}[C_2][-k]\right) 
			\end{split}
			\end{equation*}
			are $ mk $-connective in the Postnikov t-structure on $ \Mod_{\underline{\Z}} $ (Variant \ref{var:eqvtmodulespostnikov}). 
			\item The aforementioned maps $ c_{\Z[-k]} $ and $ c_{\Z[C_2][-k]} $ induce isomorphisms on $ \pi_{mk}(-)^e $ and surjections on $ \pi_{mk}(-)^{C_2} $. 	
		\end{itemize}
\end{lemma} 
\begin{proof}
		Follows from the same argument as \cite[Lemma 4.5.2]{Raksit20}. 
\end{proof}
\begin{prop}\label{prop:LSym_connectivegr}
	Let $ \underline{\Z} $ be the constant $ C_2 $-Mackey functor valued at $ \Z $, regarded as a connective derived algebra with involution.  
	The graded derived symmetric algebra monad $ \LSym^{\sigma}_{\underline{\Z}} \colon \Gr\left(\underline{\Mod}_{\underline{\Z}}\right) \to \Gr\left(\underline{\Mod}_{\underline{\Z}}\right) $ of Variant \ref{variant:otherdalg_inv} preserves the full $C_2$-subcategory $ \Gr^{\leq 0}(\underline{\Mod}_{\underline{\Z}})_{\geq * } $ consisting of those graded modules $ X^* $ such that $ X^n \simeq 0 $ for $ n> 0 $ and $ X^n \in \tau_{\geq n} \Mod_{\underline{k}} $ is $ n $-connective in the Postnikov t-structure on $ \Mod_{\underline{\Z}} $ for $ n \leq 0 $. 
\end{prop}
\begin{proof}
	The subcategory $ \Gr^{\leq 0}\left(\underline{\Mod}_{\underline{\Z}}\right)_{\geq *} $ is closed under colimits, tensor products, and norms, and has compact projective generators given by finite coproducts of objects $ C_2 \otimes \Sigma^{n} \underline{\Z}(n) $ and $ \Sigma^{n} \underline{\Z}(n) $. 
	Now a $ \underline{\Z} $-module $ X $ is $ n $-connective if and only if $ X^e $ and $ X^{C_2} $ are both $ n $-connective if and only if $ X^e $ and $ X^{\varphi C_2} $ are both $ n $-connective. 
	Notice that $ \Phi^{C_2} \LSym_{\underline{\Z}}^{\sigma,m} (M) = \tau_{\geq 0} \left(\left(\pi_0 M^{\otimes m}_{h\Sigma_m}\right)^{tC_2} \right) $ for $ M \in \Mod_{\underline{\Z}}^0 $, that is the functor $ \Phi^{C_2} \LSym_{\underline{\Z}}^{\sigma,m} $ is degree $ m $ on $ \Mod_{\underline{\Z}}^0 $. 
	Therefore $ \Phi^{C_2} \LSym_{\underline{\Z}}^{\sigma,m} $ is $ m $-excisive on $ \Mod_{\underline{\Z}} $. 
	The result follows from noting that the fiber of the map $ \Phi^{C_2} \circ \theta \colon \Phi^{C_2} C_2\Sym_{\underline{\Z}}^m(M) \to \Phi^{C_2} \LSym_{\underline{\Z}}^{\sigma,m}(M) $ is connective for $ M = \underline{\Z}, \underline{\Z}[C_2] $ and the same downward induction argument of \cite[Lemma 4.5.3]{Raksit20}. 
\end{proof}
\begin{defn}\label{defn:inv_gr_calg}
	Define an \emph{graded involutive commutative algebra} monad 
	\begin{equation*}
		\CSym^\sigma \colon \Gr\left(\underline{\slice}_{=0}\right) \to \Gr\left(\underline{\slice}_{=0}\right) 
	\end{equation*} 
	which sends a graded zero-slice over $ \underline{\Z} $ to the free $ C_2 $-graded-commutative graded cohomological Tambara functor (compare Remark \ref{rmk:equivariant_koszul_rule}). 
\end{defn}
\begin{ntn}\label{ntn:graded_zero_slices_as_localization}
	We have a fully faithful embedding 
	\begin{equation*}
	 	\iota \colon \Gr^{\leq0}\left(\underline{\slice}_{=0}\right) \xrightarrow{[*]} \Gr^{\leq 0}\left(\underline{\Mod}_{\underline{\Z}}\right)^{\heartsuit}_{+} \to \Gr^{\leq 0}\left(\underline{\Mod}_{\underline{\Z}}\right)_{\geq 0+}
	\end{equation*}
	which admits a left $ C_2 $-adjoint
	\begin{equation*}
	 	\pi \colon \Gr^{\leq 0}\left(\underline{\Mod}_{\underline{\Z}}\right)_{\geq 0+} \xrightarrow{[-*]} \Gr^{\leq 0}\left(\underline{\Mod}_{\underline{\Z}}\right)^{\heartsuit}_{+} \xrightarrow{P^0} \Gr^{\leq0}\left(\underline{\slice}_{=0}\right)
	\end{equation*} 
	where the last functor is the zeroth slice functor pointwise. 
\end{ntn}
\begin{lemma}\label{lemma:free_dalg_graded_truncation}
	There is a commutative diagram
	\begin{equation*}
		\begin{tikzcd}
			\Gr^{\leq 0}\left(\underline{\Mod}_{\underline{\Z}} \right)_{\geq 0^+} \ar[d,"\pi"] \ar[r,"{\LSym^\sigma}"] & \Gr^{\leq 0}\left(\underline{\Mod}_{\underline{\Z}} \right)_{\geq 0^+} \ar[d,"\pi"] \\ 
			\Gr^{\leq 0}\left(\underline{\slice}_{=0}\right) \ar[r,"{\CSym^\sigma}"] & \Gr^{\leq 0}\left(\underline{\slice}_{=0}\right) 
		\end{tikzcd} 
	\end{equation*}
	of $ C_2 $-$ \infty $-categories, where $ \LSym^\sigma $ is the restriction of the free graded derived involutive algebra from Proposition \ref{prop:LSym_connectivegr}, $ \pi $ is from Notation \ref{ntn:graded_zero_slices_as_localization}, and $ \CSym^\sigma $ is the functor of Definition \ref{defn:inv_gr_calg}.  
\end{lemma}
\begin{proof}
	For each $ M \in \Gr^{\leq 0}\left(\Mod_{\underline{\Z}} \right)_{\geq 0+} $, there is a unit map $ M \to \LSym_{\underline{\Z}}^\sigma(M) $ which induces a map $ \pi(M) \to \pi \LSym_{\underline{\Z}}^\sigma(M) $. 
	Since $ \pi $ is $ C_2 $-symmetric monoidal with respect to the Koszul $ C_2 $-symmetric monoidal structure on the target\footnote{This is no longer true if we replace $ P^0 \underline{\pi}_0 $ by $ \underline{\pi}_0 $.}, the $ C_2 $-$ \E_\infty $-algebra structure on $ \LSym_{\underline{\Z}}^\sigma(M) $ induces a natural graded $ C_2 $-commutative algebra structure on $ \pi \LSym_{\underline{\Z}}^\sigma(M) $, hence there is a canonical map 
	\begin{equation}\label{eq:free_dalg_graded_truncation_comparison}
		c_M \colon \CSym_{\underline{\Z}}^\sigma(\pi(M))\to \pi\left(\LSym_{\underline{\Z}}^{\sigma}(M)\right) \,. 
	\end{equation}
	Since the functors preserve sifted colimits and send direct sums to tensor products, it suffices to show that (\ref{eq:free_dalg_graded_truncation_comparison}) is an equivalence for $ M = \underline{\Z}[-k](k) $ and $ \underline{\Z}[C_2][-k](k) $. 
	The result follows from Lemma \ref{lemma:C2sym_LSym_connectivity_comparison}, combined with the fact that, given two $ n $-connective $ \underline{\Z} $-modules $ X, Y $ and a map $ f \colon X \to Y $ so that $ \pi_n f^e\colon \pi_nX^e \to \pi_nY^e $ is an isomorphism and $ \pi_nf^{C_2} \colon \pi_nX^{C_2} \to \pi_nY^{C_2} $ is surjective, then $ P^0 \left(\underline{\pi}_nX \to \underline{\pi}_n Y\right) $ is an isomorphism. 
\end{proof}
\begin{prop}\label{prop:inv_grdalg}
	The adjunction $ \iota \dashv \pi $ of Notation \ref{ntn:graded_zero_slices_as_localization} induces an equivalence between the $ C_2 $-$ \infty $-categories 
	\begin{itemize}
		\item The full $ C_2 $-subcategory of $ \Gr^{\leq 0}\underline{\DAlg}_{\underline{\Z}}^\sigma $ on those graded derived involutive algebras $ A $ such that $ A^n $ has (Mackey functor) homotopy groups concentrated in degree $ -n $ for all $ n \leq 0 $ and the restriction map on $ \underline{\pi}_{-n}A^n $ is injective. 
		\item The full $ C_2 $-subcategory of $ C_2\E_\infty\underline{\Alg}\Gr^{\leq 0}\left(\underline{\Mod}^{\heartsuit}_{\underline{\Z}}\right)^{\otimes_K}  $ spanned by those ordinary graded $ C_2 $-Tambara functors $ B $ such that $ B^n \simeq 0 $ for $ n > 0 $ and the restriction map on each $ B^n $, $ n \leq 0 $ is injective. 
	\end{itemize}
\end{prop}
\begin{proof}
	Using the result of Lemma \ref{lemma:free_dalg_graded_truncation}, the result follows from the same argument as in \cite[Proposition 4.5.6]{Raksit20}. 
\end{proof}
\begin{lemma}
	Let $ \underline{\Z} $ be the constant $ C_2 $-Mackey functor at $ \Z $. 
	The filtered derived symmetric algebra monad $ \LSym^{\leq * } \colon \Fil\left(\underline{\Mod}_{\underline{\Z}}\right) \to \Fil\left(\underline{\Mod}_{\underline{\Z}}\right) $ of Construction \ref{cons:gr_fil_dalg_inv} preserves the full $ C_2 $-subcategory $ \Fil^{\leq 0}\left(\underline{\Mod}_{\underline{\Z}}\right)_{\geq * } $ consisting of those filtered modules $ X^* $ such that $ X^n \simeq 0 $ for $ n> 0 $ and $ X^n \in \tau_{\geq n}^{\mathrm{Post}} \Mod_{\underline{\Z}} $ is $ n $-connective with respect to the Postnikov t-structure on $ \Mod_{\underline{\Z}} $ (Variant \ref{var:eqvtmodulespostnikov}) for $ n \leq 0 $. 
\end{lemma}
\begin{proof}
	Similar to proof of Proposition \ref{prop:LSym_connectivegr}. 
\end{proof}
\begin{lemma}\label{lemma:free_inv_fildalg_colim_compatible}
	Let $ \underline{\Z} $ be the constant $ C_2 $-Mackey functor at $ \Z $. 
	There is a commutative diagram of $ C_2 $-$ \infty $-categories 
	\begin{equation*}
	\begin{tikzcd}[column sep=large]
		\Fil^{\leq 0}\left( \underline{\Mod}_{\underline{\Z},\leq 0} \right)_{\geq 0} \ar[d,"{\LSym^\sigma}"'] \ar[r,"{\colim}"] & \underline{\Mod}_{\underline{\Z}}^{\leq 0} \ar[d,"{\LSym^\sigma}"] \\
		\Fil^{\leq 0}\left( \underline{\Mod}_{\underline{\Z},\leq 0} \right)_{\geq 0} \ar[r,"{\colim}"] & \underline{\Mod}_{\underline{\Z}}^{\leq 0}
	\end{tikzcd}
	\end{equation*}
	where $ \Fil^{\leq 0}\left( \underline{\Mod}_{\underline{\Z},\leq 0} \right)_{\geq 0} $ denotes those filtered $ \underline{\Z} $-modules $ M^{\geq 0} $ such that $ M^n \simeq 0 $ for $ n > 0 $ and $ M^n $ is $ n $-connective and ($ 0 $-)coconnective for all $ n \leq 0 $. 
\end{lemma} 
Its proof is similar to that of Lemma \ref{lemma:free_dalg_graded_truncation}, so we omit it. 
\begin{prop}\label{prop:nonconn_dalg_Post_filt}
	Let $ \underline{\Z} $ be the constant $ C_2 $-Mackey functor at $ \Z $. 
	Then the Postnikov filtration functor $ \tau_{\geq *}^{\mathrm{Post}} \colon \underline{\Mod}_{\underline{\Z}} \to \Fil\left(\underline{\Mod}_{\underline{\Z}}\right) $ induces a fully faithful embedding $ \underline{\DAlg}_{\underline{\Z}}^{\sigma,\mathrm{ccn}} \inj \Fil\underline{\DAlg}_{\underline{\Z}}^\sigma $ where $ \underline{\DAlg}_{\underline{\Z}}^{\sigma,\mathrm{ccn}} $ denotes the full $ C_2 $-subcategory on those derived involutive algebras whose underlying $ \underline{\Z} $-module is coconnective.
\end{prop}
\begin{proof}
	Note that because the t-structure on $ \Mod_{\underline{\Z}} $ is right-complete (Lemma \ref{lemma:postnikov_rightcomplete}), $ \tau_{\geq *} $ is fully faithful. 
	By Lemma \ref{lemma:free_inv_fildalg_colim_compatible}, we have an equivalence $ \colim \circ \LSym^\sigma \circ \tau_{\geq *} \simeq \LSym^\sigma $ of $ C_2 $-endofunctors on $ \underline{\Mod}_{\underline{\Z}}^{\leq 0} $. 
	Then \cite[Proposition 4.1.9]{Raksit20} implies there is a localizing adjunction on categories of modules over the respective monads. 
\end{proof}

\section{Involutive cohomological invariants}\label{section:cochainwinvolution} 
In this section, we introduce involutive enhancements of the cotangent complex and de Rham complex for the involutive algebras defined in \S\ref{section:scr_with_inv}.  

We construct the involutive cotangent complex in much the same way as its classical counterpart: as the $ C_2 $-left adjoint to a square-zero extension functor. 
However, we remark that in pivoting from $ \E_\infty $-algebras to $ C_2 $-$ \E_\infty $-algebras, we should regard such a functor as associating to a pair $ (A,M) $ an square- and \emph{norm-zero} $ C_2 $-$ \E_\infty $ $ A $-algebra with underlying object $ A \oplus M $, and likewise for ordinary derived algebras and derived involutive algebras

The ordinary derived de Rham complex is characterized by the structure present on it: $ \L \Omega^{^\bullet}_{A/k} $ is the initial $ h_+ $-differential graded derived commutative $ k $-algebra under $ A $. 
The involutive (derived) de Rham complex $ \L\Omega^{\sigma,\bullet}_{-/\underline{k}} $ is similarly characterized by a universal property: To define $ \L\Omega^\sigma_{-/\underline{k}} $, we specify the type of structure it should have.  
Non-equivariantly, one typically starts by defining differential graded objects, then introducing differential graded algebras. 
Here, the involutive enhancement of differential graded objects is \emph{not} differential graded objects in $ \underline{k} $-modules as one might expect: In \cite[\S2]{SVP96}, Solotar--Vigué-Poirrier show that given an algebra $ A $ with an involution $ \omega $, its de Rham complex naturally acquires the structure of a dg module so that the differential $ d $ is \emph{antilinear} with respect to the involution; that is, $ d\omega = - \omega d $. 
Thus, the appropriate enhancement of $ h_+ $-differential graded objects for our purposes are the $ h_+^\sigma $-differential graded objects of Definition \ref{defn:dg_involutive_module}. 
Our $ h^\sigma_+ $-dg objects still have a relationship to complete filtered objects, albeit with a twist (see Proposition \ref{prop:shear_gr_on_dg_mod}, which the interested reader may contrast with \cite[Remark 5.1.12]{Raksit20}).

\begin{rmk}\label{rmk:involutive_deRham_cplx_and_exterior_powers_conj}
	An essential feature of the ordinary (derived) de Rham complex $ \L\Omega^{\bullet}_{A/k} $ is that the individual terms appearing in it can be expressed as shifts of higher exterior powers of the cotangent complex $ \L_A $ (cf. \cite[Theorem 5.3.6]{Raksit20}).  
	In this context, we expect a corresponding involutive enhancement of such a statement. 
	However, in attempting to prove such a statement, we quickly find ourselves straying from the goal of this work.  
	Thus, we set this question aside for the moment and hope to return to it in future work.  
\end{rmk}

\subsection{The involutive cotangent complex}\label{subsection:C2cotangentcplx}
In this subsection, we introduce a definition of involutive cotangent complexes for derived involutive algebras. 

We show that the involutive cotangent complex of a derived involutive ring is a genuine equivariant enhancement of the ordinary cotangent complex considered in \cite[\S4.4]{Raksit20}. 
We show that the involutive cotangent complex is computable like its classical counterpart; the reader who wishes to get a feel for computational aspects of the involutive cotangent complex is invited to proceed directly to Example \ref{ex:cotangent_computations}. 

Throughout this section, we will work with a derived involutive algebraic context $ \cat $ equipped with a fixed map of derived involutive algebraic contexts $ \underline{\Mod}_{\underline{\Z}} \to \cat $ (see Example \ref{ex:inv_dalg_over_const_Mackey}). 
\begin{ntn}
	Let $ C_2 \E_\infty\underline{\Alg\Mod}(\cat) $ denote the $ C_2 $-$ \infty $-category whose $ \mathcal{O}^\op_{C_2}$-objects consist of pairs $ (A, M) $ where $ A $ is a $ C_2 $-$ \E_\infty $-algebra in $ \cat $ and $ M $ is an $ A $-module in $ \cat $. 
	Write 
	\begin{equation*}
		\underline{\DAlg^\sigma\Mod}(\cat) := \underline{\DAlg}^\sigma(\cat ) \fiberproduct_{C_2\E_\infty\underline{\Alg}(\cat)} C_2\E_\infty\underline{\Alg\Mod}(\cat) \,.
	\end{equation*}
	If $ A $ is a derived involutive algebra object of $ \cat $, write $ \underline{\DAlg^\sigma\Mod}_A := \underline{\DAlg^\sigma\Mod}(\cat)_{\underline{(A,0)}/-} $. 
\end{ntn}
\begin{prop}\label{prop:modules_and_gr01}
\begin{enumerate}[label=(\alph*)]
	\item There is an equivalence of $ C_2 $-$ \infty $-categories between $ \Gr^{\{0,1\}}C_2\E_\infty\underline{\Alg}(\cat) $ and $ C_2\E_\infty\underline{\Alg\Mod}(\cat) $ which commutes with the forgetful functors to $ \Gr^{\{0,1\}}\left(\cat\right) $.
	\item \label{propitem:modules_and_gr01_dalginv} There is an equivalence of $ C_2 $-$ \infty $-categories between $ \Gr^{\{0,1\}}\underline{\DAlg}^\sigma(\cat) $ of Example \ref{ex:dalg_in_gr_01} and $ \underline{\DAlg^\sigma\Mod}(\cat) $ which commutes with the forgetful functors to $ \Gr^{\{0,1\}}\left(\cat\right) $. 
\end{enumerate}	
\end{prop}
\begin{proof}
	To prove part (a), we proceed as in the proof of \cite[Lemma 4.4.2]{Raksit20}. 
	We first regard $ \Gr^{\{0,1\}}C_2\E_\infty\underline{\Alg}(\cat) $ as a subcategory of $ \Gr^{\geq 0}C_2\E_\infty\underline{\Alg}(\cat) $. 
	Now $ \Gr^{\geq 0}C_2\E_\infty\underline{\Alg}(\cat) $ can be identified with the $ \infty $-category of lax $ C_2 $-symmetric monoidal functors $ \Z_{\geq 0}^\delta \to \cat $. 
	Then $ 0 \in \left(\Z_{\geq 0}^\delta\right)^{C_2/C_2} $ has a unique $ C_2 $-commutative algebra structure, and $ 1 \in \left(\Z_{\geq 0}^\delta\right)^{C_2/C_2} $ has a unique module structure over $ 0 $. 
	Thus there is an induced $ C_2 $-functor $ \alpha \colon \Gr^{\{0,1\}}C_2\E_\infty\underline{\Alg}(\cat) \to C_2\E_\infty\underline{\Alg\Mod}(\cat) $. 
	The result follows from observing that $ \alpha $ commutes with the respective forgetful $ C_2 $-functors to $ C_2\E_\infty\underline{\Alg}(\cat) $ and appealing to the Barr--Beck--Lurie theorem pointwise \cite[Corollary 4.7.3.16]{LurHA}. 

	To prove part (b), notice that Remark \ref{rmk:zeroth_graded_dalg} and the argument of \cite[Lemma 4.4.3]{Raksit20} furnish a $ C_2 $-functor $ \alpha' \colon \Gr^{\{0,1\}}\underline{\DAlg}^\sigma(\cat)  \to \underline{\DAlg^\sigma\Mod}(\cat) $ making the following diagram commute
	\begin{equation*}
	\begin{tikzcd}[column sep=tiny]
		\Gr^{\{0,1\}}\underline{\DAlg}^\sigma(\cat) \ar[rd,"G"'] \ar[rr,"{\alpha'}"] & & \underline{\DAlg^\sigma\Mod}(\cat) \ar[ld,"G'"] \\
		& \Gr^{\{0,1\}}\left( \cat \right) & 
	\end{tikzcd}. 
	\end{equation*}
	Now we show that the functors $ G $, $ G' $ admit left $C_2$-adjoints $ F $, $ F' $ respectively. 
	The functors $ G$, $ G' $ evidently admit fiberwise left adjoints whose values on the fiber over $ C_2/C_2 $ are given by $ F = \LSym^\sigma_{\Gr^{\{0,1\}}(\cat)} $ and $ F'(X^0, X^1 ) = \left(\LSym^\sigma_{\cat} (X^0), \LSym^\sigma_{\cat} (X^0)\otimes X^1 \right) $. 
	The functors $ F $, $ F' $ promote to left $ C_2 $-adjoints by Corollary \ref{cor:C2_left_adjoint_local_crit}; the functors $ G $ and $ G' $ evidently commute with $ C_2 $-products. 
	Moreover, both $ C_2 $-adjunctions are fiberwise monadic, the former by definition.  
	Unravelling the definition of $ \LSym^\sigma $ on $ \Gr^{\{0,1\}}\left(\cat\right) $, we see that the canonical map $ \alpha' \circ F \to F' $ is an equivalence, hence $ \alpha' $ is an equivalence by \cite[Corollary 4.7.3.16]{LurHA}. 
\end{proof}
\begin{ntn}\label{ntn:square_zero_graded_alg}
	Let $ \D^\vee, \D^\vee_{C_2} \in \Gr\DAlg^{\sigma,\slice=0}_{\underline{\Z}} $ denote the graded $ C_2 $-cohomological Tambara functors with underlying objects given by $ \D^\vee := \mathbbm{1} \oplus  \mathbbm{1}(-1) $ and $ \D^\vee_{C_2} := \mathbbm{1} \oplus  \mathbbm{1}[C_2](-1) $. 
	By Variant \ref{variant:otherdalg_inv}.\ref{varitem:discrete_dalg_inv} and Construction \ref{cons:gr_fil_dalg_inv}, there is a canonical embedding $ \Gr\DAlg^{\sigma,\slice=0}_{\underline{\Z}} \to \Gr\DAlg^\sigma_{\underline{\Z}} $, so we naturally regard $ \D^\vee $ and $ \D^\vee_{C_2} $ as objects of $ \Gr\DAlg^\sigma_{\underline{\Z}} $. 
\end{ntn}
\begin{rmk}\label{rmk:dalgmod_separate_components}
		Writing $ U $ for the forgetful $ C_2 $-functor $ \underline{\DAlg^\sigma\Mod}(\cat) \to \underline{\DAlg}^\sigma(\cat) $, $ U \colon (A , M ) \mapsto A $, the unit $ \mathbbm{1} \to \D^\vee $ and counit $ \D^\vee \to \mathbbm{1} $ induce natural transformations $ \eta \colon U \to G $ and $ \varepsilon \colon G \to U $ so that $ \varepsilon \circ \eta \simeq \id $. 
		Similar statements hold for the forgetful $ C_2 $-functor $ U \colon \underline{\DAlg^\sigma\Mod}_A \to \underline{\DAlg}^\sigma_A $. 
\end{rmk}
\begin{cons}\label{cons:sq_nm_zero_ext}
	Then the \emph{square-zero and norm-zero extension functor} is given by the $ C_2 $-functor 
	\begin{equation*} 
	\begin{split}
		G \colon \underline{\DAlg^\sigma\Mod}(\cat) \simeq \Gr^{\{0,1\}}\underline{\DAlg}^\sigma(\cat) \to \underline{\DAlg}^\sigma (\cat) \\
		(A, M) \mapsto \ev_0( (A, M) \ostar \D^\vee)\,,
	\end{split}
	\end{equation*} 
	where we have used the equivalence of Proposition \ref{prop:modules_and_gr01}\ref{propitem:modules_and_gr01_dalginv}. 
	Furthermore, $ G $ induces a $ C_2 $-functor $ G \colon \underline{\DAlg^\sigma\Mod}_A \to \underline{\DAlg}_A^\sigma $. 
\end{cons}
\begin{defn}\label{defn:derivations} (cf. \cite[Definition 4.4.7]{Raksit20})
		Let $ A \in \DAlg^\sigma(\cat) $ and let $ B \in \DAlg^\sigma_{A/}(\cat) $ and $ M \in \Mod_B(\cat) $. 
		An \emph{$ A $-linear derivation of $ B $ into $ M $} is a morphism of derived involutive algebras $ \delta \colon B \to B \oplus M $ in $ \DAlg^\sigma(\cat)_{A/-/B} $. 
		We may abuse notation by identifying a derivation $ \delta $ with the map of $ A $-modules $ d \colon B \to M $ obtained by post-composing $ \delta $ with the projection $ B \oplus M \to M $. 
\end{defn}
\begin{prop}\label{prop:sq_zero_has_left_adjoint}
		The square-zero extension functor $ G \colon \underline{\DAlg^\sigma\Mod}_A \to \underline{\DAlg}_A^\sigma $ of Construction \ref{cons:sq_nm_zero_ext} admits a left $ C_2 $-adjoint $ \L $. 
		On underlying $ \infty $-categories, this recovers the adjunction of \cite[Proposition 4.4.8]{Raksit20}. 
\end{prop}
\begin{proof}
		Observe that $ C_2 $-$ \infty $-categories $ \underline{\DAlg}^\sigma_A $ and $ \underline{\DAlg^\sigma\Mod}_A $ admit finite $ C_2 $-products; we will write $ \prod_{C_2} $ for the right adjoints to both restriction functors. 
		By Corollary \ref{cor:C2_left_adjoint_local_crit}, it suffices to show that $ G $ admits fiberwise left adjoints and $ G $ preserves finite $ C_2 $-products (or equivalently, that the left adjoints satisfy a Beck-Chevalley condition). 
		The former is true because $ G $ evidently preserves all limits. 
		To show that $ G $ preserves finite $ C_2 $-products, we must check that the canonical transformation 
		\begin{equation*}
		\begin{tikzcd}
			\DAlg_{A^e} \ar[r,"\prod_{C_2}"] & \DAlg^\sigma_A \\
			\DAlg\Mod_{A^e} \ar[r,"\prod_{C_2} "'] \ar[u,"G"]  & \DAlg^\sigma\Mod_A \ar[u,"G"'] \ar[lu, Rightarrow,shorten=4mm] 
		\end{tikzcd}
		\end{equation*}
		is an equivalence. 
		This is true because the forgetful $ C_2 $-functor $ \underline{\DAlg}^\sigma_A \to \underline{\Mod}_A $ is conservative, and the transformation is evidently an equivalence on underlying objects.  
\end{proof} 
\begin{defn}\label{defn:involutive_cotangent}
		Let $ A \in \DAlg^\sigma(\cat) $, let $ B $ be an involutive $ A $-algebra. 
		By Remark \ref{rmk:dalgmod_separate_components}, there is a unique $ B $-module $ \L_{B/A} $ equipped with an equivalence $ \L(B) \simeq (B, \L_{B/A}) $ in $ \DAlg^\sigma \Mod_{A} $, where $ \L $ is the left $ C_2 $-adjoint of Proposition \ref{prop:sq_zero_has_left_adjoint}. 
		We will refer to $ \L_{B/A} $ as the \emph{involutive cotangent complex of $ B $ over $ A $}. 
\end{defn}
The unit of the adjunction equips $ \L_{B/A} $ with a universal $ A $-linear derivation $ d \colon B \to \L_{B/A} $. 
\begin{rmk}
		Given a pushout diagram in $ \DAlg^\sigma(\cat) $ 
		\begin{equation*}
		\begin{tikzcd}
				A' \ar[d] \ar[r] & A\ar[d] \\
				B' \ar[r] & B \,,
		\end{tikzcd}
		\end{equation*}
		the universal property of the involutive cotangent complex immediately implies that there is a canonical equivalence of $ B $-modules $ B \otimes_{B'} \L_{B'/A'} \simeq \L_{B/A} $. 
\end{rmk}
For the remainder of this subsection, we specialize to the derived involutive algebra contexts of Example \ref{ex:inv_dalg_over_const_Mackey}. 
\begin{rmk} 	
	Let $ k $ be a discrete commutative ring with involution and let $ \underline{k} $ denote the associated fixed point $ C_2 $-Green functor. 
	Let $ \DAlg^\sigma\Mod^\conn_{\underline{k}} $ denote the full $ C_2 $-subcategory on those pairs $ (A, M) $ where both $ A $ and $ M $ are connective. 
	This $ C_2 $-$ \infty $-category is generated under pointwise sifted colimits and restrictions by the full subcategory $ \DAlg^\sigma\Mod_{\underline{k}}^0 $ on those $ (A, M) $ where $ A = \LSym_{\underline{k}}^\sigma(M) $ for some $ M \in \Mod^0_{\underline{k}} $ and $ M \simeq A \otimes_{\underline{k}} \underline{k}[U] $ for some finite $ C_2 $-set $ U $. 
	Since the square-zero extension functor $ G $ preserves sifted colimits, $ G|_{\DAlg^\sigma\Mod^\conn_{\underline{k}}} $ is the left derived functor of its restriction $ G|_{\DAlg^\sigma\Mod^0_{\underline{k}}} $. 
	Now the restriction $ G|_{\DAlg^\sigma\Mod^0_{\underline{k}}} $ takes values in $ \DAlg^{\sigma, \slice=0}_{\underline{k}} $, and it follows that this restriction agrees with the square-zero extension considered in \cites[\S3]{Hill_Tambara_Kahlerdiffl}[\S14]{strickland_tambara}. 
	The notion of square-zero extension, derivation, and cotangent complex in this setting are explored in the following example.  
\end{rmk}
\begin{ex} [Free derived involutive algebras] \label{ex:cotangent_computations}
	Consider a free derived involutive $ \underline{k} $-algebra on a connective $ \underline{k} $-module $ P $. 
	Then there is a canonical equivalence of $ \LSym^\sigma_{\underline{k}}(P) $-modules 
	\begin{equation*}
		\L_{\LSym^\sigma_{\underline{k}}(P)/\underline{k}} \simeq \LSym^\sigma_{\underline{k}}(P) \otimes_{\underline{k}} P . 
	\end{equation*} 
	Let's unravel two special cases to see how this works. 
	Let $ (B, M) \in \DAlg^\sigma\Mod^{ \slice=0} $ be arbitrary. 
	\begin{itemize} 
		\item Suppose $ P = \underline{k} $. 
		We want to find a pair $ \mathbb{L}(\underline{k}[x]) \in \DAlg^\sigma\Mod_{\underline{k}} $ such that 
		\begin{equation*} 
			\hom_{\DAlg^\sigma\Mod}(\mathbb{L}(\underline{k}[x]), (B,M)) \simeq \hom_{\DAlg^\sigma}(\underline{k}[x], B \oplus M) 
		\end{equation*}
		It follows from the definitions that a morphism of involutive algebra objects $ \varphi \colon \underline{k}[x] \to B\oplus M $ in $ \DAlg^{\sigma, \slice=0}_{\underline{k}} $ 
		is equivalent to the data of a map of involutive algebras $ \pi_1 \circ \varphi \colon \underline{k}[x] \to B^{C_2} $ and the image of $ {\pi_2 \circ \varphi^{C_2}(x) \in M^{C_2}} $. 
		Hence $ \mathbb{L}(\underline{k}[x]) \simeq (\underline{k}[x], \underline{k}[x]\{dx\}) $. 

		\item Suppose $ P = \underline{k}[C_2] $. 
		We want to find a pair $ \mathbb{L}(\underline{k}[x,x_\sigma]) \in \DAlg^\sigma\Mod_{\underline{k}} $ such that 
		\begin{equation*}
			\hom_{\DAlg^\sigma\Mod}(\mathbb{L}(\underline{k}[x]), (B,M)) \simeq \hom_{\DAlg^\sigma}(\underline{k}[x,x_\sigma], B \oplus M) . 
		\end{equation*}
		We see that a morphism of $ \varphi \colon \underline{k}[x,x_\sigma] \to B\oplus M $ in $ \DAlg^{\sigma, \slice=0}_{\underline{k}} $ determines a diagram
		\begin{equation*}
		\begin{tikzcd}
			k[x^i_T, x_N]_{i \in \nat} \ar[d, bend right=20,"\mathrm{Res}"'] \ar[r,"{\varphi^{C_2}}"] & B^{C_2} \oplus M^{C_2} \ar[d, bend right=20,"\mathrm{Res}"']\\
			k[x,x_\sigma] \ar[u, bend right=20,"\mathrm{Tr}"'] \ar[loop below,"{x \mapsto x_\sigma}",out=-60, in=240,distance=15] \ar[r,"{\varphi^{e}}"]  & B^e \oplus M^e \ar[u, bend right=20,"\mathrm{Tr}"'] 
		\end{tikzcd}
		\end{equation*}
		in which certain squares commute; in particular, 
		\begin{align*}
			\varphi^{C_2}(x^n_T) &= \varphi^{C_2}(\Tr(x^n)) = \Tr(\varphi^e(x^n)) = \Tr(\varphi^e_1(x)^n + n\varphi^e_1(x)^{i-1}\varphi^e_2(x)) . 
		\end{align*}
		Since $ xx_\sigma $ is in the image of $ \pi_0\LSym^2(k[x,x_\sigma])^{hC_2} \to k[x,x_\sigma]^{C_2} $, we deduce that the image of $ \varphi(xx_\sigma) $ is determined by the commutativity of the following diagram
		\begin{equation*}
		\begin{tikzcd}
			\pi_0\LSym^2(k[x,x_\sigma])^{hC_2} \ar[d] \ar[r]& \pi_0\LSym^2(B^e \oplus M^e)^{hC_2} \ar[d] \\
			 k[x,x_\sigma]^{C_2} \ar[r, "\varphi"] & B^{C_2} \oplus M^{C_2} \,.
		\end{tikzcd} 
		\end{equation*}
		In particular, since the restriction map $ B^{C_2} \oplus M^{C_2} \to B^ e \oplus M^e $ is injective, we have $ \varphi^{C_2}_2(x_N) = \Tr(\varphi^e_1(x) \varphi^e_2(x_\sigma)) $. 
		Hence the data of $ \varphi $ is determined by a map of rings $ \underline{k}[x,x_\sigma] \to B $ and $ \varphi^e_2(x) \in M^e $. 
		Hence we see that $ \mathbb{L}(\underline{k}[x,x_\sigma]) \simeq (\underline{k}[x,x_\sigma], \underline{k}[x,x_\sigma] \otimes C_2 ) $ where we identify $ C_2 \simeq \{dx, dx_\sigma \}$.
	\end{itemize}
\end{ex}
\begin{ex} [Hyperelliptic involutions] 
		Let $ f(x) $ be a polynomial with $ \C $ coefficients and consider the algebra $ \C[x,y]/(y^2-f(x)) $ with the involution which sends $ y $ to $ -y $ and acts by the identity on $ \C $ and $ x $. 
		Let us regard $ \C[x,y]/(y^2-f(x)) $ as the underlying algebra with $ C_2 $-action of a $ C_2 $-Green functor $ A $ over the constant $ C_2 $-Green functor $ \underline{\C} $.  
		Now there is a pushout square in $ \DAlg^\sigma_{\underline{\C}} $:
		\begin{equation*}
		\begin{tikzcd}
				\underline{\C}[z,w] \ar[r,"{\varphi}"] \ar[d] & \underline{\C}[y,y_\sigma, x] \ar[d,"\psi"] \\
				\underline{\C} \ar[r] & \underline{\C}[x,y]/(y^2-f(x)) 
		\end{tikzcd}\,,
		\end{equation*}
		where $ C_2 $ acts trivially on $ z, w, x $ and permutes $ y $ and $ y_\sigma $. 
		The morphisms satisfy $ \varphi(z) = y + y_\sigma $, $ \varphi(w) = -yy_\sigma - f(x) $ and $ \psi(y) = y $, $ \psi(x) = x $. 
		Since $ \DAlg^\sigma_{\underline{\C}}\Mod \to \DAlg^\sigma_{\underline{\C}} $ is a cocartesian fibration and the cocartesian pushforward maps preserve colimits (of modules), the involutive cotangent complex of $ A= \underline{\C}[x,y]/(y^2-f(x)) $ is determined by the cofiber sequence
		\begin{equation*}
		\begin{tikzcd}[column sep=large]
					A \{dz, dw\} \ar[rrr,"{dz \mapsto dy + dy_\sigma}", "{dw \mapsto -ydy_\sigma -y_\sigma dy -f'(x)dx}"'] &&& A\{dy, dy_\sigma, dx\} \ar[r] & \L_{A/\underline{\C}}	
		\end{tikzcd}
		\end{equation*}
		of $ A $-modules. In particular, its Lewis diagram is
	\begin{equation*}
	\begin{tikzcd}
		\C[x]/f(x)\{dx\}/(f'(x) dx) \ar[d, bend right=20,"\mathrm{Res}"'] \\
		\C[x,y]/(y^2-f(x)) \{dx,dy\}/(2y dy-f'(x)dx) \ar[u, bend right=20,"\mathrm{Tr}"'] \ar[loop below,out=-60, in=240,distance=15] 
	\end{tikzcd}  
	\end{equation*}
	where $ C_2 $ acts trivially on $ dx $ and by multiplication by $ -1 $ on $ dy $.  
\end{ex}

\subsection{Involutive derived de Rham complex} \label{subsection:C2deRham}
In this section, we will introduce involutive enhancements of the derived de Rham complex and de Rham cohomology. 
We begin with equivariant analogues of differential graded objects. 
Recall that non-equivariantly, a $ h_{\pm} $-dg structure on a graded object $ X_* $ encodes differentials of the form $ X_i[\pm1] \to X_{i+1} $ (see \cite[\S5]{Raksit20}). 
While we wish to define equivariant dg objects which specialize to the non-equivariant notions on underlying objects, in doing so, we must make a choice. 
For instance, maps $ X_i[-1] \to X_{i+1} $ and $ X_i[-\sigma] \to X_{i+1} $ in any $ C_2 $-stable $ C_2 $-$ \infty $-category both induce maps $ X_i^e[-1] \to X_{i+1}^e $ on underlying objects. 
In this work, we promote $ h_{-} $-dg structures to the equivariant setting naïvely (i.e. so that the differentials are of the form $ X_i[-1] \to X_{i+1} $) and we promote $ h_{+} $-dg structures to the equivariant setting with a twist: An $ h_+^\sigma $-differential graded object has differentials of the form $ X_i[\sigma] \to X_{i+1} $. 
We will define the desired differential graded objects as modules over certain algebras $ \D_{-} $ and $ \D_{+}^\sigma $ which are trivial square- and norm-zero extensions of $ \underline{\Z}(0) $. 
Next, we show that there is a natural notion of the cohomology of a $ h_+^\sigma $-differential graded object. 
That $ \D_{-} $ and $ \D_{+}^\sigma $ are both $ C_2 $-$ \E_\infty $-bialgebras allows us to (by \S\ref{subsection:param_bialg}) define $ C_2 $-symmetric monoidal structures on $ h_+^\sigma $- and $ h_{-} $-differential graded objects and $ h_{+}^\sigma $- and $ h_{-} $-differential graded $ C_2 $-$ \E_\infty $-algebras. 
We will see that $ \D_+^\sigma $ has more structure than $ \D_{-} $; its dual may be canonically regarded as a derived involutive bialgebra. 
This structure on $ \D_{+}^\sigma $ will allow us to define a derived involutive enhancement of strictly commutative differential graded algebras. 
The involutive derived de Rham complex of $ A $ over $ B $ is defined as the universal involutive $ h^\sigma_+ $-differential graded $ B $-algebra under $ A $. 

We begin by showing that certain graded $ \underline{\Z} $-modules have $ C_2 $-bialgebra structures. 

\begin{lemma}\label{lemma:sigma_discrete}
	Write $ \underline{\Z} $ for the constant $ C_2 $-Mackey functor at $ \Z $. 
	Write $ \Z_{-} $ for the Mackey functor
	\begin{equation*}
	\begin{tikzcd}
		\Z_{-}^{C_2} = 0 \ar[d, bend right=20,"\mathrm{Res}"'] \\
		\Z_{-}^e = \Z \ar[u, bend right=20,"\mathrm{Tr}"'] \ar[loop below,"{m \mapsto -m}",out=-60, in=240,distance=15]
	\end{tikzcd} .
	\end{equation*}
	Then there is an equivalence of $ \underline{\Z} $-modules
	\begin{equation*}
		\Sigma^{-\sigma} \underline{\Z} \simeq \Sigma^{-1} \Z_{-}.
	\end{equation*}
	In particular, $ \Z_{-} $ is a regular $ (-1) $-slice. 
\end{lemma}
We thank Michael Hill for assistance with this aspect of the story. 
\begin{proof}
	Follows from computation of $ \pi_{-1}\T^{-\sigma} $ of Proposition \ref{prop:graded_fil_Ssigma}. 
\end{proof}
Recall that there are $ C_2 $-symmetric monoidal structures on the $ C_2 $-$\infty$-categories $ \Gr\left(\underline{\Mod}_{\underline{k}}\right)^\heartsuit_{\rslice^\pm} $ so that the truncation functors $ \Gr\left(\underline{\Mod}_{\underline{k}}\right)_{\rslice\geq 0^{\pm}} \to \Gr\left(\underline{\Mod}_{\underline{k}}\right)^\heartsuit_{\rslice^\pm} $ are $ C_2 $-symmetric monoidal (Proposition \ref{prop:filtrations_are_compatible}). 
\begin{cons}\label{cons:differential_algebra} 
		We will construct $ C_2 $-$ \E_\infty $-co-$C_2$-$ \E_\infty $ algebras $ \D^\sigma_{+} $ and $ \D_{-} $ in $ \Gr\left(\underline{\Mod}_{\underline{\Z}}\right) $ so that their underlying objects are given by
		\begin{equation*}
			 \D^\sigma_{+} \simeq \left(\underline{\Z} \oplus \Sigma^\sigma \underline{\Z}(1)\right) \qquad \qquad  \D_{-} \simeq \left(\underline{\Z} \oplus \Sigma^{-1}\underline{\Z}(1)\right) 
		\end{equation*}
		and so that there are equivalences of bicommutative bialgebras $ \left(\D^\sigma_+\right)^e \simeq \D_+ $ and $ \left(\D_{-}\right)^e \simeq \D_{-} $ over $ \Z $, where $ \D_{\pm} $ are the graded bialgebras from \cite[Construction 5.1.1]{Raksit20}. 
\end{cons}
There are graded $ C_2 $-$ \E_\infty $-bialgebras $ \underline{\Z}[\varepsilon_\pm] = \underline{\Z} \oplus \underline{\Z}[\varepsilon_{\pm}]/(\varepsilon^2_{\pm}) $ in $ \Gr\left(\underline{\Mod}_{\Z}\right)^{\heartsuit,\ostar_K}_{r^\pm} $, where $ \varepsilon $ is in grading degree $ \pm 1 $, the comultiplication takes $ \varepsilon \mapsto \varepsilon \otimes 1 - 1 \otimes \varepsilon $ and the $C_2$-coalgebra structure takes $ \varepsilon \mapsto e \otimes \varepsilon - \sigma \otimes \varepsilon $. 
Notice that $ \underline{\Z}[\varepsilon_{\pm}] $ are in the essential image of the functors of Observation \ref{obs:graded_rslice_heart_as_graded_Mackey_functors}. 
Now define 
\begin{equation*}
			 \D^\sigma_{+} = \left(\iota_{+}\underline{\Z} [\varepsilon_{-}]\right)^\vee \qquad \qquad  \D_{-} = \left(\iota_{-}\underline{\Z} [\varepsilon_{-}]\right)^\vee \,. 	
\end{equation*}
Note that under the (inverse) equivalences of Observation \ref{obs:graded_rslice_heart_as_graded_Mackey_functors}, $ \D^\sigma_{+} $ and $ \D_{-} $ have the desired underlying objects. 

Now, by Proposition \ref{prop:compatible_filt_implies_C2_monoidal_heart}, the inclusions $ \iota_\pm \colon \Gr\left(\underline{\Mod}_{\underline{\Z}}\right)^\heartsuit_{\rslice^\pm} \to \Gr\left(\underline{\Mod}_{\underline{\Z}}\right)_{\rslice\geq 0^{\pm}} $ are lax $ C_2 $-symmetric monoidal. 
By Example \ref{ex:spheres_as_slices}, the functors $ \iota_\pm $ are $ C_2 $-symmetric monoidal on the full $ C_2 $-subcategory of $ \Gr\left(\underline{\Mod}_{\underline{\Z}}\right)^\heartsuit_{\rslice^\pm} $ on those graded objects $ X^* $ so that $ X^n $ is a free $ \underline{\Z} $-module on a direct sum of regular slice cells of dimension $ \pm n $. 
Therefore, $ \D^\sigma_{+} $ and $ \D_{-} $ inherit $ C_2 $-$ \E_\infty $-co-$ C_2 $-$ \E_\infty $ bialgebra structures from $ \underline{\Z}[\varepsilon] $. 

Finally, the identification of the underlying graded bialgebras over $ \Z $ follows from noting that $ \underline{\Z}[\varepsilon]^e $ agrees with the graded bialgebra $ \Z[\varepsilon] $ considered in \cite[Construction 5.1.1]{Raksit20}. 
\begin{ntn}\label{ntn:D_Dsigma_in_any_C2_sym_cat}
		Let $ \cat $ be a $ \underline{\Z} $-linear $ C_2 $-stable $ C_2 $-presentable $ C_2 $-symmetric monoidal $ C_2 $-$ \infty $-category. 
		The unique $ C_2 $-symmetric monoidal $ C_2 $-colimit-preserving $ C_2 $-functor $ \underline{\Mod}_{\underline{\Z}} \to \cat $ induces a $ C_2 $-symmetric monoidal $ C_2 $-functor $ \Gr\left(\underline{\Mod}_{\underline{\Z}}\right) \to \Gr(\cat) $, so we may regard $ \D^\sigma_{+} $, $ \D_{-} $ as bi-$ C_2 $-$ \E_\infty $ bialgebra objects of $ \Gr(\cat) $. 
\end{ntn}
We now have what we need to introduce the involutive enhancements of cochain complexes in which the differential anti-commutes with the $ C_2 $-action. 
\begin{defn}\label{defn:dg_involutive_module}
	Let $ \cat $ be a $ \underline{\Z} $-linear $ C_2 $-stable $ C_2 $-presentable $ C_2 $-symmetric monoidal $ C_2 $-$ \infty $-category. 
	The $ C_2 $-$ \infty $-category of \emph{$h_{+}^\sigma$- (resp. $ h_{-}$-)differential graded objects of $ \cat $ with involution} is given by 
	\begin{equation*}
	\begin{split}
			\underline{\DG}^\sigma_+(\cat) &:= \underline{\Mod}_{\D^\sigma_{+}}(\Gr(\cat)) \\
			\underline{\DG}_{-}(\cat) &:= \underline{\Mod}_{\D_{-}}(\Gr(\cat)) .
	\end{split}	
	\end{equation*}
	By Proposition \ref{prop:Dsigma_alg} and Variant \ref{variant:module_pointwise_param_tensor}, $ \underline{\DG}^\sigma_+(\cat) $ and $ \underline{\DG}_{-}(\cat) $ are $C_2$-symmetric monoidal $ C_2 $-$ \infty $-categories.
\end{defn}
\begin{rmk}\label{rmk:inv_dg_objects_unravelled} 
	Consider the derived involutive algebraic context of Example \ref{ex:inv_dalg_over_const_Mackey} where $ k = \Z $. 
	By Lemma \ref{lemma:sigma_discrete}, we can identify $ (\D^{\sigma,\vee}_+)^e \simeq \underline{\Z} \oplus \Sigma^{-1}\Z_{-1} (-1) $. 

	Thus informally, an object $ X_* $ of $h_{+}^\sigma $ cochain complexes/involutive differential graded objects of $ \cat $ is given by a collection $ X_* $ of $ \underline{\Z} $-modules with $ C_2 $-action equipped with a $ \D^\sigma_+ = \underline{\Z} \oplus \underline{\Z}[\sigma](1) $-module structure on $ X_* $. 
	In particular, for each $ n $, there is a map $ d \colon \Sigma^\sigma X_n \to X_{n+1} $ so that $ d \circ (\Sigma^\sigma d) = 0 $. 
	\begin{itemize}  
		\item Consider the action of $ d $ on the underlying graded $ \Z $-module $ (X_*)^e $. 
		Because the module structure map $ \D^\sigma_+ \otimes X_* \to X_* $ is $ C_2 $-equivariant with respect to the diagonal action on the source and the given action on the target, if $ m \in \pi_*X_k $ and $ \sigma \neq e \in C_2 $, 
		\begin{equation*} 
			\sigma(\varepsilon \cdot m) = (\sigma \varepsilon)\cdot \sigma(m) = (-\varepsilon) \cdot \sigma(m) .
		\end{equation*} 
		In particular, the differential is $ \sigma $-antilinear. 

		If $ X_*^e $ consists of $ \underline{\Z\left[\frac{1}{2}\right]} $-modules, then we can canonically write $ X_*^e = X^+_* \oplus X^-_* $ as the direct sum of the $ \pm 1 $ eigenspaces, and we see that $ \cdot \varepsilon: X^+_n \to X^-_{n+1} $ and vice versa (cf. \cite{SVP96}).	
		\item We may adjoin $ d $ to obtain a map $ X_n \to \hom (S^\sigma, X_{n+1}) $. 
		In particular, on homotopy groups, $ d $ induces a map of abelian groups 
		\begin{equation*}
				\pi_{*} (X_n^{C_2}) \to \pi_* \fib \left(X_{n+1}^{C_2} \to X_{n+1}^e\right) \,.
		\end{equation*} 
		so that the composite
		\begin{equation*}
		\begin{tikzcd}[row sep=tiny]
				\pi_{*} (X_n^{C_2}) \ar[r] & {\pi_* \fib \left(X_{n+1}^{C_2} \to X_{n+1}^e\right)} \ar[d] & \\
				& \pi_{*} (X_{n+1}^{C_2}) \ar[r] & {\pi_* \fib \left(X_{n+2}^{C_2} \to X_{n+2}^e\right)} &
		\end{tikzcd}
		\end{equation*}
		is zero. 
		Alternatively, we may regard $ d $ as inducing maps $ \pi_{V} X_n \to \pi_{V + \sigma }X_{n+1} $ on $ RO(C_2) $-graded homotopy groups.  
	\end{itemize} 
\end{rmk}

\begin{rmk}\label{rmk:involutive_dg_vs_dg_with_invltn}
	We contrast the structure present on $ X_* \in \DG^\sigma_{+}\left(\Gr(\underline{\Mod}_{\underline{k}})\right) $ with that on $ Y_{*} \in \DG_{-}\left(\Gr(\underline{\Mod}_{\underline{k}})\right) $: $ Y_* $ can be regarded as having a differential $ d $ which lowers degree (or in other words, $ Y_*$  is a chain complex), but the differential $ d $ commutes with the $ C_2 $-action. 
	In other words, the $ C_2 $-action on $ Y_* $ is an action via maps of chain complexes, while the $ C_2 $-action on $ X_* $ is not an action via maps of cochain complexes. 
	Thus, we refer to objects of $ \DG^\sigma_{+}\left(\Gr(\underline{\Mod}_{\underline{k}})\right) $ as involutive cochain complexes, while we refer to objects of $ \DG_{-}\left(\Gr(\underline{\Mod}_{\underline{k}})\right) $ as chain complexes \emph{with involution}.
\end{rmk}
Using the $ C_2 $-symmetric monoidal structures on $ \underline{\DG}^\sigma_+(\cat) $ and $ \underline{\DG}_{-}(\cat) $, we can formulate notions of homotopy coherent ($h^\sigma_{+}$- or $ h_{-}$-) differential graded $ C_2 $-$ \E_\infty $ algebras. 
However, in the $ h^\sigma_{+} $ setting, we may define a variant of this algebraic structure (to be thought of as endowed with strictly commutative multiplication and strictly equivariant norm maps) using additional structure on $ \D^\sigma_{+} $ and the theory of derived involutive algebra of \S\ref{section:scr_with_inv}.  
\begin{obs}\label{obs:Dsigma_bialg}
	Note that $ \D^\sigma_+ $ is dualizable--denote the dual by $ (\D^\sigma_+)^\vee $.  
	By Corollary \ref{cor:mod_comod_equivalence}, $ (\D^\sigma_+)^\vee $ is a $ C_2 $-bicommutative bialgebra. 
\end{obs}
\begin{prop}\label{prop:Dsigma_alg}
		There is a unique $C_2$-bicommutative derived involutive algebra structure on $ (\D^\sigma_+)^\vee $ in $ \Gr\left(\underline{\Mod}_{\underline{\Z}}\right) $ promoting its $C_2$-bicommutative bialgebra structure. 
		Moreover, under (the graded variant of) the forgetful functor of Remark \ref{rmk:inv_dalg_underlying}, the derived algebra structure on $ \left((\D^\sigma_+)^\vee\right)^e $ agrees with that of $ \D_{+}^\vee $ from \cite[Proposition 5.1.7]{Raksit20}. 
\end{prop}
\begin{proof}
	It suffices to observe that both $ \underline{\Z} $ and $ \underline{\Z}_{-} $ are zero slices, and $ \Z \oplus \Z_{-}(-1) $ admits the structure of an involutive commutative derived algebra (Definition \ref{defn:inv_gr_calg}). 
	The result follows from Proposition \ref{prop:inv_grdalg}. 
	The second statement follows from the description of the Postnikov filtration on $ \Mod_{\underline{k}} $ in Variant \ref{var:eqvtmodulespostnikov}. 
\end{proof} 
\begin{rmk}\label{rmk:square_zero_graded_alg}
	Recall $ \D^\vee, \D^\vee_{C_2} \in \Gr\DAlg^{\sigma}_{\underline{\Z}} $ of Notation \ref{ntn:square_zero_graded_alg}. 
	Note that there is a canonical identification $ \D^{\sigma,\vee}_{+} \simeq \mathbbm{1} \times_{\D^\vee_{C_2}} \D^\vee $ in $ \Gr\DAlg^{\sigma}_{\underline{\Z}} $. 
\end{rmk}
\begin{defn}\label{defn:dg_involutive_alg}
	Let $ \cat $ be a $ \underline{\Z} $-linear derived involutive algebraic context and let $ A $ be a derived involutive algebra object of $ \cat $. 
	The $ C_2 $-$ \infty $-category of \emph{$h_{+}^\sigma$-differential graded derived involutive $ A $-algebras} is given by 
	\begin{equation*}
		\DG^\sigma_+\underline{\DAlg}_{A}^\sigma := \underline{\coMod}_{(\D^\sigma_+)^\vee}\left(\Gr\underline{\DAlg}^\sigma(\cat)\right) . 
	\end{equation*} 
\end{defn} 
Finally, despite the caveat of Remark \ref{rmk:involutive_dg_vs_dg_with_invltn}, the shear functor of (\ref{eq:shear_gr_by_reg_rep_sphere}) allows us to relate involutive cochain complexes with chain complexes with involution.  
\begin{prop} \label{prop:shear_gr_on_dg_mod}  
Let $ \cat $ be a $ \underline{\Z} $-linear derived involutive algebraic context. 
\begin{enumerate}[label=(\arabic*)]
	\item The functor $ [-\rho *] $ of (\ref{eq:shear_gr_by_reg_rep_sphere}) lifts to a $ C_2 $-symmetric monoidal equivalence of $ C_2 $-$ \infty $-categories $ \underline{\Mod}_{\D^\sigma_{+}}(\Gr(\cat)) \xrightarrow{[-\rho *]} \underline{\Mod}_{\D_{-}} (\Gr(\cat)) $ with inverse $ [\rho *] $. 
	\item On underlying spectra, $ [\pm \rho *] $ recovers the equivalences $ [\pm 2 *] $ of \cite[Remark 5.1.12]{Raksit20}, i.e. there is a commutative diagram
	\begin{equation*}
	\begin{tikzcd}
		\underline{\Mod}_{\D^{\sigma}_{+}}(\Gr(\cat))^{C_2} \ar[d,"(-)^e"] \ar[r,"{\left[-\rho *\right]}","\sim"'] & \underline{\Mod}_{\D_{-}} (\Gr(\cat))^{C_2} \ar[d,"(-)^e"] \\
		\Mod_{\D_{+}}(\Gr(\cat^e)) \ar[r,"{\left[- 2 *\right]}","\sim"'] & \Mod_{\D_{-}}(\Gr(\cat^e)) 
	\end{tikzcd}
	\end{equation*}
	where the horizontal equivalences have inverses $ [\rho *] $ and $ [2*] $, respectively. 
\end{enumerate}
\end{prop} 
\begin{rmk}
		Compare \cite[Corollaries 2.20 and 2.21]{MR3007090}. 
\end{rmk}
\begin{proof} [Proof of Proposition \ref{prop:shear_gr_on_dg_mod}]
\begin{enumerate}[label=(\arabic*)] 
	\item It suffices to observe that the equivalence $ [-\rho*] $ takes the bicommutative bialgebra $ \D^\sigma_+ $ to the bicommutative bialgebra $ \D_{-} $. 
	The statement about the algebra structures follows from the assertion about the hearts in Proposition \ref{prop:shear_gr_by_reg_rep_sphere_is_mult}. 
	\item This follows from the observation that $ (S^\rho)^e\simeq S^2 $. \qedhere
\end{enumerate}
\end{proof}
Given a homotopy coherent cochain complex, we can consider cocycles and coboundaries to be the orbits and fixed points (resp.) with respect to $ \D_{\pm} $. 
Then, we can think of the \emph{cohomology} of said cochain complex to be the Tate construction with respect to $ \D_\pm $. 
The relevant question for our purposes is: What is the cohomology of an involutive cochain complex in the sense of Definition \ref{defn:dg_involutive_module}? 
However, the question is more subtle due to Remark \ref{rmk:involutive_dg_vs_dg_with_invltn}. 
Suppose $ M^* $ is an involutive cochain complex of $ \underline{\Z} $-Mackey functors, and let us consider the graded abelian group with $ C_2 $-action $ (M^*)^e $. 
In particular, $ (M^n)^e $ is concentrated in degree $ n $.  
Let us define a new graded abelian group with $ C_2 $-action $ N^* $ as follows: As graded abelian groups, $ N^\ell = M^\ell $ and if $ \ell $ is odd, allow $ \sigma \neq e \in C_2 $ to act on $ N^\ell $ by $ M^\ell \xrightarrow{\sigma(-)} M^\ell \xrightarrow{(-1)\cdot} M^\ell $.  
Otherwise if $ \ell $ is even, then the $ C_2 $-action on $ N^\ell $ agrees with the $ C_2 $-action on $ M^\ell $. 
The differential on $ M^\ell $ induces a differential on $ N^\ell $; however, the differential on $ N^\ell $ \emph{is $ C_2 $-equivariant}. 
In particular, we can take the cohomology of $ N $ in the usual sense, and it is a graded $ \Z[C_2] $-module.  

\begin{rmk}
		The bi-$ C_2 $-$ \E_\infty $-bialgebras $ \D^\sigma_+ $ and $ \D_{-} $ satisfy the assumptions of Proposition \ref{prop:tate_cons_is_param_compatibility}, so we have a norm map and parametrized Tate construction for objects of $ \underline{\LMod}_{\D^\sigma_+}(\cat) $, $ \underline{\LMod}_{\D_{-}}(\cat) $. 
		In particular, we can take $ \omega_{\D^\sigma_{+}} = \underline{\Z}[\sigma](1) $.  
\end{rmk}
\begin{defn}\label{defn:cohomology_of_inv_chain_cplx} 
	Let $ \cat $ be a $ \underline{\Z} $-linear $ C_2 $-stable $ C_2 $-presentable $ C_2 $-symmetric monoidal $ C_2 $-$ \infty $-category (see Remark \ref{ntn:D_Dsigma_in_any_C2_sym_cat}).
	Let $ M_* \in \Mod_{\D^\sigma_{+}}(\cat) $ (resp. $ M_* \in \Mod_{\D^\sigma_{-}}(\cat) $) an involutive $ h_+^\sigma $ (resp. $ h_{-} $) chain complex in $ \cat $. 
	Then the \emph{cohomology of $ M $} is the Tate construction
	\begin{equation*}
		H^*(M) \simeq M^{t\D^\sigma_{+}} \in \Gr(\cat) \qquad \left(\text{resp. } H^*(M) \simeq M^{t\D_{-}} \in \Gr(\cat)\right)\,.
	\end{equation*}
\end{defn}
\begin{rmk}\label{rmk:cohomology_is_normed_alg}
		Suppose that $ \cat $ is a $\underline{\Z}$-linear $C_2$-symmetric monoidal $ C_2 $-$ \infty $-category which is $C_2$-complete, $C_2$-cocomplete, and fiberwise compactly generated. 
		Then by Proposition \ref{prop:param_tate_lax_monoidal}, the functors $ H^*(-) \colon \DG_{-}(\cat)\to \Gr(\cat) $ and $ \DG_{+}^\sigma(\cat)\to \Gr(\cat)$ are lax $ C_2 $-symmetric monoidal. 
		Given a $ C_2 $-commutative algebra in $ \DG_{-}(\cat) $ or $ \DG^\sigma_{-}(\cat) $, its cohomology is canonically a $ C_2 $-commutative algebra in $ \Gr(\cat) $. 
\end{rmk}
We can relate the (co)homology with respect to $ \D^\sigma_\pm $ using the following. 
\begin{defn}
		[{\cite[Definition 5.2.4]{Raksit20}}]\label{defn:cohomology_type} 
		Suppose that $ \cat $ is a $\underline{\Z}$-linear $C_2$-symmetric monoidal $ C_2 $-$ \infty $-category which is $C_2$-complete, $C_2$-cocomplete, and fiberwise compactly generated.
		Let $ X \in \underline{\Mod}_{\mathbb{D}_{-}}(\cat) $. 
		Write $ |X|^{\geq *} \in \underline{\Fil}^{\wedge}(\cat) $ denote the image of $ X $ under the $C_2$-functor of Proposition \ref{prop:complete_fil_as_cochain_cplx_monoidal}. 
		Let $ |X| := \underline{\colim} |X|^{\geq *} \in \cat $. 
		Given $ X \in \underline{\Mod}_{\mathbb{D}^\sigma_+}(\cat) $, we define $ |X|^{\geq *} = |X[-\rho *]|^{\geq *} \in \underline{\Fil}^{\wedge}(\cat) $ and $ |X| := \underline{\colim} |X[-\rho *]|^{\geq *} \in \cat $ by transporting their definitions on $ \underline{\Mod}_{\mathbb{D}_{-}}(\cat) $ along the equivalence of Proposition \ref{prop:shear_gr_by_reg_rep_sphere_is_mult}. 
		In both cases, we will refer to $ |X | $ as the \emph{cohomology type} of $ X $ and $ |X|^{\geq *} $ as the \emph{brutal filtration} on the cohomology type of $ X $. 
\end{defn}
\begin{prop}\label{prop:param_coh_formula}
		Suppose that $ \cat $ is a $\underline{\Z}$-linear $C_2$-symmetric monoidal $ C_2 $-$ \infty $-category which is $C_2$-complete, $C_2$-cocomplete, and fiberwise compactly generated. 
		Let $ \delta_{\mathrm{gr}} \colon \cat \to \Gr(\cat) $ denote the diagonal functor. 
		Then
		\begin{enumerate}[label=(\alph*)]
			\item \label{prop_item:param_coh_formula_Dminus_mod} For $ X \in \underline{\Mod}_{\mathbb{D}_{-}}(\cat) $, there is an equivalence $ H^*(X) \simeq \delta_{\mathrm{gr}}\left(|X|\right) $
			\item \label{prop_item:param_coh_formula_Dsigmaplus_mod} For $ X \in \underline{\Mod}_{\mathbb{D}_{+}^\sigma}(\cat) $, there is an equivalence $ H^*(X) \simeq \delta_{\mathrm{gr}}\left(|X|\right)[\rho *] $, where $ [\rho *] $ is the $ C_2 $-functor of (\ref{eq:shear_gr_by_reg_rep_sphere}). 
		\end{enumerate}
		Moreover, the equivalences here recover those of \cite[Proposition 5.2.6]{Raksit20} on underlying objects $ X^e \in \cat^e $. 
\end{prop}
\begin{proof}
		Part \ref{prop_item:param_coh_formula_Dminus_mod} follows immediately from \cite[Proposition 5.2.6(a)]{Raksit20} and the fiberwise description of parametrized modules of Corollary \ref{cor:param_mod_fiberwise_description}. 
		Part \ref{prop_item:param_coh_formula_Dsigmaplus_mod} follows from part \ref{prop_item:param_coh_formula_Dminus_mod} and naturality of the parametrized Tate construction (Remark \ref{rmk:param_tate_naturality}). 
\end{proof}
Now, we show how the aforementioned constructions allow us to define the involutive derived de Rham complex via a universal property. 
Moreover, we use the preceding shear equivalence to define involutive derived de Rham cohomology. 
For the rest of the section, let us fix a $ \underline{\Z} $-linear derived involutive algebraic context $ \cat $ and a derived involutive algebra $ A $ in $ \DAlg^\sigma(\cat) $. 
\begin{ntn} 
	Let $ \DG_+^{\sigma, \geq 0}\underline{\DAlg}^\sigma_A $ denote the fiber product $ \DG_+^\sigma\underline{\DAlg}^\sigma_A \fiberproduct_{\Gr\underline{\DAlg}^\sigma_A} \Gr^{\geq 0}\underline{\DAlg}^\sigma_A $ (see Definition \ref{defn:dg_involutive_alg}); it is a $ C_2 $-$ \infty $-category. 
	We will refer to $ C_2 $-objects therein as \emph{nonnegative $ h_\sigma^+$-differential graded involutive $ A $-algebras.} 
\end{ntn} 
\begin{prop}\label{prop:inv_deRham_asadjoint} 
	The composite of the forgetful $ C_2 $-functor with the evaluation functor
	\begin{equation*}
		\DG_+^{\sigma, \geq 0}\underline{\DAlg}_{A}^\sigma \xrightarrow{U} \Gr^{\geq 0}\underline{\DAlg}_{A}^\sigma \xrightarrow{\ev^0} \underline{\DAlg}_{A}^\sigma
	\end{equation*}
	admits a left $ C_2 $-adjoint. 
	On underlying $ \infty $-categories, this recovers the adjunction of \cite[Proposition 5.3.2]{Raksit20}. 
\end{prop}
\begin{proof} 
	Observe that $ \DG_+^{\sigma, \geq 0}\underline{\DAlg}_{A}^\sigma $ and $ \underline{\DAlg}^\sigma_{A} $ admit finite $ C_2 $-products which are preserved by $ \ev^0 \circ U $. 
	By Corollary \ref{cor:C2_left_adjoint_local_crit}, it suffices to show that $ \ev^0 \circ U $ admits fiberwise left adjoints. 
	In view of Remark \ref{rmk:zeroth_graded_dalg}, the latter follows from a nearly identical argument to that of \cite[Proposition 5.3.2]{Raksit20}; we only need note that $ (\D^{\sigma}_+)^\vee $ is dualizable. 
\end{proof} 
\begin{defn}\label{defn:inv_deriveddeRhamcplx}
	We will refer to the left $ C_2 $-adjoint $ \L \Omega^{\bullet, \sigma}_{-/A}: \underline{\DAlg}_{A}^\sigma \to \DG_+^{\sigma, \geq 0}\underline{\DAlg}_{A}^\sigma $ of Proposition \ref{prop:inv_deRham_asadjoint} as the \emph{involutive derived de Rham complex over $ A $}.
\end{defn}
\begin{rmk}
		Let $ B $ be an involutive $ A $-algebra. 
		Then the involutive de Rham complex is equipped with a canonical map $ B \to \L\Omega^{\sigma,0}_{B/A} $ in $ \underline{\DAlg}^\sigma_{A} $, which induces a map $ B \to \L\Omega^{\sigma,\bullet}_{B/A} $ in $ \Gr\underline{\DAlg}^{\sigma}_{A} $. 
		Thus, we will regard $ \L\Omega^{\sigma,\bullet}_{B/A} $ as an object of $ \Gr\underline{\DAlg}^\sigma_B $. 
\end{rmk}
We can relate the involutive derived de Rham complex with the involutive cotangent complex as follows:
\begin{theorem}\label{thm:invdeRham_equals_LSymcotangent}
	The \emph{involutive derived de Rham complex} is computed by
	\begin{align*}
		\L \Omega^{\sigma,\bullet}_{-/A} \colon & \underline{\DAlg}_{A}^\sigma \to \DG_+^{\sigma, \geq 0}\underline{\DAlg}_{A}^\sigma \\
		& B \mapsto \LSym_{B}^\sigma( \Sigma^\sigma \L_{B/A}(1))
	\end{align*}
	where $ \L_{B/A} $ is the involutive cotangent complex of Definition \ref{defn:involutive_cotangent}. 
	Moreover, the first differential of the $ h^\sigma_+ $-cochain complex
	\begin{equation*}
		d \colon B \simeq \L\Omega^{\sigma,0}_{B/A} \to \Sigma^{-\sigma} \L\Omega^{\sigma,1}_{B/A} \simeq \L_{B/A}
	\end{equation*}
	is given by the universal $ A $-linear derivation of $ B $ in the sense of Definition \ref{defn:derivations}. 
\end{theorem}
\begin{cor}
	The involutive derived de Rham complex functor is fully faithful.  
\end{cor}
The involutive de Rham complex lifts the usual de Rham complex in the following sense. 
\begin{cor}\label{cor:inv_de_Rham_cplx_non_eqvt_comparison}
		Let $ k $ be a commutative ring with involution, and let $ \underline{k} $ denote its associated fixed point $ C_2 $-Green functor. 
		Let $ A $ be a derived involutive $ \underline{k} $-algebra and write $ A^e $ for its underlying $ k $-algebra (having forgotten the $ C_2 $-action). 
		Then there are equivalences 
		\begin{equation*}
				\left(\L\Omega^{\sigma,\bullet}_{A/\underline{k}}\right)^e \simeq \L\Omega^{\bullet}_{A^e/k}
		\end{equation*}
		of $ h_+ $-differential graded derived commutative $ k $-algebras which are natural in $ A $. 
		Here, the right-hand side denotes the ordinary derived de Rham complex. 
\end{cor}
The comparison result follows immediately from the definitions and the characterization of the derived de Rham complex in \cite[\S5.3]{Raksit20}; let us note that such a comparison is facilitated by the language of $ C_2 $-$ \infty $-categories. 
\begin{proof}
	[Proof of Theorem \ref{thm:invdeRham_equals_LSymcotangent}] 
	We construct a diagram (\ref{diagram:involutive_cotangent_formula}) of $C_2$-adjoint pairs (in which the left and upper arrows are left $ C_2 $-adjoints) lifting the diagram appearing in the proof of \cite[Theorem 5.3.6]{Raksit20}. 
	The result follows from an otherwise identical proof strategy to that in \emph{loc. cit.}: showing the $C_2$-functor given by tracing from the lower left directly up, then to the right is equivalent to the $ C_2 $-functor from the lower left to the upper right which passes through the center of the diagram. 
	Let us make the observations: 
	\begin{itemize}
		\item The forgetful $ C_2 $-functor $ U \colon \DG^\sigma_+\underline{\DAlg}_{A}^\sigma \to \Gr\underline{\DAlg}_{A}^\sigma $ admits a right $ C_2 $-adjoint $ V $. 
		This claim follows from Corollary \ref{cor:C2_left_adjoint_local_crit}; the restriction functors associated to the map $ C_2 \to C_2/C_2 $ have left adjoints by an argument similar to Proposition \ref{prop:invdalg_C2_colims}. 
		The functor $ U $ commutes with the left adjoints to the restriction functor because the $ C_2 $-symmetric monoidal structure on $ \DG_+^\sigma(\cat) $ is induced by the $ C_2 $-$ \E_\infty $-coalgebra structure on $ \D^\sigma_+ $.  
		For $ B \in \Gr\underline{\DAlg}^\sigma_{A} $, there is a canonical natural equivalence $ U(V(B)) \simeq B \ostar \D^{\sigma,\vee}_+ $. 
		\item The natural transformation
		\begin{equation*}
			\ev^0 \circ V \to \ev^0 \circ V \circ \ins^{\{0,1\}} \circ \ev^{\{0,1\}}
		\end{equation*}
		induced by the unit of the adjunction $ (\ins^{\{0,1\}} , \ev^{\{0,1\}}) $ is an equivalence because of the identity $ \ev^0 V(B) \simeq \ev^0\left(B \ostar \D^{\sigma,\vee}_+\right) \simeq B^0 \oplus \Sigma^{-\sigma}B^{1} $ in $ \cat $. 
		On underlying $ \infty $-categories, the equivalence recovers that of \cite[Lemma 5.3.12(b)]{Raksit20}.  
		It follows that the outer rectangle of (\ref{diagram:involutive_cotangent_formula}) consisting of right $ C_2 $-adjoints commutes. 
		\item The adjoint pair $ G, L $ are the square-zero extension and involutive cotangent complex functors of Proposition \ref{prop:sq_zero_has_left_adjoint} and Definition \ref{defn:involutive_cotangent}. 
			For any $ (B, M) \in \Gr^{\{0,1\}}\underline{\DAlg}_{A}^\sigma $, there is a natural equivalence $ \ev^0 (V(B, M)) \simeq G (B \oplus \loops^\sigma M) $. 
			This follows by a modification of \cite[Lemma 5.3.12(c)]{Raksit20}, in view of Remark \ref{rmk:square_zero_graded_alg}. 
			It follows that the lower trapezoid of (\ref{diagram:involutive_cotangent_formula}) consisting of right adjoints commutes. 
		\item Let $ \alpha $ be the equivalence of Proposition \ref{prop:modules_and_gr01}. 
			By Corollary \ref{cor:C2_left_adjoint_local_crit}, the statement of \cite[Lemma 5.3.11]{Raksit20} holds with adjoints replaced by $ C_2 $-adjoints and $ F $ replaced by the $ C_2 $-functor which sends $ (B,M) \mapsto \LSym^\sigma_B(M(1)) $. 
			Therefore, the right triangle of (\ref{diagram:involutive_cotangent_formula}) commutes.  
	\end{itemize}
	\begin{equation}\label{diagram:involutive_cotangent_formula}
	\begin{tikzcd}[row sep=large, column sep=large]
		\DG^{\sigma, \geq 0}_+ \underline{\DAlg}_{A} \ar[dd, shift left=1,"{\ev_0}"] \ar[r,shift left=1,"{\ins^{\geq 0}}"] & \DG^{\sigma}_+ \underline{\DAlg}_{A}\ar[rr, shift left=1,"U"] \ar[l,shift left=1,"{\ev^{\geq 0}}"] & & \Gr\underline{\DAlg}_{A}^\sigma \ar[ll, shift left=1,"V"] \ar[d,shift left=1,"{\ev^{\geq 0}}"]   \\
		& \underline{\DAlg^\sigma\Mod}_{A} \ar[ld, shift left=1,"G"] \ar[r, shift left=1,"\Sigma^\sigma"] & \underline{\DAlg^\sigma\Mod}_{A} \ar[r, shift left=1,"{\LSym_{}^\sigma}"] \ar[l, shift left=1,"\Omega^\sigma"] & \Gr^{\geq 0}\underline{\DAlg}_{A}^\sigma \ar[d, shift left=1,"{\ev^{\{0,1\}}}"] \ar[u,shift left=1,"{\ins^{\geq 0}}"] \ar[l, shift left=1] \\
		\underline{\DAlg}_{A}^\sigma \ar[uu, shift left=1,"{\L \Omega^{\sigma,\bullet}_{-/A}}"] \ar[ru, shift left=1,"L"]  & \DG^{\sigma, \geq 0}_+\underline{\DAlg}_{A} \ar[l, shift left=1,"\ev_0"] & \Gr^{\geq 0}\underline{\DAlg}_{A}^\sigma \ar[l, shift left=1,"V"] & \Gr^{\{0,1\}}\underline{\DAlg}_{A}^\sigma \ar[u, shift left=1,"{\ins^{\{0,1\}}}"] \ar[lu, "\sim"',"\alpha"] \ar[l, shift left=1,"\ins"]
	\end{tikzcd} .
	\end{equation}
	Since the diagram of right $ C_2 $-adjoints in (\ref{diagram:involutive_cotangent_formula}) commutes, the diagram of left $ C_2 $-adjoints commutes and we are done. 

	The description of the map $ B \to \L\Omega^{\sigma, 1}_{B/A} $ as the universal $ A $-linear derivation (of involutive rings) follows from a $ C_2 $-analogue of the argument appearing in \cite[Theorem 5.3.6]{Raksit20}; the main point is that the (\ref{diagram:involutive_cotangent_formula}) is a parametrized enhancement of the diagram appearing there. 
\end{proof} 
\begin{defn}\label{defn:involutive_deRham_cohomology} 
		Let $ B $ be a derived involutive algebra over $ A $. 
		Define the \emph{Hodge-filtered Hodge-completed involutive de Rham cohomology} $ \invdeRham^{\wedge, \geq \star}_{B/A} := \left|\L \Omega^{\bullet,\sigma}_{B/A}\right|^{\geq *} $ where $ |-|^{\geq *} $ is the functor of Definition \ref{defn:cohomology_type}; in words, it is the image of the involutive de Rham complex under the equivalence of Proposition \ref{prop:complete_fil_as_cochain_cplx_monoidal}. 
		In particular, involutive de Rham cohomology defines a $ C_2 $-functor
		\begin{equation*}
				\invdeRham^{\wedge, \geq \star}_{-/A} \colon \underline{\DAlg}_A^\sigma \to \underline{\E}_\infty\underline{\Alg}\left(\Fil^\wedge\left(\underline{\Mod}_A\right)\right) \,.	
		\end{equation*}
		Define the \emph{Hodge-completed involutive de Rham cohomology} to be the colimit $ \invdeRham_{-/A} := \colim_\star \invdeRham^{\wedge, \geq \star}_{-/A} $.  
		This defines a $ C_2 $-functor $ \underline{\DAlg}_A^\sigma \to \underline{\E}_\infty\underline{\Alg}\left(\underline{\Mod}_A\right) $. 
\end{defn}
The next comparison result follows immediately from the definitions of (filtered) involutive derived de Rham cohomology and the descriptions of the non-equivariant counterparts in \cite[\S5.3]{Raksit20}. 
However, let us remark that such a comparison is straightforward \emph{precisely because} we use the language of $ C_2 $-$ \infty $-categories. 
\begin{prop}\label{prop:inv_deRham_non_eqvt_comparison}
		Let $ k $ be a commutative ring with involution, and let $ \underline{k} $ denote its associated fixed point $ C_2 $-Green functor. 
		Let $ A $ be a derived involutive $ \underline{k} $-algebra and write $ A^e $ for its underlying $ k $-algebra (having forgotten the $ C_2 $-action). 
		Then there are equivalences 
		\begin{equation*}
			\left(\invdeRham^{\wedge, \geq \star}_{A/\underline{k}}\right)^e \simeq \mathrm{dR}^{\wedge, \geq \star}_{A^e/k} \qquad \qquad  \left(\invdeRham_{A/\underline{k}}\right)^e \simeq \mathrm{dR}_{A^e/k}
		\end{equation*}
		of complete filtered $ \E_\infty $-$ k $-algebras and $ \E_\infty $-$ k $-algebras, respectively which are natural in $ A $. 
		Here, the right-hand sides denote the ordinary (i.e. non-equivariant) Hodge-filtered Hodge-completed derived de Rham cohomology and de Rham cohomology, respectively. 
\end{prop} 
\begin{rmk}\label{rmk:involutive_derham_alg_structure_pending_conjecture}
		Let $ A $ be a derived involutive algebra over $ \underline{k} $. 
		The object $ \invdeRham^{\wedge, \geq \star}_{A/\underline{k}} $ may be regarded as an $ \E_\infty $-algebra in complete filtered $ \underline{k} $-modules in $ C_2 $-spectra, and $ \invdeRham_{-/\underline{k}} $ may be regarded as an $ \E_\infty $-$ \underline{k} $-algebra in $ C_2 $-spectra.  
		Assuming the conjecture of Remark \ref{rmk:complete_fil_as_cochain_cplx_C2monoidal_conj}, we may replace $ \E_\infty $ by $ C_2 $-$ \E_\infty $ in the previous sentence. 
\end{rmk}

\section{Real Hochschild homology}\label{section:realhh} 
Let $ R $ be a smooth $ \Z $-algebra. 
In their original work, Hochschild--Kostant--Rosenberg computed the homology of the cyclic bar complex $ H_*(B^{\mathrm{cyc}}R) $ and showed that it is isomorphic (as a graded abelian group) to de Rham complex of $ R $ over $ \Z $ \cite{MR142598}. 
Since the homology of a $ \Z $-module is the associated graded of the Postnikov filtration on it, that the filtration is functorial in $ R $ follows \emph{a posteriori} from this computation. 
The Hochschild homology has an action of $ S^1 $, and since the Postnikov filtration is lax symmetric monoidal, $ \mathrm{fil}_{\mathrm{HKR}}\HH(R/\Z) $ admits an action of $ \tau_{\geq *} \Z[S^1] $. 

In more recent work, Raksit formulates an elegant universal property for Hochschild homology incorporating all of the structure present on $ \HH $--its $ S^1 $-action, algebra structure, and filtration--and also specifying how they interact \cite{Raksit20}. 
Raksit identifies a canonical derived bicomutative bialgebra structure on the filtered $ \Z $-module $ \tau_{\geq *}\Z^{S^1} $, and uses the presence of this structure to define a functor $ \HH_{\mathrm{fil}}(-/\Z) \colon \DAlg_{\Z} \to \Fil_{S^1}^{\geq 0}\DAlg_{\Z} $ which interpolates between ordinary Hochschild homology and the derived de Rham complex. 
In particular, the identification of the filtration on the underlying object of $ \HH_{\mathrm{fil}} $ with the Postnikov filtration happens \emph{a posteriori}, as $ \tau_{\geq *} $ does not preserve derived algebra structures except in special cases (\cite[Remark 6.1.11]{Raksit20}). 
Beyond giving a functorial description of filtered Hochschild homology, this method is remarkably \emph{robust}: Raksit's result applies when $ \Z \to R $ are replaced by $ A \to B $ where $ A, B $ are arbitrary derived algebras. 

The robustness of the filtered circle approach to the HKR-theorem extends to the involutive setting. 
As Hornbostel--Park have already observed in \cite{PHrealTHH_perfectoid}, a HKR-style filtration on the real Hochschild homology of constant $ C_2 $-Green functors associated to commutative rings does not arise [directly] from an existing well-known filtration on $ \Mod_{\underline{\Z}} $. 
The situation is even murkier when one considers the real Hochschild homology of a $ C_2 $-Tambara functor on which the $ C_2 $-action is \emph{nontrivial} (see Corollary \ref{cor:gr_HR_via_resolutions}). 
We sidestep the question of which filtration on $ \Mod_{\underline{\Z}} $ is the `correct' one by instead studying the notion of a \emph{filtered involutive circle action}. 
Informally, we may regard a filtered $ S^\sigma $-action as a filtered $ S^1 $-action (in the sense of \cite[\S6]{Raksit20}) with a twist, i.e. the circle action simultaneously increases the filtration degree and sends a $ \Z[C_2] $-module $ M $ to $ M \otimes \Z_{-} $. 

In \S\ref{subsection:fil_inv_circ}, we introduce real Hochschild homology and discuss the involutive circle $ S^\sigma $ which acts on it.  
Adopting the strategy of Raksit outlined above, we show that $ \tau_{\geq *}\underline{\Z}^{S^\sigma} $ admits the structure of a derived bicommutative bialgebra refining that on $ \Z^{S^1} $. 
In \S\ref{subsection:realHKR}, we use the notion of filtered involutive circle actions of the previous subsection to define a filtered enhancement of real Hochschild homology and prove Theorem \ref{thm:real_hkr}. 
In \S\ref{subsection:filteredHCHPetc}, we consider whether our constructions allow us to identify filtrations on real negative cyclic homology, real cyclic homology, and real periodic cyclic homology. 
In \S\ref{subsection:computations_comparisons}, we include a few computations of filtrations on $ \HR(-/\underline{k}) $ of free derived involutive algebras and discuss relationships between our work and classical work when $ 2 $ acts invertibly on the base. 

\subsection{The involutive filtered circle}\label{subsection:fil_inv_circ}
Throughout this section, we will work with a fixed derived involutive algebraic context $ \cat $ and a derived involutive algebra object $ A $ of $ \cat $. 
\begin{defn}
	The category of \emph{derived involutive $ A $-algebras with $ S^\sigma $-action} is the $ C_2 $-$\infty$-category $ \underline{\Fun}_{C_2}(BS^\sigma,\underline{\DAlg}_{A}^\sigma) $ (Proposition \ref{prop:param_functors}). 	
\end{defn}
A choice of basepoint $ * \inj BS^\sigma $ induces a $ C_2 $-functor $ g \colon \underline{\Fun}_{C_2}\left(BS^\sigma,\underline{\DAlg}_{A}^\sigma\right) \to \underline{\DAlg}_{A}^\sigma $. 
\begin{lemma}\label{lemma:existence_freeSsigma_alg}
	The functor $ g $ admits a left $ C_2 $-adjoint 
	\begin{equation*}
		\underline{\DAlg}_{A}^\sigma \to \underline{\Fun}_{C_2}\left(BS^\sigma,\underline{\DAlg}_{A}^\sigma\right) .
	\end{equation*} 
\end{lemma}
\begin{proof}
	Follows from \cite[Corollary 9.18 \& Theorem 10.5]{Shah18} and Proposition \ref{prop:invdalg_C2_colims}.
\end{proof}
\begin{defn}\label{defn:realHH} 
	We will denote the left $ C_2 $-adjoint of Lemma \ref{lemma:existence_freeSsigma_alg} by $ \HR(-/A) $. 
	If $ B $ is a derived involutive $ A $-algebra, we will refer to $ \HR (B/A) $ refer to it as \emph{real Hochschild homology of $ B $ over $ A $}, or simply \emph{real Hochschild homology of $ B $} when $ A $ is understood. 
\end{defn}
\begin{prop}\label{prop:HR_Ssigma_colimit}
	The composite of the forgetful functor and real Hochschild homology is computed on underlying objects by a parametrized colimit over $ S^\sigma $ (Example \ref{ex:signrepsphere})
	\begin{equation*}
		g \circ \HR(B/A) \simeq B^{\otimes_{A} S^\sigma} .
	\end{equation*}
\end{prop}
\begin{proof}
	Recall the existence of a fiber sequence of $ C_2 $-spaces $ S^\sigma \to ES^\sigma \simeq \{*\} \xrightarrow{\eta} BS^\sigma $. 
	By \cite[Theorem 10.5; also see Remark 5.4]{Shah18}, the left adjoint to $ g $ is given by $ C_2 $-left Kan extension along the morphism $ \eta $. 
	By \cite[Remark 10.2]{Shah18}, the criterion that a diagram be a $ C_2 $-left Kan extension may be checked pointwise, i.e. after pulling back along a basepoint $ \underline{*} \to  BS^\sigma $ (since $ BS^\sigma $ is connected). 	
	The result follows from the description of $ C_2 $-colimits in (the $ C_2$-$\infty$-category of) $ C_2 $-$ \E_\infty $-algebras in \cite[Corollary 5.3.8]{NS22}. 
\end{proof}
\begin{rmk}\label{rmk:real_HH_as_THR_linearization}
\begin{itemize}
	\item Let $ k $ be a discrete ring with an involution, and let $ A \in \DAlg_{\underline{k}}^\sigma $. 
	The real topological Hochschild homology of (the underlying $ C_2 $-$ \E_\infty $-algebra of) $ A $ as defined in Dotto--Moi--Patchkoria--Reeh by \cite[Corollary 2.12]{MR4186464} is given by $ \THR(A) := A^{\otimes S^\sigma} = A \otimes_{N_e^{C_2}A} A $. 

	Since $ A \otimes_{\underline{N}_e^{C_2}A} A \simeq \THR(A/\sphere) \otimes_{\THR(\underline{k}/\sphere)} \underline{k} $, we see that $ \HR(-/\underline{k}) $ is a \emph{linearization} of $ \THR $. 

	\item Since $ (-)^e $ is symmetric monoidal, we see that 
	\begin{equation*} 
	\HR(A/\underline{k})^e = (\THR(A/\sphere) \otimes_{\THR(\underline{k}/\sphere)} \underline{k})^e = \THH(A^e/\sphere^e) \otimes_{\THH(k^e/\sphere^0)} k^e \simeq \HH(A^e/k^e) .
	\end{equation*}
	In particular, real Hochschild homology is an \emph{enhancement} of Hochschild homology. 

	\item This agrees with the definition of real Hochschild homology in \cite{SVP96}.
\end{itemize}
\end{rmk} 
Next, we show that a $ \underline{\Z} $-module with an $ S^\sigma $-action is equivalently a module over the $ C_2 $-group ring $ \underline{\Z}[S^\sigma] $. 
\begin{ntn}\label{ntn:Z_linear_circle}
	We let $ \T^\sigma $ denote the $ C_2 $-group ring $ \underline{\Z}[S^\sigma] $, i.e. the image of $ S^\sigma $ under the unique $C_2 $-colimit-preserving $C_2 $-symmetric monoidal functor $ \underline{\Spc}^{C_2} \to \underline{\Mod}_{\underline{\Z}} $. 
	Then the $ C_2 $-commutative monoid structure on $ S^\sigma $ induces a bicommutative $ C_2 $-bialgebra structure on $ \T^\sigma $. 

	Furthermore, $ \T^\sigma $ is dualizable with dual $ \T^{\sigma,\vee}$. 
	We may also describe $ \T^{\sigma,\vee} $ as the image of $ S^\sigma $ under the unique $ C_2 $-limit preserving $ C_2 $-functor $ \underline{\Z}^{(-)} \colon \underline{\Spc}^{C_2, \vop} \to \underline{\Mod}_{\underline{\Z}} $ sending $ C_2/C_2 $ to $ \underline{\Z} $. 
	By Proposition \ref{prop:param_dual_takes_coalg_to_alg}, the $ C_2 $-$ \E_\infty $ bialgebra structure on $ \T^{\sigma} $ induces a $ C_2 $-$ \E_\infty $ bialgebra structure on $ \T^{\sigma,\vee} $. 
\end{ntn}
\begin{lemma}\label{lemma:inv_circ_cochain_inv_dalg}
	There is a canonical derived involutive bialgebra structure on $ \T^{\sigma,\vee} $ promoting its $ C_2 $-$ \E_\infty $ bialgebra structure, and the derived bicommutative bialgebra structure on $ \T^\vee $ of \cite[Cons. 6.1.2]{Raksit20}. 
\end{lemma}
\begin{proof}
	We may also describe $ \T^{\sigma,\vee} $ as the image of $ S^{\sigma} $ under the finite-$ C_2 $-limit-preserving $C_2$-symmetric monoidal functor $ \underline{\Z}^{(-)} \colon \underline{\Spc}^{C_2,\vop} \to \underline{\DAlg}_{\underline{\Z}}^\sigma $. 
	This description lifts along the forgetful $ C_2 $-functor $ \underline{\DAlg}_{\underline{\Z}}^\sigma \to C_2\E_\infty\underline{\Alg}_{\underline{\Z}} $, so we are done. 
\end{proof}
\begin{prop}
	Let $ \cat $ be an derived involutive algebraic context and let $ A $ be an derived involutive algebra in $ \cat $. 
	Then there is an equivalence of $ C_2 $-$ \infty $-categories
	\begin{equation*}
		\underline{\Fun}_{C_2}\left(BS^\sigma,\underline{\DAlg}_{A}^\sigma\right) \simeq \underline{\coMod}_{\T^{\sigma,\vee}}\underline{\DAlg}_{A}^\sigma  \,,
	\end{equation*} 
	where $ \T^{\sigma,\vee} $ is the derived involutive bialgebra of Lemma \ref{lemma:inv_circ_cochain_inv_dalg} and the $ C_2 $-$ \infty $-category on the right-hand side is from Construction \ref{cons:dalg_in_comodules}. 
	On underlying $ \infty $-categories, this equivalence recovers the equivalence of \cite[Remark 6.1.3]{Raksit20}. 
\end{prop}
\begin{proof}
	Since $ \underline{\DAlg}^\sigma_A $ admits finite $ C_2 $-coproducts (Proposition \ref{prop:invdalg_C2_colims}), it acquires a distributive $ C_2 $-cocartesian $ C_2 $-symmetric monoidal structure \cite[Example 2.4.1]{NS22}. 
	In other words, $ \underline{\DAlg}^\sigma_A $ lifts to an object of $ \underline{\Alg}_{C_2}\left(\Cat_{C_2}^\otimes \right) $. 
	Similarly, by \cite[Corollary 6.0.12]{NS22}, $ \underline{\Fun}_{C_2}\left(BS^\sigma, \underline{\Spc}^{C_2}\right) $ acquires a $ C_2 $-symmetric monoidal structure because it is a $ C_2 $ presheaf category. 
	Finally, $ \underline{\Fun}_{C_2}\left(BS^\sigma,\underline{\DAlg}_A^\sigma\right) $ acquires a $ C_2 $-cocartesian $ C_2 $-symmetric monoidal structure from \cite[Corollary 9.18]{NS22}. 

	By Theorem 11.5 of \cite{Shah18}, the constant $ C_2 $-functor $ \left(BS^\sigma\right)^\vop \to \underline{\coMod}_{\T^{\sigma,\vee}}\underline{\DAlg}_{A}^\sigma $ induces a $ C_2 $-colimit preserving functor $ F \colon \underline{\Fun}\left(BS^\sigma, \underline{\Spc}^{C_2}\right) \to \underline{\coMod}_{\T^{\sigma,\vee}}\underline{\DAlg}_{A}^\sigma $. 
	The map of $ C_2 $-monoids $ S^\sigma \to \{*\} $ induces a map of derived involutive bicommutative bialgebras $ A \to \T^{\sigma, \vee} $. 
	This, in turn, induces a $ C_2 $-colimit preserving $ C_2 $-functor $ G \colon \underline{\DAlg}_{A}^\sigma \to \underline{\coMod}_{\T^{\sigma,\vee}}\underline{\DAlg}_{A}^\sigma $. 

	Now by \cite[Theorem 5.1.4(3)]{NS22}, $ \underline{\Alg}_{C_2}\left(\Cat_{C_2}^\otimes \right) $ has coproducts, hence the $ C_2 $-colimit-preserving functors $ F $ and $ G $ induce a $ C_2 $-functor $ \alpha \colon \underline{\Fun}\left(BS^\sigma, \underline{\Spc}^{C_2}\right) \otimes \underline{\DAlg}^\sigma_A \to \underline{\coMod}_{\T^{\sigma,\vee}}\underline{\DAlg}_{A}^\sigma $ making the diagram 
	\begin{equation*}
	\begin{tikzcd}[row sep=small]
		\underline{\Fun}\left(BS^\sigma, \Spc^{C_2}\right) \otimes \underline{\DAlg}^\sigma_{A} \ar[rr,"f"] \ar[rd] & & \underline{\coMod}_{\T^{\sigma,\vee}}\underline{\DAlg}_{A}^\sigma \ar[ld] \\
		& \underline{\DAlg}_{A}^\sigma & 
	\end{tikzcd} 
	\end{equation*}
	commute. 
	Noticing that each of the forgetful functors is fiberwise monadic such that $ f $ takes free algebras to free algebras, by the Barr--Beck--Lurie theorem \cite[Corollary 4.7.3.16]{LurHA} we conclude that $ f $ is an equivalence. 
\end{proof}
\begin{cons}\label{cons:filteredSsigma}
	We set $ \T^{\sigma,\vee}_{\fil} := \tau_{\geq*}^{\mathrm{Post}}\T^{-\sigma} \in \Fil\left(\Mod_{\underline{\Z}}\right) $ for the \emph{dual $\underline{\Z}$-linear filtered involutive circle} using the Postnikov filtration of Recollection \ref{rec:genuinepostnikov}. 
	We set $ \T^{\sigma}_{\fil} =\left(\T^{\sigma,\vee}_{\fil}\right)^{\vee} $ for the \emph{$\underline{\Z}$-linear filtered involutive circle}. 
\end{cons}
\begin{obs}\label{obs:postnikov_regslice_agreement}
	There is an equivalence $ \tau^\rslice_{\geq *} \underline{\Z}^{S^\sigma} \simeq \tau^{\mathrm{Post}}_{\geq *}\underline{\Z}^{S^\sigma} $ of filtered $ \underline{\Z} $-modules. 
	This follows from Example \ref{ex:spheres_as_slices} and Proposition \ref{prop:graded_fil_Ssigma}. 
\end{obs}
\begin{prop}\label{prop:graded_fil_Ssigma}
	The graded pieces of $ \T^{\sigma,\vee}_\fil $ are given by
	\begin{equation*}
		\gr^{-1} \T^{\sigma,\vee}_\fil = \Sigma^{-\sigma}\underline{\Z}, \qquad \gr^{0} \T^{\sigma,\vee}_\fil = \underline{\Z}, 
	\end{equation*}
	and $ \gr^n \T^{\sigma,\vee}_\fil = 0 $ otherwise. 
\end{prop}
\begin{proof}
	Dualizing the exact sequence $ \underline{\Z}[C_2] \to \underline{\Z}^{\oplus 2} \to \T^\sigma $ gives an exact sequence $ \T^{\sigma,\vee} \to \underline{\Z}^{\oplus 2} \xrightarrow{f} \underline{\Z}[C_2] $, since $ \underline{\Z} $ and $ \underline{\Z}[C_2] $ are self-dual. 
	The morphism $ f $ induces, on Mackey functor homotopy groups $ \underline{\pi}_0 $
	\begin{equation*}
	\begin{tikzcd}
		\Z \oplus \Z \ar[d, bend right=20,"\mathrm{Res}"'] \ar[rr,"{(1, 0), (0,1) \mapsto e+\sigma}"] & & \Z \{e + \sigma \} \ar[d, bend right=20,"\mathrm{Res}"'] \\
		\Z \oplus \Z \ar[u, bend right=20,"\mathrm{Tr}"'] \ar[rr,"{(1, 0), (0,1) \mapsto e+\sigma }"] & & \Z \{e, \sigma\} \ar[u, bend right=20,"\mathrm{Tr}"']  
	\end{tikzcd} .
	\end{equation*}
	Since $ \underline{\Z} $ and $ \underline{\Z}[C_2] $ are discrete, the kernel of $ \underline{\pi}_0 f $ is $ \underline{\pi}_0 \T^{\sigma,\vee} \simeq \underline{\Z} $ and the cokernel is $ \underline{\pi}_{-1} \T^{\sigma,\vee} \simeq \underline{\Z}_{-} $ (see Lemma \ref{lemma:sigma_discrete}). 
\end{proof}
\begin{cor}\label{cor:dualfilteredcircles}
	The filtered object $ \T^{\sigma, \vee}_\fil $ is dualizable, and its dual $ \T^\sigma_{\mathrm{fil}} $ satisfies $ \colim \T^\sigma_{\fil} \simeq \T^\sigma $. 
\end{cor}
\begin{proof}
	By Proposition \ref{prop:graded_fil_Ssigma}, $ \T^{\sigma,\vee}_\fil $ is an extension of an invertible object in $ \Fil\left(\Mod_{\underline{\Z}}\right) $ by another invertible object. 
	Dualizability follows from the fact that dualizable objects in a stable $ \infty $-category are stable under finite limits and colimits. 
	The identification of the colimit of $ \T^\sigma_\fil $ is a straightforward computation. 
\end{proof}
\begin{prop}\label{prop:filtered_Ssigma_bialg}
\begin{enumerate}[label=(\arabic*)]
	\item There exists a canonical $ C_2 $-$ \E_\infty $-bialgebra structure on $ \T^{\sigma, \vee}_{\mathrm{fil}} $ promoting the $ C_2 $-$ \E_\infty $-bialgebra structure on $ \T^{\sigma, \vee} $. 
	\item \label{propitem:filtered_Ssigma_dalg} There exists a unique derived involutive bialgebra structure on $ \T^{\sigma, \vee}_\fil $ promoting its $ C_2 $-$ \E_\infty $-bialgebra structure. 
	\item Under the forgetful functor of Remark \ref{rmk:inv_dalg_underlying}, we have canonical identifications $ (\T^{\sigma,\vee})^e \simeq \T^{\vee} $ of derived bicommutative bialgebras over $ \Z $ and $ (\T^{\sigma, \vee}_\fil )^e = \T^{\vee}_\fil $ of bicommutative bialgebras in filtered $ \Z $-modules, where $ \T^\vee $ and $ \T^\vee_\fil $ are from \cite[Theorem 6.1.6]{Raksit20}.  
\end{enumerate}
\end{prop} 
\begin{proof}
\begin{enumerate}[label=(\arabic*)]
	\item Because the regular slice filtration functor is lax $ C_2 $-symmetric monoidal (Proposition \ref{rmk:regslicelaxsymmmon}), $ \T^{\sigma,\vee}_\fil $ naturally inherits a $ C_2 $-$ \E_\infty $ algebra structure from $ \T^{\sigma,\vee} $. 
	In fact, the functor $ \tau_{\geq *}^\rslice $ is strict $C_2$-symmetric monoidal on tensor powers of $ \Z^{S^{\sigma}} $ by Example \ref{ex:spheres_as_slices}. 
	Then $ \T^{\sigma, \vee}_\fil $ inherits a $ C_2 $-$ \E_\infty $-bialgebra from $ \T^{\sigma,\vee} $ structure via Corollary \ref{cor:mod_comod_equivalence}. 
	Uniqueness of this $ C_2 $-$ \E_\infty $-algebra structure follows from, Corollary \ref{cor:colim_reg_slicefil}, which implies that $ \tau_{\geq *}^\rslice $ is fully faithful. 

	\item By Observation \ref{obs:postnikov_regslice_agreement}, there is an equivalence $ \tau^\rslice_{\geq *} \underline{\Z}^{S^\sigma} \simeq \tau^{\mathrm{Post}}_{\geq *}\underline{\Z}^{S^\sigma} $. 
	It suffices to observe that $ \T^{\sigma, \vee}_\fil $ is in the essential image of the functor $ \tau_{\geq *}^{\mathrm{Post}} $ of Proposition \ref{prop:nonconn_dalg_Post_filt}. 

	\item The first statement follows from the fact that the underlying diagram of a parametrized (co)limit diagram is an ordinary (co)limit, and that the functor $ \Mod_{\underline{\Z}} \to \Mod_{\Z} $ is symmetric monoidal, hence it preserves duals. 
	The second statement follows from the description of the regular slice filtration (see Lemma \ref{lemma:reg_slice_conn_conditions}). 
	 \qedhere
\end{enumerate}
\end{proof}
\begin{rmk}
	Note the use of two different filtrations on the $ C_2 $-$ \infty $-category $ \underline{\Mod}_{\underline{\Z}} $ in the proof of Proposition \ref{prop:filtered_Ssigma_bialg}. 
	We use the \emph{Postnikov} filtration to show that $ \T^{\sigma, \vee}_\fil $ admits a(n involutive) \emph{derived} algebra structure, while the $ C_2 $-$ \E_\infty $ bialgebra structure on $ \T^{\sigma, \vee}_\fil $ follows from how the \emph{regular slice filtration} interacts nicely with the cell structure on $ \T^{\sigma, \vee} $. 
\end{rmk}
\begin{ntn}
		Let $ \cat $ be a $ \underline{\Z} $-linear $ C_2 $-stable $ C_2 $-presentable $ C_2 $-symmetric monoidal $ C_2 $-$ \infty $-category. 
		The structure map $ \underline{\Mod}_{\underline{\Z}} \to \cat $ induces a $ C_2 $-symmetric monoidal functor $ \Fil\left(\underline{\Mod}_{\underline{\Z}}\right) \to \Fil\left(\cat\right) $, hence we may regard $ \T_\fil^\sigma $ and $ \T_\fil^{\sigma,\vee} $ as dualizable $ C_2 $-$ \E_\infty $-bialgebras in $ \Fil(\cat) $. 

		Now suppose $ \cat $ is additionally equipped with the structure of a derived involutive algebraic context. 
		Then by Proposition \ref{prop:filtered_Ssigma_bialg}\ref{propitem:filtered_Ssigma_dalg}, we may regard $ \T_\fil^{\sigma,\vee} $ as a derived involutive bialgebra in $ \Fil(\cat) $. 
\end{ntn}
\begin{defn}\label{defn:fil_mod_fil_signedS_action}
	Write $ \Fil_{S^\sigma}\left(\underline{\Mod}_{A}\right) := \underline{\Mod}_{\T^\sigma_\fil}\left(\Fil(\underline{\Mod}_{A})\right) $ and call this the $ C_2 $-$ \infty $-category of \emph{filtered $ A $-modules with filtered $ S^\sigma $-action}. 

	By Variant \ref{variant:module_pointwise_param_tensor}, we may regard $ \Fil_{S^\sigma}\left(\underline{\Mod}_{A}\right) $ as a $ C_2 $-symmetric monoidal $ C_2 $-$ \infty $-category. 
	We will call $C_2$-$\E_\infty$-algebra objects therein \emph{filtered $ C_2 $-$\E_\infty $-$ A $-algebras with filtered $ S^\sigma $-action}. 
\end{defn}
We are allowed to make the following definition by Proposition \ref{prop:filtered_Ssigma_bialg}\ref{propitem:filtered_Ssigma_dalg} and Construction \ref{cons:dalg_in_comodules}.
\begin{defn}
	Write $ \Fil_{S^\sigma}\underline{\DAlg}_{A}^\sigma $ for the $ C_2 $-$ \infty $-category $ \underline{\coMod}_{\T^{\sigma,\vee}_{\mathrm{fil}}}\left(\Fil\left(\underline{\DAlg}_{A}^\sigma\right)\right) $, and let us call $ C_2 $-objects of this $ C_2 $-$ \infty $-category \emph{filtered derived involutive $ A $-algebras with filtered $ S^\sigma $-action}. 
\end{defn}
\begin{rmk}\label{rmk:filteredSsigma_Ssigma_adjunction}
	Under the $ C_2 $-adjunction $ \colim \dashv \delta $ of Recollection \ref{rec:splitfiltobj}, there is a canonical equivalence of $ C_2 $-$ \E_\infty $-bialgebras $ \underline{\colim} \T^\sigma_{\mathrm{fil}} \simeq \T^\sigma $. 
	It follows that there is a $ C_2 $-symmetric monoidal $ C_2 $-adjunction
	\begin{equation*}
		\colim \colon \Mod_{\T^\sigma_\fil}\Fil\left(\underline{\Mod}_{A}\right) \rlarrows \underline{\Fun}\left(BS^\sigma, \underline{\Mod}_{A}\right) \colon c \,.
	\end{equation*}
	By Proposition \ref{prop:filtered_Ssigma_bialg}\ref{propitem:filtered_Ssigma_dalg}, since the equivalence $ \colim\left(\T^\sigma_\fil\right) \simeq \T^\sigma $ canonically promotes to one of derived involutive bialgebras, there is an induced $ C_2 $-adjunction
	\begin{equation*}
		\colim \colon  \Fil_{S^\sigma}\underline{\DAlg}_{A}^\sigma \rlarrows \underline{\Fun}\left(BS^\sigma, \underline{\DAlg}^\sigma_{A}\right) \simeq \underline{\DAlg}^\sigma\left(\underline{\Mod}_{S^\sigma}(\underline{\Mod}_{A})\right) \colon c .
	\end{equation*}
\end{rmk}
\begin{rmk}\label{rmk:filteredSsigma_gr_adjunction}
	Under the $ C_2 $-adjunction $ \gr^* \dashv \zeta $ of (\ref{eq:param_assoc_graded}), there is an equivalence of $ C_2 $-$ \E_\infty $-bialgebras $ \gr(\T^{\sigma,\vee}_\fil) \simeq \D^{\sigma, \vee}_+ $. 
	This induces a $ C_2 $-symmetric monoidal $ C_2 $-adjunction
	\begin{equation*}
		\gr \colon \underline{\Mod}_{\T^\sigma_\fil}\Fil\left(\underline{\Mod}_{A}\right) \rlarrows \Gr\left( \underline{\Mod}_{A}\right) \colon \zeta .
	\end{equation*}
	By Proposition \ref{prop:Dsigma_alg}, there is an equivalence $ \gr(\T^{\sigma,\vee}_\fil) \simeq \D^{\sigma, \vee}_+ $ of graded derived involutive bialgebras. 
	Hence by Remark \ref{rmk:dalg_functoriality_in_comodules}, there is an induced $ C_2 $-adjunction
	\begin{equation*}
		\gr \colon \Fil_{S^\sigma}\underline{\DAlg}_{A}^\sigma \rlarrows \DG_+^\sigma \underline{\DAlg}^\sigma_{A} \simeq \underline{\DAlg}^\sigma\left(\underline{\coMod}_{\D^{\sigma,\vee}_+}(\underline{\Mod}_{A})\right) \colon \zeta .
	\end{equation*}
	In words, this asserts that the associated graded of a filtered derived commutative $ A $-algebra with filtered $ S^\sigma $-action is an involutive $ h^\sigma_+ $-differential graded $ A $-algebra of Definition \ref{defn:dg_involutive_alg}. 
\end{rmk}

\subsection{HK\texorpdfstring{$\R$}{R}-filtered real Hochschild homology} \label{subsection:realHKR} 
In this section, we show how the results of \S\ref{subsection:fil_inv_circ} can be used to define a filtration on real Hochschild homology which recovers the ordinary HKR filtration on underlying $ \Z $-modules. 
To define a filtration on real Hochschild homology, we replace $ S^\sigma $-actions by filtered $ S^\sigma $-actions.  
Let $ \cat $ denote a fixed derived involutive algebraic context and let $ A $ denote a derived involutive algebra in $ \cat $. 
\begin{defn}
	We define the $ C_2 $-$ \infty $-category of \emph{nonnegative filtered derived involutive $ A $-algebras with filtered $ S^\sigma $-action} to be the pullback 
	\begin{equation*}
		\Fil^{\geq 0}_{S^{\sigma, \mathrm{fil}}}\underline{\DAlg}_{A}^\sigma := \Fil_{S^{\sigma, \mathrm{fil}}}\underline{\DAlg}_{A}^\sigma \times_{\Fil\underline{\DAlg}_{A}^\sigma} \Fil^{\geq 0}\underline{\DAlg}_{A}^\sigma .
	\end{equation*} 
\end{defn} 
In the language of Remark \ref{rmk:param_unstraighten}, the underlying $ \infty $-category of $ \Fil^{\geq 0}_{S^{\sigma, \mathrm{fil}}}\underline{\DAlg}_{A}^\sigma $ may be identified with the $ \infty $-category of nonnegative filtered derived $ A $-algebras with filtered $ S^1 $-action of \cite[Notation 6.2.3]{Raksit20}. 
\begin{rmk}
	On nonnegative filtered objects $ X \in \Fil^{\geq 0}(\cat) $, the colimit is computed by $ |X| \simeq \ev_0X $. 
\end{rmk}
\begin{prop}\label{prop:existence_freefilteredHR}
	The forgetful $ C_2 $-functor $ \Fil^{\geq 0}_{S^{\sigma, \mathrm{fil}}}\underline{\DAlg}_{A}^\sigma \to \Fil^{\geq 0}\underline{\DAlg}_{A}^\sigma \xrightarrow{\ev^0} \underline{\DAlg}_{A}^\sigma $ admits a left $ C_2 $-adjoint. 
	On underlying $ \infty $-categories, this recovers the adjunction of \cite[Proposition 6.2.4]{Raksit20}. 
\end{prop}
\begin{proof}
	Similar to the proof of Proposition \ref{prop:inv_deRham_asadjoint}. 
\end{proof}
\begin{defn}\label{defn:filtered_real_HH}
		Denote the left $C_2$-adjoint of Proposition \ref{prop:existence_freefilteredHR} by $ \HR_\fil(-/A) $. 
		If $ B $ is a derived involutive $ A $-algebra, we will refer to $ \HR_\fil(B/A) $ as the \emph{HKR-filtered real Hochschild homology of $ B $ over $ A $}. 
\end{defn}
\begin{rmk}\label{rmk:filtered_real_HH_underlying}
		By construction, the underlying filtered derived $ k^e $-algebra with filtered $ S^1 $-action associated to $ \HR_{\fil}(B/A) $ is the HKR-filtered Hochschild homology $ \HH_{\fil}(B^e/A^e ) $ of \cite[Definition 6.2.5]{Raksit20}. 
\end{rmk}
The following result shows that filtered real Hochschild homology interpolates between real Hochschild homology and the involutive de Rham complex. 
\begin{theorem}\label{thm:gr_of_filtered_HR}
	Let $ \cat $ denote a fixed derived involutive algebraic context and let $ A $ denote a derived involutive algebra in $ \cat $. 	
	Then there are canonical equivalences
	\begin{align*}
		\ev^0\HR_\fil(B/A) &\simeq \HR(B/A) \qquad \text{in } \Mod_{S^\sigma}\DAlg_{A}^\sigma \\
		\gr^\bullet \HR_\fil(B/A) &\simeq \L\Omega^{\sigma,\bullet}_{B/A} \qquad \text{ in } \DG^\sigma_+\DAlg_{A} 
	\end{align*}
	which are natural in $ B $. 
	On underlying $ \infty $-categories, these natural equivalences recover those of \cite[Theorem 6.2.6]{Raksit20}. 
\end{theorem}
\begin{proof}
	We form the following diagram of $ C_2 $-adjoint pairs 
	\begin{equation}\label{diagram:gr_of_filtered_HR}
	\begin{tikzcd}[column sep=huge,row sep=huge]
		\underline{\Fun}\left(BS^\sigma,\underline{\DAlg}_{A}^\sigma\right) \ar[r,shift left=1,"c"] \ar[rd, shift left=1,"U"] & \Fil^{\geq 0}_{S^{\sigma, \mathrm{fil}}}\underline{\DAlg}_{A}^\sigma \ar[d,shift left=1,"\ev_0"] \ar[l, shift left=1,"\ev_0"] \ar[r, shift left=1,"\gr"] & \DG^{\sigma,\geq 0}_+\underline{\DAlg}_{A}^\sigma \ar[l,shift left=1,"\zeta"] \ar[ld, shift left=1,"\ev_0", out=240,in=0] \\
		& \underline{\DAlg}_{A}^\sigma \ar[lu, shift left=1,"\HR(-/A)"] \ar[ru, shift left=1,"\L\Omega^{\sigma ,\bullet}_{-/A}", in=240,out=0] \ar[u, shift left=1,"\HR_\fil"]
	\end{tikzcd}
	\end{equation}
	where the upper left adjoint pair is from Remark \ref{rmk:filteredSsigma_Ssigma_adjunction}, the upper right adjoint pair is from Remark \ref{rmk:filteredSsigma_gr_adjunction}, the left diagonal adjoint pair is that of Definition \ref{defn:realHH}, and the right diagonal adjoint pair is from Definition \ref{defn:inv_deriveddeRhamcplx}. 
	The right $C_2$-adjoints commute, from which it follows that the left $ C_2 $-adjoints commute. 
	The statement about underlying $ \infty $-categories follows from the fact that the diagram (\ref{diagram:gr_of_filtered_HR}) specializes to the diagram considered in \cite[Theorem 6.2.6]{Raksit20} by definition. 
\end{proof}
\begin{prop}
	Suppose $ B = \LSym^\sigma_{A} (M) $ for some $ A $-module $ M $. 
	Then the filtered object $ \HR_\fil (B/A) $ is split.
\end{prop}
\begin{proof}
		In view of Example \ref{ex:cotangent_computations}, Theorem \ref{thm:invdeRham_equals_LSymcotangent}, and Theorem \ref{thm:real_hkr}, the result follows from a similar argument to that of \cite[Lemma 6.2.7]{Raksit20}. 
\end{proof}

\subsection{Filtered orbits, fixed points, Tate construction}\label{subsection:filteredHCHPetc}
In this subsection, we apply the constructions of \S\ref{subsection:param_bialg} to the setting of filtered $ S^\sigma $-actions. 
When applied to real Hochschild homology, we will obtain filtrations on negative real cyclic homology, real cyclic homology, and real periodic cyclic homology. 
\begin{obs}  
		Let $ \T^\sigma $ denote the $ \underline{\Z} $-linear involutive circle of Notation \ref{ntn:Z_linear_circle}. 
		There is an equivalence of $ \T^\sigma $-modules $ \Sigma^{-\sigma} \T^\sigma \simeq \T^{-\sigma} $. 
		Taking regular slice filtrations on both sides, we obtain an equivalence of $ \T_\fil^\sigma $-modules $ \T^\sigma_\fil \simeq \Sigma^{\sigma}\underline{\Z}(1) \otimes \T^{-\sigma}_\fil $ promoting the equivalence of \cite[Remark 6.3.1]{Raksit20}. 
\end{obs}
Thus by Proposition \ref{prop:tate_cons_is_param_compatibility}, we may take the parametrized Tate construction with respect to $ \mathbb{T}^\sigma_\fil $. 
\begin{defn}\label{defn:filtered_real_trace_theories} 
	Let $ \cat $ denote a fixed derived involutive algebraic context, let $ A $ denote a derived involutive algebra in $ \cat $, and suppose $ B $ is a derived involutive $ A $-algebra. 
	Then we define \emph{filtered real cyclic homology, real negative cyclic homology, and real periodic cyclic homology of $ B $ over $ A $} to be
		\begin{equation*}
			\HC_\fil(B/A) = \HR_\fil(B/A)_{\mathbb{T}^\sigma_\fil} \qquad  \HC^{-}_\fil(B/A) = \HR_\fil(B/A)^{\mathbb{T}^\sigma_\fil} \qquad  \HP_\fil(B/A) = \HR_\fil(B/A)^{t\mathbb{T}^\sigma_\fil}
		\end{equation*}
		where $ \HR_\fil(B/A) $ is the filtered real Hochschild homology of $ A $ of Definition \ref{defn:filtered_real_HH}. 
\end{defn} 
The following is an immediate consequence of working in the parametrized setting.  
\begin{prop}
	Taking $ (-)^e $ in Definition \ref{defn:filtered_real_trace_theories} recovers the filtered trace theories in Example 6.3.8 of \cite{Raksit20}. 
\end{prop}
\begin{lemma}\label{lemma:gr_of_filtered_orbits_fixpt_Tate}
		Let $ \cat $ be a $C_2$-stable $ C_2 $-presentable $ \underline{\Z} $-linear $ C_2 $-symmetric monoidal $ C_2 $-$ \infty $-category. 
		Then for $ X \in \Fil_{S^\sigma}(\cat) $, there are canonical natural equivalences
		\begin{align*}
				\gr\left(X_{\mathbb{T}^\sigma_\fil}\right) &\simeq \mathrm{und}\left(|\gr(X) |^{\leq *}\right)[\rho *] \\
				\gr\left(X^{\mathbb{T}^\sigma_\fil}\right) &\simeq \mathrm{und}\left(|\gr(X) |^{\geq *}\right)[\rho *] \\
				\gr\left(X^{t\mathbb{T}^\sigma_\fil}\right) &\simeq \delta\left(|\gr(X) |\right)[\rho *] 
		\end{align*} 
		in $ \Gr(\cat) $. 
\end{lemma}
\begin{proof}
		Since $ \gr $ is $ C_2 $-symmetric monoidal by Proposition \ref{prop:param_assoc_gr_is_C2_monoidal}, we have canonical natural equivalences 
		\begin{equation*}
			\gr\left(X_{\mathbb{T}^\sigma_\fil}\right)\simeq X_{\mathbb{D}^\sigma_+} \qquad \gr\left(X^{\mathbb{T}^\sigma_\fil}\right)\simeq X^{\mathbb{D}^\sigma_+} \qquad \gr\left(X^{t\mathbb{T}^\sigma_\fil}\right)\simeq X^{t\mathbb{D}^\sigma_+}  	\,.
		\end{equation*}
		The result now follows from Proposition \ref{prop:param_coh_formula}.
\end{proof}
\begin{prop}\label{prop:colim_of_fil_fix_orb_Tate}
		Let $ \cat $ be a $ \underline{\Z} $-linear distributive $ C_2 $-symmetric monoidal $ C_2 $-$ \infty $-category and let $ X \in \Fil_{S^\sigma}(\cat) $. 
		Then there is a canonical equivalence $ \colim X_{\T^{\sigma,\fil}} \simeq \left(\colim X \right)_{\T^\sigma} $ and there are canonical maps
		\begin{equation*}
				\colim \left(X^{\T^{\sigma,\fil}}\right) \to \left(\colim X\right)^{\T^\sigma} \qquad \qquad \colim \left(X^{t\T^{\sigma,\fil}}\right) \to \left(\colim X\right)^{t\T^\sigma }\,.
		\end{equation*}
\end{prop}
\begin{proof} 
		Follows from Remark \ref{rmk:param_tate_naturality}. 
\end{proof}
\begin{rmk} 
		Let $ \underline{k} $ be the fixed point $ C_2 $-Green functor associated to a commutative ring with an involution, and let $ A $ be a derived involutive algebra over $ \underline{k} $.  
		There is a canonical equivalence
		$ \colim \HC_\fil(A/\underline{k}) \simeq \HC(A/\underline{k}) $ and there are canonical maps
		\begin{equation*}
			\colim \HC^{-}_\fil(A/\underline{k}) \to \HC^{-}_\fil(A/\underline{k}) \qquad \qquad \colim \HP_\fil(A/\underline{k}) \to \HP(A/\underline{k}) \,.
		\end{equation*} 
\end{rmk}

\subsection{Computations \& comparisons}\label{subsection:computations_comparisons}
The presentation of the real HKR filtration given here is rather anachronistic--classically, one can prove the Hochschild--Kostant--Rosenberg theorem by simply computing the homotopy groups of Hochschild homology on free polynomial algebras, observing that in degree $ i $, they are given by differential $ i $-forms (hence one can use the Postnikov filtration), and concluding for smooth algebras via an étale base change argument. 
Here we sketch an approach to a `real Hochschild--Kostant--Rosenberg theorem' in the spirit of \cite[\S3.2]{MR1600246}.  
Our results show that identifying an HKR-style filtration on real Hochschild homology is more subtle than in the non-equivariant case. 
Furthermore, we show how the theory simplifies when we work over a base on which $ 2 $ is invertible, and provide a dictionary for putting our results in the context of classical results such as \cite{MR1384461,MR972090,SVP96,MR917736}. 
\begin{prop}\label{prop:resolution_by_norms}
	Let $ \underline{k} $ be the constant $ C_2 $-Mackey functor associated to a commutative ring and write $ \underline{N}^{C_2} $ for the relative norm (Definition \ref{defn:relativenorm}). Then
	\begin{enumerate}
		\item There is a resolution of $ \underline{k}[x] $ as an $ \underline{N}^{C_2}\underline{k}[x] $-module
		\begin{equation*}
			\Sigma^{\sigma-1} \underline{N}^{C_2}\underline{k}[x] \xrightarrow{d_0} \underline{N}^{C_2}\underline{k}[x] \to \underline{k}[x] .
		\end{equation*}
		Moreover, the base change along the augmentation $ d_0 \otimes_{\underline{N}^{C_2}\underline{k}[x]} \underline{k}[x] $ is the zero map. 
		\item There is a resolution of $ \underline{k}[x,x_\sigma] $ as an $ \underline{N}^{C_2} \underline{k}[x,x_\sigma] $-module
		\begin{equation*}
			\Sigma^{\sigma-1}\underline{N}^{C_2} \underline{k}[x,x_\sigma] \xrightarrow{d_1'} \underline{N}^{C_2} \underline{k}[x,x_\sigma] \otimes C_2 \xrightarrow{d_0'} \underline{N}^{C_2} \underline{k}[x,x_\sigma] \to \underline{k}[x,x_\sigma] .
		\end{equation*} 
			Moreover, the base change along the augmentation $ d_i' \otimes_{\underline{N}^{C_2}\underline{k}[x,x_\sigma]} \underline{k}[x,x_\sigma] $ is the zero map for $ i = 0, 1$. 
	\end{enumerate}
\end{prop}
The following corollary is \cite[Lemma 4.27 \& Proposition 4.35]{PHrealTHH_perfectoid}.
\begin{cor}\label{cor:gr_HR_via_resolutions}
	The real Hochschild homology of $ \underline{k}[x] $ and $ \underline{k}[x,x_\sigma] $ admit complete descending exhaustive filtrations with associated graded 
	\begin{align*}
		\gr^i \HR(\underline{k}[x]/\underline{k}) \simeq 
			\begin{cases}
				\Sigma^\sigma \underline{k}[x] & \text{ if } i=1 \\
				\underline{k}[x] & \text{ if } i=0 \\
				0 & \text{ otherwise }
			\end{cases}
		\qquad
		\gr^i \HR(\underline{k}[x,x_\sigma]/\underline{k}) \simeq 
			\begin{cases}
				\Sigma^{\sigma+1} \underline{k}[x,x_\sigma] & \text{ if } i=2 \\
				\Sigma \underline{k}[x,x_\sigma] \otimes C_2 & \text{ if } i=1 \\
				\underline{k}[x,x_\sigma] & \text{ if } i=0 \\
				0 & \text{ otherwise }
			\end{cases}.
	\end{align*}
\end{cor}
\begin{rmk}
		When $ k $ is endowed with the trivial action, the filtration on $ \HR(\underline{k}[x]) $ in Corollary \ref{cor:gr_HR_via_resolutions} manifestly agrees with that of \cite[Lemma 4.27]{PHrealTHH_perfectoid}. 
		We expect the filtration on $ \HR(\underline{k}[x,x_\sigma]/\underline{k}) $ from Corollary \ref{cor:gr_HR_via_resolutions} to agree with that of Theorem \ref{thm:real_hkr}, but defer this to future work. 
\end{rmk}
\begin{proof} [Proof of Corollary \ref{cor:gr_HR_via_resolutions}] 
	This follows from \cite[Proposition 3.16]{Ariotta} and \ref{prop:resolution_by_norms}. 
\end{proof}
\begin{proof}
	[Proof of Proposition \ref{prop:resolution_by_norms}]
	We begin with the first case. 
	Note that $ \underline{N}^{C_2}\underline{k}[x] \simeq \underline{k}[x,x_\sigma] $.  
	Write $ {\varepsilon \colon \underline{k}[x,x_\sigma] \to \underline{k}[x]} $ for the structure map. 
	Observe that we have a map $ \varphi: \underline{k} \to \underline{k}[x] $ represented by the element $ x \in \pi_0 \underline{k}[x]^{C_2} $. 
	Since the composite $ \underline{k}[C_2] \to \underline{k} \xrightarrow{\varphi}  \underline{k}[x] $ factors through $ \varepsilon $, $ \varphi $ descends to a map $ \overline{\varphi} $: 
	\begin{equation*}
	\begin{tikzcd}
		\underline{k}[x,x_\sigma] \ar[r,"\varepsilon"] & \underline{k}[x] \ar[r] & \cofib(\varepsilon) \\
		\underline{k}[C_2] \ar[r] \ar[u,dotted,"{1 \mapsto x}"] &\underline{k} \ar[u,"{1 \mapsto x}"] \ar[r] & \Sigma^\sigma \underline{k} \ar[u,dotted,"\exists","{\overline{\varphi}}"'] 
	\end{tikzcd}.
	\end{equation*}
	In particular, note that the morphism $ {\overline{\varphi}} $ induces
	\begin{align*}
		\pi_1^e{\overline{\varphi}} \colon \pi_1^e \Sigma^\sigma \underline{k} \simeq \pi_0^e \Sigma^{\sigma -1} \underline{k} \simeq k \to \pi_0^e \fib(\varepsilon) \\
			1 \mapsto x-x_\sigma .
	\end{align*}
	Taking the desuspension and noting that $ \varepsilon $ is a map of $ \underline{k}[x,x_\sigma] $-modules gives a sequence 
	\begin{equation*}
		\Sigma^{\sigma-1}\underline{k}[x,x_\sigma] \to \underline{k}[x,x_\sigma] \xrightarrow{\varepsilon} \underline{k}[x] .
	\end{equation*} 
	A computation on homotopy groups of both underlying and geometric fixed points shows that this is indeed an exact sequence of $ C_2 $-spectra. 
	Finally, the final statement follows from the fact that $ \varepsilon (x-x_\sigma) = 0 $. 

	Now we address the second case. 
	Write $ \varepsilon $ for the augmentation $ \underline{k}[x,x_\sigma,y, y_\sigma] \simeq \underline{N}^{C_2}\underline{k}[x,x_\sigma] \to \underline{k}[x,x_\sigma] $. 
	On underlying $k$-modules, the map $ \varepsilon $ satisfies $ \varepsilon(x) = x = \varepsilon(y) $. 
	The morphism
	\begin{align*}
		g \colon &\underline{k}[x,x_\sigma,y, y_\sigma] \otimes C_2 \to \underline{k}[x,x_\sigma, y, y_\sigma] \\
		& 1 \otimes e \mapsto (x-y) \\
		& 1 \otimes \sigma \mapsto (x_\sigma - y_\sigma)
	\end{align*}
	of $ \underline{k}[x,x_\sigma,y, y_\sigma] $-modules clearly factors through the fiber of $ \varepsilon $. 
	Now we notice that the morphism $ \underline{k} \to \underline{k}[x,x_\sigma,y, y_\sigma] $ represented by the element $ xx_\sigma - y y_\sigma \in \underline{k}[x,x_\sigma,y, y_\sigma]^{C_2} $ fits into a commutative diagram
	\begin{equation*}
	\begin{tikzcd}
		\underline{k}[x,x_\sigma,y, y_\sigma] \otimes C_2 \ar[r,"g"] & \underline{k}[x,x_\sigma, y, y_\sigma] \ar[r] & \cofib (g) \\
		\underline{k}[C_2] \ar[r] \ar[u,dotted,"\psi"] &\underline{k} \ar[u,"{1 \mapsto x_N - y_N}"] \ar[r] & \Sigma^\sigma \underline{k} \ar[u,dotted,"\exists","{\overline{\psi}}"'] 
	\end{tikzcd}
	\end{equation*}
	where $ \psi $ is defined by
	\begin{align*}
		\psi(e) &= x_\sigma e + y \sigma \\
		\psi(\sigma) &= y_\sigma e + x \sigma 
	\end{align*}
	and the left-hand square commutes because of the identities
	\begin{align*}
		xx_\sigma - yy_\sigma &= x_\sigma (x-y) + y(x_\sigma - y_\sigma) \\
						&=y_\sigma(x-y) + x(x_\sigma - y_\sigma) .
	\end{align*} 
	In particular, the morphism $ \pi_0^e \Sigma^{\sigma-1}\underline{k} \simeq k \to \fib (g) \to (\underline{k}[x,x_\sigma,y, y_\sigma] \otimes C_2)^e $ takes
	\begin{align*}
		\overline{\psi}(1) = \psi(e-\sigma) = x_\sigma e + y \sigma - (y_\sigma e + x\sigma) = (x_\sigma - y_\sigma)e - (x-y)\sigma
	\end{align*}
	which generates the kernel of $ \pi_0^e g $. 
	Since $ g $ is a morphism of $ \underline{k}[x,x_\sigma, y, y_\sigma] $-modules, $ \overline{\psi} $ extends to a morphism $ \Sigma^{\sigma-1}\underline{k}[x,x_\sigma, y, y_\sigma] \to \underline{k}[x,x_\sigma,y, y_\sigma] \otimes C_2 $. 
	Again a computation on homotopy groups of both underlying and geometric fixed points shows that this is indeed an exact sequence of $ C_2 $-spectra. 
	The final statement in the proposition follows from similar considerations as in the proof of the first half of this proposition. 
\end{proof}
Many existing results and computations for real trace theories assume that $ \frac{1}{2} \in k $. 
We discuss how the study of $ C_2 $-actions simplifies under the assumption that $ \frac{1}{2} \in k $ and relate our results to classical results under these assumptions.  
\begin{recollection}\label{rec:classical_HR_char_neq_2_splitting}
		Let $ k $ be a discrete commutative ring with $ \frac{1}{2} \in k $. 
		Let $ A $ be a commutative ring with involution over $ k $, or a cdga with involution over $ k $. 
		Then the Hochschild complex $ B^\cyc A $ of $ A $ admits an involution \cites[Proposition 5.2.3]{MR1600246}{MR972090} and the complex splits $ B^\cyc A \simeq B^{\cyc,+}A \oplus B^{\cyc,-}A $. 
		This splitting induces a splitting of Hochschild homology groups
		\begin{equation*}
				\HH_*(A/k) \simeq \HH_*^+(A/k) \oplus \HH_*^{-}(A/k)
		\end{equation*}
		where the generator $ \sigma \in C_2 $ acts by $ +1 $ on $ \HH_*^+(A/k) $ and by $ -1 $ on $ \HH_*^-(A/k) $. 
		Now the involution on $ B^\cyc A $ lifts to the cyclic bicomplex of $ A $, inducing a splitting of the cyclic bicomplex and hence a splitting of its homology, the cyclic homology of $ A $ over $ k $ \cite[(5.2.7.1)]{MR1600246}
		\begin{equation*}
				\HC_*(A/k) \simeq \HD_*(A/k) \oplus \HD_*'(A/k)
		\end{equation*}
		where $ \HD_*(A/k) $ ($ \HD_*'(A/k)$) is called the \emph{(skew) dihedral homology of $ A $ over $ k $}.
\end{recollection}
\begin{lemma}\label{lemma:char_neq_2_C2_action_splits}
		Let $ X \in \mathcal{D}^{BC_2} $ be an object with $ C_2 $-action (in an ordinary $ \infty $-category $ \mathcal{D} $), and suppose $ \mathcal{D} $ is stable and idempotent-complete. 
		Assume that $ \frac{1}{2} \in \pi_0\End_{\mathcal{D}}(X) $. 
		Then there is a canonical splitting $ X \simeq X^+ \oplus X^{-} $ in $ \mathcal{D}^{BC_2} $ so that for all $ a \in \pi_*X^+ $, the generator $ \sigma $ of $ C_2 $ acts by $ \sigma(a) = a $ and for all $ b \in \pi_*X^{-} $, $ \sigma(b) = -b $.  

		Moreover, the splitting is natural, i.e. if $ Y \in \left(\mathcal{D}^{BC_2}\right)_{X/} $, then the map $ f \colon X \to Y $ decomposes as a direct sum $ f = f^+ \oplus f^{-} $ where $ f^+ \colon X^+ \to Y^+ $ and $ f^- \colon X^{-} \to Y^{-} $. 
\end{lemma}
\begin{proof}
		Since $ \mathcal{D} $ is assumed to be stable, an idempotent in the homotopy category of $ \mathcal{D}$ lifts to $ \mathcal{D} $ by the proof of Remark 1.2.4.6 (also see Warning 1.2.4.8) of \cite{LurHA}, and we do not need to verify \cite[Prop 4.4.5.20]{LurHTT}. 
		Note that $ e:= \frac{1}{2}(1+\sigma) $ is a $ C_2 $-equivariant idempotent endomorphism of $ X $. 
		We define $ X^+ $ to be the equalizer of $ (e, \id_X) $. 
\end{proof}
\begin{prop}\label{prop:HR_fixpt_as_positive_HH}
		Let $ k $ be a commutative ring with $ \frac{1}{2} \in k $ and let $ A $ be a connective derived involutive algebra over $ k $. 
		Then there is an isomorphism of graded abelian groups 
		\begin{equation}
			 \pi_*\left(\HR(\underline{A}/\underline{k})^{C_2}\right) \simeq \HH_*^+(A^e/k)
		\end{equation}
		where the right-hand side is from Recollection \ref{rec:classical_HR_char_neq_2_splitting}. 
\end{prop}
\begin{proof}
		Since $ \frac{1}{2} \in k $, $ \underline{k} $ and $ A $ are Borel $ C_2 $-spectra--in particular, their geometric fixed points vanish. 
		It follows that $ \HR(\underline{A}/\underline{k}) $ is Borel, so we have that $ \HR(\underline{A}/\underline{k})^{C_2} \simeq \left(\HR(\underline{A}/\underline{k})^{e}\right)^{hC_2} \simeq \HH(A^e/k)^{hC_2} $. 
		The result follows from the observation that when $ \mathcal{D} = \Spectra $ in Lemma \ref{lemma:char_neq_2_C2_action_splits}, the homotopy fixed points of $ X $ satisfy $ X^{hC_2} \simeq (X^+)^{hC_2} \oplus (X^-)^{hC_2} \simeq X^+ $, where $ (X^-)^{hC_2} = 0 $ by the homotopy fixed point spectral sequence. 
\end{proof}
\begin{prop}
		Let $ k $ be a commutative ring with $ \frac{1}{2} \in k $ and let $ A $ be a connective derived involutive algebra over $ k $. 
		Then there is an isomorphism
		\begin{equation}
			 \pi_*\left(\HCR(A/\underline{k})^{C_2}\right) \simeq \HD_*(A^e/k)
		\end{equation}
		where the right-hand side is the dihedral homology of Recollection \ref{rec:classical_HR_char_neq_2_splitting}. 
\end{prop}
\begin{proof}
		As in the proof of Proposition \ref{prop:HR_fixpt_as_positive_HH}, our assumptions imply that $ \HCR(A/\underline{k}) $ is a Borel $ \underline{k} $-module. 
		Now, we have an equivalence $ \HCR(A/\underline{k}) \simeq \HR(A/\underline{k})_{hS^\sigma} $. 
		Endow the universal complex line bundle $ E $ on $ \C \mathrm{P}^\infty $ with a $ C_2 $-equivariant metric. 
		Then the unit disk bundle $ D(E) $ includes into the unit sphere bundle $ S(E) \simeq ES^1 $, and they are classified by $ C_2 $-functors $ f, g \colon BS^\sigma \to \Spc^{C_2} $, respectively, and the inclusion $ S(E) \subseteq D(E)$  is classified by a natural transformation $ \iota \colon g \implies f $.  
		Since colimits commute with each other, there is a cofiber sequence $ \colim g \xrightarrow{\iota} \colim f  \to \mathrm{Th}(E) \simeq \colim (\cofib (g \to f)) $ where $ \mathrm{Th}(E) $ is the $ C_2 $-equivariant Thom space. 
		Now let $ X $ be a $C_2$-spectrum with $ S^\sigma $-action classified by a functor $ BS^\sigma \xrightarrow{X} \Spectra^{C_2} $. 
		Replacing $ g $ by $ BS^\sigma \xrightarrow{g, X} \Spc^{C_2} \times \Spectra^{C_2} \xrightarrow{\otimes} \Spectra^{C_2} $ and applying an equivariant Thom isomorphism \cite[Proposition 2.3]{MR1139971}, we obtain an exact sequence $ X \to X_{hS^\sigma} \to \left(\Sigma^\rho X\right)_{hS^\sigma} $. 
		Taking $ X = \HR(A/\underline{k})$, this induces an exact sequence of $ \underline{k} $-modules
		\begin{equation*}
		\begin{tikzcd}
				\HR(A/\underline{k}) \ar[r]& \HR(A/\underline{k})_{hS^\sigma} \ar[r]& \left(\Sigma^{\rho} \HR(A/\underline{k})\right)_{hS^\sigma} \,. 
		\end{tikzcd}
		\end{equation*}
		Since $ \HCR(A/\underline{k}) $ is Borel, we have an equivalence of exact sequences
		\begin{equation*}
		\begin{tikzcd}[row sep=tiny]
				\HR(A/\underline{k})^{C_2} \ar[r]\ar[d,"\sim"] &\HR(A/\underline{k})_{hS^\sigma}^{C_2} \ar[d,"\sim"] \ar[r]& \Sigma^{\rho} \HR(A/\underline{k})_{hS^\sigma}^{C_2} \ar[d,"\sim"] \\
				\HH(A^e/\underline{k})^{hC_2} \ar[r]& \HH(A^e/\underline{k})_{hS^1}^{hC_2} \ar[r]& \Sigma^{\rho} \HH(A^e/\underline{k})_{hS^1}^{hC_2} \,. 
		\end{tikzcd}
		\end{equation*}
		Naturality of the splitting in Lemma \ref{lemma:char_neq_2_C2_action_splits} implies that the preceding sequence may be rewritten as
		\begin{equation}\label{eq:SBI_exact_seq}
		\begin{tikzcd}
				\HH(A^e/\underline{k})^+ \ar[r]& \HH(A^e/\underline{k})_{hS^1}^{+} \ar[r]& \left(\Sigma^{\rho} \HH(A^e/\underline{k})_{hS^1}\right)^{+} \,. 
		\end{tikzcd}
		\end{equation}
		Now observe that if $ Y $ is a spectrum with $ C_2 $-action and $ y \in \pi_n(Y) $ ($ z \in \pi_n(Y) $) so that $ \sigma(y) = y $ ($\sigma(z) = -z $), then $ \id \otimes y \in \pi_0 \hom\left(S^{n+\rho},S^\rho \otimes Y\right) \simeq \pi_{n+2}(S^\rho \otimes Y) $ satisfies $ \sigma(\id \otimes y) = -\id \otimes y $ ($ \sigma(\id \otimes z) = \id \otimes z $). 
		In particular, there are isomorphisms $ \pi_*(S^\rho \otimes Y)^\pm \simeq \pi_{*-2}Y^{\mp} $. 
		The result follows from applying induction on $ * $ and the five lemma to compare the homotopy long exact sequence of (\ref{eq:SBI_exact_seq}) with the SBI sequence of \cite{MR917736}. 
\end{proof}

\appendix

\section{Parametrized module categories}\label{appendix:param_module_cats}
\begin{ntn}
		Let $ \underline{\Com}^\otimes_{\mathcal{O}_G^\simeq} \to \underline{\Fin}_{G,*} $ denote the minimal $ G $-$ \infty $-operad of \cite[Example 2.4.7]{NS22}. 
		Write $ p \colon \Com^\otimes_{\mathcal{O}_G^\simeq} \to \Span(\Fin_G) $ for the corresponding fibrous pattern under the equivalence of \cite[Proposition 5.2.14]{BHS2022envelopes}. 
		Endow $ \Com^\otimes_{\mathcal{O}_G^\simeq} $ with the structure of an algebraic pattern where the elementary objects are $ G/H $ and inert morphisms are those whose image under the structure map $ p $ are inert \cite[Definition 4.1.11]{BHS2022envelopes}. 

		There is a canonical orbit functor $ (-)_G \colon \underline{\Com}^\otimes_{\mathcal{O}_G^\simeq} \xrightarrow{s} \Span(\Fin_{G,*}, all, \nabla) \xrightarrow{(-)_G} \Fin_* $. 
\end{ntn}
\begin{recollection}
		[{\cite[Definition 3.1.12]{BHS2022envelopes}}] 
		Let $ \mathcal{P} $, $ \mathcal{O} $ be algebraic patterns. 
		A \emph{morphism of algebraic patterns} is a functor $ f \colon \mathcal{O} \to \mathcal{P} $ which preserves elementary objects, inert morphisms, and active morphisms. 
		Given such an $ f $, for each $ X \in \mathcal{O} $, there is a functor
		\begin{equation}
				f^{\mathrm{el}}_{X/} \colon \mathcal{O}^{\mathrm{el}}_{X/} \to \mathcal{P}^{\mathrm{el}}_{f(X)/} \,.
		\end{equation}
		If the functor $ f^{\mathrm{el}}_{X/} $ is coinitial for all $ X \in \mathcal{O} $, we say $ f $ is \emph{strong Segal}. 
		If the functor $ f^{\mathrm{el}}_{X/} $ is an equivalence for all $ X \in \mathcal{O} $, we say $ f $ is \emph{iso-Segal}. 
\end{recollection}
\begin{lemma}\label{lemma:orbits_is_strong_Segal}
		Let $ G $ be a finite group. 
		Then the orbit functor $ (-)_G \colon \underline{\Com}^\otimes_{\mathcal{O}_G^\simeq} \to \Fin_* $ is \emph{strong Segal}. 
\end{lemma}
\begin{proof}
	Unravel the definitions! Or combine Observation 4.1.14 and with an argument as in Example 5.2.17 of \cite{BHS2022envelopes}. 
\end{proof}
\begin{defn}\label{defn:operad_induces_G_operad}
		There is a canonical functor $ \underline{(-)} \colon \Op_\infty \to \Op_{G,\infty} $, defined as a composite
		\begin{equation}
		\begin{tikzcd}
				\Op_\infty \simeq \mathrm{Fbrs}\left(\Fin_*\right) \ar[r, "{(-)_G^*}"] & \mathrm{Fbrs}\left(\underline{\Com}^\otimes_{\mathcal{O}_G^\simeq}\right) \ar[r,"{p \circ}"] & \mathrm{Fbrs}\left(\Span\left(\Fin_G\right) \right) \simeq \Op_{G,\infty}
		\end{tikzcd}
		\end{equation} 
		where pullback along the functor $ (-)_G $ of Lemma \ref{lemma:orbits_is_strong_Segal} restricts to the appropriate functor by \cite[Lemma 4.1.19 \& Corollary B]{BHS2022envelopes}. 
		Then by Corollary 4.1.17 of \emph{loc. cit.} postcomposition with $ p $ takes fibrous $ \underline{\Com}^\otimes_{\mathcal{O}_G^\simeq} $-patterns to fibrous $ \Span\left(\Fin_G\right) $-patterns. 
		Given an ordinary $ \infty $-operad $ \mathcal{O}^\otimes $, we will refer to its image under this functor as the \emph{constant $ G $-$ \infty $-operad at $ \mathcal{O} $} and denote it by $ \underline{O}^\otimes $.
\end{defn}
\begin{rmk}\label{rmk:const_G_operad_as_product}
		If $ \mathcal{O}^\otimes $ is an ordinary $ \infty $-operad, there is a canonical functor $ \mathcal{O}^\otimes \times \mathcal{T}^\op \to \underline{\mathcal{O}}^\otimes $. 	
\end{rmk}	
\begin{ntn} [{\cite[Remark 4.1.1.4, Definition 4.2.1.7, Definition 4.3.1.6]{LurHA}}]
		Write $ \underline{\Assoc}^\otimes $, $ \underline{\mathcal{L}M}^\otimes $, $ \underline{\mathcal{R}M}^\otimes $,  $ \underline{\mathcal{B}M}^\otimes $ denote the constant $ G $-$ \infty $-operads associated to the ordinary $ \infty $-operads $ \Assoc^\otimes $, $ \mathcal{L}M^\otimes $, $ \mathcal{R}M^\otimes $,  $ \mathcal{B}M^\otimes $ under the functor of Definition \ref{defn:operad_induces_G_operad}. 
\end{ntn}
\begin{rmk}
		Under the equivalence of \cite[Corollary 5.2.13]{BHS2022envelopes}, the $ G $-$ \infty $-operad $ \underline{\Assoc}^\otimes $ has the same objects as $ \Com^\otimes_{\mathcal{O}_G^\simeq} $. 
		The `forwards' morphisms are those spans 
		\begin{equation*}
		\begin{tikzcd}
			U \ar[d] & \ar[l] Z \ar[r,"f"]  \ar[d] & X \ar[d] \\
			V &  \ar[l] Y \ar[r,equals] & Y
		\end{tikzcd}
		\end{equation*}
		where the induced map $ Z \to U \times_V Y $ is a summand inclusion, plus the data of a pair $ (f, \leq) $ where $ f $ is equivalent to a fold map and for all $ x \in X $, an ordering on the orbits in the preimage $ f^{-1}(\{x\}) \subseteq Z $. 

		Notice that for subgroups $ K \leq H \leq G $, and some fold map $ g \colon G/H^{\sqcup n} \to G/H $, the set of orbits of the domain $ G/K^{\sqcup n} $ of $ g \times_{G/H} G/K \colon  G/K^{\sqcup n} \to G/K $ is canonically identified with the set of orbits of $ G/H^{\sqcup n} $. 
		This, along with the lexicographic ordering of \cite[Definition 4.1.1.4]{LurHA}, specifies composition on $ \underline{\Assoc}^\otimes $. 

		The canonical forgetful functor/operad structure map $ \underline{\Assoc}^\otimes \to \underline{\Fin}_{G,*} $ forgets the ordering. 
\end{rmk}

Using these notions, we can formulate what it means for a $ \underline{\E}_1 $-monoidal $ G $-$ \infty $-category to act on another $ G $-$ \infty $-category. 
\begin{defn}
	[{\cite[Definition 4.2.1.13]{LurHA}}] 
	Let $ q \colon \cat^\otimes \to \underline{\Assoc}^\otimes $ be a fibration of $ G $-$ \infty $-operads and let $ \mathcal{M} $ be a $ G $-$ \infty $-category. 
	A \emph{weak enrichment of $ \mathcal{M} $ over $ \cat^\otimes $} is a fibration of $ G $-$ \infty $-operads $ q \colon \mathcal{O}^\otimes \to \underline{\mathcal{L}M}^\otimes $ together with equivalences $ \mathcal{O}^\otimes_{\underline{a}} \simeq \cat^\otimes $ and $ \mathcal{O}^\otimes_{\underline{m}} \simeq \mathcal{M} $. 

	Let $ q \colon \mathcal{O}^\otimes \to \underline{\mathcal{L}M}^\otimes $ exhibit $ \mathcal{M} $ as weakly enriched over $ \cat^\otimes $. 
	Let $ \LMod(\mathcal{M}) $ denote the $ G $-$ \infty $-category $ \Alg_{/\underline{\mathcal{L}M}}(\mathcal{O}) $, and refer to it as the $ G $-$ \infty $-category of left module objects of $ \mathcal{M} $. 
	Composition with the inclusion $ \underline{\Assoc}^\otimes \to\underline{\mathcal{L}M}^\otimes $ gives a categorical fibration
	\begin{equation*}
		\LMod(\mathcal{M}) = \Alg_{/\underline{\mathcal{L}M}}(\mathcal{O}) \to \Alg_{/\underline{\Assoc}}(\mathcal{O}) = \Alg(\cat)\,.
	\end{equation*}
	If $ A $ is an $ \underline{\Assoc} $-algebra object of $ \cat $ (i.e. a global section of the $ G $-$ \infty $-category $ \Alg(\cat) $, denote the parametrized fiber (as a $ G $-$ \infty $-category) by $ \LMod_A(\mathcal{M}) $. 
\end{defn}
\begin{defn}\label{defn:param_cat_tensored} 
		Let $ q \colon \cat^\otimes \to \underline{\mathcal{L}M}^\otimes $ be a fibration of $ C_2 $-$ \infty $-operads. 
		We say that $ q $ \emph{exhibits $ \cat_{\underline{m}^{C_2/C_2}}=:\mathcal{M} $ as left-tensored over $ \cat_{\underline{a}^{C_2/C_2}} $} if $ q $ is a coCartesian fibration of $ C_2 $-$ \infty $-operads. 

		Given $ q \colon \cat^\otimes \to \underline{\mathcal{R}M}^\otimes $, there is a similar notion of a $ C_2 $-$ \infty $-category being \emph{right-tensored} over an $ \underline{\E}_1 $-monoidal $ C_2 $-$ \infty $-category. 
\end{defn}
\begin{defn}\label{defn:param_linear_functors} 
		Let $ q \colon \mathcal{M}^\otimes \to \underline{\mathcal{L}M}^\otimes $, $ p \colon \mathcal{N}^\otimes \to \underline{\mathcal{L}M}^\otimes $ be cocartesian fibrations of $ C_2 $-$ \infty $-operads so that $ \displaystyle \mathcal{M}^\otimes \times_{\underline{\mathcal{L}M}^\otimes} \underline{\Assoc}^\otimes \simeq \mathcal{N}^\otimes \times_{\underline{\mathcal{L}M}^\otimes} \underline{\Assoc}^\otimes =: \cat^\otimes $. 
		A \emph{$ \cat $-linear functor from $ \mathcal{M}^\otimes_{\mathfrak{m^{C_2/C_2}}} $ to $ \mathcal{N}^\otimes_{\mathfrak{m^{C_2/C_2}}} $} is a $ \underline{\mathcal{L}M}^\otimes $-monoidal functor from $ \mathcal{M}^\otimes $ to $ \mathcal{N}^\otimes $ which is the identity on $ \cat^\otimes $. 
\end{defn}
\begin{rmk}
		We write right- (resp. left-)tensored instead of `$ C_2 $ right- (resp. left-)tensored' because we find the proliferation of $ C_2 $'s appearing in modifiers to be a bit unwieldy. 
		We trust that one can ascertain whether we refer to Definition \ref{defn:param_cat_tensored} or \cite[Definition 4.2.1.19]{LurHA} based on the context. 
\end{rmk}
The following proposition is inspired by \cite[Corollary 2.4.15]{NS22}.  
\begin{prop}\label{prop:param_cats_monoidal_const_operads}
		For any $ \infty $-operad $ \mathcal{O}^\otimes $, pullback along the functor of Remark \ref{rmk:const_G_operad_as_product} induces a canonical identification of $ \underline{\mathcal{O}}^\otimes $-monoidal $ G $-$ \infty $-categories and $ \mathcal{O}^\op_G $-cocartesian families of $ \mathcal{O}^\otimes $-monoidal $ \infty $-categories (\cite[Definition 4.8.3.1]{LurHA}). 
\end{prop}
\begin{cor}\label{cor:param_comodules_tensored} 
	Let $ \cat $ be a $ G $-$ \infty $-category. 
	Let $ A $ be an $ \underline{\E}_1 $-coalgebra in $ \cat $, i.e. an $ \underline{\E}_1 $-algebra in $ \cat^\vop $. 
	Then left comodules over $ A $ is \emph{right-tensored} over $ \cat $. 
\end{cor}
\begin{proof} 
		Follows from the results of \cite[\S4.3.2]{LurHA} and Proposition \ref{prop:param_cats_monoidal_const_operads}. 
\end{proof}
\begin{cor}
		Let $ \cat^\otimes $ be an $ \underline{\Assoc}^\otimes $-monoidal $ C_2 $-$ \infty $-category. 
		Under the equivalence of Proposition \ref{prop:param_cats_monoidal_const_operads}, a $ C_2 $-$\infty$-category $ \mathcal{M} $ is left-tensored over $ \cat $ if and only if $ \mathcal{M} $ defines a $ \mathcal{O}_G^\op $-cocartesian family of $ \infty  $-categories left-tensored over $ \cat $ in the sense of \cite[Definition 4.8.3.9]{LurHA}. 
\end{cor}

\begin{cor}\label{cor:mod_to_alg_cart_fibration}
		Let $ \cat $ be a $ \underline{\E}_1 $-monoidal $ \mathcal{T} $-$ \infty $-category and let $ \mathcal{M} $ be a $ \mathcal{T} $-$ \infty $-category which is left-tensored over $ \cat $. 
		Then the canonical forgetful functor $ \LMod(\mathcal{M}) = \Alg_{/\underline{\mathcal{L}M}}(\mathcal{O}) \to \Alg_{/\underline{\Assoc}}(\mathcal{O}) = \Alg(\cat) $ is a $ \mathcal{T} $-cartesian fibration. 
\end{cor}
\begin{proof}
		Follows from Proposition \ref{prop:param_cats_monoidal_const_operads} and \cite[Corollary 4.2.3.2]{LurHA}. 
\end{proof}
\begin{cor}\label{cor:param_mod_fiberwise_description}
		Let $ \cat $ be an $ \underline{\E}_1 $-monoidal $ \mathcal{T} $-$ \infty $-category and let $ A $ be a $ \mathcal{T} $-object of $ \underline{\E}_1\Alg(\cat) $. 
		Let $ \mathcal{M} $ be a $ \mathcal{T}$-$ \infty $-category which is left-tensored over $\cat$ and consider the parametrized module category $ \underline{\LMod}_A(\mathcal{M}) $. 
		Then its (non-parametrized) fibers satisfy
		\begin{equation*}
				\underline{\LMod}_A(\mathcal{M})_t \simeq \LMod_{A_t}(\mathcal{M}_t)
		\end{equation*}
		for all $ t \in \mathcal{T} $, where the right-hand side denotes the non-parametrized left module category of, for instance \cite[\S4.2.1]{LurHA}. 
\end{cor}
\begin{proof}
		By definition of parametrized mapping spaces (see \cite[\S3]{Shah18}), under the equivalence of Proposition \ref{prop:param_cats_monoidal_const_operads}, the fiber $ \underline{\LMod}_A(\mathcal{M})_t $ is given by $ \left(\mathcal{T}^{/t}\right)^\op $-cocartesian families of modules over $ A_{\underline{t}} \colon \left(\mathcal{T}^{/t}\right)^\op \to \cat_{\underline{t}} $ in $ \mathcal{M}_{\underline{t}} $. 
		The result follows from noting that $ \left(\mathcal{T}^{/t}\right)^\op $ has an initial object. 
\end{proof}
\begin{ex}\label{ex:param_endomorphisms_cats}
		Let $ \cat $ be a $ C_2 $-$ \infty $-category. 
		By Proposition \ref{prop:param_cats_monoidal_const_operads}, $ \underline{\Fun}(\cat, \cat) $ is a $ \underline{\E}_1 $-monoidal $ C_2 $-$ \infty $-category. 
		Furthermore, $ \cat $ is left-tensored over $ \underline{\Fun}(\cat, \cat) $. 
		On underlying $ \infty $-categories, this recovers the action of $ \Fun(\cat^e, \cat^e) $ on $ \cat^e $ described in \cite[Introduction to \S4.7]{LurHA}. 
\end{ex}
\begin{prop}\label{prop:param_modules_limits}
		Let $ \cat $ be an $\underline{\E}_1$-monoidal $ C_2 $-$ \infty $-category, $ \mathcal{M} $ a $ C_2 $-$ \infty $-category which is left-tensored over $ \cat $, and let $ A $ be a $ \mathcal{T} $-algebra object of $ \cat $. 
		Suppose $ \mathcal{M} $ is $ C_2 $-complete. 
		Then
		\begin{enumerate}[label=(\alph*)]
			 	\item \label{prop_item:param_modules_admit_lims} the $ C_2 $-$ \infty $-category $ \underline{\LMod}_A(\mathcal{M}) $ is $C_2$-complete  
			 	\item a map $ p \colon \underline{\LMod}_A(\mathcal{M}) $ is a $ \mathcal{T} $-limit diagram if and only if the induced map $ K^{\underline{\triangleleft}} \to \mathcal{M} $ is a $ \mathcal{T} $-limit diagram. 
			 	\item \label{prop_item:param_mod_lim_change_of_alg_forget_is_right_adjoint} given a morphism of algebra objects $ A\to B $ of $ \mathcal{C} $, the induced functor $ \underline{\LMod}_B(\cat) \to \underline{\LMod}_A(\cat) $ admits a left $ \mathcal{T} $-adjoint.  
	  \end{enumerate} 
\end{prop}
\begin{proof}
\begin{enumerate}[label=(\alph*)]
		\item By (the dual to) \cite[Corollary 12.15]{Shah18}, it suffices to show that $ \underline{\LMod}_A(\mathcal{M}) $ admits totalizations and all $ C_2 $-products. 
		It follows from \cite[Proposition 4.2.3.3(1)]{LurHA} that the category admits totalizations. 
		By the dual to \cite[Proposition 5.11]{Shah18}, we must show that the restriction functor $ \LMod_A(\mathcal{M}^{C_2}) \to \LMod_{A^e}(\mathcal{M}^e) $ admits a right adjoint. 
		Since the restriction functor $ \mathcal{M}^{C_2} \to \mathcal{M}^e $ is monoidal, its right adjoint $ G $ is lax monoidal. 
		In particular, the unit promotes to a map of $ \E_1 $-algebra objects $ \eta \colon A \to G (A^e) $. 
		Now consider the composite $ \LMod_{A^e}(\mathcal{M}^e) \xrightarrow{G} \LMod_{G(A^e)}(\mathcal{M}^{C_2}) \xrightarrow{\eta^*} \LMod_{A}(\mathcal{M}^{C_2}) $. 
		It is right adjoint to the restriction functor $ \LMod_{A}(\mathcal{M}^{C_2}) \to \LMod_{A^e}(\mathcal{M}^{e}) $. 
		\item Follows from \cite[Corollary 4.2.3.3(2)]{LurHA} and the proof of part \ref{prop_item:param_modules_admit_lims}.
		\item Follows from Corollary \ref{cor:C2_left_adjoint_local_crit} and \cite[Corollary 4.2.3.3(3)]{LurHA}. \qedhere
\end{enumerate}		
\end{proof}

\begin{prop}\label{prop:param_modules_colimits}
		Let $ \cat $ be an $\underline{\E}_1$-monoidal $ C_2 $-$ \infty $-category, $ \mathcal{M} $ a $ C_2 $-$ \infty $-category which is left-tensored over $ \cat $, and let $ A $ be a $ \mathcal{T} $-algebra object of $ \cat $. 
		Suppose $ \mathcal{M} $ is $ C_2 $-cocomplete, and that for each $ \alpha \colon s \to t $ in $ \mathcal{T} $ the left adjoint to the restriction map $ \mathcal{M}_t \to \mathcal{M}_s $ is monoidal. 
		Then
		\begin{enumerate}[label=(\alph*)]
			 	\item \label{prop_item:param_modules_admit_colims} the $ C_2 $-$ \infty $-category $ \underline{\LMod}_A(\mathcal{M}) $ is $C_2$-cocomplete  
			 	\item a map $ p \colon \underline{\LMod}_A(\mathcal{M}) $ is a $ \mathcal{T} $-colimit diagram if and only if the induced map $ K^{\underline{\triangleright}} \to \mathcal{M} $ is a $ \mathcal{T} $-colimit diagram. 
			 	\item \label{prop_item:param_mod_colim_change_of_alg_tensor_is_left_adjoint} given a morphism of algebra objects $ A\to B $ of $ \mathcal{C} $, the induced functor $ \underline{\LMod}_B(\cat) \to \underline{\LMod}_A(\cat) $ is a left $ \mathcal{T} $-adjoint.  
	  \end{enumerate} 
\end{prop}
\begin{rmk}
		Note that Proposition \ref{prop:param_modules_colimits} has a stronger assumption than Proposition \ref{prop:param_modules_limits} because given $ \alpha \colon s \to t $ in $ \mathcal{T} $ so that $ \alpha^* $ is monoidal, the right adjoint of $ \alpha^* $ is automatically lax monoidal (and hence takes algebra objects to algebra objects), while the left adjoint to $ \alpha^* $ is \emph{oplax} monoidal, hence does not a priori preserve algebra objects. 
\end{rmk}
\begin{proof}
		[Proof of Proposition \ref{prop:param_modules_colimits}] 
		We prove (a); items (b) and (c) follow almost immediately from (a) as in the proof of Proposition \ref{prop:param_modules_limits}. 

		By \cite[Corollary 12.15]{Shah18}, it suffices to show that $ \underline{\LMod}_A(\mathcal{M}) $ admits geometric realizations and all $ C_2 $-coproducts. 
		It follows from \cite[Proposition 4.2.3.5(1)]{LurHA} that the category admits geometric realizations. 
		By \cite[Proposition 5.11]{Shah18}, we must show that the restriction functor $ \LMod_A(\mathcal{M}^{C_2}) \to \LMod_{A^e}(\mathcal{M}^e) $ admits a left adjoint. 
		By assumption, the left adjoint $ L\colon \mathcal{M}^e \to \mathcal{M}^{C_2} $ to the restriction functor $ \mathcal{M}^{C_2} \to \mathcal{M}^e $ is monoidal. 
		In particular, the counit promotes to a map of $ \E_1 $-algebra objects $ \varepsilon \colon L(A^e) \to A $. 
		Now consider the composite $ \LMod_{A^e}(\mathcal{M}^e) \xrightarrow{L} \LMod_{L(A^e)}(\mathcal{M}^{C_2}) \xrightarrow{\varepsilon_* = \left(- \otimes_{L(A^e)} A\right)} \LMod_{A}(\mathcal{M}^{C_2}) $. 
		It is left adjoint to the restriction functor $ \LMod_{A}(\mathcal{M}^{C_2}) \to \LMod_{A^e}(\mathcal{M}^{e}) $. 
\end{proof}

\addcontentsline{toc}{section}{\protect\numberline{\thesection}References}
\renewcommand*{\bibfont}{\small}

\printbibliography

\end{document}